\documentclass[10pt]{article}

\usepackage{hyperref}
\usepackage{amsmath}
\usepackage{epsfig, amssymb, amsfonts, dsfont}
\usepackage{amsthm}
\usepackage{amstext}
\usepackage{amsopn}
\usepackage{mathrsfs}
\usepackage{subfigure}
\usepackage{graphicx}
\usepackage[left=2.3cm,top=2cm,right=2.3cm,bottom=2cm, nohead,foot=1cm]{geometry}
\usepackage[makeroom]{cancel}
\usepackage{color}
\usepackage{enumerate}
\usepackage{comment}
\usepackage{stackrel}
\usepackage{stmaryrd}
\usepackage{mathtools}
\usepackage{amsbsy}

\def\stackrel#1#2{\mathrel{\mathop{#2}\limits^{#1}}}

\numberwithin{equation}{section}

\newtheorem{theorem}{Theorem}[section]

\newtheorem{lemma}[theorem]{Lemma}

\newtheorem{corollary}[theorem]{Corollary}
\newtheorem{proposition}[theorem]{Proposition}

\theoremstyle{definition}
\newtheorem{definition}[theorem]{Definition}
\newtheorem{remark}[theorem]{Remark}

\newcommand{\R}{{\mathbb R}}

\DeclareFontFamily{U}{mathx}{\hyphenchar\font45}
\DeclareFontShape{U}{mathx}{m}{n}{
      <5> <6> <7> <8> <9> <10>
      <10.95> <12> <14.4> <17.28> <20.74> <24.88>
      mathx10
      }{}
\DeclareSymbolFont{mathx}{U}{mathx}{m}{n}
\DeclareMathSymbol{\bigtimes}{1}{mathx}{"91}

\usepackage{relsize}

\newcommand\rotUp{\rotatebox[origin=c]{180}{$\Upsilon$}}

\newcommand{\vsm}{{\mathsmaller{\mathsmaller{V}}}}

\newcommand{\Z}{{\mathbb Z}}

\newcommand{\C}{{\mathbb C}}

\newcommand{\Da}{{D}}
\newcommand{\Rphi}{{\Upsilon}}
\newcommand{\Rvarphi}{{\rotUp}}

\title{On planar Brownian motion singularly tilted \\ through a point potential}

\date{  }

 \author{\textbf{Jeremy  Clark}\footnote{ {\tt jeremy@olemiss.edu}}\hspace{.8cm}and\hspace{.8cm}\textbf{Barkat Mian}\footnote{ {\tt bmian@go.olemiss.edu}} \vspace{.1cm}  \\  University of Mississippi, Department of Mathematics   }

\begin{document}
\maketitle

\begin{abstract} We discuss a family of time-inhomogeneous two-dimensional diffusions, defined over a finite time interval $[0,T]$,   having  transition density functions that are expressible in terms of the integral kernels for negative exponentials of the two-dimensional Schr\"odinger operator with a point potential at the origin.   These diffusions have a singular drift pointing in the direction of the origin that is strong enough to enable the possibly of visiting there, in contrast to a two-dimensional Brownian motion.   Our main focus is on  characterizing a local time process at the origin analogous to that for a one-dimensional Brownian motion and on studying the law of its process inverse.
\end{abstract}\vspace{.2cm}

\section{Introduction} 
A two-dimensional Schr\"odinger Hamiltonian with point (zero-range)   potential   at the origin is a self-adjoint operator on  $L^2(\R^2)$  that can be heuristically expressed  as
\begin{align}\label{Hamiltonian}
\mathbf{H}\,= \,-\frac{1}{2}\,\Delta\,-\,\vsm\,\delta(x)\,,
\end{align}
in which $\Delta=\frac{\partial^2}{\partial x_{1}^2}+\frac{\partial^2}{\partial x_{2}^2}$ is the two-dimensional  Laplacian, $\delta(x)$ is the  Dirac delta function on $\R^2$, and $\mathsmaller{V}>0$ is an infinitesimal coupling constant. The monograph~\cite{AGHH} by Albeverio et al.\  surveys  Schr\"odinger operators with potentials  supported on discrete sets, wherein Chapter I.5 develops theory for~(\ref{Hamiltonian}).    One mathematically rigorous method for defining  Schr\"odinger operators with zero-range potentials  at the origin is through self-adjoint extensions of the symmetric operator given by $-\frac{1}{2}\Delta$ with domain  comprised of smooth functions supported on compact subsets of $\R^2\backslash\{0\}$.  For each  $\lambda\in (0,\infty)$, there exists a unique  self-adjoint operator $-\frac{1}{2}\Delta^{\lambda}$ on  $L^2(\R^2)$ that can be roughly understood  as $-\frac{1}{2}\Delta$ over the set $\R^2\backslash \{0\}$ and with the following asymptotic  condition  at the origin for functions $\psi:\R^2 \rightarrow \C$ in its domain:  as $x\in \R^2$ approaches $0$,
\begin{align}\label{Origin}
\psi(x)\,\stackrel{x\rightarrow 0}{=}\, \mathbf{c}_{\psi} \,\left( \log |x| \,+\, \frac{1}{2}\,\log \frac{\lambda }{2} \,+\,\gamma_{\mathsmaller{\textup{EM}}}\right) + \mathit{o}(1) \,, \hspace{.3cm}\text{where} \,\hspace{.1cm}\mathbf{c}_{\psi}\,:=\,\lim_{x\rightarrow 0} \frac{ \psi(x) }{\log |x|  }\,,
\end{align}
and $\gamma_{\mathsmaller{\textup{EM}}}:=-\int_0^{\infty}e^{-x} \log x\,dx $ is the Euler-Mascheroni constant;  see~\cite[Thm.\ 5.3]{AGHH} for a  precise description of the domain of $-\frac{1}{2}\Delta^{\lambda}$, in which $\lambda$ is related to their  parameter $\alpha\in \R$  through  $-2\pi \alpha=\frac{1}{2}\log\frac{\lambda}{2}+\gamma_{\mathsmaller{\textup{EM}}}  $.  The operator  $-\frac{1}{2}\Delta^{\lambda}$ can be obtained as a limit  in the norm resolvent sense of regular Schr\"odinger operators $-\frac{1}{2}\Delta-\mathsmaller{V}^{\lambda}_{\varepsilon}\delta_{\varepsilon}(x)$ by taking $\delta_{\varepsilon}(x):=\frac{1}{\varepsilon^2} D\big(\frac{x}{\varepsilon} \big) $ for $\varepsilon>0$ and a function $D\in C_c(\R^2)$ with integral one and  tuning the coupling constant  $\mathsmaller{V}^{\lambda}_{\varepsilon}$ to vanish with the small $\varepsilon$ asymptotics 
\begin{align}\label{VSM}
\vsm^{\lambda}_{\varepsilon}\,\stackrel{\varepsilon\rightarrow 0}{=}\,\frac{ \pi  }{ \log \frac{1}{\varepsilon} }\Bigg(1+ \frac{ \frac{1}{2}\log \frac{\lambda}{2}+\gamma_{\mathsmaller{\textup{EM}}} +I_{\Da} }{ \log \frac{1}{\varepsilon} }\Bigg) +\mathit{o}\bigg(\frac{ 1}{ \log^2 \frac{1}{\varepsilon} } \bigg) \hspace{.3cm}\text{for}\,\hspace{.1cm}I_{\Da}\,:=\,\int_{(\R^2)^2}\log |x-y| \,\Da(x)\,\Da(y)\,dx\,dy\,.
\end{align}
 This convergence result can be found in~\cite[Thm.\ 5.5]{AGHH}, under  more flexible conditions on $\Da(x)$.

 The self-adjoint operator $-\frac{1}{2}\Delta^{\lambda}$ has a single eigenvalue $-\lambda$, which  is simple and has corresponding normalized eigenfunction $
\psi^{\lambda}(x):= (\frac{2\lambda}{\pi})^{1/2} K_0\big(\sqrt{2\lambda} |x|  \big) $, with $K_0$ denoting here the order-$0$ modified Bessel function of the second kind.  The interval $[0,\infty)$ comprises the remainder of the spectrum of $-\frac{1}{2}\Delta^{\lambda}$, which is absolutely continuous, and each $k\in \R^2\backslash \{0\}$ is associated with a generalized (unnormalizable) eigenfunction  $\psi^{\lambda}_k$ with value $\frac{1}{2}|k|^2$ and 
having the form of a perturbed plane wave:
\begin{align}\label{PsiLambda}
\psi^{\lambda}_k(x)\,=\, e^{ik\cdot x}\,+\,\frac{ i\pi }{\log \frac{|k|^2}{2\lambda}  -i\pi } \,H_0^{(1)}\big(|k|\,|x|\big) \,, 
\end{align}
where $H^{(1)}_0$ is the order-$0$ Hankel function of the first kind. In other terms, $\psi^{\lambda}_k$ satisfies $ \Delta\psi^{\lambda}_k(x)=-|k|^2\psi^{\lambda}_k(x)$ for $x\neq 0$ and obeys the asymptotic condition~(\ref{Origin}) at the origin.  The operators in the semigroup $\{e^{\frac{t}{2}\Delta^{\lambda}}\}_{t\in [0,\infty)}$ are bounded, self-adjoint, and have positive-valued integral kernels $f_{t}^{\lambda}(x,y)$, for which closed expressions involving double integrals were discussed by Albeverio, Brze\'zniak, and D\c{a}browski in~\cite[Sec.\ 3.2]{Albeverio}. If for a bounded Borel measurable function $\varphi:\R^2\rightarrow \R$ we define $f_{t }^{\lambda}(x,\varphi):=\int_{\R^2} f_{t}^{\lambda}(x,y)\varphi(y)dy$, then a formal application of the Feynman-Kac formula to the heuristic expression~(\ref{Hamiltonian})  yields the nonsensical result  
\begin{align}\label{DegenTilt}
f_{t }^{\lambda}(x,\varphi)\,=\,\mathbf{E}_x\Bigg[ \,\textup{exp}\Bigg\{\, \vsm\int_0^t \,\delta(\mathbf{W}_r)\,dr\Bigg\} \,\varphi(\mathbf{W}_t)\,  \Bigg]\,, 
\end{align}
in which  $\{\mathbf{W}_t\}_{t\in [0,\infty)}$ is a two-dimensional Brownian motion starting from $x\in \R^2$. The meaning of this expectation is particularly  unclear given that a two-dimensional Brownian motion will almost surely never visit the origin after time zero, however, the limiting regime involving  $\mathsmaller{V}^{\lambda}_{\varepsilon}$ and $\delta_{\varepsilon}$ can lend some significance to it.  For the measure kernel $f_{t }^{\lambda,\varepsilon}(x,\cdot)$ from $\R^2$ to $\R^2$ defined through the Feynman-Kac formula
\begin{align}\label{Chen}
f_{t }^{\lambda,\varepsilon}(x,\varphi)\,:=\,\mathbf{E}_x\Bigg[ \,\textup{exp}\Bigg\{ \,\vsm^{\lambda}_{\varepsilon}\int_0^t \,\delta_{\varepsilon}(\mathbf{W}_r)\,dr\,\Bigg\}\, \varphi(\mathbf{W}_t)   \,\Bigg]\,,
\end{align}
a recent article~\cite{Chen} by Chen  applies probabilistic techniques, motivated in part by the work of Kashara and Kotani~\cite{Kashara3}, to prove the convergence
\begin{align}\label{ConvKer}
f_{t }^{\lambda,\varepsilon}(x,\varphi)\hspace{.5cm}\stackrel{\varepsilon\rightarrow 0}{\longrightarrow} \hspace{.5cm}  f_{t }^{\lambda}(x,\varphi)
\end{align}
for every
  $x\in \R^2\backslash \{0\}$ and bounded Borel measurable  function $\varphi$ on $\R^2$.

In this article, we study the  time-inhomogeneous two-dimensional diffusion defined over a given finite time interval $[0,T]$ and whose transition densities for times $0\leq s <t\leq T$ have the form
\begin{align}\label{FirstTrans}
\mathlarger{d}_{s,t}^{T,\lambda}(x,y)\,:=\,f_{t-s}^{\lambda}(x,y)\,\frac{ F_{T-t}^{\lambda}(y) }{F_{T-s}^{\lambda}(x)  }\,,\hspace{.5cm}\text{in which} \hspace{.1cm}\,F_t^{\lambda}(x)\,:=\,\int_{\R^2}\, f_{t}^{\lambda}(x,z) \,dz\,.  
\end{align}
In the same spirit as (\ref{DegenTilt}), the law of this process can be understood as a singular ``exponential tilting" of two-dimensional Wiener measure by the ``random variable" $ \mathsmaller{V}\int_0^T \delta(\mathbf{W}_r)dr $.  To be more precise, let  $\mathbf{P}_{\mu}^T$ denote the two-dimensional Wiener measure on $C([0,T],\R^2)$ corresponding to an initial distribution $\mu $ on $\R^2$, and let $\mathbf{P}_{\mu}^{T,\lambda}$ be its counterpart with transition densities $\mathlarger{d}_{s,t}^{T,\lambda}(x,y)$. Then, in consequence of the convergence~(\ref{DegenTilt}) from \cite[Thm.\ 1]{Chen}, the exponentially-tilted Wiener measure $$  \widetilde{\mathbf{P}}_{\mu}^{T,\lambda,\varepsilon}\,:=\, \frac{  Y^{\lambda,\varepsilon}_T  }{  \mathbf{E}_{\mu}^T\big[ Y^{\lambda,\varepsilon}_T\big]  }\, \mathbf{P}_{\mu}^T \,,\hspace{.5cm} \text{where} \hspace{.1cm}\, Y^{\lambda,\varepsilon}_T\,:=\,\textup{exp}\Bigg\{ \,\vsm^{\lambda}_{\varepsilon}\int_0^T \,\delta_{\varepsilon}(\mathbf{W}_r)\,dr\,\Bigg\}\,,$$
converges to $\mathbf{P}_{\mu}^{T,\lambda}$ as $\varepsilon\searrow 0$   in the sense of finite-dimensional  distributions. Not surprisingly, a two-dimensional diffusion process $\{X_t\}_{t\in [0,T]}$ with law  $\mathbf{P}_{\mu}^{T,\lambda}$ has a positive probability of visiting the origin. Our focus is on providing  representations for the local time at the origin, characterizing the law of its process inverse, and  measuring the size of the set of times $t\in [0,T]$ at which $X$ visits the origin.

The integral kernel $f_{t}^{\lambda}(x,y)$ appears in the limiting correlation formulas derived in the regularization scheme of the two-dimensional stochastic heat equation (SHE) proposed by Bertini and Cancrini in~\cite{BC}.  The two-dimensional SHE is  an ill-posed SPDE of the form
\begin{align}\label{SHE}\frac{\partial}{\partial t}\,u_t(x)\,=\,\frac{1}{2}\,\Delta\, u_t(x)\,+\, \sqrt{\vsm }\,\xi (t,x)\,u_t(x) \,,  \end{align}
in which the coupling constant $\mathsmaller{V}>0$ is to have a similarly fine-tuned vanishing role as in~(\ref{Hamiltonian}), and $\xi(t,x)$ is a time-space white noise on $[0,\infty)\times \R^2$, that is with correlation $\mathbb{E}\big[ \xi(t,x) \xi(t',x')    \big]=\delta(t-t')\delta(x-x')$.  For  $j\in C_c^{\infty}(\R^2)$ with integral one, define the mollified Gaussian field  $\xi_{\varepsilon}(t,x):=  \frac{1}{\varepsilon^2}\int_{\R^2} \xi(t,y)j\big(\frac{x-y}{\varepsilon}\big)dy$ for $\varepsilon>0$, which has correlations
$$ \mathbb{E}\big[\, \xi_{\varepsilon}(t,x)\,\xi_{\varepsilon}(t',x') \,\big]\,=\,\delta(t-t')\,J_{\varepsilon}( x-x' ) \, ,
 $$
where $J_{\varepsilon}(x):=\frac{1}{\varepsilon^2} J( \frac{x}{\varepsilon}  )$ for the function $J(x):=\int_{\R^2} j(y)j(x+y)dy$.  Define $ \mathsmaller{V}_{\varepsilon}^{\lambda} $  as in~(\ref{VSM}) with $\Da(x):=2J\big(\sqrt{2}x\big)$, and put $\mathsmaller{\mathbf{V}}_{\varepsilon}^{\lambda}:=2\mathsmaller{V}_{\varepsilon}^{\lambda} $. In place of~(\ref{SHE}), consider the informal SPDE
\begin{align}\label{SHE2}  
\frac{\partial}{\partial t}u^{\lambda,\varepsilon}_{t}(x)\,=\,\frac{1}{2}\,\Delta\, u^{\lambda,\varepsilon}_{t}(x)\,+\, \sqrt{\mathbf{\vsm}_{\varepsilon}^{\lambda} }\,\xi_{\varepsilon}(t,x)\,u^{\lambda,\varepsilon}_t(x) \,,  
\end{align}
with $u^{\lambda,\varepsilon}_{0}=\varphi$ for some  $\varphi\in C_c(\R^2)$. We can  indicate the initial data by equating  $u^{\lambda,\varepsilon}_{t}(x,\varphi)\equiv u^{\lambda,\varepsilon}_t(x)$.  For any $x,y\in \R^2$ with $x\neq y$ and $\varphi_1,\varphi_2\in C_c(\R^2) $, the correlation between the random variables $   u^{\lambda,\varepsilon}_{t}(x,\varphi_1)$ and $u^{\lambda,\varepsilon}_{t}(y,\varphi_2)$ has the small $\varepsilon$ convergence 
\begin{align}\label{CorrelationLimit}
 \mathbb{E}\left[ \, u^{\lambda,\varepsilon}_{t}(x,\varphi_1)\,u^{\lambda,\varepsilon}_{t}(y,\varphi_2)\,\right] \, \stackrel{ \varepsilon\rightarrow 0
 }{\longrightarrow}   \, \int_{(\R^2)^2 }\, f^{\lambda}_t\Big( \frac{x-y}{\sqrt{2} }   , z \Big)\, g_t\Big( z' -\frac{x+y}{\sqrt{2}} \Big)     \, \varphi_1\Big( \frac{ z'+z  }{\sqrt{2}} \Big)\,\varphi_2\Big( \frac{z' -z }{\sqrt{2}} \Big)\, dz\,dz'\,,
\end{align}
in which $g_t:\R^2\rightarrow (0,\infty)$ is the two-dimensional  Gaussian density function $g_t(x):=\frac{ 1  }{ 2\pi t } e^{-\frac{1}{2t}|x|^2}$; see Appendix~\ref{AppendixSHE} for a proof of~(\ref{CorrelationLimit}) using~(\ref{ConvKer}).   The derivation in~\cite{BC}  of the limiting $2$-point correlation functions within the critical scaling window of the 2d SHE was extended by  Caravenna, Sun, and Zygouras in~\cite{CSZ4} to the $3$-point correlation functions and then by  Gu, Quastel, and Tsai in~\cite{GQT} to the general $n$-point correlation functions  through functional analysis methods, taking some inspiration from Dimock and Rajeev's article~\cite{Dimock} on the Schr\"odinger evolution of a  finite system of bosons in $\R^2$ interacting through point potentials (see also the related earlier  work~\cite{Dell} by Dell'Antonio, Figari, and  Teta). Although the results in~\cite{GQT} address a problem  suggested from the beginning in~\cite[Remark on p.\ 620] {BC}, they do not ensure distributional convergence of the fields  because the associated limiting moments grow too quickly as $n\rightarrow \infty$.

The \textit{critical two-dimensional stochastic heat flow} (2d SHF) introduced by Caravenna, Sun, and Zygouras in~\cite{CSZ5}  is a distributional law, depending on a fixed  $\vartheta\in \R$, for  a two-parameter   process  $\{\mathscr{Z}^{\vartheta}_{s,t}(dx,dy)\}_{0\leq s<t < \infty}$ of random Borel measures on $\R^2\times \R^2$, which arises as a universal limit of point-to-point partition functions for  two-dimensional directed polymer models within a  critical weak-coupling scaling regime; see also the more recent article~\cite{CSZ6} by the same authors.  These rescaled polymer models can be understood as providing a discrete regularization of a slightly altered  version of the 2d SHE~(\ref{SHE}), in particular with  a halved diffusion rate ($\frac{1}{2}\mapsto \frac{1}{4}$);  see~\cite[Rmk.\ 1.5]{CSZ5}.  A similar proof scheme as used for the discrete models can be applied to show distributional convergence of the regularization~(\ref{SHE2}) to a  2d SHF with modified parameters; see~\cite[Rmk.\ 1.4]{CSZ5}.  For a more precise statement of this convergence, we need a few definitions.   For  $s\geq 0$ let $ u^{\lambda,\varepsilon}_{s,t}(x,\varphi)$ denote the solution of the SPDE~(\ref{SHE2}) for times $t\in[s,\infty)$ with initial condition $u^{\lambda,\varepsilon}_{s,s}(x,\varphi)=\varphi(x)$. Using the random kernels $u^{\lambda,\varepsilon}_{s,t}(x,\cdot)$, define  the random Borel measure on $\R^2\times \R^2$ by $\mu^{\lambda,\varepsilon}_{s,t}(dx,dy):=dx\, u^{\lambda,\varepsilon}_{s,t}(x,dy)$.  Then the measure-valued process $\{\mu^{\lambda, \varepsilon}_{s,t}(dx,dy)\}_{0\leq s<t < \infty}$  converges in the  finite-dimensional sense to
$\{\widehat{\mathscr{Z}}^{\vartheta}_{s,t}(dx,dy)\}_{0\leq s<t < \infty}$ for $ \widehat{\mathscr{Z}}^{\vartheta}_{s,t}:=2\mathscr{Z}^{\vartheta+\log 2}_{2s,2t} $ with  $\vartheta:= \log \lambda +\gamma_{\mathsmaller{\textup{EM}} }$.  Our motivation for studying  planar diffusion processes  with transition densities of the form~(\ref{FirstTrans}) is to understand the correlation measure for a critical two-dimensional random continuum  polymer measure  corresponding to the 2d SHF.  The pair of articles~\cite{Clark3, Clark4} consider a   critical continuum polymer model in a hierarchical setting, and we  expect that the continuum polymer associated with the 2d SHF to have some key similarities with it. We shall defer further discussion of this topic to another article.

\section{Detailed model formulation and main results}  \label{SecMainResults} 
In Sections~\ref{SubSecTransProb} \& \ref{SubsectionStochdDiff}, we will discuss the construction and  basic properties of the $\R^2$-valued diffusion process that is the focus of this article.  We place the proofs of the propositions stated in these first two subsections in the appendix because they  offer negligible  advantage in terms of understanding later material.   Our chief results, which we present in Sections~\ref{SubsectionLocalTime}--\ref{SubsectionOriginSet}, all concern topics related to the time that the process spends in the vicinity of the origin.  Finally, we  outline the remainder of the text  in Section~\ref{SubsectionOutline}.

\subsection{The transition density function and corresponding path measure}\label{SubSecTransProb}  
The discussion below  builds up  a precise definition of our path measures and ends with a  theorem about conditioning on  the event that the process avoids the origin.
 
 For $t\geq 0$ define $h_{t}^{\lambda}:\R^2\times \R^2\rightarrow [0,\infty]$ through the double integral
\begin{align}\label{DefLittleH}
      h_{t}^{\lambda}(x, y)\,:=\, 2\,\pi\, \lambda \int_{0 < r < s < t}\,g_{r}( x)\,\nu'\big((s-r)\lambda\big)\,g_{t-s}(y)\,ds\,dr\,,
\end{align}
where as above $g_t(x):=\frac{ 1  }{ 2\pi t } e^{-\frac{1}{2t}|x|^2}$,  and $\nu'$ is the derivative of the function $\nu:\C\rightarrow \C$ defined by  $\nu(a)=\int_0^{\infty}\frac{ a^{s} }{\Gamma(s+1) }ds  $. Some authors  refer to  $\nu$ and a family of multivariate generalizations of it as the \textit{Volterra functions}; see the book~\cite{Apelblat0} by Apelblat, which compiles identities for these special functions and provides a historical survey thereof, and the  review~\cite{Garrappa} by Garrappa and Mainardi.\footnote{The name \textit{Volterra} is also applied to a function $f:\R\rightarrow \R$ whose derivative is not Riemann integrable despite existing everywhere and being bounded.}  The key properties of $\nu$ for our purposes are covered in  Section~\ref{SubsecFractExp}.  Next, we  define $H_{t}^{\lambda}:\R^2\rightarrow [0,\infty]$ as the partial integral of $h_{t}^{\lambda}$, that is
\begin{align}\label{DefH}
H_t^{\lambda}(x)\,:=\, \int_{\R^2}\,h_{t}^{\lambda}(x, y)\, dy\,=\,   \int_{0 }^t \,\frac{e^{-\frac{|x|^2}{2r}} }{r }\,\nu\big((t-r)\lambda\big)\,dr\,,
\end{align}
which satisfies the diffusion equation $    \frac{\partial}{\partial t} H_{ t}^{\lambda}(x)=\frac{1}{2}\Delta_x H_{ t}^{\lambda}(x) $
for $x\neq 0$ and has the small $|x|$ asymptotics in~(\ref{Origin}) with $\mathbf{c}_{\psi}=-2\nu(t\lambda) $.  The linear operator $e^{ \frac{t}{2}\Delta^{\lambda}  } $ has integral kernel $f_{t}^{\lambda}:\R^2\times \R^2\rightarrow [0,\infty]$ given by
\begin{align}\label{DefFullKer}
 f_{t}^{ \lambda}(x,y)\,:=\, g_{t}(x-y) \,+\,h_{t}^{ \lambda}(x,y) \,.
\end{align}
This form for $f^{\lambda}_t(x,y)$, which was used in the recent articles~\cite{GQT,Chen,CSZ5},  can  be shown to   equal that  previously given  in~\cite[Eqn.\ (3.11)]{Albeverio} through using the modified Bessel function identity $K_0(a)=\int_0^{\infty} \frac{1}{2s}e^{-s-\frac{a^2}{4s}  }ds    $; see~\cite[Rrk.\ 2.1]{GQT}.  Since the  functions in the family $\{ f_t^{\lambda}(x,y) \}_{t\in [0,\infty)}$ are the integral kernels  for the operators in the semigroup $\{e^{\frac{t}{2}\Delta^{\lambda}}\}_{t\in [0,\infty)}$, they must satisfy the semigroup property $\int_{\R^2}f_t^{\lambda}(x,y)f_{T-t}^{\lambda}(y,z)dy=f_{T}^{\lambda}(x,z) $ for $0<t<T$. This can be verified  using the following integral identity for $\nu'$
\begin{align}\label{DoubleNuPrime}
\int_{0}^t \,\int_{t}^T  \,\nu'(r\lambda)\,\frac{1}{s-r}\,\nu'\big( (T-s)\lambda\big) \,ds\,dr\,=\,\frac{1}{\lambda} \,\nu'(T\lambda)\,, 
\end{align}
which is equivalent to~\cite[Lem.\ 8.7]{GQT}.  Related identities arise naturally in connection to a jump process that we cover in Section~\ref{SubsectionProcInv}; see Remark~\ref{RemarkVartheta} and Lemma~\ref{LemmaC}.

Fix $T,\lambda >0$, and let $0\leq s<t\leq T$. We define $ \mathlarger{d}_{s,t}^{T,\lambda}:\R^2\times \R^2\rightarrow (0,\infty]$ for $x,y\in \R^2$ with $x\neq 0$ by
\begin{align} \label{Defd}
  \mathlarger{d}_{s,t}^{T,\lambda}(x,y)\,:=\, f_{t-s}^{\lambda}(x, y)\,\frac{1+H_{T-t}^{\lambda}(y) }{ 1+ H_{T-s}^{\lambda}(x) } \,, 
\end{align}
and extend this definition to the case $x=0$ through the limit
\begin{align}\label{Defd2}
\mathlarger{d}_{s,t}^{T,\lambda}(0,y)\,:=\,\lim _{x\searrow 0} \mathlarger{d}_{s,t}^{T,\lambda}(x,y)\,=\,\big(1+H_{T-t}^{\lambda}(y)\big)\, \frac{ \lambda \int_0^{t-s}g_{t-s-r}(y)\nu'( r \lambda)  dr}{\nu\big( (T-s)\lambda\big)  }  \,.  
\end{align}
The Chapman-Kolmogorov relation in the  lemma below  follows easily from the semigroup property for the family of kernels $\{ f_t^{\lambda}(x,y) \}_{t\in [0,\infty)}$ mentioned above.  
\begin{lemma}\label{PropTranKern}
 Fix some $T,\lambda>0$. The $(0,\infty]$-valued   map $(s, x; t, y)\mapsto   \mathlarger{d}_{s,t}^{T,\lambda }(x,y)  $ defined  in~(\ref{Defd})--(\ref{Defd2}) for $0\leq s<t\leq T$ and  $ x,y \in\R^2$ is a transition probability density function, meaning
 \begin{align*}
 \int_{\R^2} \,\mathlarger{d}_{s,t}^{T,\lambda }(x,y)\,dy\,=\,1  \hspace{1cm} \text{and} \hspace{1cm}
 \int_{\R^2} \,\mathlarger{d}_{r,s}^{T,\lambda}(x,y) \,\mathlarger{d}_{s,t}^{T,\lambda}(y,z)\,dy\,=\, \mathlarger{d}_{r,t}^{T,\lambda}(x,z)  \,.
 \end{align*}
\end{lemma}
As the time $s\in [0,T)$ approaches the boundary $T$, the transition density function satisfies $\mathlarger{d}_{s,t}^{T,\lambda }(x,y) \approx g_{t-s}(x-y)$ except for $x,y\in \R^2$ in a vanishing region around the origin.  Thus, $\mathlarger{d}_{s,t}^{T,\lambda }$  has a natural extension to times $0\leq s<t$ with $t > T$ through the following rules:
\begin{align}\label{JExtentd}
     \mathlarger{d}_{s,t}^{T ,\lambda }(x,y) \,=\,\begin{cases}  \displaystyle \int_{\R^2}  \, \mathlarger{d}_{s,T}^{T,\lambda  }(x,a)\,g_{t-T}(a-y)\, da   & \hspace{.3cm} s<T<t \,, \vspace{.1cm} \\ \displaystyle g_{t-s}(x-y)  & \hspace{.3cm} T\leq s<t \,. \end{cases}
\end{align}
This extended transition density function corresponds to   a stochastic process whose law is merely that of a two-dimensional Brownian motion over the time interval $[T,\infty)$.\vspace{.2cm}

  Consider the path space $\boldsymbol{\Omega}:=C\big([0,\infty), \R^2\big)  $, which is  Polish when equipped with the usual metric defined for $\omega,\omega'\in \boldsymbol{\Omega}$ by
  $$\rho(\omega,\omega')\,:=\,\sum_{n=1}^{\infty}\,\frac{1\wedge \rho_n(\omega,\omega')}{2^n}\,,\,\hspace{.7cm}\text{where}\hspace{.2cm} \rho_n(\omega,\omega')\,:=\,\sup_{t\in [0,n]}\,\big|\omega(t)-\omega'(t)\big| \,.  $$ 
  Let $\mathscr{B}(\boldsymbol{\Omega})$ and $\{X_t \}_{t\in [0,\infty)}$ respectively denote  the Borel $\sigma$-algebra and the coordinate process on $\boldsymbol{\Omega}$, meaning that $X_t$ is the $\R^2$-valued map on $\boldsymbol{\Omega} $ defined by $X_t(\omega)=\omega(t)$.  Finally, we use $\{ \mathscr{F}^{X}_t\}_{t\in [0,\infty)}$ to denote the filtration generated by $X$. We  prove  the following proposition in Appendix~\ref{AppendixYMart}.
\begin{proposition}\label{PropCP} Fix $T,\lambda >0$ and   $x\in \R^2$.  There exists a unique  probability measure  $\mathbf{P}^{T,\lambda}_{x}  $ on  $\big(\boldsymbol{\Omega},\mathscr{B}(\boldsymbol{\Omega})  \big)$ under which the coordinate process $\{X_t \}_{t\in [0,\infty)}$ has initial distribution $\delta_x$ and is  Markov with transition density function  $   \mathlarger{d}_{s,t}^{T,\lambda  } $ (with respect to   $\{\mathscr{F}^{X}_t \}_{t\in [0,\infty)}$).
The measure $ \mathbf{P}^{T,\lambda }_{x}$ depends continuously on $(x,\lambda ,T)\in \R^2\times (0,\infty)\times (0,\infty)$ under the weak topology on the set of finite Borel measures on $\boldsymbol{\Omega}$.   Moreover,  $ \mathbf{P}^{T,\lambda }_{x}$ converges weakly as $\lambda\searrow 0$ to the Wiener measure $ \mathbf{P}_{x}$. 
\end{proposition}
For  the diffusive rescaling map $\mathfrak{R}_r:\boldsymbol{\Omega}\rightarrow \boldsymbol{\Omega}$ given by $(\mathfrak{R}_r\omega)(t)= r^{-\frac{1}{2}}\omega(\frac{t}{r})$ for some $r>0$, the path measures have the scaling property $\mathbf{P}^{ T, r\lambda}_{ x} =\mathbf{P}^{r T,\lambda}_{\sqrt{r} x}\circ \mathfrak{R}_r^{-1}  $ due to the transition density function symmetry $\mathlarger{d}_{s,t}^{T ,r\lambda }(x,y)=\mathlarger{d}_{r s,r t}^{r T ,\lambda }(\sqrt{r} x,\sqrt{r} y)r$.  Given a Borel probability measure $\mu$ on $\R^2$, we define $\mathbf{P}^{T,\lambda}_{\mu}:=\int_{\R^2}\mathbf{P}^{T,\lambda}_{x}\mu(dx)  $, and we use $\mathbf{E}^{T,\lambda}_{\mu}$ to  denote the  expectation corresponding to $\mathbf{P}^{T,\lambda }_{\mu}$. 
For $\mathbf{P}_{\mu}$ denoting the  two-dimensional Wiener measure on $\big(\boldsymbol{\Omega},\mathscr{B}\big(\boldsymbol{\Omega}\big)  \big)$ with initial distribution $\mu$, we set  $\mathbf{P}^{T,\lambda}_{\mu}:=\mathbf{P}_{\mu}$ when $T\leq 0$. The next proposition addresses \textit{strong} Markovianity, which can be proven using a similar technique to that in~\cite[Sec.\ 2.6C]{Karatzas} for a $d$-dimensional Brownian motion, and we summarize the needed adjustment in Appendix~\ref{AppendixSMP}.  For  $s\geq 0$ define the  shift map  $\theta_s:\boldsymbol{\Omega}\rightarrow \boldsymbol{\Omega}$ by $\theta_s p(t):=p(t+s)$ for $t\geq 0$. 
\begin{proposition}\label{PropStrongMarkov} For any fixed $\lambda >0$, the family of probability measures  $\{\mathbf{P}^{T,\lambda}_{x}\}^{ T\in \R}_{ x\in \R^2}  $ on $\big(\boldsymbol{\Omega},\mathscr{B}( \boldsymbol{\Omega} )\big)$ is  strong Markov with respect to $\{\mathscr{F}_t^X\}_{t\in [0,\infty)}$, meaning that for any bounded Borel measurable function $G: \boldsymbol{\Omega} \rightarrow \R$ and $\mathscr{F}^X$-optional time $\mathbf{S}$   we have
$$  \mathbf{E}_{x  }^{T,\lambda}\left[\, G\circ \theta_{\mathbf{S}} \,\Big|\,\mathscr{F}_{\mathbf{S} }^X \,\right] \,=\, \mathbf{E}_{ X_{\mathbf{S}} }^{T-\mathbf{S},\lambda }[\, G\,] \hspace{.4cm} \text{$\mathbf{P}^{T,\lambda}_{x}$\,-a.s.\ on}\,\,\{\mathbf{S}<\infty  \}\,. $$
\end{proposition}

  For $x\in \R^2\backslash\{0\}$  a key difference between the path measure $\mathbf{P}^{T,\lambda}_{x}$ from Proposition~\ref{PropCP} and the Wiener measure $\mathbf{P}_{x} $ is that  $\mathbf{P}^{T,\lambda}_{x}$ assigns   positive weight to the event $\mathcal{O}\in \mathscr{B}(\boldsymbol{\Omega})$ that the coordinate process $X$ visits the origin, meaning $X_t=0$ for some $t\geq 0$. In particular, this implies that $\mathbf{P}^{T,\lambda}_{x}$ is not absolutely continuous with respect to  $\mathbf{P}_{x} $. The main result in the theorem below, which we prove in Section~\ref{SubsectionCorSubMART}, is that $\mathbf{P}_{x }^{T,\lambda}$  conditioned on the event $\mathcal{O}^c$ is equal to $\mathbf{P}_{x }$.  This may not be surprising to the reader in view of the form of the transition density function $\mathlarger{d}_{s,t}^{T,\lambda }$.  
  Define
   the  $\mathscr{F}^X$-stopping time 
  \begin{align*}
  \tau\,:=\,\inf\big\{ t\in [0,\infty) \,:\, X_t=0 \big\}\, ,  
\end{align*}
for which we can obviously write that   $\mathcal{O}=\{ \tau <\infty \} $.  The event  $ \{\tau\in [T,\infty) \}  $ is  $\mathbf{P}_{x }^{T,\lambda}$-null   because $\{X_{T+t}\}_{t\in [0,\infty)}$ is a two-dimensional Brownian motion with respect to the filtration $\{\mathscr{F}_{T+t}^{X} \}_{t\in [0,\infty)} $ having initial density $\mathlarger{d}_{0,T}^{T,\lambda }(x,\cdot)$.  For  (iii) below and  convenience in the sequel,  we  extend the definition of $H_{t}^{\lambda }(x)$ in~(\ref{DefH}) to all $t\in \R$ by setting $H_{t}^{\lambda }(x)=0$ when $t < 0$.
\begin{theorem} \label{CorSubMART}  Fix $T,\lambda > 0$ and $x\in \R^2\backslash \{0\}$. Let $\mathbf{S}$ be an $\mathscr{F}^X$-stopping time such that the event  $\mathcal{O}_{\mathbf{S}}:=\{ \tau \leq \mathbf{S}\}$ is not assigned probability one under $\mathbf{P}_{x }^{T,\lambda}$. 
\begin{enumerate}[(i)]
\item  $\mathbf{P}_{x }^{T,\lambda}[\mathcal{O}^c]=\frac{1}{1+H_{T}^{\lambda}(x)  }$

\item If $\mathbf{\widetilde{P}}_{x }^{T,\lambda} $ denotes the conditioning of $\mathbf{P}_{x}^{T,\lambda }$ to the event $\mathcal{O}^c$, then under $\mathbf{\widetilde{P}}_{x }^{T,\lambda}$  the  coordinate process $\{X_t \}_{t\in [0,\infty)}$  is a standard two-dimensional Brownian motion with initial position $x$.

\item  If $\mathbf{\widetilde{P}}_{x }^{T,\lambda,\mathbf{S}} $ denotes the conditioning of $\mathbf{P}_{x}^{T,\lambda }$ to the event $\mathcal{O}^{c}_{\mathbf{S}}$, then  the stopped coordinate  process $\{X_{t\wedge \mathbf{S}} \}_{t\in [0,\infty)}$ has the same  law   under $\mathbf{\widetilde{P}}_{x}^{T,\lambda,\mathbf{S}} $ as it does under the path measure
 $ \frac{  1+H_{T-\mathbf{S} }^{\lambda}(X_{\mathbf{S}})   }{\mathbf{E}_{x}[ 1+H_{T-\mathbf{S} }^{\lambda}(X_{\mathbf{S}})  ]  } \mathbf{P}_{x}$.

\item  The distribution of the random variable $\tau$ under $\mathbf{P}_{x }^{T,\lambda}$ conditioned on the event $\mathcal{O}$  has density
\begin{align*}
 \text{}\hspace{1cm}\frac{\mathbf{P}_{x }^{T,\lambda}\big[\tau \in dt\,\big|\,\mathcal{O}\big] }{dt}\,=\,\frac{1}{H_{T}^{\lambda}(x)  }  \, \frac{ e^{-\frac{|x|^2 }{2t } }   }{ t }\,\nu\big( (T-t)\lambda\big)\, 1_{[0,T]}(t)\,,  \hspace{1cm} t\geq 0 \,.
 \end{align*}

\end{enumerate}

\end{theorem}
\begin{remark} The Radon-Nikodym derivative $ \frac{  1+H_{T-\mathbf{S} }^{\lambda}(X_{\mathbf{S}})   }{\mathbf{E}_{x}[ 1+H_{T-\mathbf{S} }^{\lambda}(X_{\mathbf{S}})  ]  }$ appearing in (iii) is equal to one in the event that $\mathbf{S}\geq T$, and we obtain (ii) as a special case of (iii) by taking $\mathbf{S}=\infty$.    
\end{remark}

\subsection{A weak solution to the corresponding  stochastic differential equation}\label{SubsectionStochdDiff}
To apply It\^o calculus  in later sections, we will need to construct a  weak solution to the SDE associated with the path measure $ \mathbf{P}^{T,\lambda}_{x}$.  Note that we cannot simply rely on Girsanov's theorem for this purpose because  $ \mathbf{P}^{T,\lambda}_{x}$ is not absolutely continuous with respect to the Wiener measure $ \mathbf{P}_{x}$.  Given $\lambda >0$ and $t\in \R$, define 
$b_{t}^{\lambda}:\R^2\rightarrow \R^2$ for  $x\neq 0$ by
\begin{align}\label{DefDriftFun}
b_{t}^{\lambda}(x)\,:=\, \nabla_x  \,\log\big(  1+ H_{t}^{ \lambda}(x)   \big)   \,=\,\frac{ \nabla_x H_{t}^{\lambda} (x)  }{ 1+ H_{t}^{ \lambda}(x)  }\,,
\end{align}
where $\nabla_x:=\big(\frac{\partial}{\partial x_1}, \frac{\partial}{\partial x_2}\big)$ is the gradient operator, and set $b_{t}^{\lambda}(0):=0$.
 For fixed $s\in [0,\infty)$ and $x\in \R^2$,  the map  sending $ (t,y)\in [s,\infty)\times \R^2\backslash \{0\}$ to $ \mathlarger{d}_{s,t}^{T,\lambda }(x,y) $  satisfies the forward Kolmogorov equation
\begin{align}\label{KolmogorovForJ}
\frac{\partial}{\partial t} \,\mathlarger{d}_{s,t}^{T,\lambda}(x,y)\,=\,\frac{1}{2}\,\Delta_y  \,\,\mathlarger{d}_{s,t}^{T,\lambda}(x,y)\,-\,\nabla_y \,\cdot\,\Big[\, b_{T-t}^{\lambda}(y) \, \mathlarger{d}_{s,t}^{T,\lambda}(x,y) \,\Big] \,.
\end{align}
An SDE corresponding to the above forward Kolmogorov equation has the form
    \begin{align}\label{SDEToSolve}
        d\mathbf{X}_t\,=\,d\mathbf{W}_t\,+\,b_{T-t}^{\lambda}(\mathbf{X}_t)\,dt
    \end{align}
    for an $\R^2$-valued process  $\{ \mathbf{X}_t\}_{t\in [0,\infty)}$ adapted to a filtration $\mathscr{F}=\{ \mathscr{F}_t \}_{t\in [0,\infty)}$ and a two-dimensional standard Brownian motion  $\{ \mathbf{W}_t \}_{t\in [0,\infty)}$ with respect to $\mathscr{F}$, all defined over some probability space $(\Omega, \mathscr{B},\mathbb{P})$.   The norm of the drift function $b_{T-t}^{\lambda}(x)$   diverges to $\infty$ at $x=0$ (see~(\ref{bBlowUp}) below) too quickly for standard techniques for constructing weak solutions  to apply (e.g.\ Girsanov's theorem, as indicated above).
    
 The proof of the proposition below  concerning weak solutions to the SDE~(\ref{SDEToSolve}) is in Appendix~\ref{AppendixWeakSolConst}.   As a preliminary, we will augment our $\sigma$-algebras.  For $T,\lambda>0 $ and a Borel probability measure $\mu$ on $\R^2$, let  $\mathcal{N}^{T}_{ \mu}$ denote the collection of all subsets of  $\mathbf{P}^{T,\lambda}_{\mu}$-null sets.  We  have omitted $\lambda$ as a script on $\mathcal{N}^{T}_{\mu}$ because the measures $\mathbf{P}^{T,\lambda}_{\mu}$ and $\mathbf{P}^{T,\lambda'}_{\mu}$ are equivalent for any $\lambda,\lambda'>0$ by Theorem~\ref{ThmEquivalent} in the next subsection.  Next, we define the augmented  $\sigma$-algebras 
 \begin{align}\label{Augmented}
 \mathscr{F}_t^{T,\mu}\,:=\,\sigma\big\{\,\mathscr{F}_t^X \,\cup \,\mathcal{N}^{T}_{\mu} \, \big\} \hspace{.7cm}\text{and}\hspace{.7cm}  \mathscr{B}^{T}_{\mu}\,:=\, \sigma\big\{ \,\mathscr{B}(\boldsymbol{\Omega}) \,\cup\, \mathcal{N}^{T}_{\mu} \, \big\} \,.
 \end{align}
  The measure $\mathbf{P}^{T,\lambda}_{\mu}$ extends uniquely to the augmented $\sigma$-algebra $\mathscr{B}^{T}_{\mu}$, and we use the same symbol for the extension.   In the case of a Dirac measure $\mu=\delta_x$, we use the alternative notations $\mathscr{F}_t^{T,x}\equiv \mathscr{F}_t^{T,\mu}  $, $\mathscr{B}^{T}_{x}\equiv \mathscr{B}^{T}_{\mu}$, and $\mathbf{P}^{T,\lambda}_{x}\equiv \mathbf{P}^{T,\lambda}_{\mu} $. The augmented filtration $\{\mathscr{F}_t^{T,\mu}\}_{t\in [0,\infty)}$ is continuous, and the family $\{ \mathbf{P}^{T,\lambda}_{x} \}_{x\in \R^2}^{T\in \R} $ remains strong Markov with respect to $\mathscr{F}^{T,\mu}$; see~\cite[Sec.\ 2.7(A-B)]{Karatzas}. 
\begin{proposition}\label{PropStochPre} Fix some $T,\lambda>0$ and a Borel probability measure $\mu$ on $\R^2$. There exists an  $\R^2$-valued process  $\{  W^{T,\lambda }_t\}_{t\in [0,\infty)}$ on the probability space $\big(\boldsymbol{\Omega}, \mathscr{B}^{T}_{\mu},\mathbf{P}^{T,\lambda}_{\mu}\big)$  
 such that  $W^{T,\lambda }$ is a standard two-dimensional Brownian motion with respect to the filtration  $\{\mathscr{F}_t^{T,\mu}   \}_{t\in [0,\infty)}$, and the coordinate process  $X$   satisfies the following  SDE  with respect to $W^{T,\lambda }$:
$$ dX_t\,=\,dW_{t}^{T,\lambda}\,+ \,b_{T-t}^{\lambda}(X_t) \, dt \,.  $$
\end{proposition}\vspace{.2cm}

   The   $\R^2$-valued drift function $b_{T}^{\lambda}(x)$ has the radial form $b_{T}^{\lambda}(x)=-\frac{x}{|x| }\bar{b}_{T}^{\lambda}\big(|x|\big)$ for a decreasing function  $\bar{b}_{T}^{\lambda}:(0,\infty)\rightarrow (0,\infty)$ that vanishes at $\infty$ and blows up near zero for any fixed $T,\lambda>0$ with the asymptotics 
  \begin{align}\label{bBlowUp}
  \big| b_{T}^{\lambda}(x)\big| \,= \,\bar{b}_{T}^{\lambda}\big(|x|\big)\, \stackrel{x\rightarrow 0 }{=}  \,\frac{ 1 }{|x|\log \frac{1}{|x|} }\left(1 \,+\,\frac{\frac{1}{2}\log\frac{\lambda}{2} +\gamma   }{\log\frac{1}{|x|}   } \,+\,\mathit{O}\bigg(  \frac{1  }{\log^2\frac{1}{|x|}   }   \bigg)    \right)    \,.
  \end{align}
The above follows from (i) \& (iii) of Proposition~\ref{PropK}.  Next we state an analog of Proposition~\ref{PropStochPre} for the radial process $\{ R_t\}_{t\in [0,\infty)}$ defined by $R_t:=|X_t|$.
\begin{corollary}\label{CorollaryToPropStoch}  Fix some $T,\lambda >0$ and a Borel probability measure $\mu$ on $\R^2$.  Let $\{  w^{T,\lambda }_t\}_{t\in [0,\infty)}$ be the real-valued process  on $\big(\boldsymbol{\Omega}, \mathscr{B}^{T}_{\mu},\mathbf{P}^{T,\lambda}_{\mu}\big)$   given by the It\^o integral
 $$w^{T,\lambda }_t\,=\,\int_0^t \,1_{X_s\neq 0}\,\frac{X_s}{|X_s|}\cdot  dW_{s}^{T,\lambda} \hspace{.5cm}\text{a.s.    $\mathbf{P}^{T,\lambda}_{\mu}$}  \,.    $$
Then  $w^{T,\lambda }$ is a standard one-dimensional Brownian motion with respect to $\{\mathscr{F}_t^{T,\mu}   \}_{t\in [0,\infty)}$, and the radial process  $R$  satisfies the    SDE below with respect to $w^{T,\lambda }$.
\begin{align*}
dR_t\,=\,dw_t^{T,\lambda}\,+\, \bigg(\frac{ 1 }{ 2R_t }\,-\, \bar{b}_{T-t}^{\lambda  }(R_t) \bigg)\, dt 
\end{align*}
\end{corollary}
\begin{remark} Since $\bar{b}_{T-t}^{\lambda  }(a) \ll \frac{1}{a} $ for small $a>0$   as a consequence of the asymptotic~(\ref{bBlowUp}), we find that the drift term $\bar{b}_{T-t}^{\lambda  }(R_t) $ has a merely perturbative role in the above SDE when $R_t=|X_t|$ is near zero.  Nevertheless,  (i) of Theorem~\ref{CorSubMART} implies that the drift term  is strong enough to  enable the process $R$ to  visit the origin with positive probability. 
\end{remark}

\subsection{The local time at the origin}\label{SubsectionLocalTime}
We will now examine the duration for which the coordinate process $X$ occupies an infinitesimal region around the origin under the law $\mathbf{P}^{T,\lambda}_{\mu}$.
For fixed $\varepsilon \in (0,1)$ and $t\geq 0$, define  the  random variable $ L_t^{\varepsilon}$ on the measurable space $\big(  \boldsymbol{\Omega}, \mathscr{B}^{T}_{\mu}    \big)$  by
\begin{align}\label{DEFLocaltime}
    L_t^{\varepsilon}\,:=\,\frac{1}{2\varepsilon^2\log^2 \frac{1}{\varepsilon}} \,\textup{meas}\big(\,\big\{r\in [0,t]\,:\, |X_r|\leq \varepsilon  \big\}\,\big)\,,
\end{align}
where $\textup{meas}(E)$ denotes the Lebesgue measure of a  set $E\subset \R$. We refer to the process $\{\mathbf{L}_t\}_{t\in [0,\infty)}$ arising through the $\varepsilon \searrow 0$ limit  of $\{L_t^{\varepsilon}\}_{t\in [0,\infty)}$ as the \textit{local time at the origin} or just the \textit{local time}.
The proof of the following theorem is in Section~\ref{SectionLocalTimeAdditive}.
\begin{theorem}\label{ThmExistenceLocalTime}  Fix some $T,\lambda>0$ and a Borel measure $\mu $ on $\R^2$.  There exists an $\mathscr{F}^{T,\mu}  $-adapted continuous process $\{ \mathbf{L}_t\}_{t\in [0,\infty)}$ on $ \big(  \boldsymbol{\Omega}, \mathscr{B}^{T}_{\mu}    \big) $ for which $\sup_{0\leq t\leq T}\big|  L_t^{\varepsilon}  -\mathbf{L}_t  \big|  $ vanishes in $L^1\big(\mathbf{P}^{T,\lambda}_{\mu}  \big)$-norm as $\varepsilon \searrow 0$ and that is
 $ \mathbf{P}^{T,\lambda}_{\mu}$ almost surely constant over the time interval $[T,\infty)$.
\end{theorem}

The next theorem, which we prove in Section~\ref{SubsectionThmEquivalent}, connects the local time process with the Radon-Nikodym derivative of $\mathbf{P}_{x}^{T,\lambda'}$ with respect to $\mathbf{P}_{x}^{T,\lambda}$.  Given $T,\lambda,\lambda'>0$ define $R^{\lambda,\lambda'}_{T}:\R^2\rightarrow [0,\infty)$ by 
\begin{align}\label{FirstR}
R^{\lambda,\lambda'}_{T}(x)\,:=\,\frac{1+H_{T}^{\lambda'}(x)}{1+H_{T}^{\lambda}(x)  } \hspace{.5cm} \text{when $x\neq 0$ and} \hspace{.5cm}R^{\lambda,\lambda'}_{T}(0)\,:=\,\lim_{x\rightarrow 0} R^{\lambda,\lambda'}_{T}(x) \, =\,\frac{ \nu(T\lambda')  }{ \nu(T\lambda)  }  \,.
\end{align}
\begin{theorem}\label{ThmEquivalent} Fix some $T,\lambda,\lambda' > 0$ and   a Borel probability measure $\mu$ on $\R^2$.  The probability measure $\mathbf{P}_{\mu}^{T,\lambda'}$ is absolutely continuous with respect to $\mathbf{P}_{\mu}^{T,\lambda}$ and has Radon-Nikodym derivative
 \begin{align*}
 \text{}\hspace{.7cm}\frac{d\mathbf{P}_{\mu}^{T,\lambda'}}{d\mathbf{P}_{\mu}^{T,\lambda}}\,=\,R^{\lambda',\lambda}_{T}(X_0) \, \bigg(\frac{ \lambda'}{\lambda}  \bigg)^{\mathbf{L}_{T}} \hspace{.2cm}\text{a.s.}\,\,\mathbf{P}_{\mu}^{T,\lambda}\,.
 \end{align*}
\end{theorem}
\begin{remark} Although $\mathbf{P}_{\mu}^{T,\lambda}$ converges weakly to the Wiener measure $\mathbf{P}_{\mu}$ as $\lambda \searrow 0$, the Radon-Nikodym derivative above does not have a meaningful small $\lambda$ limit. This is to be expected since $\mathbf{P}_{\mu}^{T,\lambda'}$ is not absolutely continuous with respect to $\mathbf{P}_{\mu}$.
\end{remark}

Recall that the  $\tau$ denotes the first time that the coordinate process $X$ reaches the origin and $\mathcal{O}=\{ \tau<\infty\}$.  The following proposition, which we prove  in Section~\ref{SubsectionThmEquivalent} using Theorem~\ref{ThmEquivalent}, offers a closed form for the joint density $(\tau,\mathbf{L}_{T})$ conditional on $\mathcal{O}$.  Note that this generalizes (iv) of Theorem~\ref{CorSubMART}.
\begin{proposition}\label{PropLocalTimeProp} Fix some  $T,\lambda > 0$ and $x\in \R^2$.
The equality  $\tau=\inf\{t > 0 \, :\, \mathbf{L}_{t}>0 \}    $ holds almost surely $\mathbf{P}_{x }^{T,\lambda}$. The joint probability density of the random variables $\tau$ and $\mathbf{L}_T$ under $\mathbf{P}_{x }^{T,\lambda}$ conditioned on the event $\mathcal{O}$ has the closed form
$$  \text{}\hspace{2.5cm} \frac{ \mathbf{P}_{x }^{T,\lambda}[\tau\in dt,\, \mathbf{L}_T \in dv\,\big|\,\mathcal{O}]}{dt\,dv}  \,=\,  \frac{1}{H_{T}^{\lambda}(x)  } \,\frac{ e^{-\frac{ |x|^2 }{2t  }  }  }{ t } \,\frac{ (T-t)^v \,\lambda^v }{\Gamma(v+1)}\, 1_{[0,T]}(t)\,,  \hspace{.5cm}  t,v\in [0,\infty)\,. $$
\end{proposition}

\begin{remark}\label{RemarkMargLoc}  We can apply the above proposition along with (i) of Theorem~\ref{CorSubMART} to derive a closed expression for the moment generating function of $\mathbf{L}_T$ under $\mathbf{P}_{x }^{T,\lambda}$.  For $r\in \R$ we get  $\mathbf{E}_{x }^{T,\lambda}[e^{rL_T}]=\frac{H_{T}^{\lambda e^r}(x)  }{ 1+H_{T}^{\lambda}(x)   }  $.
\end{remark}

We will now present  a few alternative methods for approximating the local time process $\mathbf{L}$ in terms of the number of downcrossings of the process $|X|$ over a small interval near zero. The proofs of the theorems below are respectively placed in Sections~\ref{SectionLocalTimeDowncrossing} \&~\ref{SubsectionDowncrossing2}.  For some $\varepsilon>0$  let $\{\varrho_{n}^{\downarrow,\varepsilon}\}_{n\in \mathbb{N}}$ and $\{\varrho_{n}^{\uparrow,\varepsilon}\}_{n\in \mathbb{N}_0}$ be the  sequences of stopping times such that $\varrho_{0}^{\uparrow,\varepsilon}=0$ and for $n\in \mathbb{N}$
\begin{align}\label{VARRHOS}
\varrho_{n}^{\downarrow,\varepsilon}\,:=\,&\,\inf\big\{s\in \big(\varrho_{n-1}^{\uparrow,\varepsilon},\infty\big) \,:\,X_s = 0 \big\}\,, \nonumber   \\
\varrho_{n}^{\uparrow,\varepsilon}\,:=\,&\,\inf\big\{s\in \big(\varrho_{n}^{\downarrow,\varepsilon},\infty\big)\,:\,|X_s| = \varepsilon  \big\}\,,
\end{align}
where we interpret $\inf \emptyset =\infty$.  Moreover, for $t\geq 0$ let  $N_{t}^{\varepsilon}$ denote the number of times $\varrho_{n}^{\downarrow,\varepsilon}$  in the interval $[0,t]$. 
\begin{theorem}\label{THMLocalTime1} Fix some $T,\lambda >0$ and  a Borel probability measure $\mu$  on $\R^2$. If for $\varepsilon\in (0,1)$ we define the process $\{\ell_{t}^{\varepsilon}\}_{t\in [0,\infty)}$ by $ \ell_{t}^{\varepsilon}:=\frac{1 }{  2\log \frac{1}{\varepsilon} }N_{t}^{\varepsilon}$, then $\mathbf{E}^{T,\lambda}_{\mu}\big[\sup_{0\leq t\leq T}\big|  \ell_t^{\varepsilon}  -\mathbf{L}_t  \big|    \big]$ vanishes as $\varepsilon \searrow 0$.
\end{theorem}

A different presentation of essentially  the same result is given by the next corollary, also proved in Section~\ref{SectionLocalTimeDowncrossing}, which  approximates the local time using the running cumulative duration of the upcrossing intervals $\big[\varrho_{n}^{\downarrow,\varepsilon},  \varrho_{n}^{\uparrow,\varepsilon}\big)$.  
\begin{corollary}\label{THMLocalTime1.5} Fix some $T,\lambda >0$ and a  Borel probability measure $\mu$  on $\R^2$.  If for $\varepsilon\in (0,1)$ we
  define the process   $\{\mathbf{L}_{t}^{\varepsilon}\}_{t\in [0,\infty)}$ by $ \mathbf{L}_{t}^{\varepsilon}\,:=\,\frac{1 }{ \varepsilon^2 \log \frac{1}{\varepsilon} }\sum_{n=1}^{N_{t}^{\varepsilon}} \big(\varrho_{n}^{\uparrow,\varepsilon}\wedge t-\varrho_{n}^{\downarrow ,\varepsilon}\big) $, then $\mathbf{E}^{T,\lambda}_{\mu}\big[\sup_{0\leq t\leq T}\big|  \mathbf{L}_{t}^{\varepsilon}  -\mathbf{L}_t  \big|    \big]$ vanishes as $\varepsilon \searrow 0$.     
\end{corollary}

Notice that the normalizing factors $ ( 2 \log \frac{1}{\varepsilon} )^{-1} $ and $ ( \varepsilon^2 \log \frac{1}{\varepsilon} )^{-1} $ in the definitions of the processes $\ell^{\varepsilon}$ and $\mathbf{L}^{\varepsilon}$ do not include the square of $(\log \frac{1}{\varepsilon} )^{-1} $ in contrast to the normalizing factor in the definition~(\ref{DEFLocaltime}) of $L^{\varepsilon}$.  It is instructive to define a slight variation of $\ell^{\varepsilon}$ that serves as an intermediary between $\ell^{\varepsilon}$ and $L^{\varepsilon}$.  For $\varepsilon>0$ define the sequences of stopping times $\{ \widetilde{\varrho}_{n}^{\downarrow,\varepsilon} \}_{n\in \mathbb{N}} $ and  $\{ \widetilde{\varrho}_{n}^{\uparrow,\varepsilon} \}_{n\in \mathbb{N}} $  such that $\widetilde{\varrho}_{0}^{\uparrow,\varepsilon}=0$ and for $n\in \mathbb{N}$
\begin{align}\label{VARRHOS2}
\widetilde{\varrho}_{n}^{\downarrow,\varepsilon}\,:=\,&\,\inf\big\{s\in \big(\widetilde{\varrho}_{n-1}^{\uparrow,\varepsilon},\infty\big) \,:\,|X_s| \leq  \varepsilon \big\}\,, \nonumber \\
\widetilde{\varrho}_{n}^{\uparrow,\varepsilon}\,:=\,&\,\inf\big\{s\in \big(\widetilde{\varrho}_{n}^{\downarrow,\varepsilon},\infty\big)\,:\,|X_s| = 2\varepsilon  \big\}\,.
\end{align}
In analogy to $N_{t}^{\varepsilon}$, we define $\widetilde{N}_{t}^{\varepsilon}$ as the number of times $\widetilde{\varrho}_{n}^{\downarrow,\varepsilon}$ in the interval $[0,t]$.
\begin{theorem}\label{THMLocalTime2} Fix some $T,\lambda>0$ and a  Borel probability measure $\mu$  on $\R^2$. If for $\varepsilon\in (0,1)$ we define the process $\{\widetilde{\ell}_t^{\varepsilon}\}_{t\in [0,\infty)}  $  by $ \widetilde{\ell}_t^{\varepsilon}:= \frac{\log 2}{2\log^2\frac{1}{\varepsilon}  }\widetilde{N}_{t}^{\varepsilon} $, then $\mathbf{E}^{T,\lambda}_{\mu}\big[\sup_{0\leq t\leq T}\big|  \widetilde{\ell}_t^{\varepsilon}  -\mathbf{L}_t  \big|    \big]$ vanishes as $\varepsilon \searrow 0$.
\end{theorem}

\subsection{The process inverse of the local time at the origin}\label{SubsectionProcInv}

As before we use $\mathbf{L}$ to denote the local time process, defined on the probability space  $\big( \boldsymbol{\Omega},\mathscr{B}_{\mu}^T,  \mathbf{P}_{\mu}^{T,\lambda} \big)$.  Our focus  will be on characterizing the law of the right-continuous process inverse of $\{\mathbf{L}_t\}_{t\in [0,T]}$.  This  is a one-sided jump process $\{ \boldsymbol{\eta}_s \}_{s\in [0,\infty)} $ with a position-dependent jump-rate  measure, and its rate of  small jumps  approximates that of  the so-called \textit{Dickman subordinator} explored in~\cite{CSZ2}.  For any $s\in [0,\infty)$, we define $\boldsymbol{\eta}_s$ as the $\mathscr{F}^{T,\mu}$-stopping time
\begin{align}\label{ETA}
\boldsymbol{\eta}_s\,:=\,\inf\big\{\, t\in [0,T]  \,:\,\mathbf{L}_t >  s \,\big\} \,, 
\end{align}
where we set $\inf \emptyset :=T$ here. The  restriction of $t$ to the interval $ [0,T]$  in the definition of $\boldsymbol{\eta}$ above is reasonable because the local time process $\mathbf{L}$ is $\mathbf{P}_{\mu}^{T,\lambda }$ almost surely  constant over the time interval $[T,\infty)$.  Since the trajectories of the process $\mathbf{L}$ are increasing, continuous, and plateau at time $T$, the trajectories of  $\{\boldsymbol{\eta}_s\}_{s\in [0,\infty)}$ are increasing, right-continuous, and satisfy $\mathbf{L}_{\boldsymbol{\eta}_s}=s\wedge \mathbf{L}_T$ for all $s\in [0,\infty)$.   The excursion intervals of the coordinate process $X$ away from the origin correspond to time intervals on which  $\mathbf{L} $ does not increase, which translate into jumps for $\boldsymbol{\eta}$.  In fact, the  $[0,T]$-valued process $\boldsymbol{\eta}$ jumps continually until it reaches the terminal ``death state" value $T$ at the random time $\mathbf{L}_T$,  which we can express as $\inf\{ s\in [0,\infty)\,:\, \boldsymbol{\eta}_s=T\}$.  It is to be observed  that $X_{\boldsymbol{\eta}_s}=0$ for all $s\in [0,\mathbf{L}_T)$, because the local time process only has an opportunity to increase when $X$ visits the origin, and $\boldsymbol{\eta}_0=\tau \wedge T$ almost surely $ \mathbf{P}_{\mu}^{T,\lambda}$ in consequence of Proposition~\ref{PropLocalTimeProp}.

We will begin by defining a family of Borel measures $\{\mathscr{J}^{T,\lambda }(a,\cdot) \}_{a\in [0,T]}$  on $[0,T]$ characterizing the instantaneous jump rates for $\boldsymbol{\eta}$, meaning that 
 for $ 0\leq  s <  t <\infty$ and any test function $\phi\in C^{1}\big([0,T]\big)$, 
\begin{align*}
 \lim_{t\searrow s} \frac{\mathbf{E}^{T,\lambda}_{\mu}\big[ \phi(\boldsymbol{\eta}_t)\,\big|\, F_s \big] -\phi(\boldsymbol{\eta}_s  )}{t-s} \,=\,\int_{[0,T]}\, \big(\phi(b)-\phi(\boldsymbol{\eta}_s)  \big) \, \mathscr{J}^{T,\lambda } (\boldsymbol{\eta}_s,db)\quad \textup{a.s.}\,\,\mathbf{P}^{T,\lambda}_{\mu}\,, \hspace{.5cm}  F_s\,:=\,\mathscr{F}_{\boldsymbol{\eta}_s}^{T,\mu}\,.
\end{align*}
Define the measure 
\begin{align}\label{JUMPFORM}
\mathscr{J}^{T,\lambda } (a,db)\,=\,  \frac{1}{b-a}\,\frac{  \nu\big((T-b )\lambda \big) }{  \nu\big((T-a)\lambda \big)  }\,1_{ [a,T) }(b)\, db\,+\,  \frac{1}{\nu\big((T-a) \lambda\big)} \,\delta_{T}(db) \,.
\end{align}
  The weight  $\nu((T-a) \lambda)^{-1}$ that $\mathscr{J}^{T,\lambda}(a,\cdot)$ assigns to the singleton set $\{T\}$ is to be interpreted as  the death rate.  The measure $\mathscr{J}^{T,\lambda}(a,\cdot)$ is absolutely continuous with respect to Lebesgue measure over  $[0,T)$, and as $b\searrow a$  the  density $\frac{\mathscr{J}^{T,\lambda } (a,db)}{db}$ becomes close to $\frac{1}{b-a}$,  the jump rate density for the Dickman subordinator~\cite{CSZ2}.  Notice that the translation   symmetry $ \mathscr{J}^{T,\lambda }(a,db)= \mathscr{J}^{T-a,\lambda }\big(0,d(b-a)\big)  $ holds for $b\in [a,T]$.   Given  $\phi\in C^{1}\big([0,T]\big)$, define $\mathcal{A}^{T,\lambda}\phi$ as the real-valued function on $[0,T]$ with
\begin{align*}
\big(\mathcal{A}^{T,\lambda}\,\phi\big)(a) \,=\,
\int_{[0,T]} \, \big( \phi(b )-\phi( a ) \big)  \, \mathscr{J}^{T, \lambda}(a,db) \,. \nonumber 
\end{align*}

\begin{definition}\label{DefSub}  Fix some $T,\lambda>0$. Let $\{ \eta_s \}_{s\in [0,\infty)}$ be a $[0,T]$-valued, right-continuous process  defined on a probability space $(\Sigma, \mathcal{F},\mathcal{P})$.  Then we refer to $\eta$ as a \textit{parameter $(\lambda,T)$ Volterra jump process} with respect to a filtration  $\{ F_s \}_{s\in [0,\infty)}$ provided that  $\eta$ is  $F$-adapted and for all $\phi\in C^1\big([0,T]\big)$ and $s\in [0,\infty)$
$$\lim_{t\searrow s} \,\frac{\mathcal{E}\big[\phi( \eta_t ) \,\big|\, F_s  \big]\,-\,\phi( \eta_s )}{t-s}\,=\, \big(\mathcal{A}^{T,\lambda}\,\phi\big)(\eta_s) \quad \text{a.s.\  $\mathcal{P}$} \,.$$
\end{definition}
The proof of the theorem below is in Section~\ref{SubsecInvProc}.  Recall that  $\tau:=\inf\big\{t\in [0,\infty)\,\big|\, X_t=0 \big\}$.
\begin{theorem}\label{ThmLocalTimeToLEVY} Fix some $T,\lambda>0$ and a Borel probability measure $\mu$ on $\R^2$.  Under the probability measure $\mathbf{P}_{\mu}^{T,\lambda}$, the process $\{\boldsymbol{\eta}_s\}_{s\in [0,\infty)}$ is a parameter $(\lambda,T)$ Volterra jump process with respect to the filtration $\{F_s\}_{s\in [0,\infty)}$ defined by $ F_s:=\mathscr{F}_{\boldsymbol{\eta}_s}^{T,\mu}$, where the initial position is $\boldsymbol{\eta}_0=\tau \wedge T$.
\end{theorem}

 \subsubsection{The transition kernels for the Volterra jump process}\label{SubsectionDickman}

Next we discuss an explicit form for the transition kernels of a parameter $(\lambda,T)$ Volterra jump process.  For $s,T,\lambda >0$ and $a\in [0,T]$, let $\mathscr{T}^{T,\lambda}_s(a,\cdot)$ denote the Borel measure on $[0,T]$ given by
\begin{align}\label{TFORM}
\mathscr{T}^{T,\lambda}_{s}(a, db)\,=\, \mathcal{D}^{T-a, \lambda}_{ s}(b-a)\,1_{[a ,T
 )}(b)\,db\,+\,\mathbf{p}^{T-a,\lambda}_{s}\,\delta_{T}(db)\,,  
 \end{align}
where  $\mathcal{D}^{t,\lambda}_{s}:[0,t]\rightarrow [0,\infty) $ and $\mathbf{p}^{t,\lambda}_{s}\in [0,1]$ have the forms
 $$ 
 \mathcal{D}^{t,\lambda}_{s}(b)\,:=\,  \frac{ \nu\big((t-b)\lambda  \big)}{ \nu(t\lambda )}\,\frac{b^{s-1}\lambda^s  }{ \Gamma(s) }       \hspace{.7cm} \text{and} \hspace{.7cm} \mathbf{p}^{t,\lambda}_{s}\,:=\,\frac{1}{ \nu(t\lambda )}\,\int_0^s\, \frac{t^r \lambda^r}{\Gamma(r+1)  }\,dr   \,. $$
Naturally,  we set $\mathscr{T}^{T,\lambda}_0(a,\cdot):=\delta_a$.  Part (i) of the proposition  below states the Chapman-Kolmogorov relation for the family of transition kernels $\{ \mathscr{T}^{T,\lambda}_s \}_{s\in [0,\infty)}  $  and part (ii) implies that  a $[0,T]$-valued, right-continuous Markov process  having these transition kernels is a parameter $(\lambda,T)$ Volterra jump process.  We omit the proofs of (i) and (ii), which merely rely on the identity $\int_0^1r^{\alpha-1}(1-r)^{\beta-1}dr=\frac{ \Gamma(\alpha)\Gamma(\beta)  }{\Gamma(\alpha+\beta)  }$ for $\alpha,\beta>0$, and the proof of (iii) is in Section~\ref{SubsecInvProc}. 
\begin{proposition}\label{PropTransMeas} Fix $T,\lambda>0$. 
\begin{enumerate}[(i)]
\item The  measures $\{ \mathscr{T}^{T,\lambda}_s(a,\cdot) \}_{s\in [0,\infty)}^{a\in [0,T] }   $ have total mass one and satisfy  $\int_{[0,T]}\mathscr{T}^{T,\lambda}_{s}(a, db)\mathscr{T}^{T,\lambda}_{t}(b, \cdot)=\mathscr{T}^{T,\lambda}_{s+t}(a, \cdot)$.

\item The family  of  transition kernels $\{ \mathscr{T}^{T,\lambda}_s \}_{s\in [0,\infty)}  $ obeys the  Kolmogorov backward  equation below for any test function $\phi\in C^1\big([0,T]\big)$ and $a\in [0,T]$.
\begin{align*}
\frac{d}{ds} \mathscr{T}^{T,\lambda}_{s}(a, \phi) \,=\, \int_{[0,T]}\,\int_{[0,T]}\, \,\big( \phi(c)-\phi(b)  \big) \,\mathscr{J}^{T, \lambda}(b,dc)   \,\mathscr{T}^{T,\lambda}_{s}(a, db) 
\end{align*}

\item  If $\{ \eta_s \}_{s\in [0,\infty)}$   is a parameter $(\lambda,T)$ Volterra jump process  on some probability space $(\Sigma, \mathcal{F},\mathcal{P})$ and with respect to a filtration  $\{F_s  \}_{s\in [0,\infty)}$, then $\eta$ is a Markovian jump process having transition kernels $\{ \mathscr{T}^{T,\lambda}_s \}_{s\in [0,\infty)}   $ with respect to $\{F_s  \}_{s\in [0,\infty)}$.
\end{enumerate}

\end{proposition}

\subsubsection{The path measures for the Volterra jump process}\label{SubsubsectionNuPathMeasure}

Let $ \boldsymbol{\bar{\Omega}} :=D\big([0,\infty)\big) $ denote the set of real-valued cadlag  functions on $[0,\infty)$, which we equip with the Skorokhod topology.  Let $\{\eta_s\}_{s\in [0,\infty) }$ denote the coordinate process on $ \boldsymbol{\bar{\Omega}}$, that is $\eta_t: \boldsymbol{\bar{\Omega}}\rightarrow \R$ is defined by $\eta_t(\omega)=\omega(t)$.  Furthermore,  let $\{F_s^{\eta}\}_{s\in [0,\infty)}$ be the filtration generated by the coordinate process.   Given a Borel probability measure $\vartheta$ on $[0,T]$, let $\Pi^{T,\lambda}_{\vartheta}$ denote the unique Borel probability measure  on $  \boldsymbol{\bar{\Omega}}$ under which the coordinate process $\eta$  is Markovian with transition kernels  $\{ \mathscr{T}^{T,\lambda}_s\}_{s\in [0,\infty)}  $ and  initial distribution $\vartheta$. The following theorem is  closely related to Theorem~\ref{ThmEquivalent}, and  its proof can be found in Section~\ref{SubsectionThmProcTRAN}.
\begin{theorem}\label{ThmProcTRAN} Fix some $T,\lambda>0$ and  a Borel probability measure $\vartheta$  on $ [0,T]$.  Define the stopping time $\mathbf{S}=\inf\{ s\in [0,\infty)\,:\, \eta_s=T \}$.  For any $\lambda' >0$, the measure $\Pi^{T,\lambda'}_{\vartheta}$ is absolutely continuous with respect to $\Pi^{T,\lambda}_{\vartheta}$, with Radon-Nikodym derivative given by
$$   \frac{ d\Pi^{T,\lambda'}_{\vartheta}  }{ d\Pi^{T,\lambda}_{\vartheta}  }\,=\,  
 \frac{ \nu\big((T-\eta_0) \lambda    \big) }{ \nu\big((T-\eta_0)\lambda'   \big)  }  \bigg(\frac{\lambda'}{\lambda}\bigg)^{\mathbf{S}  }  \,.    $$
\end{theorem}

 \subsubsection{The renewal density for the Volterra jump process}\label{SubsectionRenewal}

The explicit form for the transition kernels  $\{ \mathscr{T}^{T,\lambda}_s\}_{s\in [0,\infty)}   $  of a parameter $(\lambda,T)$ Volterra jump process results in  an explicit form for its renewal density, which we present in the corollary of Proposition~\ref{PropTransMeas} below.  The renewal density  $\mathbf{G}^{T,\lambda}_a:[0,T)\rightarrow [0,\infty)$ of the process $\eta$ starting from $a\in [0,T)$ can be expressed as
\begin{align*}
\mathbf{G}^{T,\lambda}_a(b)\,=\, \int_0^{\infty}\, \frac{\mathscr{T}^{T,\lambda}_{s}(a,db)}{ db }\,ds \,.
\end{align*}
Here, $\mathbf{G}^{T,\lambda}_a$ is the Lebesgue density for the Borel measure $\vartheta^{T,\lambda}_a $ on $ [0,T) $ defined by
\begin{align}\label{RenewalMeasure}
\vartheta^{T,\lambda}_a(E)\,=\,\mathcal{E}_a^{T,\lambda}\Big[\, \textup{meas}\big(\,\big\{s\in [0,\infty)\,:\,  \eta_s\in E  \big\} \,\big) \, \Big] \,,  \hspace{.5cm} E\subset [0,T)      \,.
\end{align}
Thus, $\vartheta^{T,\lambda}_a(E)$ is the expected amount of time that $ \eta$ spends in the set $E$ before reaching the terminal state $T$.
\begin{corollary}\label{CorRenewalForm} Fix some $T,\lambda>0$. The renewal density for a parameter $(\lambda,T)$ Volterra jump process  starting from  $a\in [0,T)$ has the closed form 
\begin{align*}
\mathbf{G}^{T,\lambda}_a(b)\,=\, \lambda \, \nu'\big((b-a)\lambda \big)\,\frac{ \nu\big((T-b) \lambda\big)}{ \nu\big((T-a) \lambda \big)}\,1_{[a,T)} \,, \hspace{.5cm}b\in [0,T)\,.
\end{align*}
\end{corollary}
Formulas for several other probability densities can be readily derived from this renewal density. For instance, the  position $\eta_{ \mathbf{S}-}\in [0,T)$ from which the process $\eta$ jumps at the terminal time $\mathbf{S}:=\inf\big\{s\in [0,\infty)\,:\, \eta_s=T  \big\}$  has probability density
\begin{align*}
\frac{\mathcal{P}_a^{T,\lambda}[ \eta_{ \mathbf{S}-} \in d b  ]}{d b }\,=\,  \frac{\mathbf{G}^{T,\lambda}_a(b)}{ \nu\big((T-b)  \lambda \big)}  \,=\,\lambda \,\frac{  \nu'\big((b-a) \lambda\big) }{\nu\big((T-a) \lambda \big)   } \,1_{[a,T)} \,, \hspace{.5cm}b\in [0,T) \,,
\end{align*}
in which the first equality uses that the process $\eta_s$ has death rate $ \nu((T-b)\lambda )^{-1}    $ whilst $\eta_s=b$.  Another example is given in the following lemma, which we prove in  Section~\ref{SubSecSomeTimes}, concerning the distribution of the landing point for  the process  $\eta$ after the first time that it jumps over  $\varepsilon\in (0,T)$.
\begin{lemma}\label{LemmaEscape}Fix some $T,\lambda>0$. Let  $\{ \eta_s \}_{s\in [0,\infty)}$ be a parameter $(\lambda,T)$ Volterra jump process with respect to a filtration  $\{  F_s\}_{s\in [0,\infty)}$ on a probability space  $(\Sigma, \mathcal{F},\mathcal{P})$.  Assume that $ \eta_0=0$ almost surely.  For $\varepsilon>0$  define the stopping time $ \varpi^{\varepsilon}:=\inf\{s\in [0,\infty) \,:\,  \eta_s  \geq  \varepsilon  \}$. When $\varepsilon\in (0,T)$ the  random variable $\eta_{ \varpi^{\varepsilon} }$ takes the value $T$ with probability  $\frac{ \nu(\varepsilon  \lambda )}{ \nu(T\lambda )} $ and has the following Lebesgue density over the interval $[0,T)$:
\begin{align*}
\frac{\mathcal{P}\big[\eta_{ \varpi^{\varepsilon} } \in db\big]}{db}\,=\,  1_{ (\varepsilon,T]}(b) \,\frac{ \nu\big((T-b) \lambda\big)}{ \nu(T\lambda )}\, \int_{0}^{\varepsilon}\,\lambda \,\nu'(a\lambda )\,\frac{1}{b-a}\,da\,\,, \hspace{.5cm}b\in [0,T)\,.
\end{align*}
\end{lemma}

\begin{remark}\label{RemarkVartheta}  The fact that the above representation for the distribution of $\eta_{ \varpi^{\varepsilon} }$ is a  probability measure relies on the identity below, which reduces to~(\ref{DoubleNuPrime}) through differentiating in $T$ since $\nu(0)=0$.
$$  \int_{0}^{\varepsilon} \,\int_{\varepsilon}^{T}\,\lambda \,\nu'(a \lambda )\,\frac{1}{b-a}\,\nu\big( (T-b)\lambda \big)\,db\,da\,=\,\nu(T\lambda )\,-\,\nu(\varepsilon\lambda )  $$
\end{remark}

\subsection{Measuring the set of visitation times to the origin}\label{SubsectionOriginSet}

 For $\omega \in \boldsymbol{\Omega}$ define  $\mathscr{O}(\omega):=\{t\in [0,\infty)\,:\, X_t(\omega)=0\} $, in other terms, as the set of times that the process $X$ visits the origin. We will discuss some properties of $\mathscr{O}\equiv\mathscr{O}(\omega)$ as a random subset of $[0,\infty)$, defined on the probability space $\big( \boldsymbol{\Omega},\mathscr{B}_{\mu}^T,  \mathbf{P}_{\mu}^{T,\lambda} \big)$. Notice that part (i) of Theorem~\ref{CorSubMART} implies that $\mathbf{P}_{\mu}^{T,\lambda}[\mathscr{O}=\emptyset  ]=\int_{\R^2}\frac{1}{1+H_T^{\lambda}(x)  } \mu(dx)  $  since $\mathcal{O}^c=\{\mathscr{O}=\emptyset\}$.  We prove the results stated below in Section~\ref{SectionHausdorff}.

\begin{proposition}\label{PropZeroSetBasics}  Fix some $T,\lambda>0$ and a Borel probability measure $\mu$ on $\R^2$.
Let $\vartheta \equiv\vartheta(\omega, \cdot )$ denote the random Borel measure on $[0,\infty) $ having distribution function $t\mapsto \mathbf{L}_t(\omega)$. The following statements hold for $\mathbf{P}_{\mu}^{T,\lambda}$ almost every $\omega\in \boldsymbol{\Omega}$.
\begin{enumerate}[(i)]

\item The set $\mathscr{O}(\omega)$ is \ uncountable when $\omega\in \mathcal{O}$.

\item The set $\mathscr{O}(\omega)$   has Hausdorff dimension $0$.

\item The measure $\vartheta(\omega,\cdot)$ takes full weight  on $\mathscr{O}(\omega)$.

 \end{enumerate}
 
\end{proposition}

Parts (i) \& (ii) of Proposition~\ref{PropZeroSetBasics} offer a preliminary notion for the typical size of the random  set $\mathscr{O}(\omega)$ under $\mathbf{P}_{\mu}^{T,\lambda}$ in the event that $\mathscr{O}(\omega)\neq \emptyset$.   We can formulate a more refined characterization for the size of the set $\mathscr{O}(\omega)$
using a  generalized Hausdorff dimension formalism, capable of distinguishing between  the sizes of some nontrivial Hausdorff dimension zero sets.   Given  $h\in (0,\infty)$, define $d_{h}:[0,\infty)\rightarrow [0,\infty) $ such that $d_{h}(0)=0$ and   
$d_{h}(a)= (1+\log^+ \frac{1}{a})^{-h}    $ for $a\in (0,\infty)$. We note that $d_{h}$ is a so-called \textit{dimension zero} function because $a^{\alpha} \ll  d_{h}(a)$ as $a\searrow 0$ for every $\alpha>0$. Let $H_{h}$ denote the outer measure on $\R$ that assigns a set $S\subset \R $ the value
\begin{align}\label{LogHausMeas}
H_{h}(S)\,:=\,\lim_{\delta\searrow 0}\,H_{\delta,h}(S)  \hspace{.7cm}\text{for}\hspace{.7cm} H_{\delta,h}(S)\, :=\,\inf_{\substack{ S\subset \cup_n I_n  \\  |I_n|\leq \delta  }}\, \sum_{n}\,d_{h}\big(|I_n|\big) \,,  
\end{align}
wherein $\{I_n\}_{n\in \mathbb{N}} \subset \mathcal{P}(\R)$ is a covering of $S$, and $|I_n|$ denotes the diameter of $I_n$.  We will refer to the set $S$ as having \textit{log-Hausdorff exponent} $\mathfrak{h}\in [0,\infty]$ if
\begin{align}\label{LogHausExp}
\mathfrak{h}\,:=\,&\,\sup \big\{\, h\in [0,\infty)\, :\,  H_{h}(S)=\infty \,\big\}\,, \, \text{or equivalently} \nonumber \\   \mathfrak{h}\,:=\,&\, \inf \big\{\, h\in [0,\infty)\, :\,  H_{h}(S)=0 \,\big\}\,. 
\end{align}
 Naturally, the log-Hausdorff exponent is only meaningful for characterizing the size of sets with Hausdorff dimension zero.

\begin{theorem}\label{ThmHAUSDORF} Fix  some $T,\lambda>0$ and a Borel probability measure $\mu$ on $\R^2$.  The set $\mathscr{O}(\omega)$ has log-Hausdorff exponent one for  $\mathbf{P}_{\mu}^{T,\lambda }$ almost every  $\omega\in \mathcal{O}$.  
\end{theorem}


\subsection{Summary outline of the remaining  sections}\label{SubsectionOutline}

To bring the reader more quickly to our topic of interest, which is the anomalous behavior of the coordinate process $X$ near the origin under the law $\mathbf{P}^{T,\lambda}_{\mu}$, we have relegated the proofs of the  technical propositions in Sections~\ref{SubSecTransProb} \&~\ref{SubsectionStochdDiff} to  Appendices~\ref{AppendixYMart}--\ref{AppendixWeakSolConst}.  The following is  a summary of the content of later sections:
\begin{itemize}
    \item We prove Theorem~\ref{CorSubMART} in Section~\ref{SectionOriginEvent}.

    \item  Section~\ref{SectionGirsanov} contains the proof of a weaker version of Theorem~\ref{ThmExistenceLocalTime}, which shows that the path measures $\mathbf{P}^{T,\lambda}_{\mu}$ and $\mathbf{P}^{T,\lambda'}_{\mu}$ are mutually absolutely continuous without explicitly relating the Radon-Nikodym derivative to the local time. The proof relies on Girsanov's theorem.

    \item  The  results stated in  Section~\ref{SubsectionLocalTime} concerning the local time at the origin are proved in  Section~\ref{SectionLocalTime}, where the order of  presentation is  different.  We begin with a Tanaka-type construction of the local time process, as the increasing component in the Doob-Meyer decomposition of a certain submartingale, and we then show that this construction can be approximated using the processes defined through  downcrossings  in Theorems~\ref{THMLocalTime1} \&~\ref{THMLocalTime2}. The latter of these is the starting point in our proof of a generalization of Theorem~\ref{ThmExistenceLocalTime}.

    \item  The proofs of the results presented in  Sections~\ref{SubsectionProcInv} \& \ref{SubsectionOriginSet} are respectively proved in Sections~\ref{SectionRightCont} \&~\ref{SectionHausdorff}.

    \item The purpose of Section~\ref{SectionBasics} is to collect approximations and bounds for  $H_{T}^{\lambda}(x)$, the drift function $b_t^{\lambda}(x)$, and the kernel $h_t^{\lambda}(x,y)$.

    \item  Section~\ref{SectionMiscProofs} contains the proofs of several technical lemmas stated in Sections~\ref{SectionOriginEvent}--\ref{SectionHausdorff}.  Many of these proofs  require the estimates from Section~\ref{SectionBasics}.  

\end{itemize}
In   Sections~\ref{SectionOriginEvent}--\ref{SectionRightCont} and Appendices~\ref{AppendixYMart}--\ref{AppendixSMP}, we begin our analysis by introducing a key martingale  (in some cases, a one-parameter family thereof) or a submartingale whose increasing Doob-Meyer component satisfies a special property, such as remaining constant during excursions of the coordinate process $X$ away from the origin.  We usually confine our analysis in proofs to the time interval $[0,T]$  since the shifted process $\{X_{T +t}\}_{t\in [0,\infty)} $ is a  two-dimensional Brownian motion under  $\mathbf{P}_x^{T,\lambda}$ with respect to the filtration $\{ \mathscr{F}^{T,x}_{T+t}\}_{t\in [0,\infty)}$, thus nullifying all of our quantities of interest.

Let us  clarify the meaning of a notation used frequently in some sections.  For  nonnegative functions $f(z)$ and $g(z)$, we  write $f\preceq g$  if $f(z)$ is bounded by a constant multiple of $g(z)$ for a range of the variable $z$ that is  to be clear from context, say from the statement of the proposition being invoked.  Usually, $z$ is a tuple for which some of its components are restricted to a bounded range.

\section{The probability of visiting the  origin} \label{SectionOriginEvent}

We will now direct our focus to proving Theorem~\ref{CorSubMART}.  A natural tool for this purpose is the $\mathbf{P}_{\mu}^{T,\lambda}$-submartingale $\{\mathcal{S}^{T,\lambda}_t\}_{t\in [0,\infty)}$ given by
$ \mathcal{S}^{T,\lambda}_t:=\mathfrak{p}_{T-t}^{\lambda}(X_t)$ for the function $\mathfrak{p}_{t}^{\lambda}:\R^2\rightarrow [0,1]  $ with $\mathfrak{p}_{t}^{\lambda}(x):=(1+H_{t}^{\lambda}( x) )^{-1}$,  which satisfies the PDE
\begin{align}\label{PartialForP}
\frac{\partial}{\partial t}\,\mathfrak{p}_{t}^{\lambda}(x)\,=\,  \frac{1}{2}\, \Delta_x \,\mathfrak{p}_{t}^{\lambda}(x)\,+\, b_t^{\lambda}(x)\cdot (\nabla_x\,\mathfrak{p}_{t}^{\lambda})(x)\,, \hspace{.5cm} x\in \R^2\backslash \{0\}\,,
\end{align}
since $\frac{\partial}{\partial t} H_{t}^{\lambda}( x)= \frac{1}{2} \Delta_x H_{t}^{\lambda}( x)$.
Note that, based on the claim in (i) of Theorem~\ref{CorSubMART} and the Markov property, $ \mathcal{S}^{T,\lambda}_t$ is the probability   conditional on $\mathscr{F}^{X}_t$ that the process $X$ does not visit  the origin after time $t$.  The increasing component $ \mathcal{A}^{T,\lambda}$ in the Doob-Meyer decomposition of  $ \mathcal{S}^{T,\lambda}$ increases only when $X$ visits the origin, meaning that for $\mathbf{P}_{\mu}^{T,\lambda}$-almost every $\omega\in \boldsymbol{\Omega}  $ the  measure $\vartheta(\omega, \cdot )$ on $[0,\infty)$ with distribution function $t\mapsto  \mathcal{A}^{T,\lambda}_t(\omega) $ is  supported on $\mathscr{O}(\omega):=\big\{t\in [0,\infty)\,\big|\,X_t(\omega)=0  \big\}$. Said differently, the process $ \mathcal{S}^{T,\lambda}$ behaves as a martingale during the  excursions of $X$ away from the origin.   Although $\mathcal{A}^{T,\lambda} $ can be used to construct the local time process $\mathbf{L}$ through the integral
$
\mathbf{L}_t=\int_0^{t\wedge T} \nu\big( (T-r)\lambda\big) d\mathcal{A}^{T,\lambda}_r
$,
there is a similar but tidier construction of $ \mathbf{L}$ in Section~\ref{SubsectionTanaka}.

\subsection{A submartingale related  to the probability of avoiding the origin}

The lemma below provides an estimate for the expected amount of time (under the law $\mathbf{P}_{\mu}^{T,\lambda}$)  that the process $X$ takes to leave a small region around the origin  when starting from the origin. In addition, a bound for the higher integer moments of this passage time is provided.  The same results hold for a two-dimensional Brownian motion, so we are merely showing that the presence of the drift $b_{T-t}^{\lambda}(X_t)$, although singular and pointing towards the origin, does not significantly prolong the typical time taken for the process to escape a small region containing the origin.  The proof is  in Section~\ref{SubsectionLemmaLeave}. 
\begin{lemma}\label{LemmaLeave} Given $\varepsilon>0$  define the $\mathscr{F}^X$-stopping time    $\varrho^{\uparrow, \varepsilon}:= \inf\{t\in [0,\infty)\,:\, |X_t|\geq \varepsilon \}$. 
\begin{enumerate}[(i)]
\item  For any $L\in (0,\infty)$ and $m\in \mathbb{N}$, there exists a  $C_{L,m}>0$ such that   for all $T,\lambda \in (0,L]$,  $\varepsilon>0$, and  $x\in \R^2$
$$   \mathbf{E}_{x}^{T,\lambda}\big[ \,(\varrho^{\uparrow ,\varepsilon})^m \,\big]\,\leq \, C_{L,m}\,\varepsilon^{2m}   \,. $$

\item For any $L\in (1,\infty)$, there exists a  $C_{L}>0$ such that for all $T,\lambda\in \big[\frac{1}{L},L\big] $ and $\varepsilon \in (0,\frac{1}{2}]$
$$  \bigg| \mathbf{E}_{0}^{T,\lambda}\big[ \,\varrho^{\uparrow, \varepsilon} \,\big]\,-\,\frac{1}{2}\,\varepsilon^2 \bigg|\,\leq \, C_{L}\,\frac{\varepsilon^2}{\log \frac{1}{\varepsilon} }\,.  $$

\end{enumerate}

\end{lemma}

We use the next  lemma to verify that certain local martingales that we construct through stochastic integrals are square-integrable and, in particular, martingales.  The proof is placed in Section~\ref{SubsectionLemmaKbounds1}.  As above, we set $\mathfrak{p}_{t}^{\lambda}(x):= ( 1+H_{t}^{\lambda}( x) )^{-1}$ for  $t,\lambda > 0$  and $x\in \R^2$. 
\begin{lemma} \label{LemmaKbounds1}  Given any $L>1$, there exists a  $C_L>0$ such that for  all $x\in \R^2$ and  $T,\lambda >0$ with $ T\lambda \in \big[\frac{1}{L}, L\big]$ we have 
\begin{align}\label{SupSquar}
  \int_0^{T}\,\int_{\R^2} \, \mathlarger{d}_{0,s}^{T,\lambda}(x,y)\, \big|V_{T-s}^{\lambda}(y)\big|^2 \,   dy \,  ds \,<\,C_L \, ,
\end{align}
where   $V_t^{\lambda}(y):=\mathfrak{p}_t^{\lambda}(y)b_t^{\lambda}(y) $.
Moreover, the same result holds with $V_t^{\lambda}(y)$ replaced by $\mathring{V}_{t}^{ \lambda }(y):=  \lambda\frac{\partial }{ \partial \lambda } b_t^{\lambda}(y) $ or by $V_t^{\lambda,\lambda'}(y):=b_t^{\lambda'}(y)-b_t^{\lambda}(y)$ for all $\lambda' > 0$ with $ T\lambda' \in \big[\frac{1}{L}, L\big]$.
\end{lemma}

The proposition below defines a submartingale $\mathcal{A}^{T,\lambda}$ that is closely related to the local time of the process $X$ under $\mathbf{P}_{x  }^{T,\lambda}$, as indicated above. Since $H_T^{\lambda}(x)$ is radially symmetric in the variable $x\in \R^2$, there is a function $\bar{H}_T^{\lambda}:[0,\infty)\rightarrow [0,\infty]$ satisfying $\bar{H}_T^{\lambda}(|x|)=H_T^{\lambda}(x)$. The value is $\bar{H}_T^{\lambda}(a)$ is increasing in $T$, decreasing in $a$, and blows up as $\bar{H}_T^{\lambda}(a)\sim 2\nu(T\lambda)\log \frac{1}{a}$ with small $a$ by (i) of Proposition~\ref{PropK}. 
\begin{proposition} \label{PropSubMart}  Fix $T,\lambda > 0$ and a Borel probability measure $\mu$  on $\R^2 $. The process $\{\mathcal{S}_{t}^{T,\lambda}\}_{t\in [0,\infty)}$ defined by $\mathcal{S}_{t}^{T,\lambda}:=\mathfrak{p}_{T-t}^{\lambda}(X_t)$ is a bounded, continuous $\mathbf{P}^{T,\lambda}_{\mu}$-submartingale with respect to $\mathscr{F}^{T,\mu}$ for which the martingale $\mathcal{M}^{T,\lambda}$ and increasing $\mathcal{A}^{T,\lambda}$ components in its Doob-Meyer decomposition satisfy the following:
\begin{itemize}
    \item  $\mathcal{M}_{t}^{T,\lambda}:=\mathfrak{p}_{T}^{\lambda}(X_0)\,-\,\int_0^t\mathcal{S}_{s}^{T,\lambda}\,b_{T-s}^{\lambda}(X_s) \cdot dW_{s}^{T,\lambda}$ is in $L^2\big(\mathbf{P}^{T,\lambda}_{\mu}\big)$.

    \item  $\mathcal{A}^{T,\lambda}$ is constant during the excursions of $X$ away from $0$, that is
$  \int_0^{\infty}1_{X_t\neq 0}\,d\mathcal{A}_{t}^{T,\lambda}=0  $  almost surely $\mathbf{P}_{\mu }^{T,\lambda}$.
\end{itemize}
\end{proposition}

\begin{proof}   The continuous process $\mathcal{M}^{T,\lambda}$ defined by the It\^o integral in the first bullet point is a square-integrable martingale provided that
\begin{align}\label{DipDap}
 \mathbf{E}^{T,\lambda}_{\mu}\bigg[\,\int_0^{T} \,\big(\mathcal{S}_{s}^{T,\lambda}\big)^2\,\big|b_{T-s}^{\lambda}(X_s)\big|^2 \, ds\,\bigg]  \,<\,\infty\,.
\end{align}
 However, since $\mathcal{S}_{s}^{T,\lambda} =\mathfrak{p}_{T-s}^{\lambda}(X_s)  $, we can express the above expectation in terms of the transition  density function $\mathlarger{d}_{s,t}^{T,\lambda}(x,y)$ as
\begin{align*}
  \int_{\R^2} \, \int_0^{T}\,\int_{\R^2}\,  \mathlarger{d}_{0,s}^{T,\lambda}(x,y) \,\big|\mathfrak{p}_{T-s}^{\lambda}(y)\,b_{T-s}^{\lambda}(y)  \big|^2 \,  dy  \, ds \,  \mu(dx)  \,,
\end{align*}
which is finite by  Lemma~\ref{LemmaKbounds1}.  Notice that since $X$ satisfies the SDE $dX_t=dW_{t}^{T,\lambda} + b_{T-t}^{\lambda}(X_t)  dt$ by Proposition~\ref{PropStochPre},  It\^o's formula yields that when $X_t \neq 0$ 
\begin{align}\label{BooT}
d\mathcal{S}_{t}^{T,\lambda}
\,=\,&\, \Big(\frac{\partial}{\partial t}\,\mathfrak{p}_{T-t}^{\lambda}\Big)(X_t)\,dt \,+\,\big(\nabla_x\, \mathfrak{p}_{T-t}^{\lambda}\big)(X_t)\cdot \Big( dW_{t}^{T,\lambda}\,+ \,b_{T-t}^{\lambda}(X_t) \, dt \Big)\,+\,\frac{1}{2}\big(\Delta_x\, \mathfrak{p}_{T-t}^{\lambda}\big)(X_t)\, dt \nonumber  \\ \,=\,&\,-\mathcal{S}_{t}^{T,\lambda}\,b_{T-t}^{\lambda}(X_t) \cdot dW_{t}^{T,\lambda} \,=:\, d\mathcal{M}_{t}^{T,\lambda}  \,,
\end{align}
where the second equality follows from~(\ref{PartialForP}) and $\nabla_x \mathfrak{p}_{t}^{\lambda}(x)=-  \mathfrak{p}_{t}^{\lambda}(x)b_t^{\lambda}(x) $. Thus, $1_{X_t\neq 0}\,d\mathcal{S}_{t}^{T,\lambda} =1_{X_t\neq 0}\,d\mathcal{M}_{t}^{T,\lambda} $ under $\mathbf{P}^{T,\lambda}_{\mu}$.
 
We must confirm that the process $\mathcal{A}^{T,\lambda}:=\mathcal{S}^{T,\lambda}-\mathcal{M}_{t}^{T,\lambda}$ is increasing and has the property in the second bullet point.
Given $\varepsilon>0$ we will write the process $\mathcal{S}^{T,\lambda}$   as a sum of a martingale  $\mathcal{M}^{T,\lambda,\varepsilon}$ and a process $\mathcal{A}^{T,\lambda,\varepsilon}$ that can  suffer only a small decrease  from its previous running maximum when $\varepsilon\ll 1$.   Let $\{ \varrho_{n}^{\downarrow,\varepsilon} \}_{n\in \mathbb{N}}$ and $\{ \varrho_{n}^{\uparrow,\varepsilon} \}_{n\in \mathbb{N}_0}$ be the sequences of stopping times defined in~(\ref{VARRHOS}).  We have that $\mathcal{S}_{t}^{T,\lambda}=\mathcal{M}_{t}^{T,\lambda,\varepsilon}+\mathcal{A}_{t}^{T,\lambda,\varepsilon}$ for
\begin{align}\label{AaSum}
\mathcal{M}_{t}^{T,\lambda,\varepsilon}\,:=\,&\,\mathfrak{p}_{T}^{\lambda}(X_0)\,+\,\sum_{n=1}^{\infty}\, \mathcal{S}_{t\wedge \varrho_{n}^{\downarrow,\varepsilon}}^{T,\lambda}\,-\,\mathcal{S}_{t\wedge\varrho_{n-1}^{\uparrow,\varepsilon}}^{T,\lambda} \,, \nonumber \\
\mathcal{A}_{t}^{T,\lambda,\varepsilon}\,:=\,&\,\sum_{n=1}^{\infty}\,\mathcal{S}_{t\wedge \varrho_{n}^{\uparrow,\varepsilon}}^{T,\lambda} \,-\,  \underbracket{\mathcal{S}_{t\wedge \varrho_{n}^{\downarrow,\varepsilon}}^{T,\lambda}}\,=\,\sum_{n=1}^{\infty}\,\underbrace{\mathcal{S}_{t\wedge \varrho_{n}^{\uparrow,\varepsilon}}^{T,\lambda}}_{\leq \, \frac{1}{1+\bar{H}_{T-t}^{\lambda}(\varepsilon)}} \hspace{-.3cm} \, 1_{ \varrho_{n}^{\downarrow,\varepsilon}\leq t } \,\,,
\end{align}
where the last equality uses that the underbracketed term is $0 $ when $\varrho_{n}^{\downarrow,\varepsilon}< t$.  We will  argue below that $\mathcal{M}^{T,\lambda,\varepsilon}$ is a martingale for which $\sup_{0\leq t\leq T}\big|\mathcal{M}^{T,\lambda,\varepsilon}_t  -\mathcal{M}^{T,\lambda}_t  \big|  $ vanishes in  $L^2\big( \mathbf{P}_{\mu}^{T,\lambda} \big)$ as $\varepsilon\searrow 0$, and  it follows therefrom that $\sup_{0\leq t\leq T}\big|\mathcal{A}^{T,\lambda,\varepsilon}_t  -\mathcal{A}^{T,\lambda}_t  \big|  $ vanishes in $L^2\big( \mathbf{P}_{\mu}^{T,\lambda} \big)$ since $\mathcal{M}^{T,\lambda,\varepsilon}+\mathcal{A}^{T,\lambda,\varepsilon}=\mathcal{M}^{T,\lambda}+\mathcal{A}^{T,\lambda}$. Note that  the form~(\ref{AaSum}) implies that a regression for $\mathcal{A}^{T,\lambda,\varepsilon}$  is possible in the scenario that $|X|$ grows close to $\varepsilon$ over an upcrossing interval $[ \varrho_{n}^{\downarrow,\varepsilon}, \varrho_{n}^{\uparrow,\varepsilon})$ only to then pull back towards the origin again.  Fix some $\delta\in (0,T)$.  Over the time interval $[0,T-\delta]$,  the process $\mathcal{A}^{T,\lambda,\varepsilon}$ can decrease from its previous running maximum by at most $ (1+\bar{H}_\delta^{\lambda}(\varepsilon))^{-1}$.    Since $ (1+\bar{H}_\delta^{\lambda}(\varepsilon))^{-1}$ has order $\log^{-1}\frac{1}{\varepsilon}  $ with small $\varepsilon$ (Proposition~\ref{PropK}),  it follows that  $\mathcal{A}^{T,\lambda}$ is increasing over the interval $[0,T-\delta)$, and thus over $[0,T]$ because $\delta\in (0,T)$ is arbitrary and  $\mathcal{A}^{T,\lambda}$ is continuous.  Moreover, since $\mathcal{A}^{T,\lambda}=\mathcal{S}^{T,\lambda}-\mathcal{M}^{T,\lambda}$
and $1_{X_s\neq 0}\,d\mathcal{S}^{T,\lambda}_s =1_{X_s\neq 0}\,d\mathcal{M}^{T,\lambda}_s$  by~(\ref{BooT}), we get that $  \int_0^T\,  1_{X_s\neq 0}\,  d\mathcal{A}^{T,\lambda}_s=0$ holds almost surely $\mathbf{P}^{T,\lambda}_{\mu}$.

It only remains to verify that the difference between $\mathcal{M}^{T,\lambda,\varepsilon}$ and $\mathcal{M}^{T,\lambda}$ vanishes with small $\varepsilon$.  Define $\{ \Lambda^{\varepsilon}_s\}_{s\in [0,\infty)}$ as the  process  such that $\Lambda^{\varepsilon}_s=1$ when $s\in \big[ \varrho_{n-1}^{\uparrow,\varepsilon} ,  \varrho_{n}^{\downarrow,\varepsilon}\big)$ for some $n\in \mathbb{N}$ and $\Lambda^{\varepsilon}_s(\omega)=0$ otherwise. Then $\Lambda^{\varepsilon}$ is $\mathscr{F}^{T,\mu}$-adapted, and we can write 
\begin{align*}
\mathcal{M}_{t}^{T,\lambda,\varepsilon}\,=\,\mathfrak{p}_{T}^{\lambda}(X_0)\,+\,\int_0^t \, \Lambda^{\varepsilon}_s \,  d\mathcal{S}^{T,\lambda}_s \,\stackrel{ (\ref{BooT}) }{=}\,\mathfrak{p}_{T}^{\lambda}(X_0)\,-\,\int_0^t  \, \Lambda^{\varepsilon}_s \,\mathcal{S}_{s}^{T,\lambda}\,b_{T-s}^{\lambda}(X_s)  \cdot dW_{s}^{T,\lambda}\,.
\end{align*}
Then $\mathcal{M}^{T,\lambda,\varepsilon}$ is a square-integrable martingale as a consequence of~(\ref{DipDap}), and Doob's inequality combined with It\^o's isometry yields 
\begin{align*}
 \mathbf{E}^{T,\lambda}_{\mu}\bigg[ \,\sup_{t\in [0,T]}\,\Big|\mathcal{M}_{t}^{T,\lambda,\varepsilon}-\mathcal{M}_{t}^{T,\lambda}\Big|^2  \,\bigg]\,\leq \, &\,4 \,\mathbf{E}^{T,\lambda}_{\mu}\bigg[ \,\Big|\mathcal{M}_{T}^{T,\lambda,\varepsilon}-\mathcal{M}_{T}^{T,\lambda}\Big|^2  \,\bigg]\nonumber \\
 \,= \, &\,4\,\mathbf{E}^{T,\lambda}_{\mu}\Bigg[\, \int_0^T \big(1-  \Lambda^{\varepsilon}_s\big) \,\big(\mathcal{S}_{s}^{T,\lambda}\big)^2\,\big|b_{T-s}^{\lambda}(X_s)\big|^2  \,ds \,\Bigg] \,.
 \end{align*}
 Using that $1-  \Lambda^{\varepsilon}_s(\omega)\leq 1_{|X_s|\leq \varepsilon }  $, we can bound the above by the integral
 \begin{align*}
  \int_{\R^2} \,\int_0^{T}\,\int_{\R^2} \,1_{|y|\leq \varepsilon} \, \mathlarger{d}_{0,s}^{T,\lambda}(x,y)\,  \big|\mathfrak{p}_{T-s}^{\lambda}(y)\,b_{T-s}^{\lambda}(y)  \big|^2\,dy\, ds\, \mu(dx) \,,  \nonumber 
\end{align*}
which vanishes with small $\varepsilon$  by Lemma~\ref{LemmaKbounds1} and the dominated convergence theorem.
\end{proof}

\subsection{Proof of Theorem~\ref{CorSubMART}}\label{SubsectionCorSubMART}

Recall that   $\mathcal{O}\in \mathscr{B}(\boldsymbol{\Omega})$ denotes the event that  $X$ visits the origin at some time, and  $\tau$ denotes the first time that this happens, with $\tau:=\infty$ on $\mathcal{O}^c$.

\begin{proof} Part (i):  By virtue of Proposition~\ref{PropSubMart}, the increasing component $\mathcal{A}^{T,\lambda}$ of the Doob-Meyer decomposition of the submartingale $\mathcal{S}^{T,\lambda}$ does not change over time intervals during which the process $X$ is away from the origin. Thus, since the process $\mathcal{A}^{T,\lambda}$ is continuous and  $\mathcal{A}^{T,\lambda}_0=0$, we have $\mathcal{A}_{t}^{T,\lambda}=0$ for all $t\in [0,\tau]$.   Applying the optional stopping theorem to the  martingale $\mathcal{M}^{T,\lambda}=\mathcal{S}^{T,\lambda} -\mathcal{A}^{T,\lambda} $ with the stopping time $\tau \wedge T$  yields that
\begin{align}\label{UsingOST}
\mathfrak{p}_{T}^{\lambda}( x) =  \mathbf{E}_{x  }^{T,\lambda}\Big[\, \mathcal{S}_{0}^{T,\lambda}\, \Big]=  \mathbf{E}_{x  }^{T,\lambda}\Big[\, \mathcal{S}_{0}^{T,\lambda} -\mathcal{A}_{0}^{T,\lambda} \,\Big]=  \mathbf{E}_{x  }^{T,\lambda}\Big[\, \mathcal{S}_{\tau \wedge T}^{T,\lambda} -\mathcal{A}_{\tau \wedge T}^{T,\lambda} \,\Big] =\mathbf{E}_{x }^{T,\lambda }\Big[ \,\mathcal{S}_{\tau \wedge T}^{T,\lambda} \,\Big]\,.  
\end{align}
Observe that since $H_{0}^{\lambda}(x)=0$, we have $ \mathcal{S}_{\tau \wedge T}^{T,\lambda}:=(  1+H_{T-\tau \wedge T}^{\lambda}( X_{\tau \wedge T})  )^{-1}=1$ in the event that $\tau\geq T$.  On the other hand, $ \mathcal{S}_{\tau \wedge T}^{T,\lambda}=0$ in the event that $\tau < T$.  We can thus write
\begin{align}\label{1OComp}
\mathcal{S}_{\tau \wedge T}^{T,\lambda}\,=\, 1_{\tau\geq T}\,=\, 1_{\tau=\infty}  \,=\,1_{\mathcal{O}^c}\quad \text{a.s. }\,\mathbf{P}_{x }^{T,\lambda }\,, 
\end{align}
where the   second equality uses that the event $\tau\in [T,\infty)$ is $\mathbf{P}_{x }^{T,\lambda }$-null. This holds since the process $\{X_{T+t}  \}_{t\in [0,\infty)}$ has the law of a two-dimensional Brownian motion with initial density $  \mathlarger{d}_{0,T}^{T,\lambda}(x,\cdot)$ under $\mathbf{P}_{x }^{T,\lambda }$.  Plugging~(\ref{1OComp}) into~(\ref{UsingOST}) completes the proof.\vspace{.3cm}

\noindent  Part (ii):   It suffices to verify the conditions of L\'evy's theorem for the coordinate process $X$ under $\mathbf{\widetilde{P}}_{x}^{T,\lambda}$, which we formulate as (I)--(II) below.  
\begin{enumerate}[(I)]

\item $\{X_t\}_{t\in [0,T]}$ is a  $\mathbf{\widetilde{P}}_{x}^{T,\lambda}$-martingale.

\item $\big\{ (v\cdot X_t)^2-t \big\}_{t\in [0,T]}$ is a $\mathbf{\widetilde{P}}_{x}^{T,\lambda}$-martingale  for any unit vector $v\in \R^2$.

\end{enumerate}
We show (I)--(II) through standard arguments using Bayes' rule and It\^o calculus. Let $\{M_{t}^{T,\lambda}\}_{t\in [0,T]}$ denote the bounded stopped process defined by  $M_{t}^{T,\lambda}:=\big( 
1+H_{T}^{\lambda }(x)   \big)\mathcal{S}_{\tau\wedge t}^{T,\lambda} $. 
Observe that   $M_{T }^{T,\lambda}  =\frac{1}{\mathbf{P}_{x }^{T,\lambda }[\mathcal{O}^c  ]  }1_{ \mathcal{O}^c} $ holds almost surely $\mathbf{P}_{x }^{T,\lambda }$ by part (i) and~(\ref{1OComp}). We thus  have
\begin{align*}
 \mathbf{\widetilde{P}}_{x }^{T,\lambda }\,=\,M_{T}^{T,\lambda}\,\mathbf{P}_{x  }^{T,\lambda} \,, \end{align*}
 meaning that $M_{T}^{T,\lambda}$ is the Radon-Nikodym derivative of $\mathbf{\widetilde{P}}_{x }^{T,\lambda }$ with respect to $\mathbf{P}_{x  }^{T,\lambda} $.  It follows from Proposition~\ref{PropSubMart} and the optional stopping theorem that $M^{T,\lambda}$ is a mean-one $\mathbf{P}_{x  }^{T,\lambda}$-martingale.  Moreover, by~(\ref{BooT})  
\begin{align}\label{Emmy}
dM_{t}^{T,\lambda}\,=\,\big( 
1+H_{T}^{\lambda }(x)   \big)\,d\mathcal{S}_{\tau\wedge t}^{T,\lambda} \,=\,&\,-1_{\tau > t  } \,\big( 
1+H_{T}^{\lambda }(x)   \big)\, \mathcal{S}_{t}^{T,\lambda}\,b_{T-t}^{\lambda}(X_t) \cdot dW_{t}^{T,\lambda}\nonumber  \\ \,=\,&\,-M_{t}^{T,\lambda}\, b_{T-t}^{\lambda}(X_t)\cdot dW_{t}^{T,\lambda} \,,
\end{align}
in which the last equality uses that $M_{t}^{T,\lambda}=0$ for $t\in [\tau,T]$. \vspace{.3cm}

\noindent \textit{Verifying (I):} Applying It\^o's rule to the $\R^2$-valued process $\{X_t M_{t}^{T,\lambda} \}_{t\in [0,T]}$ yields that 
 \begin{align}\label{XM}
     d\big( X_t \,M_{t}^{T,\lambda}\big)\,=\, X_t \,d M_{t}^{T,\lambda}\,+\,  M_{t}^{T,\lambda} \,d X_t\,+\,  dX_t \,dM_{t}^{T,\lambda} \,=\, X_t\, d M_{t}^{T,\lambda}\,+\,  M_{t}^{T,\lambda}\,dW_{t}^{T,\lambda}\,,
  \end{align}
in which to obtain the second equality we have  used~(\ref{Emmy}) and  Proposition~\ref{PropStochPre} to compute  $dX_t  dM_{t}^{T,\lambda}$:
\begin{align*}
\Big( dW_{t}^{T,\lambda}\,+ \, b_{T-t}^{\lambda}(X_t)   \,dt  \Big)\Big( -M_{t}^{T,\lambda}\, b_{T-t}^{\lambda}(X_t) \cdot dW_{t}^{T,\lambda} \Big)  \,=\, -M_{t}^{T,\lambda}\, b_{T-t}^{\lambda}(X_t) \, dt
 \,=\, M_{t}^{T,\lambda}\,\big( dW_{t}^{T,\lambda}\,-\,dX_t   \big)\,.
\end{align*}
It follows from~(\ref{XM}) that the process $XM^{T,\lambda}$ is a local $\mathbf{P}_{x }^{T,\lambda }$-martingale.  However, the process $XM^{T,\lambda}$ is square-integrable, and thus a martingale, because $M^{T,\lambda}_t$ is bounded and    $X_t$ is square-integrable.

Next, we can  deduce that $X$ is a    $\mathbf{\widetilde{P}}_{x }^{T,\lambda}$-martingale through an application of Bayes' rule: for  $0\leq s<t\leq T$,
 \begin{align}\label{Bayes}
 \mathbf{\widetilde{E}}_{x }^{T,\lambda}\big[\, X_t \,\big|\,\mathscr{F}_s^{T,x}\,\big]\,=\, \frac{\mathbf{E}_{x }^{T,\lambda}\Big[ \,M_{T}^{T,\lambda} \, X_t \,\Big|\,\mathscr{F}_s^{T,x}\,\Big] }{ \mathbf{E}_{x }^{T,\lambda}\Big[\, M_{T}^{T,\lambda}  \,\Big|\,\mathscr{F}_s^{T,x}\,\Big]  }\nonumber  \,=\,&\, \frac{\mathbf{E}_{x  }^{T,\lambda}\Big[ \,\mathbf{E}_x^{T,\lambda }\Big[\,M_{T}^{T,\lambda}\,\Big|\,\mathscr{F}_t^{T,x}\,\Big] \, X_t \,\Big|\,\mathscr{F}_s^{T,x}\,\Big] }{ \mathbf{E}_{x }^{T,\lambda}\Big[\, M_{T}^{T,\lambda}  \,\Big|\,\mathscr{F}_s^{T,x}\,\Big]  }\nonumber \\ \,=\,&\, \frac{\mathbf{E}_{x }^{T,\lambda }\Big[\, M_{t}^{T,\lambda} \, X_t \,\Big|\,\mathscr{F}_s^{T,x}\,\Big] }{  M_{s}^{T,\lambda}   }\,=\, \frac{M_{s}^{T,\lambda}\, X_s  }{  M_{s}^{T,\lambda}   }\,=\,X_s\, ,  
 \end{align}
 where the  third and fourth equalities  use that $M^{T,\lambda} $ and $XM^{T,\lambda} $ are  $\mathbf{P}_{x }^{T,\lambda}$-martingales, respectively.\vspace{.3cm}
 
 \noindent \textit{Verifying (II):}
For a unit vector $v\in \R^2$, define the processes $X^{v}:=v\cdot X  $ and $ W^{T,\lambda,v}:=v\cdot W^{T,\lambda}$.  Then we have
$$dX_t^{v}\,= \, dW_{t}^{T,\lambda,v}\,+\,v\cdot b^{\lambda}_{T-t}(X_t)\,dt \hspace{.7cm}\text{and}\hspace{.7cm} dW_{t}^{T,\lambda,v}\, dW_{t}^{T,\lambda}\,=\,v\,dt\,,  $$ and so the process $\big\{ \big(X_{t}^{v}\big)^2 M_{t}^{T,\lambda} \big\}_{t\in [0,T]}$  satisfies 
  \begin{align*}
     d\Big( \big(X_{t}^{v}\big)^2 \,M_{t}^{T,\lambda}\Big)\,=\,& \, d\Big(\big(X_{t}^{v}\big)^2\Big)\, M_{t}^{T,\lambda}\,+\,  \big(X_{t}^{v}\big)^2\, d M_{t}^{T,\lambda}\,+\,d\Big(\big(X_{t}^{v}\big)^2\Big)\, dM_{t}^{T,\lambda}\\
     \,=\,& \,\Big(2\, X_{t}^{v}\, dW_{t}^{T,\lambda,v}\,+\, 2\,v\cdot b^{\lambda}_{T-t}(X_t)\, X_{t}^{v}\,dt  \,+\,dt \Big)\, M_{t}^{T,\lambda} \\ &\,+\, \big(X_{t}^{v}\big)^2 \, d M_{t}^{T,\lambda}\,-\,2 \,v\cdot b^{\lambda}_{T-t}(X_t)\,X_{t}^{v}\,M_{t}^{T,\lambda}\, dt   \\ \,=\,&\, 2\,M_{t}^{T,\lambda}\,\big(X_{t}^{v}\big)^2\, dW_{t}^{T,\lambda,u} \,+\, \big(X_{t}^{v}\big)^2\, d M_{t}^{T,\lambda}\,+\,M_{t}^{T,\lambda}\,dt
  \end{align*}
  where we have used that the product differential $d\big((X_{t}^{v})^2\big)\, dM_{t}^{T,\lambda} $ is equal to
   \begin{align*}
\Big(2\, X_{t}^{v}\, dW_{t}^{T,\lambda,v}\,+ 2 \,v\cdot b^{\lambda}_{T-t}(X_t)\,X_{t}^{v}\,dt  \,+\,dt \Big)\Big( -M_{t}^{T,\lambda}\, b_{T-t}^{\lambda}(X_t)\cdot dW_{t}^{T,\lambda}\Big) \,=\, -2 \,v\cdot b^{\lambda}_{T-t}(X_t)\,X_{t}^{v}\,M_{t}^{T,\lambda}\, dt \,.
\end{align*}
  Thus, the process $((X^{v})^2 -t)M^{T,\lambda} $ is a local $\mathbf{P}_{x }^{T,\lambda }$-martingale. As before, we can deduce that $((X^{v})^2 -t)M^{T,\lambda} $ is square-integrable and consequently a martingale  from the boundedness of $M^{T,\lambda}$ and the finite moments of $|X|$ under $\mathbf{P}_{x }^{T,\lambda }$.   It follows from another computation using Bayes' rule as in~(\ref{Bayes})   that $\{(X_{t}^{v})^2 -t\}_{t\in [0,T]}$ is a $\mathbf{\widetilde{P}}_{x  }^{T,\lambda}$-martingale.  \vspace{.3cm}

\noindent  Part (iii):   A similar argument as in (ii) shows that under the path measure $M_{\mathbf{S}}^{T,\lambda}\mathbf{P}_{x  }^{T,\lambda}  $ the stopped process $\{X_{t\wedge \mathbf{S}  }\}_{t\in [0,T]}$ has the same law as it does under the Wiener measure $\mathbf{P}_{x  }$. Since $M_{\mathbf{S}}^{T,\lambda}=1_{\tau >\mathbf{S} }\frac{1+ H^{\lambda}_{T}(x)  
}{ 1+ H^{\lambda}_{T-\mathbf{S}}(X_{\mathbf{S}})  }  $, it follows that the stopped process $\{X_{t\wedge \mathbf{S}  }\}_{t\in [0,T]}$ has the same (unnormalized) distributional measure under the path measures $1_{\tau >\mathbf{S} }\mathbf{P}_{x  }^{T,\lambda}  $ and $\frac{ 1+ H^{\lambda}_{T-\mathbf{S}}(X_{\mathbf{S}})  }{1+ H^{\lambda}_{T}(x) } \mathbf{P}_{x  }$. Dividing by the total mass of these finite measures (that is, normalizing)  yields the result.\vspace{.3cm}

\noindent  Part (iv):  Applying (i) to get the first equality below,  we can write 
\begin{align}\label{PreQee}
 \mathfrak{p}_{T}^{\lambda}( x)  \,=\,  \mathbf{P}_{x  }^{T,\lambda}\big[\, \mathcal{O}^c\,\big]    \,=\,\mathbf{E}_{x  }^{T,\lambda}\Big[\, \mathcal{S}_{\tau\wedge T }^{T,\lambda}\, \Big] \,=\,  \mathbf{E}_{x  }^{T,\lambda}\Big[ \,\mathcal{S}_{\tau \wedge T}^{T,\lambda}\, 1_{\tau > s }\,\Big] \,,
 \end{align}
 in which the second equality uses  that $\mathcal{S}_{\tau\wedge T}^{T,\lambda}=1_{\mathcal{O}^c}$ almost surely $\mathbf{P}_{x  }^{T,\lambda}$ by~(\ref{1OComp}), and the third holds  for any $s\in [0,T]$ since ${\mathcal{O}^c} \subset \{\tau > s \}$. Inserting a nested conditional expectation with respect to $\mathscr{F}_s^{T,\mu}$, we can express the above as
 \begin{align}\label{Qee}
 \mathbf{E}_{x }^{T,\lambda }\Big[ \,\mathbf{E}_{x  }^{T,\lambda}\Big[\,\mathcal{S}_{\tau\wedge T}^{T,\lambda} \,\Big|\,\mathscr{F}_s^{T,\mu}\,\Big] \, 1_{ \tau > s  }\,\Big] \nonumber
\,=\,\mathbf{E}_{x  }^{T,\lambda}\Big[ \,\mathbf{E}_{X_s  }^{T-s,\lambda}\Big[\,\mathcal{S}_{\tau\wedge (T-s)}^{T-s,\lambda} \,\Big]\, 1_{ \tau  > s  }\,\Big] 
\,=\,&\,\mathbf{E}_{x  }^{T,\lambda}\Big[\, \mathfrak{p}_{T-s}^{\lambda}( X_s)  \,  1_{\tau > s   }\,\Big] \,,
\end{align}
where  we have  applied the Markov property and~(\ref{PreQee}).
Therefore, if $\mathbf{\widetilde{E}}_{x }^{T,\lambda,s }$ denotes the expectation corresponding to the conditioning, $\mathbf{\widetilde{P}}_{x }^{T,\lambda,s }$, of the law  $\mathbf{P}_{x  }^{T,\lambda}$ to the event that $\tau > s$, we can write
\begin{align}
\mathfrak{p}_{T}^{\lambda}( x)\,=\,\mathbf{P}_{x  }^{T,\lambda}[ \,\tau > s  \, ] \,  \mathbf{\widetilde{E}}_{x }^{T,\lambda,s }\Big[\, \mathfrak{p}_{T-s}^{\lambda}( X_s)  \,\Big]
\,\stackrel{\textup{(iii)}}{=}\,&\,\mathbf{P}_{x }^{T,\lambda }[ \tau > s  ] \,  \frac{ \mathbf{E}_{x }\Big[\overbrace{\big(1+H^{\lambda}_{T-s}(X_s)   \big)\, \mathfrak{p}_{T-s}^{\lambda}( X_s) }^{=\,1} \Big]  }{ \mathbf{E}_{x }\big[1+H^{\lambda}_{T-s}(X_s)     \big]}\nonumber \\
\,=\,&\, \mathbf{P}_{x }^{T,\lambda }[ \tau > s  ] \,  \frac{ 1  }{ \int_{\R^2}g_{s}(x-y) \big( 1+H^{\lambda}_{T-s}(y) \big)  dy } \,,
\end{align}
 in which the second equality applies  (iii) with the nonrandom stopping time $\mathbf{S}=s$.  Solving for $\mathbf{P}_{x }^{T,\lambda }[ \tau \leq  s  ]=1- \mathbf{P}_{x }^{T,\lambda }[ \tau > s  ] $ within~(\ref{Qee}) and using that $\mathfrak{p}_{T}^{\lambda}( x):=\frac{1}{1 +H_T^{\lambda}(x)  }$, we find that
\begin{align}\label{Prelim}
\mathbf{P}_{x }^{T,\lambda }[\, \tau \leq  s  \,] \,=\,1\,-\,\frac{1}{1 +H_T^{\lambda}(x)  }\, \int_{\R^2}g_{s}(x-y) \big( 1+H^{\lambda}_{T-s}(y) \big)  dy \,. \end{align}
Since  $\{\tau \leq  s\}\subset \mathcal{O} $,   we have the second equality below. 
\begin{align*}
\mathbf{P}_{x  }^{T,\lambda}[ \,\tau > s  \,|\, \mathcal{O}\, ]\,=\,  \frac{ \mathbf{P}_{x }^{T,\lambda }\big[ \tau > s ,\mathcal{O}  \big]  }{  \mathbf{P}_{x }^{T,\lambda }[ \mathcal{O}  ]  } \,=\,&  \frac{ \mathbf{P}_{x  }^{T,\lambda}[ \mathcal{O} ]\,-\,\mathbf{P}_{x  }^{T,\lambda}\big[ \tau \leq  s   \big]  }{  \mathbf{P}_{x  }^{T,\lambda}[ \mathcal{O}  ]  } \,=\, \frac{\int_{\R^2}g_s(x-y)H_{T-s}^{\lambda}(y)dy   }{ H_T^{\lambda}(x) }    \end{align*}
The third equality is derived through algebra using (\ref{Prelim}) and that $\mathbf{P}_{x  }^{T,\lambda}[ \mathcal{O}  ] =\frac{H_T^{\lambda}(x) }{1 +H_T^{\lambda}(x)  }$ by part (i).  Now, with the expression~(\ref{DefH}) for $H_{T-s}^{\lambda}(x) $, we can compute as follows with the obvious   change of integration variable.
$$ \int_{\R^2}\,g_s(x-y)\,H_{T-s}^{\lambda}(y)\,dy  \,=\,\int_0^{T-s}\,\frac{ e^{- \frac{|x|^2}{2(r+s  )} }  }{r+s }\,\nu\big((T-s-r)\lambda\big)\,dr \,=\,\int_s^{T}\,\frac{ e^{ -\frac{|x|^2}{2a } }  }{a }\,\nu\big((T-a)\lambda\big)\,da $$
From the last two displays, we can see that the probability density $-\frac{\partial}{\partial s}\mathbf{P}_{x }^{T,\lambda }[ \tau >  s  ]$ has the claimed form.
\end{proof}

\section{Equivalence of the path measures} \label{SectionGirsanov}

Given $T>0$ and a  Borel probability measure $\mu$ on $\R^2$, our next goal is to prove that the measures $\mathbf{P}^{T,\lambda}_{\mu}$ and $\mathbf{P}^{T,\lambda'}_{\mu}$ are equivalent for any $\lambda,\lambda' >0$. In particular, this justifies our omissions of the script $\lambda$ from the notations for the augmented $\sigma$-algebras $\mathscr{B}^T_{\mu}$ and $\mathscr{F}^{T,\mu}_t$ defined in (\ref{Augmented}). Our analysis in this section concludes with the proof of Theorem~\ref{ThmPreGirsanov}, which is similar to Theorem~\ref{ThmExistenceLocalTime} but offers less information about the form of the Radon-Nikodym derivative, as we have not yet undergone  a  rigorous discussion of the local time process $\mathbf{L}$. Our method of proof relies on Girsanov's theorem, and we present two forms for  the Girsanov martingale in Section~\ref{SubsecGirsanov}.  As a preliminary for one of these constructions, in Section~\ref{SubsectMartFamily} we study a family of processes $\{\mathcal{S}^{T,\lambda,\lambda'}\}_{\lambda'>0  }$ that behave as $\mathbf{P}^{T,\lambda}_{\mu}$-martingales  over time intervals when $X$ is away from the origin, in the same sense as the submartingale $\mathcal{S}^{T,\lambda}$ of Proposition~\ref{PropSubMart}.  In fact, $\mathcal{S}^{T,\lambda}$ is the limit of $\mathcal{S}^{T,\lambda,\lambda'}$ as $\lambda'\searrow 0$.

\subsection{A family of processes that behave as martingales away from the origin}\label{SubsectMartFamily}

 Given $T,\lambda,\lambda'>0$ recall that we define $R^{\lambda,\lambda'}_{T}:\R^2\rightarrow [0,\infty)$ as in~(\ref{FirstR}).
Moreover, let us put $R^{\lambda,\lambda'}_{T}(x)=1$ when $T\leq 0$.  The function $R^{\lambda,\lambda'}_{T}$ is radially symmetric, and we define $ \bar{R}^{\lambda,\lambda'}_{T}: [0,\infty)\rightarrow [0,\infty) $ such that $\bar{R}^{\lambda,\lambda'}_{T}(|x|)= R^{\lambda,\lambda'}_{T}(x) $.  The next lemma addresses the monotonicity of $\bar{R}^{\lambda,\lambda'}_{T}(a)$ in the parameters $T,a>0$, and its  proof  is in Section~\ref{SubsectionLemmaRIncrease}.
\begin{lemma}\label{LemmaRIncrease} For $\lambda,\lambda',T,a>0$, let $\bar{R}^{\lambda,\lambda'}_{T}(a)$ be defined as above.
\begin{enumerate}[(i)]

    \item $\bar{R}^{\lambda,\lambda'}_{T}(a)$ is  increasing (resp.\ decreasing) in $a$ when $\lambda'\leq \lambda$ (resp.\ $\lambda'\geq \lambda$).

     \item $ \bar{R}^{\lambda,\lambda'}_{T}(a)$   is decreasing (resp.\ increasing) in $T$ when $\lambda'\leq \lambda$ (resp.\ $\lambda'\geq \lambda$).
\end{enumerate}
\end{lemma}

Recall from Lemma~\ref{LemmaKbounds1} that  $ V^{\lambda ,\lambda'}_{t}:\R^2\rightarrow \R^2 $  is defined by
$ V^{\lambda, \lambda'}_{t }(x):=  b_t^{\lambda'}(x) -  b_t^{\lambda}(x) $.
\begin{proposition}\label{PropSubMartII} Fix some  $T,\lambda, \lambda'>0$  and  a Borel probability    measure $\mu$ on $\R^2$.    The process $\{\mathcal{S}_{t}^{T,\lambda,\lambda'}\}_{t\in [0,\infty)}$  defined by $\mathcal{S}_{t}^{T,\lambda,\lambda'}:=R^{\lambda,\lambda'}_{T-t}(X_t)$ is a bounded, continuous $\mathbf{P}_{\mu}^{T,\lambda}   $-submartingale (resp.\ supermartingale) when $\lambda'\leq \lambda$ (resp.\ $\lambda' \geq \lambda)$ with respect to $\{\mathscr{F}_t^{T,\mu }\}_{t\in [0,\infty)}$ 
 for which the martingale $\mathcal{M}^{T, \lambda,\lambda'}$ and increasing (resp.\ decreasing) $\mathcal{A}^{T, \lambda,\lambda'}$ components in its Doob-Meyer decomposition satisfy the following:
\begin{itemize}
    \item  
$ \mathcal{M}_{t }^{T, \lambda,\lambda'}:=R^{\lambda,\lambda'}_{T}(X_0) +\int_0^t\mathcal{S}_{s }^{T, \lambda,\lambda'}V^{\lambda, \lambda'}_{T-s }(X_s)
\cdot dW_{s}^{T,\lambda} $ is in $L^2\big( \mathbf{P}_{\mu}^{T,\lambda}  \big)$.

    \item $\mathcal{A}^{T, \lambda,\lambda'}$ is constant during the excursions of $X$  from $0$, that is  $\int_0^\infty 1_{ X_t\neq 0 }\, d\mathcal{A}_{t }^{T, \lambda,\lambda'} =0 $ almost surely $\mathbf{P}_{\mu}^{T,\lambda}   $.

\end{itemize}

\end{proposition}

\begin{proof}  The process $\mathcal{S}^{T, \lambda,\lambda'}$ is bounded from above by
$$C^{\lambda,\lambda'}_T\,:=\, \sup_{ \substack{ t\in [0,T] \\ x\in \R  }  }\,R^{\lambda,\lambda'}_{t}(x)\,=\,1\vee \frac{ \nu( T\lambda'  )   }{ \nu( T\lambda  )   }    \,.$$
 Hence we have the bound
\begin{align}\label{Ebby}
    \mathbf{E}^{T,\lambda}_{\mu}\bigg[ \,\int_0^T\,\Big| \mathcal{S}_{s }^{T, \lambda,\lambda'}\,V^{\lambda, \lambda'}_{T-s }(X_s)   \Big|^2 \, ds \, \bigg]\,\leq \,&\,\big(C^{\lambda,\lambda'}_T\big)^2 \,\mathbf{E}^{T,\lambda}_{\mu}\bigg[ \,\int_0^T\,\Big| V^{\lambda, \lambda'}_{T-s }(X_s)   \Big|^2 \, ds\,  \bigg]\nonumber \\
    \,= \,&\,\big(C^{\lambda,\lambda'}_T\big)^2\, \int_{\R^2}\,\int_0^T\, \int_{\R^2}\, \mathlarger{d}_{0,s}^{T,\lambda}(x,y) \,\Big| V^{\lambda, \lambda'}_{T-s }(y)   \Big|^2  \,dy \, ds \,  \mu(dx)\,.
\end{align}
The above is finite by Lemma~\ref{LemmaKbounds1}, and hence the process $\mathcal{M}^{T, \lambda,\lambda'}$, as defined through the It\^o integral in the first bullet point, is in $L^2\big(\mathbf{P}^{T,\lambda}_{\mu}\big)$.  
In the analysis below, we will verify that $1_{X_t\neq 0}\,d\mathcal{S}_{t}^{T, \lambda ,\lambda'}=1_{X_t\neq 0}\,d\mathcal{M}_{t}^{T, \lambda ,\lambda'}$  and show that the process $\mathcal{A}^{T, \lambda ,\lambda'}:=\mathcal{S}^{T, \lambda ,\lambda'}-\mathcal{M}^{T, \lambda ,\lambda'}$ is increasing (resp.\  decreasing) when $\lambda'\leq \lambda$ (resp.\ $\lambda'\geq \lambda$).  A similar argument as in the proof of Proposition~\ref{PropSubMart}
can be used to show the property in the second bullet point.

   Writing $\mathcal{S}_{t}^{T,\lambda,\lambda'}=\mathcal{S}_{t}^{T,\lambda}\big(1+H_{T-t}^{\lambda'}(X_t)\big)$ for the process $\mathcal{S}^{T,\lambda}$ defined  in Proposition~\ref{PropSubMart}  and applying It\^o's product rule, we get the first equality below.
\begin{align} \label{ItoProd}
d\mathcal{S}_{t}^{T, \lambda ,\lambda'}\,=\,&\, \mathcal{S}_{t}^{T,\lambda}\, d\Big(1+H_{T-t}^{\lambda'}(X_t)\Big)\,+\,\Big(1+H_{T-t}^{\lambda'}(X_t)\Big)\,d\mathcal{S}_{t}^{T,\lambda}\,+\,d\mathcal{S}_{t}^{T,\lambda}\, d\Big(1+H_{T-t}^{\lambda'}(X_t)\Big) \nonumber  \\
 \,\stackbin[ ]{X_t\neq 0}{=}\,&\,\mathcal{S}_{t }^{T, \lambda,\lambda'}\,V_{T-t}^{\lambda,\lambda'}(X_t)\cdot dW_{t}^{T,\lambda}  \,=: \, d\mathcal{M}_{t}^{T, \lambda  ,\lambda'}  
\end{align}
The second equality holds when $X_t\neq 0$, and to perform this computation we have used that $d\mathcal{S}_{t}^{T,\lambda} =- \mathcal{S}_{t}^{T,\lambda}  b_{T-t}^{\lambda}( X_t)   \cdot dW_{t}^{T,\lambda}$ when $X_t\neq 0$ by Proposition~\ref{PropSubMart} and  applied It\^o's chain rule with $dX_t=dW_{t}^{T,\lambda}+b_{T-t}^{\lambda}(X_t)dt$ to get that
\begin{align*}
d\Big( 1+H_{T-t}^{\lambda'}(X_t)\Big)  
\,=\,&\,\underbracket{\Big( \frac{ \partial }{\partial t }\, H_{T-t}^{\lambda'} \Big)(X_t)\,dt} 
\,+\,\big(\nabla_x\, H_{T-t}^{\lambda'}\big)(X_t)\, dX_t\,+\,\underbracket{\frac{1}{2}\big(\Delta_x\, H_{T-t}^{\lambda'}\big)(X_t)\, dt} \\
\,=\,&\,\Big(1+H_{T-t}^{\lambda'}(X_t)\Big)\, b_{T-t}^{\lambda'}(X_t) \cdot  dX_{t} \\
\,=\,&\,\big(\mathcal{S}_{t}^{T,\lambda} \big)^{-1}\,\mathcal{S}_{t}^{T,\lambda,\lambda'}\, b_{T-t}^{\lambda'}(X_t) \cdot  dX_{t}
\\
\,=\,&\,\big(\mathcal{S}_{t}^{T,\lambda} \big)^{-1}\,\mathcal{S}_{t}^{T,\lambda,\lambda'}\, b_{T-t}^{\lambda'}(X_t) \cdot  dW_{t}^{T,\lambda}\,+\,\big(\mathcal{S}_{t}^{T,\lambda} \big)^{-1}\,\mathcal{S}_{t}^{T,\lambda,\lambda'}\, b_{T-t}^{\lambda'}(X_t) \cdot b_{T-t}^{\lambda}(X_t)\,dt\,,
\end{align*}
in which the underbracketed terms cancel  because $\frac{ \partial }{\partial t } H_{t}^{\lambda'}(x)= \frac{1}{2}\Delta_x H_{t}^{\lambda'}(x)$ when $x\neq 0$. From the form of the It\^o differentials  $d\mathcal{S}_{t}^{T,\lambda} $ and $ d\big(1+H_{T-t}^{\lambda'}(X_t)\big)$, we see that their product is equal to $-\mathcal{S}_{t}^{T,\lambda,\lambda'} b_{T-t}^{\lambda'}(X_t) \cdot b_{T-t}^{\lambda}(X_t)dt $, and plugging everything back into~(\ref{ItoProd}) yields the second equality. The third equality in~(\ref{ItoProd}) follows from the definition of $\mathcal{M}^{T, \lambda ,\lambda'}$, and thus we have that $ 1_{X_t\neq 0}\,d\mathcal{S}_{t}^{T, \lambda ,\lambda'}=1_{X_t\neq 0}\,d\mathcal{M}_{t}^{T, \lambda ,\lambda'} $.

Next we address the monotonicity of  the process $\mathcal{A}^{T,\lambda,\lambda'}$. 
 Given $\varepsilon>0$ let the sequences of stopping times $\{ \varrho_{n}^{\downarrow,\varepsilon} \}_{n\in \mathbb{N}}$ and $\{ \varrho_{n}^{\uparrow,\varepsilon}\}_{n\in \mathbb{N}_0}$ be defined as in~(\ref{VARRHOS}).  Define the processes  $\mathcal{A}^{T,\lambda,\lambda',\varepsilon}$ and $\mathcal{M}^{T,\lambda,\lambda',\varepsilon}$ such that 
\begin{align}\label{ASum}
\mathcal{A}_{t}^{T,\lambda,\lambda',\varepsilon}\,:=\,\sum_{n=1}^{\infty}\,\mathcal{S}_{t\wedge \varrho_{n}^{\uparrow,\varepsilon}}^{T,\lambda,\lambda'} \,-\,  \mathcal{S}_{t\wedge \varrho_{n}^{\downarrow,\varepsilon}}^{T,\lambda,\lambda'}\,=\,\sum_{n=1}^{\infty}\,\Big( R^{\lambda,\lambda'  }_{T-t\wedge \varrho_{n}^{\uparrow,\varepsilon}  } \big( X_{ t\wedge \varrho_{n}^{\uparrow,\varepsilon}}\big) \,-\,  R^{\lambda,\lambda'  }_{T-\varrho_{n}^{\downarrow,\varepsilon}  } ( 0)\Big)\,1_{ \varrho_{n}^{\downarrow,\varepsilon}< t  } \,,
\end{align}
and $\mathcal{M}_{t}^{T,\lambda ,\lambda',\varepsilon}:=\mathcal{S}_{t}^{T,\lambda,\lambda'}-\mathcal{A}_{t}^{T,\lambda,\lambda',\varepsilon}$. A similar argument as in the  proof of Proposition~\ref{PropSubMart} shows that $\sup_{0\leq t\leq T}\big|\mathcal{M}_{t}^{T,\lambda ,\lambda',\varepsilon}- \mathcal{M}_{t}^{T,\lambda ,\lambda'}\big| $ vanishes in $L^2\big(\mathbf{P}^{T,\lambda}_{\mu}\big)$ as $\varepsilon\searrow 0$, and hence $\sup_{0\leq t\leq T}\big|\mathcal{A}_{t}^{T,\lambda ,\lambda',\varepsilon}- \mathcal{A}_{t}^{T,\lambda ,\lambda'}\big| $ vanishes in $L^2\big(\mathbf{P}^{T,\lambda}_{\mu}\big)$.  We can write a single term from the sum~(\ref{ASum}) as
\begin{align*}
R^{\lambda,\lambda'  }_{T-t\wedge \varrho_{n}^{\uparrow,\varepsilon}  } \big( X_{ t\wedge \varrho_{n}^{\uparrow,\varepsilon}}\big) -  R^{\lambda,\lambda'  }_{T- \varrho_{n}^{\downarrow,\varepsilon}  } ( 0)=\Big( R^{\lambda,\lambda'  }_{T-t\wedge\varrho_{n}^{\uparrow,\varepsilon}  } \big( X_{ t\wedge \varrho_{n}^{\uparrow,\varepsilon}}\big) -  R^{\lambda,\lambda'  }_{T-t\wedge\varrho_{n}^{\uparrow,\varepsilon}  } ( 0)\Big)\,+\,\Big( R^{\lambda,\lambda'  }_{T-t\wedge\varrho_{n}^{\uparrow,\varepsilon}  } ( 0) -  R^{\lambda,\lambda'  }_{T-\varrho_{n}^{\downarrow,\varepsilon}  } ( 0)\Big)\,.
\end{align*}
Both of the above terms have the same sign as $\lambda-\lambda'$ by Lemma~\ref{LemmaRIncrease} (recall $\bar{R}^{\lambda,\lambda'  }_{t  }(|x|)=R^{\lambda,\lambda'  }_{t  }(x) $). Thus, when  $\lambda'\leq \lambda$ (resp.\ $ \lambda'\geq \lambda  $) the process $\mathcal{A}^{T,\lambda,\lambda',\varepsilon} $ is ``nearly increasing" (resp.\ decreasing) in the sense that $\mathcal{A}^{T,\lambda,\lambda',\varepsilon}$ can decrease (resp. increase) from its running maximum (resp.\ minimum) by at most 
\begin{align*}
\sup_{\substack{ |y|,|z|\leq \varepsilon \\  0\leq s<t\leq T}}\,\textup{sgn}(\lambda'-\lambda)\, \Big( \bar{R}^{\lambda,\lambda'  }_{T-t  }(y)\,-\,\bar{R}^{\lambda,\lambda'  }_{T-s  }(z)\Big)\,=\, \Big| \bar{R}^{\lambda,\lambda'  }_{T  }(\varepsilon)\,-\,\bar{R}^{\lambda,\lambda'  }_{T  }(0)\Big|  \,.
\end{align*}
 The above vanishes as $\varepsilon \searrow 0$, and so the process $\mathcal{A}^{T,\lambda,\lambda'}$ is increasing (resp.\ decreasing) when $\lambda'\leq \lambda$ (resp.\ $ \lambda'\geq \lambda $).
\end{proof}

\subsection{The Girsanov martingale}\label{SubsecGirsanov}

 Recall from Proposition~\ref{PropSubMartII} that the process $\{\mathcal{S}_{t}^{T,\lambda,\lambda'}\}_{t\in [0,\infty)}$ is a $\mathbf{P}_{\mu}^{T,\lambda}$-submartingale (resp.\ supermartingale) when $\lambda'\leq \lambda$ (resp.\ $\lambda'\geq \lambda$) and that  $\mathcal{M}^{T,\lambda,\lambda'} $ and $\mathcal{A}^{T,\lambda,\lambda'} $ respectively denote the martingale and predictable components in its Doob-Myer decomposition. In Section~\ref{SubsectionThmEquivalent} we will show that the process $\mathbf{L}^{T,\lambda,\lambda'}$ appearing in the theorem below is $\mathbf{P}_{\mu}^{T,\lambda} $ almost surely equal to $\log\big(\frac{\lambda'}{\lambda}\big) \,\mathbf{L} $, where $\mathbf{L}$ is the local time process at the origin. Notice that~(\ref{Deriv}) has the form of a Girsanov martingale, and the Novikov condition is trivially satisfied due to its square integrability. Given $T,\lambda,\lambda'>0$ the $\R^2$-valued function  
$ V^{\lambda, \lambda'}_{T }(x):=  b_T^{\lambda'}(x) -  b_T^{\lambda}(x) $ appearing in~(\ref{Deriv}) vanishes as $|x|\rightarrow \infty$, and its norm blows up near the origin with the following asymptotics:
\begin{align}\label{VLL}
 \Big|V^{\lambda, \lambda'}_{T }(x)\Big|\,=\, \Big| \big|b_T^{\lambda'}(x)\big| -  \big|b_T^{\lambda}(x)\big| \Big|\,\stackbin[ (\ref{DefDriftFun}) ]{x\rightarrow 0}{=}\,\frac{\frac{1}{2} |\log \frac{\lambda'}{\lambda}| }{|x|\log^2 \frac{1}{|x|} }  \,+ \,\mathit{O}\bigg(\frac{ 1 }{ |x|\log^3 \frac{1}{|x|} }\bigg)\,,
\end{align}
wherein the first equality holds because the vectors $b_T^{\lambda}(x)$ and   $b_T^{\lambda'}(x)$ both point in the direction of $-x$.
\begin{theorem}\label{ThmPreGirsanov} Fix some $T,\lambda,\lambda' >0$ and  a Borel probability measure $\mu$ on $\R^2$.  Define the  process $\{\mathbf{M}^{T,\lambda,\lambda'}_{t}\}_{t\in [0,\infty)}$  by $  \mathbf{M}^{T,\lambda,\lambda'}_{t}=\mathcal{S}_{0}^{T,\lambda',\lambda} \textup{exp}\big\{\mathbf{L}_{t}^{T,\lambda ,\lambda'}  \big\}\mathcal{S}_{t}^{T,\lambda,\lambda'} $ for
\begin{align*}
\mathbf{L}_{t }^{T,\lambda,\lambda'}\,:=\,-\int_0^{t\wedge T}\, \frac{\nu\big( (T-s) \lambda\big)  }{\nu\big((T-s)\lambda' \big)  }\,d\mathcal{A}_{s}^{T,\lambda,\lambda'} \,. 
\end{align*}
 Then $\mathbf{M}^{T,\lambda,\lambda'}$ is a continuous, mean-one, square-integrable  $\mathbf{P}^{T,\lambda}_{\mu}$-martingale with respect to $\{\mathscr{F}_t^{T,\mu }\}_{t\in [0,\infty)}$, which can be expressed  in the  form \begin{align}\label{Deriv} \mathbf{M}_{t}^{T,\lambda,\lambda'} \,=\, \textup{exp}\bigg\{ \int_0^t \,V_{T-s}^{\lambda,\lambda'}(X_s) \cdot   dW_{s}^{T,\lambda}  \,-\,\frac{1}{2}\,\int_0^t  \, \Big|V^{\lambda,\lambda'}_{T-s}(X_s)\Big|^2 \,ds \bigg\}\hspace{.5cm} a.s.\  \,  \mathbf{P}^{T,\lambda}_{\mu} \,.
 \end{align}
\end{theorem}

\begin{proof}  Recall from  the proof of Proposition~\ref{PropSubMartII} that the process $\mathcal{S}^{T,\lambda,\lambda'}_t=R_{T-t}^{\lambda,\lambda'}(X_t)$ is bounded from above by the constant $C_T^{\lambda,\lambda'}:=1\vee \frac{\nu(T\lambda')  }{\nu(T\lambda )}  $, and hence from below by $\big(C_T^{\lambda',\lambda}\big)^{-1}$ because $\mathcal{S}^{T,\lambda',\lambda}_t$ is the multiplicative inverse of $\mathcal{S}^{T,\lambda,\lambda'}_t$. Note that the It\^o  integral of $\textup{exp}\big\{\mathbf{L}_{t}^{T,\lambda ,\lambda'}  \big\}$ with respect to $\mathcal{M}^{T,\lambda ,\lambda'}$ is well-defined since
\begin{align}\label{Zibble}
    \int_0^{T} \,\Big( \textup{exp}\Big\{\mathbf{L}_{t}^{T,\lambda ,\lambda'}  \Big\}\Big)^2 \,  d  \big\langle \mathcal{M}^{T,\lambda ,\lambda'}\big\rangle_t \,=\,&\, \int_0^{T}\,\textup{exp}\Big\{\,2\,\mathbf{L}_{t}^{T,\lambda ,\lambda'}  \,\Big\} \,\Big(\mathcal{S}^{T,\lambda,\lambda'}_t\Big)^2\,\Big|  V^{\lambda,\lambda'}_{T-t}(X_t) \Big|^2\,dt \nonumber   \\
    \,\leq \,&\, \Big(1\vee \textup{exp}\Big\{- 2\,C_T^{\lambda',\lambda}   \mathcal{A}^{T,\lambda,\lambda'}_T \,\Big\}\Big)  \int_0^{T} \,\Big(\mathcal{S}^{T,\lambda,\lambda'}_t\Big)^2\,\Big|  V^{\lambda,\lambda'}_{T-t}(X_t) \Big|^2\,dt \,, 
\end{align}
and the above is $\mathbf{P}_{\mu}^{T,\lambda} $ almost surely finite as a consequence of the expectation~(\ref{Ebby}) being finite.  In fact, when $\lambda'\leq \lambda$ the random variable $\mathcal{A}^{T,\lambda,\lambda'}_T $ is nonnegative, and we can infer that~(\ref{Zibble}) has finite expectation.   We can verify  that the process $\mathbf{M}^{T,\lambda,\lambda'}_t=\mathcal{S}_{0}^{T,\lambda' ,\lambda}\textup{exp}\big\{\mathbf{L}_{t}^{T,\lambda ,\lambda'}  \big\}\,\mathcal{S}_{t}^{T,\lambda ,\lambda'}$ is a local martingale though the following computation: 
\begin{align}  \label{Itoy}
 d\Big(&\textup{exp}\Big\{\mathbf{L}_{t}^{T,\lambda ,\lambda'}  \Big\}\,\mathcal{S}_{t}^{T,\lambda ,\lambda'}\Big) \nonumber  \\ \,=\,&\,\mathcal{S}_{t}^{T,\lambda,\lambda'} \, d\Big(\textup{exp}\Big\{\mathbf{L}_{t}^{T,\lambda ,\lambda'}  \Big\}\Big)\,+\,  \textup{exp}\Big\{\mathbf{L}_{t}^{T,\lambda,\lambda'}  \Big\}\,d\mathcal{S}_{t}^{T,\lambda,\lambda'} \nonumber  \\
\,=\,& \,-\underbracket{\mathcal{S}_{t}^{T,\lambda,\lambda'}\,\textup{exp}\Big\{\mathbf{L}_{t}^{T,\lambda,\lambda'}  \Big\}\,\frac{\nu\big((T-t)\lambda \big)  }{\nu\big((T-t)\lambda' \big)  }\,d\mathcal{A}_{t}^{T,\lambda,\lambda'}}  \,+\,  \textup{exp}\Big\{\mathbf{L}_{t}^{T,\lambda,\lambda'}  \Big\}\,d\mathcal{M}_{t}^{T,\lambda,\lambda'}\,+\, \underbracket{\textup{exp}\Big\{\mathbf{L}_{t}^{T,\lambda,\lambda'}  \Big\}\,d\mathcal{A}_{t}^{T,\lambda,\lambda'}} \nonumber  \\
\,=\,&  \, \textup{exp}\Big\{\mathbf{L}_{t}^{T,\lambda,\lambda'}  \Big\}\,d\mathcal{M}_{t}^{T,\lambda,\lambda'} \,,
\end{align}
where we have used the chain rule and that $\mathcal{S}^{T,\lambda,\lambda'}=\mathcal{M}^{T,\lambda,\lambda'}+\mathcal{A}^{T,\lambda,\lambda'}$  to arrive at the second equality.
The underbracketed terms are equal, thus yielding a cancellation,  because  the monotonic process $\mathcal{A}^{T,\lambda ,\lambda'} $ is constant during the excursion intervals of  $X$ away from the origin (Proposition~\ref{PropSubMartII}, second bullet point) and    $S_{t}^{T,\lambda,\lambda'}=\frac{\nu((T-t) \lambda')  }{\nu((T-t)\lambda )  }$  when $X_t=0$. Since $\mathbf{M}^{T,\lambda,\lambda'}_0=1$,  Ito's isometry yields that 
\begin{align*}
\mathbf{E}^{T,\lambda}_{\mu}\bigg[ \,\Big( \mathbf{M}^{T,\lambda,\lambda'}_T\Big)^2 \,\bigg]\,=\, 1\,+\, \mathbf{E}^{T,\lambda}_{\mu}\Bigg[\,\Big(\mathcal{S}_{0}^{T,\lambda' ,\lambda}\Big)^2\,\int_0^{T} \,\Big( \textup{exp}\Big\{\mathbf{L}_{t}^{T,\lambda ,\lambda'}  \Big\}\Big)^2 \,  d  \big\langle \mathcal{M}^{T,\lambda ,\lambda'}\big\rangle_t \,\Bigg]\,,
\end{align*}
which is finite when $\lambda'\leq \lambda$ by the observation~(\ref{Zibble}) and because the random variable $\mathcal{S}_{0}^{T,\lambda' ,\lambda}$ is bounded by $C_T^{\lambda',\lambda}$.  We will argue that the second moment of $\mathbf{M}^{T,\lambda,\lambda'}_T$ is also finite in the case $\lambda'\geq \lambda$ after verifying~(\ref{Deriv}). \vspace{.2cm}

To derive~(\ref{Deriv}), we first observe that after taking the log of the definition of $\mathbf{M}_{t}^{T,\lambda,\lambda'}$ one can write
 \begin{align}\label{MToS}
 \log\Big(\mathbf{M}_{t}^{T,\lambda,\lambda'}\Big)\,-\,\mathbf{L}_{t}^{T,\lambda,\lambda'}\,=\, \log\Big(\mathcal{S}_{t}^{T,\lambda,\lambda'}\Big)\,-\,  \log\Big(\mathcal{S}_{0}^{T,\lambda,\lambda'}\Big) \,.  
 \end{align}
 It\^o's rule and Proposition~\ref{PropSubMartII} give us the first and second equalities below.
\begin{align*} 
d\,\log\Big(\mathcal{S}_{t}^{T,\lambda,\lambda'}\Big)\,=\,&\,\frac{1}{\mathcal{S}_{t}^{T,\lambda,\lambda'} }\,d\mathcal{S}_{t}^{T,\lambda,\lambda'}\,-\,\frac{1}{2}\,\frac{1}{\big(\mathcal{S}_{t}^{T,\lambda,\lambda'}\big)^2 }\,\Big(d\mathcal{S}_{t}^{T,\lambda,\lambda'}\Big)^2 \nonumber  \\ \,=\,&\,\mathcal{S}_{t}^{T,\lambda',\lambda} \,\bigg( \underbrace{\mathcal{S}_{t}^{T,\lambda,\lambda'}\,V_{T-t}^{\lambda,\lambda'}(X_t) \cdot dW_{t}^{T,\lambda}}_{ d\mathcal{M}_{t}^{T,\lambda,\lambda'}  }\,+\,d\mathcal{A}_{t}^{T,\lambda,\lambda'}\bigg)\,-\,\frac{1}{2} \, \Big|V^{\lambda,\lambda'}_{T-t}(X_t)\Big|^2 \, dt  \nonumber 
 \\ \,=\,& \, V_{T-t}^{\lambda,\lambda'}(X_t) \cdot dW_{t}^{T,\lambda}\,+\,\underbrace{\frac{\nu\big((T-t) \lambda \big)   }{ \nu\big((T-t) \lambda' \big) }\,d\mathcal{A}_{t}^{T,\lambda,\lambda'}}_{ -d\mathbf{L}_{t}^{T,\lambda,\lambda'}}\,-\,\frac{1}{2}\, \Big|V^{\lambda,\lambda'}_{T-t}(X_t)\Big|^2\, dt 
\end{align*}
The third equality relies on the observation that  $\mathcal{S}_{t}^{T,\lambda',\lambda} d\mathcal{A}_{t}^{T,\lambda,\lambda'} =  \frac{\nu((T-t)\lambda )  }{\nu((T-t)\lambda' )  } d\mathcal{A}_{t}^{T,\lambda,\lambda'} $,  which holds by the reasoning below~(\ref{Itoy}).  Integrating the above SDE yields that
$$ \log\Big(\mathcal{S}_{t}^{T,\lambda,\lambda'}\Big)\,-\,\log\Big(\mathcal{S}_{0}^{T,\lambda,\lambda'}\Big)\,=\,-\mathbf{L}_{t}^{T,\lambda,\lambda'}\,+\, \int_0^t\, V_{T-s}^{\lambda,\lambda'}(X_s)\cdot dW_{s}^{T,\lambda}  \,-\,\frac{1}{2}\,\int_0^t  \,\Big|V^{\lambda,\lambda'}_{T-s}(X_s)\Big|^2\, ds\,.    $$
Substituting~(\ref{MToS}) into this equation yields~(\ref{Deriv}) after canceling $-\mathbf{L}_{t}^{T,\lambda,\lambda'}$ and exponentiating.

Due to~(\ref{Deriv}), we can express the second moment of $\mathbf{M}^{T,\lambda,\lambda'}_T$ as
\begin{align}\label{MSquared}
\mathbf{E}^{T,\lambda}_{\mu}\bigg[ \,\Big( \mathbf{M}^{T,\lambda,\lambda'}_T\Big)^2 \,\bigg]\,=\,  \mathbf{E}^{T,\lambda}_{\mu}\Bigg[\,\textup{exp}\bigg\{ \int_0^T  \,\Big|V^{\lambda,\lambda'}_{T-t}(X_t)\Big|^2 \,ds \bigg\} \,\Bigg]\,.
\end{align}
We have previously observed that the above is finite when $\lambda'\leq \lambda$, and we will now argue that this remains the case when $\lambda'>\lambda$.   The integrand $\big|V^{\lambda,\lambda'}_{T-t}(X_t)\big|^2$  only becomes large when $X_t$ is near the origin, in which case $\big|V^{\lambda,\lambda'}_{T-t}(X_t)\big|^2$ is equal to $ \frac{ |\log \frac{\lambda'}{\lambda}|^2 }{4|X_t|\log^4|X_t| } $ plus a smaller error term, in consequence of~(\ref{VLL}).    Fix some $\lambda'\in (\lambda,\infty)$ and pick any $\lambda^*\in \big(0,\frac{\lambda^2}{\lambda'}\big)$. Then  $\big|V^{\lambda,\lambda^*}_{T-t}(X_t)\big|^2$ is larger than $\big|V^{\lambda,\lambda'}_{T-t}(X_t)\big|^2$ when $X_t$ is near the origin, implying that~(\ref{MSquared}) is finite since $\lambda^*<\lambda$.  Therefore, $\mathbf{M}^{T,\lambda,\lambda'} $ is a  square-integrable martingale for any $\lambda'>0$ and has mean one,  as $\mathbf{M}^{T,\lambda,\lambda'}_0=1 $.
\end{proof}

\subsection{Relating the path measures through  Girsanov transformations}\label{SubsectionGirsanov}

Let the processes $\mathbf{M}^{T,\lambda,\lambda'} $ and $\mathbf{L}^{T,\lambda,\lambda'}$ be defined as in Theorem~\ref{ThmPreGirsanov}. Recall that the $\sigma$-algebra
$\mathscr{B}_{\mu}^T$ is defined as the augmentation of  the Borel $\sigma$-algebra $\mathscr{B}(\boldsymbol{\Omega})$ by the collection of subsets of $\mathbf{P}_{\mu}^{T,\lambda}$-null sets, i.e., $\mathscr{B}_{\mu}^T$ is the completion of $\mathscr{B}(\boldsymbol{\Omega})$ under $\mathbf{P}_{\mu}^{T,\lambda}$.  For clarity, we will temporarily append $\lambda$ as a superscript to our augmented $\sigma$-algebras: $\mathscr{B}_{\mu}^{T}\equiv \mathscr{B}_{\mu}^{T,\lambda} $ and $\mathscr{F}_t^{T,\mu}\equiv \mathscr{F}_t^{T,\mu,\lambda} $.    For the purpose of the proof of the next theorem, it suffices for us to take the path space $\boldsymbol{\Omega}$ to be $C\big([0,T],\R^2\big)$ because, as emphasized before, the shifted coordinate process
 $\{X_{T+t}\}_{t\in [0,\infty)}$ is merely a  two-dimensional Brownian motion under $\mathbf{P}_{\mu}^{T,\lambda}$.
\begin{theorem}\label{ThmGirsanov} Fix some $T,\lambda,\lambda'>0$ and a Borel probability measure  $\mu$ on $\R^2$.  
 Then $\mathbf{P}_{\mu}^{T,\lambda'}$ is absolutely continuous with respect to $\mathbf{P}_{\mu}^{T,\lambda}$ with Radon-Nikodym derivative
 \begin{align*}
 \frac{d\mathbf{P}_{\mu}^{T,\lambda'}}{d\mathbf{P}_{\mu}^{T,\lambda}}\,=\, \mathbf{M}^{T,\lambda,\lambda'}_{T}\,:=\,R^{\lambda',\lambda}_T(X_0) \,\textup{exp}\Big\{  \mathbf{L}_{T}^{T,\lambda,\lambda'} \Big\}\,.
 \end{align*}
In particular, the measures $\mathbf{P}_{\mu}^{T,\lambda}$ and  $\mathbf{P}_{\mu}^{T,\lambda'}$ are mutually absolutely continuous.
\end{theorem}
\begin{proof}  Recall that $\{\mathbf{M}^{T,\lambda,\lambda'}_{t}\}_{t\in [0,\infty)}$ is a mean-one, square-integrable martingale under $\mathbf{P}_{\mu}^{T,\lambda}$ that can be expressed as~(\ref{Deriv}).  Define the probability measure $\mathbf{\widetilde{P}}_{\mu}^{T,\lambda,\lambda'}:=\mathbf{M}^{T,\lambda,\lambda'}_{T}\mathbf{P}_{\mu}^{T,\lambda}$ on the measurable space $\big(\boldsymbol{\Omega}, \mathscr{B}^{T,\lambda}_\mu   \big)$.    By Girsanov's theorem,  the process $\big\{\widetilde{W}_{t}^{T,\lambda,\lambda'}\big\}_{t\in [0,T]} $ defined by
\begin{align}\label{TwoWs}
\widetilde{W}_{t}^{T,\lambda,\lambda'} \,:=\,W_{t}^{T,\lambda} \,-\,\int_0^t\, V_{T-s}^{\lambda,\lambda'}(X_s)\,ds\,. 
\end{align}
is a standard two-dimensional $\mathbf{\widetilde{P}}_{\mu}^{T,\lambda,\lambda'}$-Brownian motion with respect  to the filtration $\big\{ \mathscr{F}^{T,\mu, \lambda}_t \big\}_{t\in [0,T]}$.   Since $V_{t}^{\lambda,\lambda'}(x):=b_t^{\lambda'}(x)-b_t^{\lambda}(x)$, we can write
\begin{align*}
X_t\,= \,X_0\, +\, W_{t}^{T,\lambda}\,+\,\int_0^t \,b^{\lambda}_{T-s}(X_s)\,ds\,\stackrel{(\ref{TwoWs})}{=}\,X_0 \,+ \,\widetilde{W}_{t}^{T,\lambda,\lambda'}\,+\int_0^t\, b^{\lambda'}_{T-s}(X_s)\,ds \,.
\end{align*}
Thus the pair $\big(X ,\widetilde{W}^{T,\lambda,\lambda'} \big)$ on the filtered probability space $\big(\boldsymbol{\Omega}, \mathscr{B}^{T,\lambda}_\mu ,\mathscr{F}^{T,\mu,\lambda},\mathbf{\widetilde{P}}_{\mu}^{T,\lambda,\lambda'}  \big)$ satisfies the same SDE as the pair $\big(X ,W^{T,\lambda'} \big)$ on the filtered probability space $\big(\boldsymbol{\Omega}, \mathscr{B}^{T,\lambda'}_\mu , \mathscr{F}^{T,\mu,\lambda'} ,\mathbf{P}_{\mu}^{T,\lambda'}  \big)$.  It follows that the finite-dimensional distributions of $X$ under $\mathbf{\widetilde{P}}_{\mu}^{T,\lambda,\lambda'}$ agree with those under $\mathbf{P}_{\mu}^{T,\lambda'} $, as they are both determined by the transition density function, $\mathlarger{d}^{T,\lambda'}_{s,t}(x,y)$.
  Hence $\mathbf{\widetilde{P}}_{\mu}^{T,\lambda,\lambda'}$ agrees with $\mathbf{P}_{\mu}^{T,\lambda'} $ on the Borel $\sigma$-algebra $\mathscr{B}(\boldsymbol{\Omega})$.  Since $\mathscr{B}^{T,\lambda'}_\mu $ is the completion of $\mathscr{B}(\boldsymbol{\Omega})$ under $\mathbf{P}_{\mu}^{T,\lambda'} $, we can conclude that $\mathscr{B}^{T,\lambda'}_\mu \subset \mathscr{B}^{T,\lambda}_\mu$ and $\mathbf{P}_{\mu}^{T,\lambda'}$ is absolutely continuous with respect to $\mathbf{P}_{\mu}^{T,\lambda'}$.  Swapping the roles of $\lambda$ and $\lambda'$ yields these relations in the opposite direction.
\end{proof}

The next lemma  comments on the relationship between the processes highlighted  above for different values of $\lambda$.
\begin{lemma}\label{LemmaRelate} Fix some $\lambda, \lambda',\lambda'',T>0$ and a Borel probability measure $\mu$.  The following equalities  hold for all $t\geq 0$  almost surely $\mathbf{P}_{\mu}^{T,\lambda'}$: 
\begin{align*}
W_{t}^{T,\lambda'}\,=\, W_{t}^{T,\lambda} \,-\,\int_0^t\, V_{T-s}^{\lambda,\lambda'}(X_s)\,ds\,, \hspace{.5cm}  \mathbf{M}^{T,\lambda,\lambda''}_{t}
\,=\, \mathbf{M}^{T,\lambda,\lambda'}_{t}\, \mathbf{M}^{T,\lambda',\lambda''}_{t}
\,, \hspace{.5cm}  \mathbf{L}^{T,\lambda,\lambda''}_{t}
\,=\, \mathbf{L}^{T,\lambda,\lambda'}_{t}\,+\, \mathbf{L}^{T,\lambda',\lambda''}_{t} \,.
\end{align*}

\end{lemma}

\begin{proof} By equating the stochastic integral equation $X_t= X_0 + W_{t}^{T,\lambda}+\int_0^t \,b^{\lambda}_{T-r}(X_r)\,dr$ with its counterpart in which $\lambda$ is replaced by $\lambda'$, we arrive at  the above relationship between $W^{T,\lambda} $ and $W^{T,\lambda'}$.  We can then verify the product formula for  $\mathbf{M}^{T,\lambda,\lambda''}_{t}$ through using the expression~(\ref{Deriv}) for  $\mathbf{M}^{T,\lambda,\lambda'}$  and $\mathbf{M}^{T,\lambda',\lambda''}$, and  substituting  $dW_{t}^{T,\lambda'}= dW_{t}^{T,\lambda} - V_{T-t}^{\lambda,\lambda'}(X_t)dt$ within the latter. Since $  \mathbf{M}^{T,\lambda,\lambda'}_{t}:=\mathcal{S}_{0}^{T,\lambda',\lambda} \textup{exp}\big\{\mathbf{L}_{t}^{T,\lambda ,\lambda'}  \big\}\mathcal{S}_{t}^{T,\lambda,\lambda'} $, we can obtain the additive formula for $\mathbf{L}^{T,\lambda,\lambda''}$ by observing that $\mathcal{S}_{r}^{T,\lambda,\lambda''}=\mathcal{S}_{r}^{T,\lambda,\lambda'}\mathcal{S}_{r}^{T,\lambda',\lambda''} $ for all $r\geq 0$. 
\end{proof}

\section{About the local time  at the origin}\label{SectionLocalTime}

Recall that for a standard one-dimensional Brownian motion $\{B_t\}_{t\in [0,\infty)}$ the local time process at zero $\{L_t\}_{t\in [0,\infty)}$ can be expressed through the Tanaka formula
\begin{align*}
    L_t\,=\,|B_t|\,-\,|B_0|\,-\,\int_0^t\,\textup{sgn}(B_s)\,dB_s\,,\hspace{.5cm}t\in [0,\infty)\,,
\end{align*}
which can be derived heuristically  through a formal application of It\^o's rule with the function $f(x)=|x|$ and by identifying $L_t$ with ``$ \int_0^t\delta (B_r)dr$." Note that the process $  L$ is the increasing component in the Doob-Meyer decomposition for the submartingale $|B|$.  The definition for the local time process at the origin that we  begin with here is $\mathbf{L}:=-\mathring{\mathcal{A}}^{T,\lambda}$, where the process $\mathring{\mathcal{A}}^{T,\lambda}$ is the decreasing Doob-Meyer component for a nonnegative, bounded supermartingale  $\mathring{\mathcal{S}}^{T,\lambda}$ that we discuss in Section~\ref{SubsectionTanaka}.  In Sections~\ref{SectionLocalTimeDowncrossing} \&~\ref{SubsectionDowncrossing2}, we prove the approximations for the local time process using downcrossings in Theorems~\ref{THMLocalTime1} \&~\ref{THMLocalTime2},  although starting from  the Tanaka-type definition of the local time just described.  Next, in Section~\ref{SectionLocalTimeAdditive}, we prove a continuous additive integral approximation for the local time, which includes Theorem~\ref{ThmExistenceLocalTime} as a special case.  Finally, Section~\ref{SubsectionThmEquivalent} contains the proofs of 
 Theorem~\ref{ThmEquivalent} and Proposition~\ref{PropLocalTimeProp}.

\subsection{A Tanaka-type construction of the local time at the origin}\label{SubsectionTanaka}

The supermartingale $\mathring{\mathcal{S}}^{T,\lambda}$  in the proposition below  is related to the process $\mathcal{S}^{T,\lambda,\lambda'}$ defined in Proposition~\ref{PropSubMartII} through $ \mathring{\mathcal{S}}^{T,\lambda}=  \lambda\frac{ 
\partial }{ \partial \lambda' }\mathcal{S}^{T,\lambda,\lambda'}|_{\lambda'=\lambda}$, and consequently  it has a  similar Doob-Meyer decomposition. Given $T,\lambda>0$ define  $R^{\lambda}_T:\R^2\rightarrow [0,\infty)$ by 
\begin{align}\label{RThing}
R^{\lambda}_T(x) \, :=\, \lambda\frac{\partial }{\partial \lambda'}  R^{\lambda,\lambda'}_T(x) \,\Big|_{\lambda'=\lambda}  \,  = \,\frac{\lambda \frac{\partial}{\partial \lambda} H_{T}^{\lambda}( x) }{1+H_{T}^{\lambda}( x)  } \hspace{.4cm}\text{for $x\neq 0$ and}\hspace{.4cm} R^{\lambda}_{T}(0)\,:=\,\lim_{x\rightarrow 0} R^{\lambda}_{T}(x)   \,=\,T\lambda\, \frac{ \nu'( T\lambda  )   }{ \nu( T\lambda  )   }  \,. 
\end{align}
 When $T\leq 0$ we set $R^{\lambda}_T(x):=0$.   As in  Lemma~\ref{LemmaKbounds1}, we define  $ \mathring{V}_{T}^{\lambda}:\R^2\rightarrow \R^2 $ by $\mathring{V}_{T}^{ \lambda }(x):=\lambda \frac{\partial}{\partial \lambda}b_T^{\lambda}(x)$.  Notice that  $\mathring{V}_{T}^{ \lambda }(x)= \nabla_x R^{\lambda}_{T}(x)$ since $\frac{1}{\lambda}R^{\lambda}_{T}(x)$ and $b_t^{\lambda}(x)$ are merely different partial derivatives of $\log(1+H_T^{\lambda}(x))$.

\begin{proposition}\label{PropSubMartIII} Fix some  $T,\lambda >0$  and a Borel probability measure  $\mu$  on $ \R^2$.   The process $\{\mathring{\mathcal{S}}_{t}^{T,\lambda}\}_{t\in [0,\infty)}$ defined by $\mathring{\mathcal{S}}_{t}^{T,\lambda}:=R^{\lambda}_{T-t}( X_t)  $ is a  nonnegative, bounded, continuous $\mathbf{P}^{T,\lambda}_{\mu}$-supermartingale with respect to the filtration $\{\mathscr{F}_t^{T,\mu   }\}_{t\in [0,\infty)}$, for which the martingale component $\mathring{\mathcal{M}}^{T,\lambda}$ and the decreasing component $\mathring{\mathcal{A}}^{T,\lambda}$ satisfy the following:

\begin{itemize}
    \item $\mathring{\mathcal{M}}_{t  }^{T,\lambda}:=R^{\lambda}_{T}(X_0)   +\int_0^t\mathring{V}_{T-s }^{\lambda}(X_s)
\cdot dW_{s}^{T,\lambda} $ is in $L^2\big(\mathbf{P}^{T,\lambda}_{\mu}\big)$.

    \item $\mathring{\mathcal{A}}^{T,\lambda}$  is constant during the excursions of $X$ away from $0$, that is $\int_0^{\infty}1_{X_s\neq 0 }\, d\mathring{\mathcal{A}}^{T,\lambda}_s=0 $ almost surely $\mathbf{P}^{T,\lambda}_{\mu}$.
    
\end{itemize}

\end{proposition}

\begin{proof}  The It\^o integral in the first bullet point defines a square-integrable $\mathbf{P}^{T,\lambda}_{\mu}$-martingale because 
$$ \mathbf{E}^{T,\lambda}_{\mu}\bigg[ \,\int_0^T \,\big|\mathring{V}_{T-s }^{\lambda}(X_s)\big|^2\,ds\,\bigg]\,=\,\int_{\R^2} \, \int_0^T \,\int_{\R^2} \, \mathlarger{d}^{T,\lambda}_{0,s}(x,y)\,\big|\mathring{V}_{T-s }^{\lambda}(y)\big|^2\,dy   \,ds \,  \mu(dx) $$
is finite by Lemma~\ref{LemmaKbounds1}.  For the martingale $\mathring{\mathcal{M}}^{T,\lambda}$ defined in the first bullet point,  we put $\mathring{\mathcal{A}}^{T,\lambda}:=\mathring{\mathcal{S}}^{T,\lambda}-\mathring{\mathcal{M}}^{T,\lambda}$. Also, given $\lambda'>\lambda$, we define the processes 
$$ \mathring{\mathcal{S}}^{T,\lambda,\lambda'}\,:=\, \frac{\lambda}{\lambda'-\lambda}\, \big(\mathcal{S}^{T,\lambda,\lambda' }-1\big)\,, \hspace{.3cm}   \mathring{\mathcal{M}}^{T,\lambda, \lambda'}\,:=\, \frac{\lambda}{\lambda'-\lambda}\, \big(\mathcal{M}^{T,\lambda,\lambda'}-1\big)\,,  \hspace{.3cm}  \mathring{\mathcal{A}}^{T,\lambda,\lambda'} \,:=\,\mathring{\mathcal{S}}^{T,\lambda,\lambda'}\,-\, \mathring{\mathcal{M}}^{T,\lambda,\lambda'}  \,.$$
It follows from Proposition~\ref{PropSubMartII} that  the process $ \mathring{\mathcal{A}}^{T,\lambda,\lambda'}  $  is decreasing  and almost surely satisfies $\int_0^{\infty}1_{X_s\neq 0 }\, d\mathring{\mathcal{A}}^{T,\lambda,\lambda'}_s=0 $ and that $\mathring{\mathcal{M}}^{T,\lambda, \lambda'}$ is a martingale with  $d\mathring{\mathcal{M}}^{T,\lambda, \lambda'}_t=\frac{\lambda}{\lambda'-\lambda} \,\mathcal{S}_t^{T,\lambda,\lambda' } V_{T-t }^{\lambda,\lambda'}(X_t) \cdot dW_{t}^{T,\lambda} $.  Doob's maximal inequality and It\^o's isometry give us the first two relations below.
\begin{align*}
&\mathbf{E}^{T,\lambda}_{\mu}\bigg[\, \sup_{0\leq t\leq T} \,\Big|\mathring{\mathcal{M}}^{T,\lambda,\lambda'}_t -\mathring{\mathcal{M}}^{T,\lambda}_t \Big|^2 \, \bigg]\\ &\,\leq \,   4\,  \mathbf{E}^{T,\lambda}_{\mu}\bigg[\,  \Big|\mathring{\mathcal{M}}^{T,\lambda,\lambda'}_T -\mathring{\mathcal{M}}^{T,\lambda}_T 
 \Big|^2\,\bigg]\\
 &=\,  4\, \mathbf{E}^{T,\lambda}_{\mu}\Bigg[\, \bigg| \frac{\lambda}{\lambda'-\lambda}\, \Big(R_{T }^{\lambda,\lambda'}(X_0) -1\Big)   - R_{T }^{\lambda}(X_0) 
 \bigg|^2   \,\Bigg]        \,+\,  4\, \mathbf{E}^{T,\lambda}_{\mu}\Bigg[\, \int_0^T\, \bigg| \frac{\lambda}{\lambda'-\lambda} \,\mathcal{S}_s^{T,\lambda,\lambda' }\, V_{T-s }^{\lambda,\lambda'}(X_s)    - \mathring{V}_{T-s }^{\lambda}(X_s) 
 \bigg|^2 ds\,\Bigg]\\
 &= \, 4\,\int_{\R^2} \,  \bigg| \frac{\lambda}{\lambda'-\lambda} \, \Big(R_{T }^{\lambda,\lambda' }(x) -1\Big)   - R_{T }^{\lambda}(x) 
 \bigg|^2    \mu(dx)\nonumber \\  &\text{}\,\,\,\,\,\,+\,  4\, \int_{\R^2} \, \int_0^T\, \int_{\R^2} \, \mathlarger{d}^{T,\lambda}_{0,s}(x,y) \,\bigg|\frac{\lambda}{\lambda'-\lambda} \,R^{\lambda,\lambda' }_{T-s}(y)\, V_{T-s }^{\lambda,\lambda'}(y)    - \mathring{V}_{T-s }^{\lambda}(y) 
 \bigg|^2 \,dy   \,ds \,  \mu(dx)
\end{align*}
The above terms can be shown to vanish as $\lambda'\rightarrow \lambda$ using the dominated convergence theorem and Lemma~\ref{LemmaKbounds1}.  Furthermore, we have 
\begin{align*}
\mathbf{E}^{T,\lambda}_{\mu}\bigg[ \,\sup_{0\leq t\leq T}\, \Big|\mathring{\mathcal{S}}^{T,\lambda,\lambda'}_t -\mathring{\mathcal{S}}^{T,\lambda}_t 
 \Big|\, \bigg]
 \,\leq  \, 4 \, \sup_{ \substack{ 0\leq t\leq T \\ x\in \R^2 } } \, \bigg|  \frac{\lambda}{\lambda'-\lambda} \, \Big(R_{t}^{\lambda,\lambda'}(x) -1\Big)   - R_{t }^{\lambda}(x) 
 \bigg| \hspace{.4cm}  \stackrel{\lambda'\rightarrow \lambda }{\longrightarrow }  \hspace{.4cm}  0\,.
\end{align*}
Since $\mathring{\mathcal{A}}^{T,\lambda}:=\mathring{\mathcal{S}}^{T,\lambda}-\mathring{\mathcal{M}}^{T,\lambda}$ and $\mathring{\mathcal{A}}^{T,\lambda,\lambda'} :=\mathring{\mathcal{S}}^{T,\lambda,\lambda'}- \mathring{\mathcal{M}}^{T,\lambda,\lambda'} $, it follows that $\sup_{0\leq t\leq T} \big|\mathring{\mathcal{A}}^{T,\lambda,\lambda'}_t -\mathring{\mathcal{A}}^{T,\lambda}_t 
 \big|$ vanishes in $L^1\big(  \mathbf{P}^{T,\lambda}_{\mu} \big)$ as $\lambda'\rightarrow \lambda$.  Since the process $ \mathring{\mathcal{A}}^{T,\lambda,\lambda'}  $ is decreasing for any $\lambda'>\lambda $, the limit process $\mathring{\mathcal{A}}^{T,\lambda}  $ is also decreasing.  Thus,  $\mathring{\mathcal{S}}^{T,\lambda}= \mathring{\mathcal{M}}^{T,\lambda,\lambda'}+\mathring{\mathcal{A}}^{T,\lambda,\lambda'}$ is a supermartingale.  We can also use that $\int_0^{T}1_{X_s\neq 0 }\, d\mathring{\mathcal{A}}^{T,\lambda,\lambda'}_s=0 $ holds almost surely for each $\lambda'>\lambda$ and  the above sense of convergence of $\mathring{\mathcal{A}}^{T,\lambda,\lambda'}$ to $\mathring{\mathcal{A}}^{T,\lambda} $ to deduce that  $\int_0^{T}1_{X_s\neq 0 }\, d\mathring{\mathcal{A}}^{T,\lambda}_s=0 $ holds almost surely.
\end{proof}

In the sequel, we refer to the process  $\mathbf{L}^{T,\lambda} :=-\mathring{\mathcal{A}}^{T,\lambda}$,  defined on the probability space $\big(\boldsymbol{\Omega}, \mathscr{B}_{\mu}^T , \mathbf{P}_{\mu}^{T,\lambda} \big)$, as the \textit{local time at the origin}.  We retain the superscripts on $\mathbf{L}^{T,\lambda}\equiv \mathbf{L}$ only in this subsection and the next because we see the approximation scheme in Theorem~\ref{ThmLocalTime1} as justifying their removal. Technically, when the local time is viewed as a $C\big([0,\infty),\R^2  \big)$-valued random element, it depends on $T$ and $\mu$ because it is an equivalence class of maps differing on $ \mathbf{P}_{\mu}^{T,\lambda}$-null sets, with $\lambda$ being irrelevant by Theorem~\ref{ThmGirsanov}.

\subsection{The first downcrossing approximation for the  local time at the origin}\label{SectionLocalTimeDowncrossing}
The aim of this subsection is to prove  the theorem below, which is essentially the same as Theorem~\ref{THMLocalTime1}, except that we begin with the   construction of $\mathbf{L}^{T,\lambda}$ in the previous subsection. Recall that for $\varepsilon >0$ we define the sequences of stopping times $\{\varrho_{n}^{\downarrow,\varepsilon}\}_{n\in \mathbb{N}}$ and $\{\varrho_{n}^{\uparrow,\varepsilon}\}_{n\in \mathbb{N}_0}$ as in~(\ref{VARRHOS}) and that $N_{t}^{\varepsilon}$  denotes the   number of times $\varrho_{n}^{\downarrow,\varepsilon}$  in the interval $[0,t]$.
\begin{theorem}\label{ThmLocalTime1} Fix some $T,\lambda >0$ and  a Borel probability measure $\mu$  on $\R^2$. For $\varepsilon\in (0,1)$ 
define the process $\{\ell_{t}^{\varepsilon}\}_{t\in [0,\infty)}$ by $ \ell_{t}^{\varepsilon}\,:=\,\frac{N_{t}^{\varepsilon} }{  2\log \frac{1}{\varepsilon} }$.  The random variable 
$\sup_{0\leq t\leq T}\big| \ell_{t}^{\varepsilon}\,-\,\mathbf{L}_{t}^{T,\lambda}\big|$ vanishes in $L^1\big(\mathbf{P}_{\mu}^{T,\lambda}\big)$-norm as $\varepsilon \searrow 0$.  
\end{theorem}
 Before commencing with the proof of Theorem~\ref{ThmLocalTime1}, we  state three lemmas to be used therein.   The following lemma provides a bound for the expectation of $N_{T}^{\varepsilon}$ under $\mathbf{P}^{T,\lambda}_{x}$. The proof, which relies on the submartingale upcrossing inequality, is in Section~\ref{SubsectionUpcroossingInequalityPre}.
\begin{lemma}\label{UpcroossingInequalityPre}  For any $L>0$ there exists a  $C_L>0$ such that for all $T,\lambda\in (0,L]$, $\varepsilon>0$, and $x\in \R^2$
$$   \mathbf{E}^{T,\lambda}_{x}\big[\, N_{T}^{\varepsilon}\,\big] \,\leq \,C_L \,\frac{ 1+\log^+\frac{T}{\varepsilon^2 }  }{1+\log^+ \frac{1}{T\lambda} }  \,. $$
In particular, there exists a $C_{T,\lambda}>0$ such that the expectation of $N_{T}^{\varepsilon}$ under $\mathbf{P}^{T,\lambda}_{x}$ is bounded from above by $ C_{T,\lambda}\log \frac{1}{\varepsilon} $ for all $\varepsilon\in (0,\frac{1}{2})$ and $x\in \R^2$.
\end{lemma}

The next lemma, when applied in combination with the strong Markov property,  provides an approximation for the expected amount of  local time at the origin that is accrued over a single upcrossing interval $\big[\varrho_{n}^{\downarrow,\varepsilon}, \varrho_{n}^{\uparrow,\varepsilon}\big)$. The lemma's statement is reminiscent of Lemma~\ref{LemmaLeave}, and its proof is placed in Section~\ref{SubsectionLemmaUprossingLocal}.
\begin{lemma}\label{LemmaUprossingLocal} Given $\varepsilon>0$ define the $\mathscr{F}^X$-stopping time $\varrho^{\uparrow, \varepsilon}:= \inf\{t\in (0,\infty)\,:\, |X_t|\geq \varepsilon \}$. 
\begin{enumerate}[(i)]

\item  For any $L\in (1,\infty)$ there exists a  $C_{L}>0$ such that for all  $T,\lambda \in (0, L]$, $m\in \mathbb{N}$,  $\varepsilon\in (0,\frac{1}{2})$, and $x\in \R^2$
\begin{align}\label{BaseII}
\mathbf{E}^{T,\lambda}_{x}\Big[\,\big(\mathbf{L}_{\varrho^{\uparrow, \varepsilon} }^{T,\lambda}\big)^m\,\Big]\,\leq \,\frac{ m!\, C_L^m  }{\log^m \frac{1}{\varepsilon} }   \,.
\end{align}

\item For any $L\in (1,\infty)$ there exists a  $C_{L}>0$ such that for all $T,\lambda \in \big[\frac{1}{L}, L\big]$ and  $\varepsilon\in (0,\frac{1}{2})$
\begin{align*}
\bigg| \mathbf{E}^{T,\lambda}_{0}\Big[ \,\mathbf{L}_{\varrho^{\uparrow, \varepsilon} }^{T,\lambda} \,\Big]\,-\,\frac{ 1 }{2\log \frac{1}{\varepsilon} }    \bigg| \,\leq \, C_{L}\,\varepsilon^2   \,.
\end{align*}

\end{enumerate}

\end{lemma}

We use the  lemma below to bound the largest change accrued over a single upcrossing interval to various processes that we encounter.  Since the proof is short, it is not too distracting to include here.
\begin{lemma}\label{LemmaSupTime}  For any $ L>0$ and $p>1$, there exists a $C_{L,p}>0$ such that for all $T,\lambda \in (0, L]$, $x\in \R^2$, and $\varepsilon\in (0,\frac{1}{2})$ we have (i) \& (ii) below. 
\begin{enumerate}[(i)]

\item $\mathbf{E}^{T,\lambda}_{x}\Big[\,\max_{1\leq n \leq N_T^{\varepsilon} }\,\big( \mathbf{L}_{\varrho_{n}^{\uparrow ,\varepsilon}}^{T,\lambda}- \mathbf{L}_{\varrho_{n}^{\downarrow ,\varepsilon}}^{T,\lambda} \big)^p \,\Big] \,\leq\, C_{L,p} \,\log^{\frac{1}{2}-p}\frac{1}{\varepsilon}  $

\item $\mathbf{E}^{T,\lambda}_{x}\Big[\,\max_{1\leq n \leq N_T^{\varepsilon} }\,\big( \varrho_{n}^{\uparrow ,\varepsilon}- \varrho_{n}^{\downarrow ,\varepsilon} \big)^p\, \Big] \,\leq\, C_{L,p}\, \varepsilon^{2p}\,\log^{\frac{1}{2}} \frac{1}{\varepsilon}  $
    
\end{enumerate}

\end{lemma}

\begin{proof} By Jensen's inequality, the expectation in (i)  is bounded by the square root of 
\begin{align*}
 \mathbf{E}^{T,\lambda}_{x}&\left[\,\max_{1\leq n \leq N_{T}^{\varepsilon} }\,\Big(\mathbf{L}_{\varrho_{n}^{\uparrow,\varepsilon}}^{T,\lambda}- \mathbf{L}_{\varrho_{n}^{\downarrow,\varepsilon}}^{T,\lambda}  \Big)^{2p} \,\right] \,\leq \,\mathbf{E}^{T,\lambda}_{x}\left[\,\sum_{n=1}^{N_{T}^{\varepsilon} }\,\Big(\mathbf{L}_{\varrho_{n}^{\uparrow,\varepsilon}}^{T,\lambda}- \mathbf{L}_{\varrho_{n}^{\downarrow,\varepsilon}}^{T,\lambda}  \Big)^{2p} \,\right] \,.
 \end{align*}
 We can express the right side above as
 \begin{align*}
\mathbf{E}^{T,\lambda}_{x}\left[\,\sum_{n=1}^{\infty}\,1_{\varrho_{n}^{\downarrow,\varepsilon}\leq T }\,\Big(\mathbf{L}_{\varrho_{n}^{\uparrow,\varepsilon}}^{T,\lambda}- \mathbf{L}_{\varrho_{n}^{\downarrow,\varepsilon}}^{T,\lambda}  \Big)^{2p} \,\right]\, = \,&\, \sum_{n=1}^{\infty }\,\mathbf{E}^{T,\lambda}_{x}\left[\,1_{ \varrho_{n}^{\downarrow,\varepsilon}\leq T }\,\mathbf{E}^{T,\lambda}_{x}\bigg[\,\Big(\mathbf{L}_{\varrho_{n}^{\uparrow,\varepsilon}}^{T,\lambda}- \mathbf{L}_{\varrho_{n}^{\downarrow,\varepsilon}}^{T,\lambda}  \Big)^{2p}\,\bigg|\,\mathscr{F}_{\varrho_{n}^{\downarrow, \varepsilon} }^{T,\mu}\,\bigg] \,\right]  \\  \, = \,&\, \sum_{n=1}^{\infty }\,\mathbf{E}^{T,\lambda}_{x}\left[\,1_{ \varrho_{n}^{\downarrow,\varepsilon}\leq T }\,\mathbf{E}^{T-\varrho_{n}^{\downarrow ,\varepsilon},\lambda}_{0}\left[\,\Big(\mathbf{L}_{\varrho^{\uparrow,\varepsilon}}^{T-\varrho_{n}^{\downarrow ,\varepsilon},\lambda}\Big)^{2p}\,\right] \,\right] \,,\end{align*}
 where we have applied the strong Markov property.  By part (i) of Lemma~\ref{LemmaUprossingLocal}, the above is bounded uniformly in $\varepsilon\in (0,\frac{1}{2})$  by a constant multiple (depending on $L$ and $p$) of
 \begin{align*}
   \frac{1}{\log^{2p}\frac{1}{\varepsilon}  }\, \sum_{n=1}^{\infty }\,\mathbf{E}^{T,\lambda}_{x}\Big[\,1_{ \varrho_{n}^{\downarrow,\varepsilon}\leq T  } \,\Big]  \, =  \,  \frac{1}{\log^{2p}\frac{1}{\varepsilon}  } \,\mathbf{E}^{T,\lambda}_{x}\left[\,N_{T}^{\varepsilon}\,  \right]\,\preceq \, \frac{1}{\log^{2p-1}\frac{1}{\varepsilon}  }  \,,
\end{align*}
in which the inequality applies  Lemma~\ref{UpcroossingInequalityPre}. This establishes (i),  and the proof of~(ii) is similar, except that   (i) of Lemma~\ref{LemmaLeave} is applied instead of Lemma~\ref{LemmaUprossingLocal}.
\end{proof}

The proof of Theorem~\ref{ThmLocalTime1}  proceeds by introducing intermediary processes $\mathfrak{l}^{T,\lambda,\varepsilon}$ and $\overline{\mathfrak{l}}^{T,\lambda,\varepsilon}$ and showing that  the differences $\mathfrak{l}^{T,\lambda,\varepsilon}-\mathbf{L}^{T,\lambda} $, $\overline{\mathfrak{l}}^{T,\lambda,\varepsilon}-\mathfrak{l}^{T,\lambda,\varepsilon} $, and $\ell^{\varepsilon}-\overline{\mathfrak{l}}^{T,\lambda,\varepsilon} $ vanish in the appropriate supremum-$L^1\big(\mathbf{P}^{T,\lambda}_{\mu}\big)$ sense.  The second difference $\mathfrak{m}^{T,\lambda,\varepsilon}:=\overline{\mathfrak{l}}^{T,\lambda,\varepsilon}-\mathfrak{l}^{T,\lambda,\varepsilon}$ is a martingale with respect to an altered filtration $\boldsymbol{\mathscr{F}}_t^{T,\mu,\varepsilon}$ that contains information  up to the future  upcrossing time $\varrho_{n}^{\uparrow,\varepsilon}$ when $t\in\big[ \varrho_{n}^{\downarrow ,\varepsilon},  \varrho_{n}^{\uparrow,\varepsilon} \big)$.  This allows us to apply Doob's maximal inequality, and  the predictable quadratic variation $\langle \mathfrak{m}^{T,\lambda,\varepsilon}\rangle$  is bounded using Lemmas~\ref{UpcroossingInequalityPre} \& \ref{LemmaUprossingLocal}.

\begin{proof}[Proof of Theorem~\ref{ThmLocalTime1}]
Let $\varepsilon \in (0,  1)$ and $t\in [0,T]$.  Since the process $ \mathbf{L}^{T,\lambda}$ starts at $0$ and is constant over the downcrossing intervals $\big[ \varrho_{n-1}^{\uparrow,\varepsilon}, \varrho_{n}^{\downarrow,\varepsilon}\big] $, we can write 
$\mathbf{L}_{t}^{T,\lambda}=\sum_{n=1}^{ N_{t}^{\varepsilon} }\, \mathbf{L}_{\varrho_{n}^{\uparrow,\varepsilon}\wedge t}^{T,\lambda}-\mathbf{L}_{\varrho_{n}^{\downarrow,\varepsilon}}^{T,\lambda}  $.
For the process $ \{\mathfrak{l}_{t}^{T,\lambda,\varepsilon}\}_{t\in [0,T]}$ defined by
$   \mathfrak{l}_{t}^{T,\lambda,\varepsilon}:=\sum_{n=1}^{ N_{t}^{\varepsilon} }\mathbf{L}_{\varrho_{n}^{\uparrow,\varepsilon}}^{T,\lambda}-\mathbf{L}_{\varrho_{n}^{\downarrow,\varepsilon}}^{T,\lambda}    
$, the difference between $\mathbf{L}^{T,\lambda }$ and $\mathfrak{l}^{T,\lambda,\varepsilon}$ has the obvious bound 
\begin{align*}
  \sup_{0\leq t\leq T} \,\Big|  \mathfrak{l}_{t}^{T,\lambda,\varepsilon}\,-\,\mathbf{L}_{t}^{T,\lambda }\Big|\,\leq \, \max_{1\leq n \leq N_T^{\varepsilon} }\,\mathbf{L}_{\varrho_{n}^{\uparrow,\varepsilon}}^{T,\lambda}\,-\,\mathbf{L}_{\varrho_{n}^{\downarrow,\varepsilon}}^{T,\lambda}   \,.
\end{align*}
As a consequence of (i) of Lemma~\ref{LemmaSupTime}, we have the following inequality for small  $\varepsilon$:
\begin{align*}
\mathbf{E}_{\mu}^{T,\lambda} \bigg[ \,\sup_{0\leq t\leq T} \,\Big|  \mathfrak{l}_{t}^{T,\lambda,\varepsilon}\,-\,\mathbf{L}_{t}^{T,\lambda}\Big|^2 \,\bigg]\,\preceq \, \frac{1}{\log^{\frac{3}{2}}\frac{1}{\varepsilon} }\hspace{.2cm}\stackrel{\varepsilon\rightarrow 0 }{\longrightarrow }  \hspace{.2cm}0 \,.
\end{align*}
Thus, it suffices for us to show that $ \sup_{0\leq t\leq T} \big|  \ell_{t}^{\varepsilon}-\mathfrak{l}_{t}^{T,\lambda,\varepsilon}\big|$ vanishes in $L^1\big(\mathbf{P}_{\mu}^{T,\lambda} \big)$-norm as $\varepsilon\searrow 0$. \vspace{.2cm}

The process $ \mathfrak{l}^{T,\lambda,\varepsilon}$ is not adapted to the filtration $\mathscr{F}^{T,\mu}$ because $ \mathfrak{l}^{T,\lambda,\varepsilon}_t$ can depend on the outcome for $\mathbf{L}^{T,\lambda}$ up to the future time $\varrho_{n}^{\uparrow,\varepsilon}$ when $t\in \big[ \varrho_{n}^{\downarrow ,\varepsilon},  \varrho_{n}^{\uparrow,\varepsilon} \big)$. We will construct an altered filtration $\{\boldsymbol{\mathscr{F}}_t^{T,\mu,\varepsilon}\}_{t\in [0,T]}$  to which  $ \mathfrak{l}^{T,\lambda,\varepsilon}$ is adapted.  For any $t\in [0,T]$, define the $\mathscr{F}^{T,\mu}$ stopping time
$$ \tau_{t}^{\varepsilon}\,:=\,\begin{cases} \varrho_{n}^{\uparrow,\varepsilon} & \text{when   $t\in \big[ \varrho_{n}^{\downarrow ,\varepsilon},  \varrho_{n}^{\uparrow,\varepsilon} \big)$ for $n=N_{t}^{\varepsilon}$\,,} \\ t  & \text{otherwise.}    \end{cases}  $$
  The family of stopping times $\{\tau_{t}^{\varepsilon}\}_{t\in [0,T]}$ is increasing and right-continuous.  We define the filtration  $\{\boldsymbol{\mathscr{F}}_t^{T,\mu,\varepsilon}\}_{t\in [0,T]}$ by   $\boldsymbol{\mathscr{F}}_t^{T,\mu,\varepsilon}:= \mathscr{F}_{\tau_{t}^{\varepsilon}}^{T,\mu} $.  Note that $\boldsymbol{\mathscr{F}}_t^{T,\mu,\varepsilon}$
  is   right-continuous because $\{\tau_{t}^{\varepsilon}\}_{t\in [0,T]}$ and $ \mathscr{F}^{T,\mu}$ are. Moreover, since $\tau_t^{\varepsilon}=t$ for $t\in \big[ \varrho_{n-1}^{\uparrow ,\varepsilon},  \varrho_{n}^{\downarrow,\varepsilon} \big)$, the left  continuity of $\mathscr{F}^{T,\mu}$ implies that 
\begin{align}\label{FLeft}
\mathscr{F}_{\varrho_{n}^{\downarrow,\varepsilon}-}^{T,\mu,\varepsilon}\,=\,\mathscr{F}_{\varrho_{n}^{\downarrow, \varepsilon}-}^{T,\mu} \,=\,\mathscr{F}_{\varrho_{n}^{\downarrow,\varepsilon}}^{T,\mu} \,.
\end{align}

Let us define the $\boldsymbol{\mathscr{F}}^{T,\mu,\varepsilon}$-predictable (and, in fact, $\mathscr{F}^{T,\mu}$-predictable)  process $\big\{\overline{\mathfrak{l}}_{t}^{T,\lambda,\varepsilon}\big\}_{t\in [0,T]}$ by
\begin{align*}
\overline{\mathfrak{l}}_{t}^{T,\lambda,\varepsilon}\,:=\,\sum_{n=1}^{ N_{t}^{\varepsilon} }\,\mathbf{E}_{\mu}^{T,\lambda}\Big[\,\mathbf{L}_{\varrho_{n}^{\uparrow,\varepsilon}}^{T,\lambda}\,-\,\mathbf{L}_{\varrho_{n}^{\downarrow,\varepsilon}}^{T,\lambda} \,\Big|\,\mathscr{F}_{\varrho_{n}^{\downarrow,\varepsilon}-}^{T,\mu,\varepsilon}\,\Big]  \,=\,\sum_{n=1}^{ N_{t}^{\varepsilon} }  \,\mathbf{E}_{0}^{T-\varrho_{n}^{\downarrow,\varepsilon},\lambda}\Big[\,\mathbf{L}_{\varrho^{\uparrow,\varepsilon}}^{T-\varrho_{n}^{\downarrow,\varepsilon},\lambda}\,\Big] \,, \nonumber
\end{align*}
where the second equality uses~(\ref{FLeft}), the strong Markov property, and that $X_{\varrho_{n}^{\downarrow,\varepsilon}}=0$ when $\varrho_{n}^{\downarrow,\varepsilon}<\infty$.    The process  $\mathfrak{m}^{T,\lambda,\varepsilon}:=\mathfrak{l}^{T,\lambda,\varepsilon}-\overline{\mathfrak{l}}^{T,\lambda,\varepsilon}$ is a mean zero  $\mathbf{P}_{\mu}^{T,\lambda} $-martingale with respect to $\boldsymbol{\mathscr{F}}^{T,\mu,\varepsilon}$.  The predictable quadratic variation of $\mathfrak{m}^{T,\lambda,\varepsilon}$ has the form
\begin{align*}
    \big\langle \mathfrak{m}^{T,\lambda,\varepsilon}\big\rangle_t \,=\,\sum_{n=1}^{ N_{t}^{\varepsilon} } \,  \mathbf{Var}_{\mu}^{T,\lambda}\Big[\,\mathbf{L}_{\varrho_{n}^{\uparrow,\varepsilon}}^{T,\lambda}\,-\,\mathbf{L}_{\varrho_{n}^{\downarrow,\varepsilon}}^{T,\lambda} \,\Big|\,\mathscr{F}_{\varrho_{n}^{\downarrow,\varepsilon}-}^{T,\mu,\varepsilon}\,\Big]\,=\,\sum_{n=1}^{ N_{t}^{\varepsilon} }  \, \mathbf{Var}_{0}^{T-\varrho_{n}^{\downarrow,\varepsilon},\lambda}\Big[\,\mathbf{L}_{\varrho^{\uparrow,\varepsilon}}^{T-\varrho_{n}^{\downarrow,\varepsilon},\lambda}\,\Big]\,,
\end{align*}
where $ \mathbf{Var}_{\mu}^{t,\lambda}$ refers to  the variance of a random variable  under $ \mathbf{P}_{\mu}^{t,\lambda}$, and we have again applied~(\ref{FLeft}) along with the strong Markov property. By (i) of Lemma~\ref{LemmaUprossingLocal}, there is a $C_{T,\lambda}>0$ such that for all $t\geq 0$ and $\varepsilon\in (0,1)$
\begin{align}\label{QuadVarBound}
    \big\langle \mathfrak{m}^{T,\lambda,\varepsilon}\big\rangle_t \,\leq \,  N_{t}^{\varepsilon }\,\frac{ C_{T,\lambda} }{ \log^2\frac{1}{\varepsilon} }  \,.
\end{align}
Using Doob's maximal inequality, we get that
\begin{align*}
\mathbf{E}_{\mu}^{T,\lambda}\bigg[\,\sup_{0\leq t\leq T}\,\big(\mathbf{m}_{t}^{T,\lambda,\varepsilon}\big)^2\,\bigg]\,\leq \,&\,
4\,\mathbf{E}_{\mu}^{T,\lambda}\Big[\,\big(\mathbf{m}_{T}^{T,\lambda,\varepsilon}\big)^2\,\Big]\,=\,4\,\mathbf{E}_{\mu}^{T,\lambda}\Big[ \,\big\langle \mathbf{m}^{T,\lambda,\varepsilon}\big\rangle_T \,\Big] \,\stackrel{(\ref{QuadVarBound})}{\leq} \,\frac{ 4 C_{T,\lambda} }{ \log^2\frac{1}{\varepsilon}  }\, \mathbf{E}_{\mu}^{T,\lambda}\big[ \,N_{T }^{\varepsilon}\,\big] 
\,\preceq \,\frac{1}{ \log \frac{1}{\varepsilon}  } \,,
\end{align*}
in which  the 
third inequality holds  by Lemma~\ref{UpcroossingInequalityPre}.  In particular, $\sup_{0\leq t\leq T}\big|\mathfrak{l}^{T,\lambda,\varepsilon}-\overline{\mathfrak{l}}^{T,\lambda,\varepsilon}\big|$
vanishes in $L^1\big(\mathbf{P}_{\mu}^{T,\lambda}\big)$-norm with small $\varepsilon$.

It is now enough for us to prove that $\sup_{0\leq t\leq T}\big|\overline{\mathfrak{l}}_{t}^{T,\lambda,\varepsilon}-\ell_{t}^{\varepsilon} \big|$ vanishes in $L^1\big(\mathbf{P}_{\mu}^{T,\lambda}\big)$-norm as $\varepsilon \searrow 0$. For any $\delta\in (0,T)$, observe that we have the following inequalities:
\begin{align}\label{LSSS}
\mathbf{E}_{\mu}^{T,\lambda}\bigg[\,\sup_{0\leq t\leq T}\,\Big|\overline{\mathfrak{l}}_{t}^{T,\lambda,\varepsilon}-\ell_{t}^{\varepsilon} \Big|\,\bigg]\nonumber   \,\leq \,&\,
\mathbf{E}_{\mu}^{T,\lambda}\Bigg[\,\sum_{n=1}^{ N_{T}^{\varepsilon} } \, \Bigg|\mathbf{E}_{0}^{T-\varrho_{n}^{\downarrow ,\varepsilon},\lambda}\Big[\,\mathbf{L}_{\varrho^{\uparrow,\varepsilon}}^{T-\varrho_{n}^{\downarrow ,\varepsilon},\lambda}  \,\Big] \,-\,\frac{  1}{2\log \frac{1}{\varepsilon}}\Bigg|\,\Bigg] \nonumber \\
\,\leq \,&\,
\mathbf{E}_{\mu}^{T,\lambda}\Bigg[\,\sum_{n=1}^{ N_{T-\delta}^{\varepsilon} }  \,\Bigg|\mathbf{E}_{0}^{T-\varrho_{n}^{\downarrow ,\varepsilon},\lambda}\Big[\,\mathbf{L}_{\varrho^{\uparrow,\varepsilon}}^{T-\varrho_{n}^{\downarrow ,\varepsilon},\lambda} \,\Big] \,-\,\frac{  1}{2\log \frac{1}{\varepsilon}}\Bigg|\,\Bigg] \, +\,\frac{C_{T,\lambda} }{\log \frac{1}{\varepsilon} }\,\mathbf{E}_{\mu}^{T,\lambda}\big[\, N_{T}^{\varepsilon}-N_{T-\delta}^{\varepsilon}  \, \big] \nonumber 
 \\
\,\preceq \,&\,
\varepsilon^2\,\mathbf{E}_{\mu}^{T,\lambda}\big[ \, N_{T-\delta}^{\varepsilon}\, \big] \, +\,\frac{1}{\log \frac{1}{\varepsilon}  }\,\mathbf{E}_{\mu}^{T,\lambda}\Big[\, \mathbf{E}_{X_{T-\delta}}^{\delta,\lambda}\big[\,  N_{\delta}^{\varepsilon} \,  \big]  \,  \Big] 
\,.
\end{align}
For the second inequality, we have separated  out the last $N_{T}^{\varepsilon}-N_{T-\delta}^{\varepsilon} $ terms from the sum and bounded them using the $m=1$ case of (i) of Lemma~\ref{LemmaUprossingLocal}: 
$$ \Bigg|\mathbf{E}_{0}^{T-\varrho_{n}^{\downarrow ,\varepsilon},\lambda}\Big[\,\mathbf{L}_{\varrho^{\uparrow,\varepsilon}}^{T-\varrho_{n}^{\downarrow ,\varepsilon},\lambda} \,\Big] \,-\,\frac{  1}{2\log \frac{1}{\varepsilon}}\Bigg|\,\leq\, \mathbf{E}_{0}^{T-\varrho_{n}^{\downarrow ,\varepsilon},\lambda}\Big[\,\mathbf{L}_{\varrho^{\uparrow,\varepsilon}}^{T-\varrho_{n}^{\downarrow ,\varepsilon},\lambda}\, \Big] \,+\,\frac{  1}{2\log \frac{1}{\varepsilon}} \,\leq \, \frac{  C_{T,\lambda} }{ \log \frac{1}{\varepsilon} } \,. $$
For the third inequality in~(\ref{LSSS}), we have applied Lemma~\ref{LemmaUprossingLocal} to the first term and used the Markov property to bound the second term.  Finally, Lemma~\ref{UpcroossingInequalityPre} implies that~(\ref{LSSS}) is bounded  by a constant multiple of 
$$ \varepsilon^2\,\frac{ 1+\log^+\frac{T}{\varepsilon^2 }  }{1+\log^+ \frac{1}{T\lambda} }
\, +\,\frac{1}{\log \frac{1}{\varepsilon}  }\,\frac{ 1+\log^+\frac{\delta}{\varepsilon^2 }  }{1+\log^+ \frac{1}{\delta\lambda} }  \hspace{.3cm}  \stackrel{\varepsilon\rightarrow 0  }{\longrightarrow  } \hspace{.3cm} \frac{2}{\log^+ \frac{ 1}{\delta \lambda } } \hspace{.3cm}  \stackrel{\delta\rightarrow 0  }{\longrightarrow  } \hspace{.3cm} 0
 $$
for all $\varepsilon\in (0,1)$, $\delta\in (0,T)$. Since the above converges to $ 2 (\log^+ \frac{ 1}{\delta \lambda } )^{-1}$ as $\varepsilon \searrow 0$ and  $\delta$ is arbitrary, we have shown that  $\sup_{0\leq t\leq T}\big|\overline{\mathfrak{l}}_t^{T,\lambda,\varepsilon}-\ell_{t}^{\varepsilon}\big|$ vanishes in $L^1\big( \mathbf{P}_{\mu}^{T,\lambda}\big)$-norm with small $\varepsilon$.    Therefore, putting our results together,  $\sup_{0\leq t\leq T}\big|\ell_{t}^{\varepsilon} -\mathbf{L}_{t}^{T,\lambda}\big|$ vanishes in $L^1\big( \mathbf{P}_{\mu}^{T,\lambda} \big)$-norm as $\varepsilon \searrow 0$.
\end{proof}

The following corollary of Theorem~\ref{ThmLocalTime1} is the same as Corollary~\ref{THMLocalTime1.5} except that it uses the definition $\mathbf{L}^{T,\lambda} :=-\mathring{\mathcal{A}}^{T,\lambda}$ of the local time.
\begin{corollary}\label{ThmLocalTime2} Fix some $T,\lambda >0$ and a  Borel probability measure $\mu$  on $\R^2$. For $\varepsilon\in (0,1)$
  define the process   $\{\mathbf{L}_{t}^{\varepsilon}\}_{t\in [0,\infty)}$ by $ \mathbf{L}_{t}^{\varepsilon}\,:=\,\frac{1 }{ \varepsilon^2 \log \frac{1}{\varepsilon} }\sum_{n=1}^{N_{t}^{\varepsilon}} \big(\varrho_{n}^{\uparrow,\varepsilon}\wedge t-\varrho_{n}^{\downarrow ,\varepsilon}\big) $.    
Then the random variable
 $\sup_{0\leq t\leq T}\big| \mathbf{L}_{t}^{\varepsilon}-\mathbf{L}_{t}^{T,\lambda}\big|$ vanishes  in $L^1\big(\mathbf{P}_{\mu}^{T,\lambda}\big)$-norm as $\varepsilon \searrow 0$.  
\end{corollary}

\begin{proof} In view of Theorem~\ref{ThmLocalTime1}, 
 it is adequate for us to demonstrate that  $\sup_{0\leq t\leq T}\big|\mathbf{L}_{t}^{\varepsilon}-\ell_{t}^{\varepsilon} \big|$ vanishes in $L^1\big( \mathbf{P}_{\mu}^{T,\lambda} \big)$-norm as $ \varepsilon\searrow 0$.
Define the processes $\{  \mathfrak{L}_t^{\varepsilon} \}_{t\in [0,T]}$ and $\{  \overline{\mathfrak{L}}_t^{\varepsilon} \}_{t\in [0, T]}$ as follows:
\begin{align*}
\mathfrak{L}_t^{\varepsilon}  \,:=\, & \, \frac{1 }{ \varepsilon^2 \log \frac{1}{\varepsilon} }\,\sum_{n=1}^{ N_{t}^{\varepsilon} }\,  \varrho_{n}^{\uparrow,\varepsilon}-\varrho_{n}^{\downarrow,\varepsilon}\,, \\
    \mathbf{\overline{\mathfrak{L}}}_{t}^{\varepsilon}\,:=\,&\,\frac{1 }{ \varepsilon^2 \log \frac{1}{\varepsilon} }\,\sum_{n=1}^{ N_{t}^{\varepsilon} }\,  \mathbf{E}_{\mu}^{T,\lambda}\Big[ \,  \varrho_{n}^{\uparrow,\varepsilon}-\varrho_{n}^{\downarrow,\varepsilon}\,\Big|\,\mathscr{F}_{ \varrho_{n}^{\downarrow,\varepsilon}- }^{T,\mu,\varepsilon} \, \Big] \,=\,\frac{1 }{ \varepsilon^2 \log \frac{1}{\varepsilon} }\,\sum_{n=1}^{ N_{t}^{\varepsilon} }\,  \mathbf{E}_{0}^{T-\varrho_{n}^{\downarrow,\varepsilon},\lambda}\big[\, \varrho^{\uparrow, \varepsilon}\,  \big]  \,,
\end{align*}
where we have used that $\mathscr{F}_{ \varrho_{n}^{\downarrow,\varepsilon}- }^{T,\mu,\varepsilon}=\mathscr{F}_{ \varrho_{n}^{\downarrow,\varepsilon} }^{T,\mu}$, the strong Markov property, and that $X_{\varrho_{n}^{\downarrow,\varepsilon}}=0$ when $\varrho_{n}^{\downarrow,\varepsilon} <\infty$.
The $L^2\big(\mathbf{P}^{T,\lambda}_{\mu}\big)$-norm of $\sup_{0\leq t\leq T}\big|\mathfrak{L}_t^{\varepsilon} -\mathbf{L}_{t}^{\varepsilon}\big|$ vanishes with small $\varepsilon$ as a consequence of (ii) of Lemma~\ref{LemmaSupTime} in the case $p=2$ because
\begin{align*}
\sup_{0\leq t\leq T}\,\big|\mathfrak{L}_t^{\varepsilon}\,-\, \mathbf{L}_{t}^{\varepsilon}  \big|\,\leq \,\frac{1 }{ \varepsilon^2 \log \frac{1}{\varepsilon} }\,\max_{1\leq n\leq N_{T}^{\varepsilon}}\,\varrho_{n}^{\uparrow,\varepsilon}-\varrho_{n}^{\downarrow,\varepsilon} \,.
\end{align*}
Moreover, $\sup_{0\leq t\leq T}\big|\mathfrak{L}_t^{\varepsilon}  - \mathbf{\overline{\mathfrak{L}}}_{t}^{\varepsilon} \big|$     vanishes in $L^2\big(\mathbf{P}^{T,\lambda}_{\mu}\big)$-norm with small   $\varepsilon $ by a similar argument as used for $\sup_{0\leq t\leq T}\big| \mathfrak{l}_t^{T,\lambda,\varepsilon} - \mathbf{\overline{\mathfrak{l}}}_{t}^{T,\lambda,\varepsilon}  \big|$ in the proof of Theorem~\ref{ThmLocalTime1}.  Finally, given some $\delta\in (0,T)$, the difference between the processes $\overline{\mathfrak{L}}^{T,\varepsilon}$ and $\ell^{\varepsilon}$ is bounded as follows: 
\begin{align*}
\mathbf{E}_{\mu}^{T,\lambda}\bigg[\,\sup_{0\leq t\leq T}\,\Big|\overline{\mathfrak{L}}_{t}^{\varepsilon}- \ell_{t}^{\varepsilon} \Big|\,\bigg] \,\leq \,&\,
\frac{  1}{\varepsilon^2\log \frac{1}{\varepsilon}}\,\mathbf{E}_{\mu}^{T,\lambda}\Bigg[\,\sum_{n=1}^{ N_{T}^{\varepsilon} }\,  \bigg|\mathbf{E}_{0}^{T-\varrho_{n}^{\downarrow ,\varepsilon},\lambda}\big[\,\varrho^{\uparrow,\varepsilon} \,\big] \,-\,\frac{  \varepsilon^2}{2}\bigg|\,\Bigg] \nonumber \\
\,\leq \,&\,
\frac{  1}{\varepsilon^2\log\frac{1}{\varepsilon}}\,\mathbf{E}_{\mu}^{T,\lambda}\Bigg[\,\sum_{n=1}^{ N_{T-\delta}^{\varepsilon} } \, \bigg|\mathbf{E}_{0}^{T-\varrho_{n}^{\downarrow ,\varepsilon},\lambda}\big[\,\varrho^{\uparrow,\varepsilon} \,\big] \,-\,\frac{  \varepsilon^2}{2}\bigg|\,\Bigg]  \, +\,\frac{C_{T,\lambda}}{\log \frac{1}{\varepsilon} }\,\mathbf{E}_{\mu}^{T,\lambda}\big[\, N_{T}^{\varepsilon}-N_{T-\delta}^{\varepsilon}  \, \big] \nonumber 
 \\
\,\leq \,&\,
\frac{C_{T,\lambda} }{\log^2 \frac{1}{\varepsilon}  }\,\mathbf{E}_{\mu}^{T,\lambda}\big[\,  N_{T-\delta}^{\varepsilon} \,\big] \, +\,\frac{C_{T,\lambda}}{\log\frac{1}{\varepsilon}  }\,\mathbf{E}_{\mu}^{T,\lambda}\Big[ \,\mathbf{E}_{X_{T-\delta}}^{\delta,\lambda}\big[\,  N_{\delta}^{\varepsilon}\,   \big] \,   \Big] \nonumber  \\
\,\preceq \,&\,
\frac{1}{\log^2 \frac{1}{\varepsilon}  }\,\frac{ 1+\log^+\frac{T}{\varepsilon^2 }  }{1+\log^+ \frac{1}{T\lambda} } \, +\,\frac{1}{\log\frac{1}{\varepsilon}  }\,\frac{ 1+\log^+\frac{\delta}{\varepsilon^2 }  }{1+\log^+ \frac{1}{\delta\lambda} } \hspace{.3cm}\stackrel{ \varepsilon\rightarrow 0 }{\longrightarrow  } \hspace{.3cm}\frac{ 2  }{1+\log^+ \frac{1}{\delta\lambda} } \,,
\end{align*}
where we have applied  Lemma~\ref{LemmaLeave} to get the second and third inequalities  and Lemma~\ref{UpcroossingInequalityPre} for the fourth.  The above vanishes with small $\varepsilon$ and $\delta$.
\end{proof}

\subsection{The second downcrossing approximation for the  local time  at the origin}\label{SubsectionDowncrossing2}
The focus of this subsection is on proving Theorem~\ref{ThmLocalTime3} below, which has the same content as Theorem~\ref{THMLocalTime2}, except that we continue to use the local time  definition $\mathbf{L}:=-\mathring{\mathcal{A}}^{T,\lambda}$. Given $\varepsilon>0$ recall that the sequences of stopping times $\{ \widetilde{\varrho}_{n}^{\downarrow,\varepsilon} \}_{n\in \mathbb{N}} $ and  $\{ \widetilde{\varrho}_{n}^{\uparrow,\varepsilon} \}_{n\in \mathbb{N}_0} $ are defined as in~(\ref{VARRHOS2}) and that $\widetilde{N}_{t}^{\varepsilon}$ denotes the number of times $ \widetilde{\varrho}_{n}^{\downarrow,\varepsilon} $ in the interval $[0,t]$. 
\begin{theorem}\label{ThmLocalTime3} Fix some $T,\lambda > 0$ and  a Borel probability measure $\mu$ on $\R^2$. For $\varepsilon\in (0,1)$ define the process $\{\widetilde{\ell}_t^{\varepsilon}\}_{t\in [0,\infty)}  $ by $ \widetilde{\ell}_t^{\varepsilon}= \frac{\log 2}{2\log^2\frac{1}{\varepsilon}  }\widetilde{N}_{t}^{\varepsilon} $.   The random variable $\sup_{0\leq t\leq T}\big|  \widetilde{\ell}_t^{\varepsilon}\,-\,\mathbf{L}_{t} \big|  $  vanishes in $L^1\big(\mathbf{P}^{T,\lambda}_{\mu}\big)$-norm as $\varepsilon \searrow 0$.
\end{theorem}
 
Next, we state three lemmas that will be employed in the proof of Theorem~\ref{ThmLocalTime3}. The following is an analog of Lemma~\ref{UpcroossingInequalityPre}, and we omit the proof.
\begin{lemma}\label{UpcroossingInequalityPreII}     For any $L>0$  there exists a  $C_L>0$ such that for all $x\in \R^2$, $\varepsilon \in (0,\frac{1}{2})$, and $T,\lambda>0$,  with $T\lambda\leq L$
$$   \mathbf{E}^{T,\lambda}_{x}\big[\, \widetilde{N}_{T}^{\varepsilon}\,\big] \,\leq \,   C_L\,\frac{ \big(1+\log^+\frac{T}{\varepsilon^2}\big)^2 }{1+\log^+\frac{1}{T\lambda} }  \,. $$
In particular, for any  fixed $T,\lambda>0$ there exists a $C_{T,\lambda}>0$ such that $ \mathbf{E}^{T,\lambda}_{x}\big[ \widetilde{N}_{T}^{\varepsilon}\big] $ is bounded by $C_{T,\lambda}\log^2 \frac{1}{\varepsilon}  $ for all $\varepsilon\in (0,\frac{1}{2})$.
\end{lemma}
The  lemma below will be applied in a similar way as  Lemma~\ref{LemmaUprossingLocal} in the proof of Theorem~\ref{ThmLocalTime1}, and its 
proof is in Section~\ref{SubsectionLemmaDowncrossingNumber}. 
\begin{lemma}\label{LemmaDowncrossingNumber} Given  $\varepsilon>0$  define the $\mathscr{F}^X$-stopping time
$\varrho^{\uparrow, \varepsilon}=\inf\{ t\in [0,\infty)\,:\, |X_t| \geq \varepsilon   \}$, and let the sequence of stopping times $\{ \varrho_{n}^{\downarrow,\varepsilon} \}_{n\in \mathbb{N}} $ be defined as in~(\ref{VARRHOS}).   Let $\mathbf{n}^{\varepsilon}$ denote the number of $n\in \mathbb{N}$ with $ \varrho_{n}^{\downarrow,\varepsilon}<\varrho^{\uparrow, 2\varepsilon}$.
\begin{enumerate}[(i)]
\item For any $L >0$ and $p\in \mathbb{N}$, there exists a $C_{L,p}>0 $  such that for all $\varepsilon \in (0,\frac{1}{2})$, $T,\lambda \in (0,  L]$, and  $x\in \R^2 $ with $|x|=\varepsilon$
$$  \mathbf{E}_{x}^{T,\lambda}\big[\, (\mathbf{n}^{\varepsilon})^p\,\big]\,\leq \, \frac{ C_{L,p}}{\log \frac{1}{\varepsilon}  } \,.  $$

\item For any $L >1$ there exists a $C_{L}>0$ such that for all $\varepsilon \in (0,\frac{1}{2})$, $T,\lambda\in \big[\frac{1}{L}, L\big]$, and  $x\in \R^2 $ with $|x|=\varepsilon$
$$ \bigg| \mathbf{E}_{x}^{T,\lambda}\big[ \,\mathbf{n}^{\varepsilon} \,\big]\,-\,\frac{ \log 2  }{\log  \frac{1}{\varepsilon}  }  \bigg| \,\leq \, \frac{ C_{L}}{\log^2 \frac{1}{\varepsilon}  } \,.  $$

\end{enumerate}

\end{lemma}
\begin{remark} The content of the above lemma can be summarized by saying that  $\mathbf{n}^{\varepsilon}$ is  roughly a Bernoulli random variable under $\mathbf{P}_{x}^{T,\lambda}$, taking the value $1$ with probability $\mathbf{p}_{\varepsilon}=\frac{ \log 2  }{\log  \frac{1}{\varepsilon}  } $.  Hence, the event that $\varrho_{1}^{\downarrow,\varepsilon}< \varrho^{\uparrow, 2\varepsilon}$ happens with a low probability when $\varepsilon\ll 1$, and there is a negligible probability of $\varrho_{2}^{\downarrow,\varepsilon}< \varrho^{\uparrow, 2\varepsilon}$ occurring.
\end{remark}

Observe that each upcrossing interval $\big[\varrho_{m}^{\downarrow,\varepsilon},\varrho_{m}^{\uparrow,\varepsilon}\big]$ is a subset of some upcrossing interval $\big[\widetilde{\varrho}_{n}^{\downarrow,\varepsilon},\widetilde{\varrho}_{n}^{\uparrow,\varepsilon}\big]$.  Let $\mathbf{n}^{\varepsilon}_n$ denote the number of times $\varrho_{m}^{\downarrow,\varepsilon}$ in the interval $\big[\widetilde{\varrho}_{n}^{\downarrow,\varepsilon},\widetilde{\varrho}_{n}^{\uparrow,\varepsilon}\big]$.  The lemma below can be proven using (i) of Lemma~\ref{LemmaDowncrossingNumber} and the same  argument as for  Lemma~\ref{LemmaSupTime}.
\begin{lemma}\label{LemmaSupTime2}  For any $ L>0$ and $p>1$, there exists a $C_{L,p}>0$ such that for all $T,\lambda \in (0, L]$, $x\in \R^2$, and $\varepsilon \in (0,\frac{1}{2})$ we have 
 $$\mathbf{E}^{T,\lambda}_{x}\bigg[\,\max_{1\leq n \leq \widetilde{N}_T^{\varepsilon} }\,\big(\mathbf{n}^{\varepsilon}_n \big)^p\, \bigg] \,\leq\, C_{L,p}\,\log\frac{1}{\varepsilon} \,. $$
\end{lemma}

The proof of Theorem~\ref{ThmLocalTime3} has a similar structure to that of Theorem~\ref{ThmLocalTime1}.  We introduce intermediary processes $\boldsymbol{\ell}^{\varepsilon}$ and $\boldsymbol{\bar{\ell}}^{\varepsilon}$, which have analogous roles to $\mathfrak{l}^{T,\lambda,\varepsilon}$ and $\overline{\mathfrak{l}}^{T,\lambda,\varepsilon}$ in the proof of Theorem~\ref{ThmLocalTime1}, and we show that the process differences $\boldsymbol{\bar{\ell}}^{\varepsilon}-\widetilde{\ell}^{\varepsilon}$, $\boldsymbol{\ell}^{\varepsilon}-\boldsymbol{\bar{\ell}}^{\varepsilon}$, and $\ell^{\varepsilon}-\boldsymbol{\ell}^{\varepsilon}$  vanish as $\varepsilon\searrow 0$.

\begin{proof}[Proof of Theorem~\ref{ThmLocalTime3}] Recall that we define the process $\{\ell_t^{\varepsilon} \}_{t\in [0,T] } $ by $ \ell_t^{\varepsilon}= \frac{N_{t}^{\varepsilon} }{2\log \frac{1}{\varepsilon}  }$.  In consequence of Theorem~\ref{ThmLocalTime1}, it suffices for us to show that $\sup_{0\leq t\leq T}\big|\ell_t^{\varepsilon}-\widetilde{\ell}_t^{\varepsilon}  \big| $
vanishes in $L^1\big(\mathbf{P}^{T,\lambda}_{\mu}\big)$-norm as $\varepsilon\searrow 0$.  Let $\mathbf{n}^{\varepsilon}_{n}$ denote the number of  times in the sequence $\{ \varrho_{m}^{\downarrow,\varepsilon}\}_{m\in \mathbb{N}}$ to occur in the upcrossing interval $\big[\widetilde{\varrho}_{n}^{\downarrow,\varepsilon}, \widetilde{\varrho}_{n}^{\uparrow,\varepsilon}     \big]$.
Define the process $\{ \boldsymbol{\bar{\ell}}_t^{\varepsilon}\}_{t\in [0,T]} $ by
$$\boldsymbol{\bar{\ell}}_t^{\varepsilon}\,:=\, \frac{1}{2\log \frac{1}{\varepsilon} } \,\sum_{n=1}^{\widetilde{N}_t^{\varepsilon}} \, \mathbf{E}^{T-\widetilde{\varrho}_{n}^{\downarrow,\varepsilon},\lambda}_{X_{\widetilde{\varrho}_{n}^{\downarrow,\varepsilon}}}[\, \mathbf{n}^{\varepsilon}\,]\,, $$
 where  $\mathbf{n}^{\varepsilon} $ is  defined  as in Lemma~\ref{LemmaDowncrossingNumber}. 
Our first step will be to show that $\sup_{0\leq t\leq T}\big|  \boldsymbol{\bar{\ell}}_t^{\varepsilon}-\widetilde{\ell}_t^{\varepsilon}\big| $
vanishes in $L^1\big(\mathbf{P}^{T,\lambda}_{\mu}\big)$-norm as $\varepsilon\searrow 0$.  Since the difference between $\widetilde{\ell}_t^{\varepsilon}$  and $ \boldsymbol{\bar{\ell}}_t^{\varepsilon}$ can be expressed as
\begin{align*}
\boldsymbol{\bar{\ell}}_t^{\varepsilon}\,-\,\widetilde{\ell}_t^{\varepsilon} \,=\,  \frac{1}{2\log \frac{1}{\varepsilon} }\, \sum_{n=1}^{\widetilde{N}_t^{\varepsilon}} \,\bigg(  \mathbf{E}^{T-\widetilde{\varrho}_{n}^{\downarrow,\varepsilon},\lambda}_{X_{\widetilde{\varrho}_{n}^{\downarrow,\varepsilon}}}[\, \mathbf{n}^{\varepsilon}\,] \,-\, \frac{ \log 2 }{\log \frac{1}{\varepsilon}  } \bigg)\,,
\end{align*}
we have the following inequalities for any choice of $ \delta \in (0,T)$:
\begin{align}\label{Toof}
 \mathbf{E}^{T,\lambda}_{\mu}\bigg[ \,\sup_{0\leq t\leq T}\,\Big| \boldsymbol{\bar{\ell}}_t^{\varepsilon} -\widetilde{\ell}_t^{\varepsilon} \Big|\,\bigg]   \,\leq \,&\,\frac{1}{2\log \frac{1}{\varepsilon} }\,\mathbf{E}^{T,\lambda}_{\mu}\Bigg[\,\sum_{n=1}^{\widetilde{N}_T^{\varepsilon}}\, \Bigg|\mathbf{E}^{T-\,\widetilde{\varrho}_{n}^{\downarrow,\varepsilon},\lambda}_{X_{\widetilde{\varrho}_{n}^{\downarrow,\varepsilon}}}[ \,\mathbf{n}^{\varepsilon}\,]\,-\,\frac{\log 2}{\log \frac{1}{\varepsilon}  }\Bigg|\, \Bigg] \nonumber \\  \,\leq \,&\,\frac{C_{T,\lambda}}{\log^3 \frac{1}{\varepsilon} }\,\mathbf{E}^{T,\lambda}_{\mu}\big[\,\widetilde{N}_{T-\delta}^{\varepsilon} \,\big] \,+\, \frac{C_{T,\lambda} }{\log^2 \frac{1}{\varepsilon}  }\,\mathbf{E}^{T,\lambda}_{\mu}\big[\,\widetilde{N}_{T}^{\varepsilon}-\widetilde{N}_{T-\delta}^{\varepsilon} \,\big] \nonumber  \\  \,\preceq\,&\,\frac{1 }{\log^3 \frac{1}{\varepsilon} }\,\frac{ \big(1+\log^+\frac{T}{\varepsilon^2}\big)^2 }{1+\log^+\frac{1}{T\lambda} } \,+\, \frac{1 }{\log^2 \frac{1}{\varepsilon}  }\,\frac{ \big(1+\log^+\frac{\delta}{\varepsilon^2}\big)^2 }{1+\log^+\frac{1}{\delta\lambda} } \hspace{.3cm}\stackrel{\varepsilon\rightarrow 0}{\longrightarrow} \hspace{.3cm}\frac{4}{1+\log^+\frac{1}{\delta\lambda} }\,. 
\end{align}
The second inequality applies (i) of Lemma~\ref{LemmaDowncrossingNumber} to the terms in the sum $\sum_{n=1}^{\widetilde{N}_T^{\varepsilon}}$ with $n>\widetilde{N}_{T-\delta}^{\varepsilon}$ and (ii) of Lemma~\ref{LemmaDowncrossingNumber} to the remaining terms.  The third inequality uses Lemma~\ref{UpcroossingInequalityPreII} twice, where for the second term, we invoke the Markov property:
\begin{align*}
\mathbf{E}^{T,\lambda}_{\mu}\big[\,\widetilde{N}_{T}^{\varepsilon}-\widetilde{N}_{T-\delta}^{\varepsilon}\, \big] \,=\, \mathbf{E}^{T,\lambda}_{\mu}\Big[\,\mathbf{E}^{T,\lambda}_{\mu}\big[\, \widetilde{N}_{T}^{\varepsilon}-\widetilde{N}_{T-\delta}^{\varepsilon}\,\big|\, \mathscr{F}_{T-\delta}^{T,\mu} \,\big]\,\Big]\,\leq  \,\mathbf{E}^{T,\lambda}_{\mu}\Big[\,\mathbf{E}^{\delta,\lambda}_{X_{T-\delta}}\big[\, \widetilde{N}_{\delta}^{\varepsilon}\,\big]\,\Big] \,\leq \,\sup_{x\in \R^2}  \,\mathbf{E}^{\delta,\lambda}_{x}\big[\, \widetilde{N}_{\delta}^{\varepsilon}\,\big]\,.\nonumber 
\end{align*}
Since $\delta\in (0,T)$ is arbitrary and the right side of~(\ref{Toof}) approaches $ 4(1+\log^+\frac{1}{\delta\lambda})^{-1}$ with small $\varepsilon$, we have shown that  $\sup_{0\leq t\leq T}\big|  \boldsymbol{\bar{\ell}}_t^{\varepsilon}-\widetilde{\ell}_t^{\varepsilon}\big| $
vanishes in $L^1\big(\mathbf{P}^{T,\lambda}_{\mu}\big)$-norm  as $\varepsilon\searrow 0$.

For $\varepsilon\in (0,1)$ and $t\in [0,\infty)$, let   $\widetilde{\tau}_t^{\varepsilon}$ denote the $\mathscr{F}^{T,\mu}$-stopping time given  by
\begin{align*}
    \widetilde{\tau}_t^{\varepsilon}\,=\,\begin{cases}    \widetilde{\varrho}_{n}^{\uparrow,\varepsilon}   &   \,\,\,\text{when}\,\, t\in  \big[\widetilde{\varrho}_{n}^{\downarrow,\varepsilon}\,, \widetilde{\varrho}_{n}^{\uparrow,\varepsilon}\big) \,\text{for\,  $n=\widetilde{N}_{t}^{\varepsilon}$,} \\ t & \,\,\,\text{otherwise.}      \end{cases}
\end{align*}
Define the filtration $\{\widetilde{\mathscr{F}}^{T,\mu,\varepsilon}_{t}\}_{t\in [0,\infty)}$ by $\widetilde{\mathscr{F}}_{t}^{T,\mu,\varepsilon}:=\mathscr{F}_{\widetilde{\tau}_t^{\varepsilon}}^{T,\mu}$, which inherits right-continuity  from $\mathscr{F}^{T,\mu}$ and $\{\widetilde{\tau}^{\varepsilon}_t\}_{t\in [0,\infty)}$.  Define the $\widetilde{\mathscr{F}}^{T,\mu,\varepsilon}$-adapted process  $\{ \boldsymbol{\ell}_t^{\varepsilon}\}_{t\in [0,T]}$ by $\boldsymbol{\ell}_t^{\varepsilon}:=\widetilde{\ell}_{\widetilde{\tau}_t^{\varepsilon}}^{\varepsilon} $. 
 We next show that $\sup_{0\leq t\leq T}\big| \boldsymbol{\ell}_t^{\varepsilon}- \boldsymbol{\bar{\ell}}_t^{\varepsilon}\big| $
vanishes in $L^2\big(\mathbf{P}^{T,\lambda}_{\mu}\big)$-norm as $\varepsilon\searrow 0$ using a similar argument as for $\sup_{0\leq t\leq T}\big| \mathfrak{l}_t^{T,\lambda,\varepsilon}- \bar{\mathfrak{l}}_t^{T,\lambda,\varepsilon}\big| $ in the proof of Theorem~\ref{ThmLocalTime1}.  To avoid excessive repetition, we will not elaborate in detail here. The process $\mathbf{m}^{\varepsilon}:=\boldsymbol{\ell}^{\varepsilon}-\boldsymbol{\bar{\ell}}^{\varepsilon} $ is a mean-zero $\mathbf{P}^{T,\lambda}_{\mu }$-martingale with respect to the filtration $\widetilde{\mathscr{F}}^{T,\mu,\varepsilon}$ because  it can be written in the form
\begin{align*}
\mathbf{m}^{\varepsilon}_t\,=\,   \frac{1}{2\log \frac{1}{\varepsilon} } \,\sum_{n=1}^{\widetilde{N}_t^{\varepsilon}}\,  \mathbf{n}^{\varepsilon}_n-\mathbf{E}^{T-\widetilde{\varrho}_{n}^{\downarrow,\varepsilon},\lambda}_{X_{\widetilde{\varrho}_{n}^{\downarrow,\varepsilon}} }\big[\, \mathbf{n}^{\varepsilon}\,\big]  \,=\, \frac{1}{2\log \frac{1}{\varepsilon}  } \,\sum_{n=1}^{\widetilde{N}_t^{\varepsilon}} \,\bigg(  \mathbf{n}^{\varepsilon}_{n} - \mathbf{E}^{T,\lambda}_{\mu }\Big[\, \mathbf{n}^{\varepsilon}_n\,\Big|\, \widetilde{\mathscr{F}}_{\widetilde{\varrho}_{n}^{\downarrow,\varepsilon}-}^{T,\mu,\varepsilon}\, \Big]\bigg)\,,
\end{align*}
where we have used that $\widetilde{\mathscr{F}}_{\widetilde{\varrho}_{n}^{\downarrow,\varepsilon}-}^{T,\mu,\varepsilon}=\mathscr{F}_{\widetilde{\varrho}_{n}^{\downarrow,\varepsilon}-}^{T,\mu} =\mathscr{F}_{\widetilde{\varrho}_{n}^{\downarrow,\varepsilon}}^{T,\mu} $ and the strong Markov property.
  The predictable quadratic variation of the martingale $ \mathbf{m}^{\varepsilon}$ has the form
\begin{align*}
 \big\langle \mathbf{m}^{\varepsilon}\big\rangle_t\,=\,  \frac{1}{4\log^2 \frac{1}{\varepsilon} } \,\sum_{n=1}^{\widetilde{N}_t^{\varepsilon}}\, \mathbf{Var}^{T,\lambda}_{\mu }\Big[\,  \mathbf{n}^{\varepsilon}_n \,\Big|\,   \widetilde{\mathscr{F}}_{\widetilde{\varrho}_{n}^{\downarrow,\varepsilon}-}^{T,\mu,\varepsilon}\,\Big] \,=\,  \frac{1}{4\log^2 \frac{1}{\varepsilon} }\, \sum_{n=1}^{\widetilde{N}_t^{\varepsilon}}\, \mathbf{Var}^{T-\widetilde{\varrho}_{n}^{\downarrow,\varepsilon},\lambda}_{X_{\widetilde{\varrho}_{n}^{\downarrow,\varepsilon} }}\big[\,  \mathbf{n}^{\varepsilon}\,\big] \,,\nonumber
\end{align*}
which we can bound  as follows:
\begin{align} \label{Folp}
 \langle \mathbf{m}^{\varepsilon}\rangle_t\,\leq \, \frac{1}{4\log^2 \frac{1}{\varepsilon} }\, \sum_{n=1}^{\widetilde{N}_t^{\varepsilon}} \,\mathbf{E}^{T-\widetilde{\varrho}_{n}^{\downarrow,\varepsilon},\lambda}_{X_{\widetilde{\varrho}_{n}^{\downarrow,\varepsilon}} }\Big[ \, \big(\mathbf{n}^{\varepsilon}\big)^2\, \Big] \,\leq \,  \frac{ C_{T,\lambda}\widetilde{N}_t^{\varepsilon} }{ \log^3 \frac{1}{\varepsilon}   }   \,,
\end{align}
with the second inequality holding for some constant $C_{T,\lambda}>0$ by  (i) of Lemma~\ref{LemmaDowncrossingNumber}.  Thus, by Doob's maximal inequality we have
\begin{align*}
\mathbf{E}^{T,\lambda}_{\mu }\bigg[ \,\sup_{0\leq t\leq T} \, \big(\mathbf{m}_t^{\varepsilon} \big)^2 \, \bigg] \,\leq \,4\,\mathbf{E}^{T,\lambda}_{\mu }\Big[ \,\big(\mathbf{m}_T^{\varepsilon} \big)^2 \, \Big] \,=\,4 \,
 \mathbf{E}^{T,\lambda}_{\mu }\Big[ \, \big\langle \mathbf{m}^{\varepsilon}\big\rangle_T \, \Big] \,\stackrel{(\ref{Folp})}{\leq} \,\frac{ 4\,C_{T,\lambda} }{ \log^3 \frac{1}{\varepsilon}   }\,  
 \mathbf{E}^{T,\lambda}_{\mu }\big[ \, \widetilde{N}_t^{\varepsilon}\,\big] \,\preceq \, \frac{ 1 }{ \log \frac{1}{\varepsilon}   }   \,,
\end{align*}
where the third inequality holds by Lemma~\ref{UpcroossingInequalityPreII}. The above implies that   $ \sup_{0\leq t\leq T}\big| \boldsymbol{\ell}_t^{\varepsilon} -\boldsymbol{\bar{\ell}}_t^{\varepsilon} \big|   $ vanishes in $L^2\big(\mathbf{P}^{T,\lambda}_{\mu }\big)$-norm as $\varepsilon \searrow 0$.

Our last step is to show that  $\sup_{0\leq t\leq T}\big|\ell_t^{\varepsilon}-\boldsymbol{\ell}_t^{\varepsilon}\big| $
vanishes in $L^1\big(\mathbf{P}^{T,\lambda}_{\mu}\big)$-norm, which we can do by observing that $\sup_{0\leq t\leq T}\,\big|\ell_t^{\varepsilon}- \boldsymbol{\ell}_t^{\varepsilon} \big| =\frac{1}{\log \frac{1}{\varepsilon}  } \sup_{1\leq n \leq\widetilde{N}_T^{\varepsilon}   }  \mathbf{n}^{\varepsilon}_n $, and then applying Lemma~\ref{LemmaSupTime2} in the case $p=2$.  
\end{proof}

\subsection{Approximating the  local time  using continuous additive functionals}\label{SectionLocalTimeAdditive}

In this subsection, we will prove the theorem below, which specializes to Theorem~\ref{ThmExistenceLocalTime} by taking the family of functions $\{\phi_{\varepsilon}\}_{\varepsilon \in (0,1)}$ on $[0,\infty)$ to have the form $\phi_{\varepsilon}=\frac{1}{2\varepsilon^2\log^2\frac{1}{\varepsilon} }1_{[0,\varepsilon]} $.

\begin{theorem}\label{TheoremLocalTimeAdditive} Fix some $T, \lambda>0$ and a Borel probability measure $\mu$ on $\R^2$.  Given  a measurable function $\phi:[0,\infty)\rightarrow [0,\infty)$,   define the process $\{ L_t^{\phi}  \}_{t\in [0,\infty)}$ by 
$L_t^{\phi}=\int_0^t\phi\big(|X_r|\big)dr$.   Let  $\{\phi_{\varepsilon}\}_{\varepsilon\in (0,1)}$ be a family of nonnegative measurable functions on $[0,\infty)$ such that
\begin{enumerate}[(I)]

\item  $\phi_{\varepsilon}$ is supported on $[0,\varepsilon ]$,

\item $\sup_{\varepsilon \in (0,1)}  \varepsilon^2  \big( \log^2  \frac{1}{\varepsilon} \big) \,\|\phi_{\varepsilon}\|_{\infty} < \infty$, and

\item $ \mathcal{E}(\varepsilon):=1-  4 \big(\log^2 \frac{1}{\varepsilon}\big) \int_{0}^{\varepsilon} a\phi_{\varepsilon}(a)da $ vanishes as $\varepsilon \searrow 0$.

\end{enumerate}
Then   the random variable $\sup_{0\leq t\leq T}\big| L_t^{\phi_\varepsilon} - \mathbf{L}_t \big|  $ vanishes in $L^1\big(\mathbf{P}_{\mu}^{T,\lambda}\big)$-norm as $\varepsilon\searrow 0$.

\end{theorem}

Our proof of Theorem~\ref{TheoremLocalTimeAdditive} will rely on the next lemma, which we prove in Section~\ref{SubsectionLemmaDowncrossingIntegral}.
\begin{lemma}\label{LemmaDowncrossingIntegral} Given  $\varepsilon>0$  define the $\mathscr{F}^X$-stopping time
$\varrho^{\uparrow, \varepsilon}=\inf\{ t\in [0,\infty)\,:\, |X_t| \geq \varepsilon   \}$.   Let the family of functions $\{\phi_{\varepsilon}\}_{\varepsilon\in (0,1)}$ be as in Theorem~\ref{TheoremLocalTimeAdditive}.
\begin{enumerate}[(i)]
\item For any $L >0$ and $p\in \mathbb{N}$, there exists a $C_{L,p}>0 $  such that for all $\varepsilon \in (0,\frac{1}{2})$, $T,\lambda \in (0,  L]$, and  $x\in \R^2 $
$$  \mathbf{E}^{T ,\lambda}_{x}\Bigg[ \,\bigg(\int_{ 0}^{\varrho^{\uparrow,2\varepsilon} }\,      \phi_{\varepsilon}\big(|X_r|\big)\,dr\bigg)^p \, \Bigg] \,\leq \, \frac{ C_{L,p}}{\log^{2p} \frac{1}{\varepsilon}  } \,.  $$

\item For any $L >1$ there exists a $C_{L}>0$ such that for all $\varepsilon \in (0,\frac{1}{2})$, $T,\lambda\in \big[\frac{1}{L}, L]$ and  $x\in \R^2 $ with $|x|=\varepsilon$
$$  
\Bigg|  \mathbf{E}^{T ,\lambda}_{x}\bigg[\, \int_{ 0}^{\varrho^{\uparrow,2\varepsilon} } \,     \phi_{\varepsilon}\big(|X_r|\big)\,dr \, \bigg] \,-\, \frac{ \log 2 }{2\log^2 \frac{1}{\varepsilon  }   } \Bigg| \,\leq \, C_L\,\bigg(\,\frac{ 1}{ \log^{\frac{5}{2}} \frac{1}{\varepsilon  }  }\,+\,\frac{ |\mathcal{E}(\varepsilon)|}{  \log^{2} \frac{1}{\varepsilon  } }\,\bigg)  
 \,.  $$

\end{enumerate}

\end{lemma}

The proof format for Theorem~\ref{TheoremLocalTimeAdditive} is similar to the proofs of  Theorems~\ref{ThmLocalTime1} \& \ref{ThmLocalTime3}, and we will use the special filtration $\{\widetilde{\mathscr{F}}^{T,\mu,\varepsilon}\}_{t\in [0,T]} $ defined in the latter. 
\begin{proof}[Proof of Theorem~\ref{TheoremLocalTimeAdditive}] For convenience, we will assume that $\int_{0}^{\infty} x\phi_{\varepsilon}(x)dx= ( 2\log\frac{1}{\varepsilon})^{-2} $ for all $\varepsilon\in (0,1)$, that is $\mathcal{E}=0$. 
Recall that the sequences of stopping times $ \{\widetilde{\varrho}_{n}^{\downarrow,\varepsilon}\}_{n\in \mathbb{N}}$ and $\{\widetilde{\varrho}_n^{\uparrow, \varepsilon}\}_{n\in \mathbb{N}_0}$ are defined as in~(\ref{VARRHOS2}), $\widetilde{N}_{t}^{\varepsilon}$ denotes the number of times $\widetilde{\varrho}_{n}^{\downarrow,\varepsilon}$ in the interval $[0,t]$, and  $\widetilde{\ell}_t^{\varepsilon}:=\frac{\log 2}{2\log^2\frac{1}{\varepsilon}  }\widetilde{N}_{t}^{\varepsilon}$.   By Theorem~\ref{ThmLocalTime3}, it suffices for us to show that  $ \sup_{0\leq t\leq T}\big|L^{\phi_\varepsilon}_t -\widetilde{\ell}^{\varepsilon}_t\big| $ vanishes in $L^1\big(\mathbf{P}_{\mu}^{T,\lambda}\big)$-norm  as $\varepsilon \searrow 0$. We can write
\begin{align*}
L_t^{\phi_\varepsilon}\,=\, \sum_{n=1  }^{\infty}1_{ \widetilde{\varrho}_{n}^{\downarrow,\varepsilon}\leq t} \, \int_{ \widetilde{\varrho}_{n}^{\downarrow,\varepsilon} }^{\widetilde{\varrho}_{n}^{\uparrow,\varepsilon}\wedge t } \,     \phi_{\varepsilon}\big(|X_r|\big)\,dr \,=\,   \sum_{n=1  }^{\widetilde{N}_{t}^{\varepsilon}}\, \int_{ \widetilde{\varrho}_{n}^{\downarrow,\varepsilon} }^{\widetilde{\varrho}_{n}^{\uparrow,\varepsilon}\wedge t } \,     \phi_{\varepsilon}\big(|X_r|\big)\,dr 
 \,.   
 \end{align*}
Define the processes $\{ \mathbf{L}_t^{\phi_\varepsilon} \}_{t\in [0,T]}$ and  $\{ \mathbf{\bar{L}}_t^{\phi_\varepsilon}\}_{t\in [0,T]}$ respectively by $\mathbf{L}_t^{\phi_\varepsilon} :=  \sum_{n=1  }^{\widetilde{N}_{t}^{\varepsilon}} \int_{ \widetilde{\varrho}_{n}^{\downarrow,\varepsilon} }^{\widetilde{\varrho}_{n}^{\uparrow,\varepsilon}}   \,   \phi_{\varepsilon}\big(|X_r|\big) dr  $ and
$$\mathbf{\bar{L}}_t^{\phi_\varepsilon} \,:=\,   \sum_{n=1  }^{\widetilde{N}_{t}^{\varepsilon}}\, \mathbf{E}^{T,\lambda}_{\mu}\Bigg[\, \int_{ \widetilde{\varrho}_{n}^{\downarrow,\varepsilon} }^{\widetilde{\varrho}_{n}^{\uparrow,\varepsilon}} \,     \phi_{\varepsilon}\big(|X_r|\big)\,dr  \,\Bigg|\,  \mathscr{F}_{\widetilde{\varrho}_{n}^{\downarrow,\varepsilon}-  }^{T,\mu,\varepsilon}\,\Bigg]\,=\,\sum_{n=1  }^{\widetilde{N}_{t}^{\varepsilon}}\, \mathbf{E}^{T-\widetilde{\varrho}_{n}^{\downarrow,\varepsilon}  ,\lambda}_{X_{ \widetilde{\varrho}_{n}^{\downarrow,\varepsilon} }}\bigg[\, \int_{ 0}^{\widetilde{\varrho}^{\uparrow,\varepsilon} }  \,    \phi_{\varepsilon}\big(|X_r|\big)\,dr \, \bigg]\,,
    $$
in which the second equality holds by the strong Markov property since $ \mathscr{F}_{\widetilde{\varrho}_{n}^{\downarrow,\varepsilon}-  }^{T,\mu,\varepsilon}=\mathscr{F}_{\widetilde{\varrho}_{n}^{\downarrow,\varepsilon} }^{T,\mu}$.  It suffices to show that the process differences $ L^{\phi_\varepsilon} - \mathbf{L}^{\phi_\varepsilon} $,  $  \mathbf{L}^{\phi_\varepsilon}-\mathbf{\bar{L}}^{\phi_\varepsilon}   $, and $ \mathbf{\bar{L}}^{\phi_\varepsilon} - \widetilde{\ell}^{\varepsilon}  $ vanish in the supremum-$L^1\big(\mathbf{P}^{T,\lambda}_{\mu}\big)$ sense.

We will bound the difference between $\widetilde{\ell}^{\varepsilon}$ and $\mathbf{\bar{L}}^{\phi_\varepsilon}$ using steps similar to those in~(\ref{Toof}).
We observe that for any $\delta\in (0,T)$
\begin{align*}
\mathbf{E}^{T ,\lambda}_{\mu }\bigg[ \,\sup_{0\leq t\leq T}\,\big| \mathbf{\bar{L}}_t^{\phi_\varepsilon} -\widetilde{\ell}_t^{\varepsilon}   \big|\,\bigg]\,\leq \,&\,\mathbf{E}^{T ,\lambda}_{\mu }\Bigg[\,\sum_{n=1}^{ \widetilde{N}_{T}^{\varepsilon} }\,\Bigg| \mathbf{E}^{T-\widetilde{\varrho}_{n}^{\downarrow,\varepsilon}  ,\lambda}_{X_{\widetilde{\varrho}_{n}^{\downarrow,\varepsilon}} }\bigg[\, \int_{ 0}^{\widetilde{\varrho}^{\uparrow,\varepsilon} } \,     \phi_{\varepsilon}\big(|X_r|\big)\,dr \, \bigg]     \,-\,\frac{ \log 2 }{2\log^2 \frac{1}{\varepsilon  }   } \Bigg|\,\Bigg]  \nonumber  \\  \,\leq \,&\,\frac{  C_{T,\lambda}}{ \log^{\frac{5}{2}} \frac{1}{\varepsilon  }  }\,\mathbf{E}^{T ,\lambda}_{\mu }\big[\,\widetilde{N}_{T-\delta}^{\varepsilon}\, \big] \,+\, \frac{  C_{T,\lambda}}{ \log^{2} \frac{1}{\varepsilon  }  }\,\mathbf{E}^{T ,\lambda}_{\mu }\big[\,\widetilde{N}_{T}^{\varepsilon} -\widetilde{N}_{T-\delta}^{\varepsilon}\, \big]\nonumber 
 \\ \,\preceq \,&\,\frac{ 1 }{ \log^{\frac{5}{2}} \frac{1}{\varepsilon  }  }\frac{ \big(1+\log^+\frac{T}{\varepsilon^2}\big)^2 }{1+\log^+\frac{1}{T\lambda} }\,+\,\frac{ 1 }{ \log^{2} \frac{1}{\varepsilon  }  }\frac{ \big(1+\log^+\frac{\delta}{\varepsilon^2}\big)^2 }{1+\log^+\frac{1}{\delta\lambda} } \hspace{.3cm}\stackrel{\varepsilon\rightarrow 0}{\longrightarrow}\hspace{.3cm}\frac{ 4 }{1+\log^+\frac{1}{\delta\lambda} }  \,,
\end{align*}
where  the  second inequality holds for some $C_{T,\lambda}>0$ by Lemma~\ref{LemmaDowncrossingIntegral}, using part (ii) for the terms  in the sum with $n\leq \widetilde{N}_{T-\delta}^{\varepsilon} $ and part (i) for the remaining terms.  The third inequality applies Lemma~\ref{UpcroossingInequalityPreII} and the Markov property.  The right side above vanishes with small $\varepsilon,\delta>0$, and so  $ \sup_{0\leq t\leq T}\,\big| \mathbf{\bar{L}}_t^{\phi,\varepsilon} -\widetilde{\ell}_t^{\varepsilon}   \big|$ vanishes in $L^1\big(\mathbf{P}^{T,\lambda}_{\mu}\big)$-norm.

    Next we will show that $\sup_{0\leq t\leq T}\big|\mathbf{L}_t^{\phi_\varepsilon} - \mathbf{\bar{L}}_t^{\phi_\varepsilon}\big| $ vanishes in $L^2\big(\mathbf{P}^{T,\lambda}_{\mu}\big)$-norm as $\varepsilon \searrow 0$  through a similar argument as  used for $\sup_{0\leq t\leq T}\big| \boldsymbol{\ell}_t^{\varepsilon}- \boldsymbol{\bar{\ell}}_t^{\varepsilon}\big| $ in the proof of Theorem~\ref{ThmLocalTime3}.  The process  $\mathbf{m}^{\phi_\varepsilon} :=\mathbf{L}^{\phi_\varepsilon} - \mathbf{\bar{L}}^{\phi_\varepsilon}$ is a mean zero  $\mathbf{P}^{T,\lambda}_{\mu}$-martingale with respect to $\widetilde{\mathscr{F}}^{T,\mu,\varepsilon}  $
     having predictable quadratic variation process given by 
\begin{align} \label{Farhba}
\big\langle \mathbf{m}^{\phi_\varepsilon} \big\rangle_t \,=\,&\, \sum_{n=1  }^{\widetilde{N}_{t}^{\varepsilon}} \,\mathbf{Var}^{T,\lambda}_{\mu}\Bigg[\,\int_{ \widetilde{\varrho}_{n}^{\downarrow,\varepsilon} }^{\widetilde{\varrho}_{n}^{\uparrow,\varepsilon}}    \,  \phi_{\varepsilon}\big(|X_r|\big)\,dr  \,\Bigg|\, \widetilde{\mathscr{F}}_{\widetilde{\varrho}_{n}^{\downarrow,\varepsilon}-}^{T,\mu,\varepsilon} \,\Bigg] \nonumber  \\
\,=\,&\, \sum_{n=1  }^{\widetilde{N}_{t}^{\varepsilon}} \,\mathbf{Var}^{T-\widetilde{\varrho}_{n}^{\uparrow,\varepsilon},\lambda}_{X_{\widetilde{\varrho}_{n}^{\uparrow,\varepsilon}} }\Bigg[\,\int_{0}^{\widetilde{\varrho}^{\uparrow,2\varepsilon}}      \phi_{\varepsilon}\big(|X_r|\big)\,dr \, \Bigg] \nonumber  \\
\,\leq\,&\, \sum_{n=1  }^{\widetilde{N}_{t}^{\varepsilon}}\, \|\phi_{\varepsilon}\|_{\infty}^2 \,\mathbf{E}^{T-\widetilde{\varrho}_{n}^{\downarrow,\varepsilon},\lambda}_{X_{\widetilde{\varrho}_{n}^{\downarrow,\varepsilon} }}\Big[\, \big(\varrho^{\uparrow,2\varepsilon}\big)^2\,\Big] \,\preceq\,  \frac{\widetilde{N}_{t}^{\varepsilon}}{\log^4 \frac{1}{\varepsilon}  }\,.
\end{align}
The last inequality holds  by (i) of Lemma~\ref{LemmaLeave} and assumption (II) on the family $\{\phi_{\varepsilon}\}_{\varepsilon\in (0,1)}$. Starting with  Doob's inequality, we  have the relations
\begin{align*}
\mathbf{E}^{T,\lambda}_{\mu}\bigg[\,\sup_{0\leq t\leq T}\, \big(\mathbf{m}^{\phi_\varepsilon}_t\big)^2  \, \bigg] \,\leq \,4 \,\mathbf{E}^{T,\lambda}_{\mu}\Big[ \,\big(\mathbf{m}^{\phi_\varepsilon}_T\big)^2   \,\Big]\, =\,4\, \mathbf{E}^{T,\lambda}_{\mu}\Big[ \big\langle \mathbf{m}^{\phi_\varepsilon}\big\rangle_T  \Big] \,\stackrel{(\ref{Farhba})}{\preceq}\, \frac{\mathbf{E}^{T,\lambda}_{\mu}[ \widetilde{N}_T^{\varepsilon}] }{\log^4 \frac{1}{\varepsilon}  } \,\preceq  \, \frac{1}{\log^2 \frac{1}{\varepsilon}  }\,,
\end{align*}
in which   the last inequality uses Lemma~\ref{UpcroossingInequalityPreII}. Consequently, the difference between the processes $\mathbf{L}^{\phi_\varepsilon} $ and $\mathbf{\bar{L}}^{\phi_\varepsilon}$ vanishes in the  supremum-$L^2\big(\mathbf{P}^{T,\lambda}_{\mu}\big)$ sense with small $\varepsilon$.

Lastly, for the difference between $L^{\phi_\varepsilon}$ and $\mathbf{L}^{\phi_\varepsilon}  $, we observe that
\begin{align*}
    \sup_{0\leq t\leq T}\,\big| L_t^{\phi_\varepsilon}  - \mathbf{L}_t^{\phi_\varepsilon}  \big|\,=\,\max_{1\leq n \leq\widetilde{N}_T^{\varepsilon} } \, \int_{ \widetilde{\varrho}_{n}^{\downarrow,\varepsilon} }^{\widetilde{\varrho}_{n}^{\uparrow,\varepsilon} } \,     \phi_{\varepsilon}\big(|X_r|\big)\,dr \,\leq \, \|\phi_{\varepsilon}\|_{\infty}\,\max_{1\leq n \leq\widetilde{N}_T^{\varepsilon} }\,  \widetilde{\varrho}_{n}^{\uparrow,\varepsilon} -\widetilde{\varrho}_{n}^{\downarrow,\varepsilon} \,.
\end{align*}
The right side above vanishes in $L^2\big(\mathbf{P}^{T,\lambda}_{\mu}\big)$ as $\varepsilon\searrow 0$  by  (ii) of Lemma~\ref{LemmaSupTime} in the $p=2$ case and assumption (II) on $\|\phi_{\varepsilon}\|_{\infty}$. Thus $\sup_{0\leq t\leq T}\big|  L_t^{\phi_\varepsilon} - \mathbf{L}_t^{\phi_\varepsilon}  \big|$ vanishes in the $L^1\big(\mathbf{P}^{T,\lambda}_{\mu}\big)$-norm.
\end{proof}

\subsection{Proofs of Theorem~\ref{ThmEquivalent} and Proposition~\ref{PropLocalTimeProp}}\label{SubsectionThmEquivalent}

\begin{proof}[Proof of Theorem~\ref{ThmEquivalent}] We will show that the  equality of processes $\mathbf{L}^{T,\lambda,\lambda'}=\log\big(\frac{\lambda'}{\lambda} \big) \mathbf{L}$ holds almost surely under $ \mathbf{P}_{\mu}^{T,\lambda}$ for any $T,\lambda,\lambda'>0$, which is adequate  in light of Theorem~\ref{ThmGirsanov}. By Lemma~\ref{LemmaRelate}, we have that
$\mathbf{L}^{T,\lambda,\lambda''}=\mathbf{L}^{T,\lambda,\lambda'}+\mathbf{L}^{T,\lambda',\lambda''}$  almost surely $\mathbf{P}^{T,\lambda}_{\mu} $  for any $ \lambda,\lambda',\lambda''>0$.  It suffices now to verify that $ \frac{\partial}{\partial \lambda'}\mathbf{L}^{T,\lambda,\lambda'}|_{\lambda'=\lambda}=\frac{1}{\lambda}\mathbf{L}$.  The processes   $  \mathcal{S}^{T,\lambda,\lambda'}$ and $\mathcal{\mathring{S}}^{T,\lambda}$, which are respectively defined  in Propositions~\ref{ThmPreGirsanov} and~\ref{PropSubMartIII},  are related through  $ \frac{1}{\lambda}\mathcal{\mathring{S}}^{T,\lambda}_t=   \frac{\partial}{\partial \lambda'}\mathcal{S}^{T,\lambda,\lambda'}_t|_{\lambda'=\lambda} $, and  the predictable components  $  \mathcal{A}^{T,\lambda,\lambda'}$ and $\mathcal{\mathring{A}}^{T,\lambda}$ in their Doob-Meyer decomposition also satisfy  $ \frac{1}{\lambda}\mathcal{\mathring{A}}^{T,\lambda}_t=  \frac{\partial}{\partial \lambda'}\mathcal{A}^{T,\lambda,\lambda'}_t\big|_{\lambda'=\lambda} $.   Since $\mathbf{L}_{t}^{T,\lambda,\lambda'}:=-\int_0^{t} \,\frac{ \nu((T-s)\lambda)  }{ \nu((T-s)\lambda')  }  d\mathcal{A}_s^{T,\lambda,\lambda'}$, we can compute   using the chain rule as follows:
\begin{align*}
    \frac{\partial}{\partial \lambda'} \mathbf{L}_{t}^{T,\lambda,\lambda'}\,\bigg|_{\lambda'=\lambda}\,=\,&\,-  \frac{\partial}{\partial \lambda'}\int_0^{t} \,\frac{ \nu\big((T-s)\lambda\big)  }{ \nu\big((T-s)\lambda'\big)  } \, d\mathcal{A}_s^{T,\lambda,\lambda'} \, \bigg|_{\lambda'=\lambda} \nonumber    \\ \,=\,&\,\int_0^{t} \,(T-s)\,\frac{\nu'\big((T-s)\lambda\big)  }{ \nu\big((T-s)\lambda\big)  } \, \underbrace{d\mathcal{A}_s^{T,\lambda,\lambda}}_{=\,0} \,-\,\int_0^{t} \, \frac{1}{\lambda}\,d\mathcal{\mathring{A}}_s^{T,\lambda}    \,=\,-\frac{1}{\lambda}\,\mathcal{\mathring{A}}_t^{T,\lambda}\,=:\,\frac{1}{\lambda}\,\mathbf{L}_t\,, \nonumber
\end{align*}
where we have used that $  \mathcal{A}^{T,\lambda,\lambda}=0$ since    $  \mathcal{S}^{T,\lambda,\lambda}=1$ is constant.
\end{proof}

The proof of the technical lemma below can be found in Section~\ref{SubsectionLemmaDetClass}. Given $T,\lambda,\lambda'>0$ we define the function $\varphi_T^{\lambda,\lambda'}:[0,T]\rightarrow [0,\infty) $ by  
\begin{align}\label{PHI}
\varphi_T^{\lambda,\lambda'}(a)\, :=\,\frac{ \nu\big((T-a)\lambda'\big)   }{ \nu\big((T-a)\lambda\big)  } \hspace{.4cm} \text{ for $a\in [0,T)$ and}   \hspace{.4cm} \varphi_T^{\lambda,\lambda'}(T):=\lim_{a\nearrow T}\,\varphi_T^{\lambda,\lambda'}(a)=1\,.
\end{align}
We refer to a family $\{f_{\alpha}\}_{\alpha\in I}$ of complex-valued  bounded measurable functions  on a measurable  space $(S,\mathcal{S})$ as a \textit{determining class} if two finite  measures $\vartheta_1$ and $\vartheta_2$ on $S$ are necessarily equal provided that  $\vartheta_1(f_{\alpha})=\vartheta_2(f_{\alpha})$ for all $\alpha\in I$. 
\begin{lemma} \label{LemmaDetClass}
Fix some  $T,\lambda>0$. The family of functions $\big\{ \varphi_T^{\lambda,\lambda'} \big\}_{\lambda'\in (0,\infty)}$ is a determining class  on the measurable space $\big([0,T],\mathscr{B}([0,T])\big)$.
\end{lemma}

\begin{proof}[Proof of Proposition~\ref{PropLocalTimeProp}] We will first derive the joint density for $(\tau,\mathbf{L}_T)$.  By Theorem~\ref{ThmEquivalent}, we have $ \frac{d\mathbf{P}_{x}^{T,\lambda'}}{d\mathbf{P}_{x}^{T,\lambda}}=R^{\lambda',\lambda}_{T}(x)  \big(\frac{ \lambda'}{\lambda}  \big)^{\mathbf{L}_{T}} $ almost surely under $\mathbf{P}_{\mu}^{T,\lambda}$ for any $\lambda'>0$, and consequently 
\begin{align} \label{FormWithR}
  R^{\lambda,\lambda'}_{T}(x)\,=\,\mathbf{E}_{x}^{T,\lambda}\bigg[ \,\Big(\frac{ \lambda'}{\lambda}  \Big)^{\mathbf{L}_{T}} \,\bigg]\,,
\end{align}
  since $ \big(R^{\lambda,\lambda'}_{T}(x)\big)^{-1}=R^{\lambda',\lambda}_{T}(x)$.  In particular, for $\beta\in \R$  and  $\lambda'=e^{\beta}\lambda$ we can use that $ R^{\lambda,\lambda'}_{T}(0):=\frac{\nu(T\lambda')  }{ \nu(T\lambda)  }$ to get the following formula for the moment generating function of $\mathbf{L}_{T}$ under $\mathbf{P}_{0}^{T,\lambda}$:
\begin{align}  \label{RatioE}
\mathbf{E}_{0}^{T,\lambda}\big[\, e^{\beta \mathbf{L}_{T}} \,\big]\,=\, \frac{\nu\big(e^{\beta}T\lambda \big)  }{ \nu(T\lambda)  }\,.
\end{align}
This uniquely determines the probability density of  $ \mathbf{L}_{T}$ under $\mathbf{P}_{0}^{T,\lambda}$ to be  $\mathcal{D}(v)=\frac{1}{\nu(T\lambda)}   \frac{ (T\lambda )^v  }{\Gamma(v+1)  }1_{v\geq 0}$; see Appendix~\ref{AppendixSectPoisson} for a brief discussion of distributions having  densities of this form.  Going back to~(\ref{FormWithR}), since $\mathbf{L}_{T}=0$ in the event $\mathcal{O}^c$, we can write
\begin{align*}
  R^{\lambda,\lambda'}_{T}(x)\,=\,&\,\mathbf{P}_{x }^{T,\lambda}\big[\,\mathcal{O}^c\,\big]\,+\,\mathbf{P}_{x }^{T,\lambda}\big[\,\mathcal{O}\,\big]\,\mathbf{E}_{x}^{T,\lambda}\bigg[\, \Big(\frac{ \lambda'}{\lambda}  \Big)^{\mathbf{L}_{T}}\, \bigg|\, \mathcal{O}\,\bigg]  \,.
 \end{align*}
Using that $ R^{\lambda,\lambda'}_{T}(x):= \frac{1+ H^{\lambda'}_T(x)  }{1+ H^{\lambda}_T(x)   }$ and that $\mathbf{P}_{x }^{T,\lambda}\big[\mathcal{O}^c\big]=\frac{ 1 }{ 1+H^{\lambda}_T(x)   }$ by (i) of Theorem~\ref{CorSubMART}, we have the first equality below:
 \begin{align}\label{FormWithR2}
  \frac{ H^{\lambda'}_T(x)  }{ H^{\lambda}_T(x)   }\,=\,\mathbf{E}_{x}^{T,\lambda}\bigg[\, \Big(\frac{ \lambda'}{\lambda}  \Big)^{\mathbf{L}_{T}}\, \bigg|\, \mathcal{O}\,\bigg]\nonumber \,=\,&\,\int_{[0,T]}\, \mathbf{E}_{x}^{T,\lambda}\bigg[\, \Big(\frac{ \lambda'}{\lambda}  \Big)^{\mathbf{L}_{T}}\, \bigg|\, \tau=t\, \bigg]\,\mathbf{P}_{x }^{T,\lambda}\big[\,\tau \in dt\,\big|\,\mathcal{O}\,\big] \nonumber
  \\ \,=\,&\,\int_{[0,T] }\, \underbrace{\frac{\nu\big((T-t)\lambda' \big)  }{ \nu\big((T-t)\lambda\big)  }}_{ =:\,\varphi_{T}^{\lambda,\lambda'}(t)   }\,\mathbf{P}_{x }^{T,\lambda}\big[\,\tau \in dt\,\big|\,\mathcal{O}\,\big] \,.
\end{align}
The third equality uses that $X_{\tau}=0$ when $\tau<\infty$, invokes the strong Markov property, and applies~(\ref{RatioE}) with $\beta=\log\frac{\lambda'}{\lambda} $.  Since $H^{\lambda'}_T(x)=\int_0^T \frac{ e^{-\frac{|x|^2  }{2t} } }{ t } \nu\big((T-t)\lambda'\big) dt $, equation~(\ref{FormWithR2}) holds when 
\begin{align}\label{TauCond}
\text{} \hspace{1.5cm}\mathbf{P}_{x }^{T,\lambda}\big[\tau \in dt\,\big|\,\mathcal{O}\big]\,=\,\frac{ 1 }{ H^{\lambda}_T(x)   } \,\frac{ e^{-\frac{|x|^2  }{2t} } }{ t } \,\nu\big((T-t)\lambda\big)\,dt\,,   \hspace{1cm}t\in [0,T]\,. 
\end{align}
Moreover,  the family of  functions $\{ \varphi_{T}^{\lambda,\lambda'} \}_{\lambda'>0}$ on $[0,T]$  defined in~(\ref{PHI})  is a determining class for Borel measures on $[0,T]$ by Lemma~\ref{LemmaDetClass}, so the above uniquely determines  the distribution of $\tau$ under $\mathbf{P}_{x }^{T,\lambda}$ conditioned on $\mathcal{O}$.  

The joint density for $\tau$ and  $\mathbf{L}_T$ can  be deduced from our observations above because 
\begin{align*} \mathbf{P}_{x }^{T,\lambda}\big[\,\tau \in dt, \,\mathbf{L}_T\in dv\,\big|\,\mathcal{O}\,\big] \,=\,&\,\mathbf{P}_{x }^{T,\lambda}\big[\,\tau \in dt\,\big|\,\mathcal{O}\,\big] \, \frac{\mathbf{P}_{x }^{T,\lambda}\big[\,\mathbf{L}_T\in dv\,\big|\,\tau=t\,\big] }{dv}\,dv
\\
\,=\,&\,\frac{ 1 }{ H^{\lambda}_T(x)   }\, \frac{ e^{-\frac{|x|^2  }{2t} } }{ t } \,\frac{ \big((T-t)\lambda \big)^v  }{\Gamma(v+1)  }\,dt\,dv\,,
\end{align*}
which uses~(\ref{TauCond}), and then the strong Markov property to write
$$ \frac{\mathbf{P}_{x }^{T,\lambda}\big[\,\mathbf{L}_T\in dv\,\big|\,\tau=t\,\big] }{dv}\,=\, \frac{\mathbf{P}_{0 }^{T-t,\lambda}\big[\,\mathbf{L}_{T-t}\in dv\,\big] }{dv}\,=\, \frac{1}{\nu\big((T-t)\lambda\big)}  \, \frac{ \big((T-t)\lambda \big)^v  }{\Gamma(v+1)  }\,. $$

Now we will address the equality between $\tau$ and $\inf\{t>0\,:\,\mathbf{L}_t>0  \}$.   It is clear  from our  constructions of the  local time process that
$\tau\leq \inf\{t>0\,:\,\mathbf{L}_t>0  \}$ holds $\mathbf{P}_{x }^{T,\lambda}$ almost surely, so we can focus on showing that this inequality also holds in the other direction.  In the event  $\mathcal{O}^c$, we have $\tau=\infty$ by definition, so $\tau\geq \inf\{t>0\,:\,\mathbf{L}_t>0  \}$ holds vacuously. Thus, it remains to consider the event $\mathcal{O}$.
The joint density for $(\tau,\mathbf{L}_{\tau})$ conditional on $\mathcal{O}$ derived above implies that $\mathbf{L}_T>0$ occurs $\mathbf{P}_{x }^{T,\lambda}$ almost surely in the event $\mathcal{O}$. 
Thus, for any $\varepsilon>0$ we can write
\begin{align*}
 \mathbf{P}_{x }^{T,\lambda}[ \mathcal{O}]\,=\,    \mathbf{P}_{x }^{T,\lambda}\big[\,\mathbf{L}_T >0\,, \,\mathcal{O}\,\big]\,=\,&\,\mathbf{P}_{x }^{T,\lambda}\big[\, \mathbf{L}_{\tau+\varepsilon} >0\, , \,\mathcal{O}\,\big]\\ &\,+\,\mathbf{P}_{x }^{T,\lambda}\big[ \,\mathbf{L}_{\tau+\varepsilon} =0\, , \,\mathcal{O}\,\big]\, \underbrace{\mathbf{P}_{x }^{T,\lambda}\big[\,\mathbf{L}_T -\mathbf{L}_{\tau+\varepsilon}>0 \,\big| \,  \mathbf{L}_{\tau+\varepsilon} =0\,\big]}_{ <1}
\end{align*}
An  argument using the strong Markov property   shows that the underbraced conditional probability is strictly less than one. The above leads to a contradiction unless $\mathbf{P}_{x }^{T,\lambda}\big[ \mathbf{L}_{\tau+\varepsilon} =0 , \,\mathcal{O}\big]=0$, and thus $ \mathbf{L}_{\tau+\varepsilon}>0$ holds almost surely $\mathbf{P}_{x }^{T,\lambda}$  in the event $\mathcal{O}$.  By taking a vanishing sequence of $\varepsilon_n>0$, we can conclude that $\tau\geq \inf\{t>0\,:\,\mathbf{L}_t>0  \}$ holds  almost surely $\mathbf{P}_{x }^{T,\lambda}$ in the event $\mathcal{O}$, giving us our desired equality. 
\end{proof}

\section{The Volterra jump process}\label{SectionRightCont}

Now we turn to the proofs of the theorems stated in Section~\ref{SubsectionProcInv}.  Recall from~(\ref{ETA}) that $\{\boldsymbol{\eta}_s\}_{s\in [0,\infty)}$ is defined as the right-continuous process inverse of the local time  $\{\mathbf{L}_t\}_{t\in [0,T]}$ and that the family of transition kernels $\{ \mathscr{T}^{T,\lambda}_s\}_{s\in [0,\infty)} $  on $[0,T]$ are defined as in~(\ref{TFORM}).
For an increasing, right-continuous,  $[0,T]$-valued process $\{ \eta_s \}_{s\in [0,\infty)}$ on some probability space,  define the stopping time $\mathbf{S}=\min\{s\in [0,\infty)\,: \,  \eta_s=T \}  $.  In Section~\ref{SubsectionDetMart} we show that 
  $\{ \eta_s \}_{s\in [0,\infty)}$  is Markovian with transition kernels $\{ \mathscr{T}^{T,\lambda}_s \}_{s\in [0,\infty)}   $  with respect to a filtration $\{ F_s\}_{s\in [0,\infty)}$ if and only if the  process $\{ \mathbf{m}_s \}_{s\in [0,\infty)  }$ given by $\mathbf{m}_s:=(\frac{\lambda'}{\lambda})^{s\wedge \mathbf{S}} \varphi_T^{\lambda,\lambda'}( \eta_s)$ is a martingale with respect to $\{ F_s\}_{s\in [0,\infty)}$ for every $\lambda'>0$, where the function $\varphi_T^{\lambda,\lambda'}:[0,T]\rightarrow [0,\infty) $ is defined as in~(\ref{PHI}).
  In Section~\ref{SubsecInvProc}, we apply this idea in combination with Theorem~\ref{ThmPreGirsanov}   to provide a short proof of   Theorem~\ref{ThmLocalTimeToLEVY}.  Moreover, these martingales  have a role in the proof of Theorem~\ref{ThmProcTRAN}   in
 Section~\ref{SubsectionThmProcTRAN}.

\subsection{A one-parameter family of martingales}\label{SubsectionDetMart}

For  $s,t,\lambda>0$  we will use the fractional integral identity
\begin{align} \label{RemarkConv}
\frac{ \lambda^s }{  \Gamma(s)} \,\int_0^t\, b^{s-1}\,\nu\big( (t-b)\lambda \big)\,  db\,=\,\nu(t\lambda  ) \,-\,\int_0^{s}\,\frac{ t^r\lambda^r  }{ \Gamma(r+1)  } \, dr\,=\, \nu(t\lambda  )\,\big(1-\mathbf{p}_s^{t,\lambda}\big) \,, 
\end{align}
which is equivalent to the basic identity for the Volterra function $\nu$ in (i) of  Proposition~\ref{PropEFunctForm} after a rescaling of the integration variables above.  The second equality holds for the probability $\mathbf{p}_s^{t,\lambda}$ defined in~(\ref{TFORM}).

\begin{proposition} \label{PropMartEta} Fix some $T,\lambda>0$. Let  $\{ \eta_s \}_{s\in [0,\infty)}$ be an increasing, right-continuous,  $[0,T]$-valued Markov process on a probability space  $(\Sigma, \mathcal{F},\mathcal{P})$ and having transition kernels $\{ \mathscr{T}^{T,\lambda}_s \}_{s\in [0,\infty)}  $ with respect to a filtration $\{ F_s\}_{s\in [0,\infty)}$.  For any $\lambda' >0$ the stopped process $\{ \mathbf{m}_s^{T,\lambda,\lambda'}\}_{s\in [0,\infty)  }$ defined by $ \mathbf{m}_s^{T,\lambda,\lambda'}:= \big(\frac{\lambda'}{\lambda}\big)^{s\wedge \mathbf{S}} \varphi_T^{\lambda,\lambda'}( \eta_s)$  is a uniformly integrable martingale with respect to $\{ F_s\}_{s\in [0,\infty)}$.
\end{proposition}

\begin{proof} We can verify that the random variable  $ \mathbf{m}_s^{T,\lambda,\lambda'}$ is integrable by observing that it is bounded by $1$ in the case $\lambda'\leq \lambda$, and  when $\lambda'>\lambda$  it again has finite uniform norm:
$$  \,\big\|\, \mathbf{m}_s^{T,\lambda,\lambda'}\,\big\|_{\infty}\,\leq\, \bigg( \frac{\lambda'}{\lambda}\bigg)^{s}\, \frac{\nu(T\lambda')  }{ \nu(T\lambda) }\,<\,\infty \,, $$
where we have  used that $\varphi_T^{\lambda,\lambda'}(t)$ is bounded by $\frac{\nu(T\lambda')  }{ \nu(T\lambda) }$. We will address uniform integrability at the end of the proof.

Now we check that $\{ \mathbf{m}_s^{T,\lambda,\lambda'}\}_{s\in [0,\infty)  }$ is a martingale. Since  $\eta$ is a Markov process governed by the  transition  kernels $\{ \mathscr{T}^{T,\lambda}_s\}_{s\in [0,\infty)} $ and $\mathbf{S}$ is the time at which the stationary Markov process $\eta$ first reaches the absorbing state $T$, it suffices to verify that the expression below is equal to $0$ for any $t\in [0,\infty)$ and $a\in [0,T)$.
\begin{align*}
\lim_{s\searrow 0}  \frac{\mathcal{E}\big[ \mathbf{m}_{t+s}^{T,\lambda,\lambda'}  \,\big|\,F_t,\,\eta_t=a \big]- (\frac{\lambda'}{\lambda})^{t} \varphi_T^{\lambda,\lambda'}(a) }{ s}   \,=\, \frac{d}{dt} \Big(\frac{\lambda'}{\lambda}\Big)^{t}  \, \varphi_T^{\lambda,\lambda'}( a)\,+\, \Big(\frac{\lambda'}{\lambda}\Big)^{t}\,\frac{d}{ds} \,\int_{[0,T]} \, \varphi_T^{\lambda,\lambda'}( b)\,\mathscr{T}^{T,\lambda}_{s}(a,db)\,\bigg|_{s=0}
\end{align*}
The right side above is null provided that the $s=0$ case of the identity below holds, which we will confirm presently.
\begin{align}\label{CheckThis}
 \frac{d}{ds} \,\int_{[0,T]} \, \varphi_T^{\lambda,\lambda'}(b)\,\mathscr{T}^{T,\lambda}_{s}(a,db) \,=\, \log\Big(\frac{\lambda}{\lambda'}\Big)\, \int_{[0,T)} \,\varphi_T^{\lambda,\lambda'}(b)\,\mathscr{T}^{T,\lambda}_{s}(a,db) 
\end{align}
 Since   $\varphi_T^{\lambda,\lambda'}(T)=1$ and $\mathscr{T}^{T,\lambda}_{s}\big(a, \{T\}\big)=\mathbf{p}^{T-a,\lambda}_{s} $, we have that
\begin{align}\label{ClosedThing}
    \int_{[0,T]}  \,\varphi_T^{\lambda,\lambda'}(b)\,\mathscr{T}^{T,\lambda}_{s}(a, db) \,=\, &\,\underbracket{\int_{[0,T)} \, \varphi_T^{\lambda,\lambda'}(b)\,\mathscr{T}^{T,\lambda}_{s}(a,db)}\,+\, \mathbf{p}^{T-a,\lambda}_{s} \nonumber 
    \\ \,=\,&\,\frac{ 1  }{  \nu\big((T-a)\lambda \big)    }  \int_a^T  \, \nu\big((T-b)\lambda' \big) \,\frac{(b-a)^{s-1}\lambda^s }{ \Gamma(s) } \,  db\,+\, \mathbf{p}^{T-a,\lambda}_{s} \nonumber 
    \\ \,=\,&\,\Big(\frac{\lambda}{\lambda'}\Big)^s \,\frac{ 1  }{  \nu\big((T-a) \lambda\big)    }  \,\int_0^{T-a}\,   \nu\big((T-a-b) \lambda'\big) \,\frac{b^{s-1} (\lambda')^s  }{ \Gamma(s) }\,   db\,+\, \mathbf{p}^{T-a,\lambda}_{s} \nonumber
    \\ \,& \hspace{-.55cm}\stackrel{(\ref{RemarkConv}) }{=}   \underbracket{\Big(\frac{\lambda}{\lambda'}\Big)^s \,\frac{ \nu\big((T-a) \lambda' \big)    }{  \nu\big((T-a) \lambda  \big)    } \,\Big(1- \mathbf{p}^{T-a,\lambda'}_{s}\Big)}\,+\, \mathbf{p}^{T-a,\lambda}_{s}\,,
\end{align}
where the second equality merely invokes the definitions for $\varphi_T^{\lambda,\lambda'}$ and  $\mathscr{T}^{T,\lambda}_{s}(a,\cdot)$. 
 Differentiating both sides of~(\ref{ClosedThing}) with respect to $s$, the product rule  yields the first equality below.
\begin{align*}
    \frac{\partial}{\partial s}\,\int_{[0,T]}  \,  \varphi_T^{\lambda,\lambda'}(b)\,\mathscr{T}^{T,\lambda}_{s}(a,db)  \,=\,&\,\log\Big(\frac{\lambda}{\lambda'}\Big)\, \Big(\frac{\lambda}{\lambda'}\Big)^s \,\frac{ \nu\big((T-a)\lambda' \big)    }{  \nu\big((T-a) \lambda\big)    } \,\Big(1- \mathbf{p}^{T-a,\lambda'}_{s}\Big)\\  &\,\,-\,\underbracket{\Big(\frac{\lambda}{\lambda'}\Big)^s\,\frac{ \nu\big((T-a) \lambda'\big)  }{  \nu\big((T-a) \lambda\big)    }\,  \frac{\partial}{\partial s} \mathbf{p}^{T-a,\lambda'}_{s}}\,+\, \underbracket{\frac{\partial}{\partial s}\mathbf{p}^{T-a,\lambda}_{s}}
     \\   \,=\,&\, \log\Big(\frac{\lambda}{\lambda'}\Big)\,\Big(\frac{\lambda}{\lambda'}\Big)^s \,\frac{ \nu\big((T-a)\lambda' \big)    }{  \nu\big((T-a)\lambda \big)    } \,\Big(1- \mathbf{p}^{T-a,\lambda'}_{s}\Big)\\ \,=\,&\,\log\Big(\frac{\lambda}{\lambda'} \Big)\,\int_{[0,T)}\,  \varphi_T^{\lambda,\lambda'}(b)\,\mathscr{T}^{T,\lambda}_{s}(a,db) 
\end{align*}
The pair of underbracketed terms are equal, and thus cancel, because  $\frac{\partial}{\partial s}\mathbf{p}^{t,\lambda}_{s}=\frac{1}{  \nu( t \lambda)  }\frac{ t^s \lambda^s }{\Gamma(s+1) }$.  For  the third equality, we  have used that the underbracketed terms in~(\ref{ClosedThing})  are equal.  Thus we have verified~(\ref{CheckThis}).

 Lastly, we will argue that the martingale $\{\mathbf{m}_s^{T,\lambda,\lambda'}\}_{s\in [0,\infty)} $ is uniformly integrable for any $\lambda'>0$. This is trivial in the case $\lambda'\leq \lambda$ since the random variable $\mathbf{m}_s^{T,\lambda,\lambda'}$ is bounded by $1$ for each $s\geq 0$. When $\lambda'> \lambda$ we have the uniform bound
 $  \mathbf{m}_s^{T,\lambda,\lambda'} \leq   (\frac{\lambda'}{\lambda})^{ \mathbf{S}} \frac{\nu(T\lambda')   }{\nu(T\lambda)}$.   Thus, it suffices to show that the random variable $(\frac{\lambda'}{\lambda})^{ \mathbf{S}} $ is integrable, which we  can  achieve as follows:
\begin{align*}
  \mathcal{E}\Bigg[ \, \bigg(\frac{\lambda'}{\lambda}\bigg)^{ \mathbf{S}} \,\Bigg]  \,=\,\lim_{n\rightarrow \infty} \mathcal{E}\Bigg[ \, \bigg(\frac{\lambda'}{\lambda}\bigg)^{n\wedge \mathbf{S}}\, \Bigg] \,\leq \, \lim_{n\rightarrow \infty} \mathcal{E}\Big[ \,\mathbf{m}_n^{T,\lambda,\lambda'} \, \Big]\,=\, \mathcal{E}\left[  \,\mathbf{m}_0^{T,\lambda,\lambda'} \,\right]\,\leq \, \frac{\nu(T\lambda')   }{\nu(T\lambda)} \,,
\end{align*}
where we have used the monotone convergence theorem,   that $\varphi_T^{\lambda,\lambda'}\geq 1$ when $\lambda'>\lambda$, and the martingale property.
\end{proof}

We can derive the following  identity for the Volterra function $\nu$ using the equation~(\ref{CheckThis}).
\begin{lemma}\label{LemmaC}   For all   $\beta,\beta' >  0$ we have
\begin{align*}
     \int_{0}^1  \,\Big(  \nu\big((1-\theta)\beta' \big)\,  \nu(\beta)   \,-\, \nu(\beta') \,\nu\big((1-\theta)\beta\big)      \Big)\, \frac{1}{\theta }\, d\theta  \,=\,  \log \Big(\frac{\beta}{\beta'}\Big)\,  \nu(\beta')\,    \nu(\beta)     \,+\, \nu(\beta')   \,-\, \nu(\beta) \,.
\end{align*}
\end{lemma}

\begin{proof}  Applying (ii) of Proposition~\ref{PropTransMeas} followed by~(\ref{CheckThis}) with $s=0$ yields that for $a\in [0,T)$
\begin{align*}
 \int_{[0,T]}\, \big( \varphi_T^{\lambda,\lambda'}(b)\,-\, \varphi_T^{\lambda,\lambda'}(a)  \big)\,  \mathscr{J}^{T,\lambda }(a,db)\,=\, &\,
\int_{[0,T]} \, \varphi_T^{\lambda,\lambda'}(b)\frac{d}{ds} \,\mathscr{T}^{T,\lambda}_{s}(a,db) \Big|_{s=0}\nonumber \\ \,=\, &\,\frac{d}{ds} \,\int_{[0,T]} \, \varphi_T^{\lambda,\lambda'}(b)\,\mathscr{T}^{T,\lambda}_{s}(a,db)  \bigg|_{s=0} \,\stackrel{(\ref{CheckThis})}{=}\,  \log\Big(\frac{\lambda}{\lambda'}\Big) \,\varphi_T^{\lambda,\lambda'}(a)\,. 
\end{align*}
Next, plugging in the definitions of  $\varphi_T^{\lambda,\lambda'}$ and $ \mathscr{J}^{T,\lambda }(a,\cdot)$ respectively from~(\ref{PHI}) and~(\ref{JUMPFORM}) yields that
\begin{align*}
 \int_{a}^T \,\Bigg( &\frac{ \nu\big((T-b) \lambda' \big)    }{  \nu\big((T-b)\lambda \big)    } \,-\,\frac{ \nu\big((T-a)\lambda'  \big)    }{  \nu\big( (T-a)  \lambda\big)    }   \Bigg) \, \frac{1}{b-a}\,\frac{  \nu\big((T-b )\lambda \big) }{  \nu\big((T-a)\lambda\big)  }\,db \\   \,+\, &\,\Bigg( 1 \,-\,\frac{ \nu\big((T-a)\lambda'  \big)    }{  \nu\big((T-a)\lambda  \big)    }   \Bigg)\,\frac{ 1   }{  \nu\big((T-a)\lambda  \big)    } \,=\, \log\Big(\frac{\lambda}{\lambda'}\Big)\,  \frac{ \nu\big((T-a)\lambda'  \big)    }{  \nu\big((T-a)\lambda  \big)    } \,. 
\end{align*}
After changing the integration variable to $ \theta= \frac{b-a}{T-a}$ and  putting $\beta=(T-a)\lambda $ and $\beta'=(T-a) \lambda' $, we arrive at the desired identity by multiplying both sides by $  \nu^2(\beta) $.
\end{proof}

We will apply the   converse of Proposition~\ref{PropMartEta} below in the next subsection.  Our proof relies on Lemma~\ref{LemmaDetClass}.
\begin{proposition}\label{PropDickmanConv} Fix some $T,\lambda>0$.  Let  $\{ \eta_s \}_{s\in [0,\infty)}$ be an increasing, right-continuous, $[0,T]$-valued process on a probability space  $(\Sigma, \mathcal{F},\mathcal{P})$.    If for each $\lambda'>0$  the process $\{ \mathbf{m}_s^{T,\lambda,\lambda'}\}_{s\in [0,\infty)} $ defined by $ \mathbf{m}_s^{T,\lambda,\lambda'}:= \big(\frac{\lambda'}{\lambda}\big)^{s\wedge \mathbf{S}} \varphi_T^{\lambda,\lambda'}( \eta_s)$  is a martingale
with respect to a filtration  $\{  F_s\}_{s\in [0,\infty)}$, then the process $\{ \eta_s \}_{s\in [0,\infty)}$ is Markovian with transition kernels  $\{ \mathscr{T}^{T,\lambda}_s \}_{s\in [0,\infty)}   $ with respect to $\{  F_s\}_{s\in [0,\infty)}$.
\end{proposition}

\begin{proof} The argument at the end of the proof of Proposition~\ref{PropMartEta} shows that the martingale $\{ \mathbf{m}_s^{T,\lambda,\lambda'}\}_{s\in [0,\infty)} $ is uniformly integrable. 
Applying the optional stopping theorem to this uniformly integrable  martingale   with the stopping time $\mathbf{S}$  yields that for any $s\in [0,\infty)$
\begin{align}\label{MGF}
\frac{ \nu\big((T-\eta_s)\lambda' \big)   }{  \nu\big((T-\eta_s)\lambda \big)    }\,=\,&\,\mathcal{E}\bigg[ \, \Big(\frac{\lambda'}{\lambda}\Big)^{\mathbf{S}-s}   \,\bigg|\,F_s \,\bigg] \quad \text{a.s.\ $\mathcal{P}$ on $\{s<\mathbf{S}\}  $} \,,
\end{align}
because $\eta_\mathbf{S}=T$ and $\varphi_T^{\lambda,\lambda'}(T)=1$.  For $r\in \R$, writing $\lambda'=\lambda e^r$ and $(\frac{\lambda'}{\lambda})^{\mathbf{S}-s}=e^{(\mathbf{S}-s)r }   $, the above provides a closed form for the   moment generating function for the random variable $\mathbf{S}-s$ conditional on $F_s$.  Furthermore, the left side of~(\ref{MGF}) (with $\lambda'=\lambda e^r$)  is the moment generating function for a random variable having the  probability density
\begin{align}\label{DenForTerm}
\mathfrak{D}(t)= \frac{1 }{  \nu\big((T-\eta_s) \lambda\big)    }\frac{ (T-\eta_s)^t  \lambda^t }{ \Gamma(t+1)  }\,,   \hspace{.5cm} t\in [0,\infty)  \,; 
\end{align}
see (i) of Proposition~\ref{PropVolterraPoisson}.  Thus, we can infer that  the distribution of the random variable $\mathbf{S}-s$ conditional on $F_s$ in the event $\{s<\mathbf{S}\}  $  has density of the form~(\ref{DenForTerm}). It follows that for $t\in (s,\infty)$
\begin{align}\label{DenForTerm2}
\mathcal{P}\big[\, \mathbf{S}\leq t\,\big|\,F_s \,\big]\,=\,\frac{1 }{  \nu\big((T-\eta_s)\lambda \big)    }\,\int_0^{t-s}\, \frac{ (T-\eta_s)^r \lambda^r  }{ \Gamma(r+1)  }\, dr\,=:\,\mathbf{p}_{t-s}^{T-\eta_s,\lambda  } \quad \text{a.s.\ $\mathcal{P}$ on $\{s<\mathbf{S}\}  $} \,.
\end{align}
  By our assumption that  $\mathbf{m}^{T,\lambda,\lambda'} $ is a martingale,  we have 
\begin{align*}
\Big(\frac{\lambda'}{\lambda}\Big)^{s\wedge \mathbf{S}} \,\varphi_T^{\lambda,\lambda'}(\eta_s) \,=\,&\,\mathcal{E}\bigg[\, \Big(\frac{\lambda'}{\lambda}\Big)^{t\wedge \mathbf{S}} \,\varphi_T^{\lambda,\lambda'}(\eta_t)  \,\bigg|\,F_s \,\bigg] \quad \text{a.s.\ $\mathcal{P}$} \,.
\end{align*}
In the event that $s<\mathbf{S}$,  we can write
\begin{align*}
\varphi_T^{\lambda,\lambda'}(\eta_s)   \,=\,&\, \Big(\frac{\lambda'}{\lambda}\Big)^{t-s}\,\mathcal{E}\Big[\, \varphi_T^{\lambda,\lambda'}(\eta_t)  \,
 1_{\mathbf{S} > t }\,\Big|\,F_s \,\Big] \,+\, \mathcal{E}\bigg[\, \Big(\frac{\lambda'}{\lambda}\Big)^{\mathbf{S}-s}\, 1_{\mathbf{S}\leq t  } \,\bigg|\,F_s \,\bigg] \nonumber\\ \,=\,&\, \Big(\frac{\lambda'}{\lambda}\Big)^{t-s}\,\mathcal{E}\Big[\, \varphi_T^{\lambda,\lambda'}(\eta_t)  
 \,\Big|\,F_s \,\Big] \,-\,\Big(\frac{\lambda'}{\lambda}\Big)^{t-s}\,\mathcal{P}\big[\, 
 \mathbf{S}\leq t\,\big|\,F_s \,\big]\,+\, \mathcal{E}\bigg[\, \Big(\frac{\lambda'}{\lambda}\Big)^{\mathbf{S}-s}\, 1_{\mathbf{S}\leq t  } \,\bigg|\,F_s \,\bigg]\nonumber
 \\ \,=\,&\, \Big(\frac{\lambda'}{\lambda}\Big)^{t-s}\,\mathcal{E}\Big[\, \varphi_T^{\lambda,\lambda'}(\eta_t)\,\Big|\,F_s \,\Big]\,-\,  \Big(\frac{\lambda'}{\lambda}\Big)^{t-s} \,\mathbf{p}_{t-s}^{T-\eta_s,\lambda  } \,+\, \frac{ 1 }{  \nu\big((T-\eta_s)\lambda \big)    }\int_0^{t-s}\, \frac{ (T-\eta_s)^r (\lambda')^r }{ \Gamma(r+1)  }\, dr \,,
\end{align*}
where  the second equality holds since $\eta_t=T$ when $\mathbf{S}\leq t$ and $\varphi_T^{\lambda,\lambda'}(T)=1 $.  The third equality follows from the  law of $\mathbf{S}-s$ having the probability density~(\ref{DenForTerm})  when conditioned on $F_s$ in the event $\{s<\mathbf{S}\}$, along with the observation~(\ref{DenForTerm2}). Using the above to solve for $\mathcal{E}\big[ \varphi_T^{\lambda,\lambda'}(\eta_t)\,\big|\,F_s \big] $, we find that 
\begin{align*}
\mathcal{E}\Big[\, \varphi_T^{\lambda,\lambda'}(\eta_t)\,\Big|\,F_s \,\Big] \,=\,&\,\mathbf{p}_{t-s}^{T-\eta_s,\lambda  }\,+\,\Big(\frac{\lambda}{\lambda'}\Big)^{t-s} \,\Bigg(\varphi_T^{\lambda,\lambda'}(\eta_s)  \,-\, \frac{ 1  }{  \nu\big((T-\eta_s)\lambda \big)    }\,\int_0^{t-s}\, \frac{  (T-\eta_s)^r \big(\lambda'\big)^r }{ \Gamma(r+1)  } \,dr\Bigg)  \nonumber \\
\,=\,&\,   \mathbf{p}_{t-s}^{T-\eta_s,\lambda  }\,+\,\Big(\frac{\lambda}{\lambda'}\Big)^{t-s} \, \varphi_T^{\lambda,\lambda'}(\eta_s)\,\Big(1-\mathbf{p}_{t-s}^{T-\eta_s,\lambda'  }  \Big) \,, \nonumber 
\end{align*}
in which the second equality follows from the definitions of  $   \varphi_T^{\lambda,\lambda'}$ of $\mathbf{p}_{t}^{T,\lambda  } $.
In consequence of Lemma~\ref{LemmaDetClass}, the above uniquely determines the  conditional distribution for $\eta_t$ given $F_s $ in the event $\{s<\mathbf{S} \}$.  Furthermore,  the measure $\mathscr{T}^{T,\lambda}_{t-s}(a, \cdot)$ for $a=\eta_s$ satisfies the equation
$$ 
\mathscr{T}^{T,\lambda}_{t-s}\Big(a,\, \varphi_T^{\lambda,\lambda'}\Big)\,=\, \mathbf{p}_{t-s}^{T-a,\lambda  } \,+\,\Big(\frac{\lambda}{\lambda'}\Big)^{t-s} \, \varphi_T^{\lambda,\lambda'}(a)\,\Big(1-\mathbf{p}_{t-s}^{T-a,\lambda'  }  \Big) \nonumber  \,,
$$
which can be verified using the identity~(\ref{RemarkConv}).  It follows that the process $\{ \eta_s \}_{s\in [0,\infty)}$ is a  Markovian process with transition kernels $\{\mathscr{T}^{T,\lambda}_{s}\}_{ s\in [0,\infty)  }$ with respect to $\{  F_s\}_{s\in [0,\infty)}$. 
\end{proof}

\subsection{Proofs of Theorem~\ref{ThmLocalTimeToLEVY} and (iii) of Proposition~\ref{PropTransMeas}}\label{SubsecInvProc}

\begin{proof}[Proof of Theorem~\ref{ThmLocalTimeToLEVY}]   From the proof of Theorem~\ref{ThmEquivalent}, we have that $\mathbf{L}^{T,\lambda,\lambda'}=\log(\frac{\lambda'}{\lambda}) \mathbf{L}$.  Thus, by Theorem~\ref{ThmPreGirsanov}  the process $\{ (\frac{\lambda'}{\lambda})^{\mathbf{L}_{t} } \mathcal{S}_{t}^{T,\lambda,\lambda'}\}_{t\in [0,T]} $ is a $\mathbf{P}^{T,\lambda}_{\mu}$-martingale with respect to the filtration $\mathscr{F}^{T,\mu}$, where  $\mathcal{S}_{t}^{T,\lambda,\lambda'}:=R^{\lambda,\lambda'}_{T-t}(X_t)$  for  the function $R^{\lambda,\lambda'}_{T-t}:\R^2\rightarrow [0,\infty)$ defined as in~(\ref{FirstR}), which satisfies $R^{\lambda,\lambda'}_{T-t}(0) =\varphi^{\lambda,\lambda'}_T(t) $.  The optional stopping theorem implies that  $\{ (\frac{\lambda'}{\lambda})^{\mathbf{L}_{\boldsymbol{\eta}_s} }\mathcal{S}_{\boldsymbol{\eta}_s}^{T,\lambda,\lambda'}\}_{s\in [0,\infty)} $ is  a martingale with respect to the filtration $\{ F_s\}_{s\in [0,\infty)}$ defined by $F_s:=\mathscr{F}_{\boldsymbol{\eta}_s }^{T,\mu}$. Since $\boldsymbol{\eta}$ is the right-continuous process inverse of $\mathbf{L}$ and $\mathbf{S}:=\inf\{s\in [0,\infty) \,:\, \boldsymbol{\eta}_s=T  \}=\mathbf{L}_T$, we have that $\mathbf{L}_{\boldsymbol{\eta}_s} = s\wedge \mathbf{S} $ for all $s\geq 0$. Furthermore, $\mathcal{S}_{\boldsymbol{\eta}_s}^{T,\lambda,\lambda'}=\varphi^{\lambda,\lambda'}_T(\boldsymbol{\eta}_s) $ for each $s\in [0,\mathbf{S})$, as $X_{ \boldsymbol{\eta}_s }=0$.  In fact, $\mathcal{S}_{\boldsymbol{\eta}_s}^{T,\lambda,\lambda'}=\varphi^{\lambda,\lambda'}_T(\boldsymbol{\eta}_s) $ holds for all $s\geq 0$ because $\boldsymbol{\eta}_s=T$ when $s\geq \mathbf{S}$ and   $\mathcal{S}_{T}^{T,\lambda,\lambda'}=1=\varphi^{\lambda,\lambda'}_T(T) $. Putting these observations together gives us that $ (\frac{\lambda'}{\lambda})^{\mathbf{L}_{\boldsymbol{\eta}_s} }\mathcal{S}_{\boldsymbol{\eta}_s}^{T,\lambda,\lambda'}= (\frac{\lambda'}{\lambda})^{s\wedge \mathbf{S}}\,\varphi^{\lambda,\lambda'}_T(\boldsymbol{\eta}_s)$ for all $s\geq 0$. We have shown that  $\{ (\frac{\lambda'}{\lambda})^{s\wedge \mathbf{S}}\varphi^{\lambda,\lambda'}_T(\boldsymbol{\eta}_s) \}_{s\in [0,\infty)}$ is a martingale with respect to $\{ F_s\}_{s\in [0,\infty)}$, and thus the process $\{ \boldsymbol{\eta}_s\}_{s\in [0,\infty)}$  is Markovian with transition kernels $\{ \mathscr{T}^{T,\lambda}_s\}_{s\in [0,\infty)}  $ with respect to $\{  F_s\}_{s\in [0,\infty)}$ by Proposition~\ref{PropDickmanConv}.  It follows from (ii) of Proposition~\ref{PropTransMeas} that $\{ \boldsymbol{\eta}_s\}_{s\in [0,\infty)}$ is a parameter $(\lambda,T)$ Volterra jump process.
\end{proof}

\begin{proof}[Proof of (iii) of Proposition~\ref{PropTransMeas}] As before,  for $s\in [0,\infty)$ let us put $\mathbf{m}_s^{T,\lambda,\lambda'}:=(\frac{\lambda'}{\lambda})^{s\wedge \mathbf{S}}\varphi_T^{\lambda,\lambda'}( \eta_s  ) $.  In light of Proposition~\ref{PropDickmanConv}, we can obtain our result by demonstrating that the  process $\{ \mathbf{m}_s^{T,\lambda,\lambda'}\}_{s\in [0,\infty)  }$ is a martingale with respect to  $\{  F_s\}_{s\in [0,\infty)}$  for any $\lambda' >0$.  The process $\{ \mathbf{m}_s^{T,\lambda,\lambda'}\}_{s\in [0,\infty)}$ is constant and equal to $(\frac{\lambda'}{\lambda})^{ \mathbf{S}}$ over the interval $[\mathbf{S},\infty)$ since  $\varphi_T^{\lambda,\lambda'}( T )=1$ and $\eta_s=T$ for $s\geq \mathbf{S}$.  It suffices to show that 
$$ \lim_{t\searrow s}\, \frac{\mathcal{E}\big[\mathbf{m}_t^{T,\lambda,\lambda'}\,\big|\, F_s  \big]-\mathbf{m}_s^{T,\lambda,\lambda'} }{t-s}\,=\,0  \hspace{.5cm}\text{a.s.\  $\mathcal{P}$ on  $\{s<\mathbf{S}\}$} \,. $$
  The product rule and our assumption on $\eta$ yield the following equalities in the event that $s<\mathbf{S}$:
\begin{align*}
  \lim_{t\searrow s}\,  \frac{\mathcal{E}\big[ \mathbf{m}_t^{T,\lambda,\lambda'}\,\big|\, F_s \big]\,-\,\mathbf{m}_s^{T,\lambda,\lambda'} }{t-s}   \,=\, &\,  \Big(\frac{\lambda'}{\lambda}\Big)^s \,\lim_{t\searrow s} \, \frac{\mathcal{E}\big[\varphi_T^{\lambda,\lambda'}( \eta_t )\,\big|\, F_s  \big]\,-\,\varphi_T^{\lambda,\lambda'}( \eta_s ) }{t-s}\,+\,\log\Big(\frac{\lambda'}{\lambda}\Big)\,\mathbf{m}_s^{T,\lambda,\lambda'} \nonumber \\  =\,&\, \Big(\frac{\lambda'}{\lambda}\Big)^s\,
\int_{[0,T]}\,\Big( \varphi_T^{\lambda,\lambda'}(b )\,-\,\varphi_T^{\lambda,\lambda'}( \eta_s ) \Big)  \, \mathscr{J}^{T, \lambda}(\eta_s,db) \,+\,\log\Big(\frac{\lambda'}{\lambda}\Big)\,\mathbf{m}_s^{T,\lambda,\lambda'} \,.
\end{align*}
  The right side above is equal to zero provided that for $t=T-\eta_s$ we have
\begin{align}\label{Dap}
\int_{[0,t]}\,\Big( \varphi_t^{\lambda,\lambda'}(a) \,-\,\varphi_t^{\lambda,\lambda'}(0)      \Big) \, \mathscr{J}^{t,\lambda}(0,da)\,=\,\log\Big(\frac{\lambda}{\lambda'}\Big) \,\varphi_t^{\lambda,\lambda'}(0) \,,
\end{align}
in which we have factored out $ (\frac{\lambda'}{\lambda})^s$ and used that $  \varphi_T^{\lambda,\lambda'}(b )=\varphi_{T-r}^{\lambda,\lambda'}(b-r)  $ and $\mathscr{J}^{T, \lambda}(r,db)= \mathscr{J}^{T-r,\lambda}\big(0,d(b-r)\big) $ for $r\in [0,T]$, where  $\mathscr{J}^{t,\lambda}(0,da):=\frac{1}{a} \frac{ \nu((t-a)\lambda)   }{ \nu(t\lambda)  }da+\frac{ 1  }{ \nu(t\lambda)  }\delta_t(da)$. However, after invoking the definitions of $\varphi_t^{\lambda,\lambda'}$ and $\mathscr{J}^{t,\lambda} $, the equation (\ref{Dap}) can be verified using Lemma~\ref{LemmaC}.
\end{proof}

\subsection{Proof of Theorem~\ref{ThmProcTRAN}}\label{SubsectionThmProcTRAN}

Let us recall some relevant notations.  We denote the space of $\R$-valued cadlag functions on $[0,\infty)$  by $ \boldsymbol{\bar{\Omega}}$, the coordinate process on $  \boldsymbol{\bar{\Omega}}$ by  $\{\eta_s\}_{s\in [0,\infty) }$, and the filtration that the coordinate process generates by $\{F_s^{\eta}\}_{s\in [0,\infty)}$, for which we will omit the superscript $\eta$ here: $F_s\equiv F_s^{\eta}$. Given $T,\lambda>0$ and a Borel probability measure $\vartheta$ on $[0,T]$, we use $\Pi^{T,\lambda}_{\vartheta}$ to denote the Borel probability measure on the path space $ \boldsymbol{\bar{\Omega}}$ with initial distribution $\vartheta$ and with finite-dimensional distributions determined by the family of transition kernels $\{ \mathscr{T}^{T,\lambda}_s\}_{s\in [0,\infty)}  $.  The expectation corresponding to $\Pi^{T,\lambda}_{\vartheta}$ is denoted by  $E^{T,\lambda }_{\vartheta}$.

\begin{proof}[Proof of Theorem~\ref{ThmProcTRAN}] Define the process $\{\mathbf{\bar{m}}_s^{T,\lambda,\lambda'}\}_{s\in [0,\infty]}$ by $  \mathbf{\bar{m}}_s^{T,\lambda,\lambda'}=\varphi^{\lambda',\lambda}_T(\eta_0) \mathbf{m}_s^{T,\lambda,\lambda'}  $ for the process $\mathbf{m}^{T,\lambda,\lambda'} $ defined as in  Proposition~\ref{PropMartEta}.  Notice that $\mathbf{\bar{m}}_0^{T,\lambda,\lambda'}=1$ and   that $\mathbf{\bar{m}}^{T,\lambda,\lambda'}$ is a  uniformly integrable martingale by Proposition~\ref{PropMartEta}.  As $s\rightarrow \infty$ the random variable $\mathbf{\bar{m}}_s^{T,\lambda,\lambda'}$ converges  in $L^1\big(\Pi^{T,\lambda}_{\vartheta}\big)$-norm to $\mathbf{\bar{m}}_{\infty}^{T,\lambda,\lambda'}:= \varphi^{\lambda',\lambda}_T(\eta_0)  (\frac{\lambda'}{\lambda})^{\mathbf{S}  }  $.
Define the probability measure $\widetilde{\Pi}^{T,\lambda,\lambda'}_{\vartheta}=\mathbf{\bar{m}}_{\infty}^{T,\lambda,\lambda'}\Pi^{T,\lambda}_{\vartheta}  $, and let $\widetilde{E}^{T,\lambda,\lambda'}_{\vartheta}$ denote its associated expectation.  For $0\leq s<t<\infty$ and a test function $\psi\in C\big([0,T]\big)$, Bayes' law gives us the first equality below.
\begin{align}\label{EtaBayes}
\widetilde{E}^{T,\lambda,\lambda'}_{\vartheta}\big[\,\psi(\eta_t)\,\big|\, F_s \,\big]\,=\,\frac{ E^{T,\lambda }_{\vartheta}\Big[\mathbf{\bar{m}}_{\infty}^{T,\lambda,\lambda'}\psi(\eta_t)\,\Big|\, F_s \Big]  }{ E^{T,\lambda }_{\vartheta}\Big[\,\mathbf{\bar{m}}_{\infty}^{T,\lambda,\lambda'}\,\Big|\, F_s \,\Big] }
\,=\,&\,\frac{ E^{T,\lambda }_{\vartheta}\Big[  E^{T,\lambda }_{\vartheta}\Big[\,\mathbf{\bar{m}}_{\infty}^{T,\lambda,\lambda'}\,\Big|\, F_t \Big]\psi(\eta_t)\,\Big|\, F_s \Big]  }{ E^{T,\lambda }_{\vartheta}\Big[\mathbf{\bar{m}}_{\infty}^{T,\lambda,\lambda'}\,\Big|\, F_s \Big]  }\nonumber \\
\,=\,&\,\frac{ E^{T,\lambda }_{\vartheta}\Big[  \mathbf{\bar{m}}_{t}^{T,\lambda,\lambda'}\psi(\eta_t)\,\Big|\, F_s \Big]  }{ \mathbf{\bar{m}}_{s}^{T,\lambda,\lambda'} }\nonumber \\
\,=\,&\, \int_{[0,T]}\,\psi(b) \,  \mathscr{T}^{T,\lambda'}_{t-s}(\eta_s,db)
\end{align}
The third equality above uses that $\mathbf{\bar{m}}^{T,\lambda,\lambda'}$ is a martingale, and we will prove the last equality in the analysis below.  It follows from~(\ref{EtaBayes}) that $\{\eta_s\}_{s\in [0,\infty)}  $ is a Markovian process under  $\widetilde{\Pi}^{T,\lambda,\lambda'}_{\vartheta}$ (and with respect to $\{F_s\}_{s\in [0,\infty)} $) having  transition kernels $\{ \mathscr{T}^{T,\lambda'}_s \}_{s\in [0,\infty)}  $, and hence  $\widetilde{\Pi}^{T,\lambda,\lambda'}_{\vartheta}=\Pi^{T,\lambda'}_{\vartheta}$.  Therefore, $\Pi^{T,\lambda'}_{\vartheta}\ll \Pi^{T,\lambda}_{\vartheta}$ and the Radon-Nikodym derivative is $\mathbf{\bar{m}}_{\infty}^{T,\lambda,\lambda'}$.

It remains to show the last equality in~(\ref{EtaBayes}).  In the event $s\geq \mathbf{S} $, we have $\Pi^{T,\lambda}_{\vartheta}[\eta_t=T\,|\, F_s] =1   $, and hence
$$   \frac{ E^{T,\lambda }_{\vartheta}\Big[  \mathbf{\bar{m}}_{t}^{T,\lambda,\lambda'}\psi(\eta_t)\,\Big|\, F_s \Big]  }{ \mathbf{\bar{m}}_{s}^{T,\lambda,\lambda'} }\,=\, \psi(T) \,\frac{ E^{T,\lambda }_{\vartheta}\Big[  \mathbf{\bar{m}}_{t}^{T,\lambda,\lambda'}\,\Big|\, F_s \Big]  }{ \mathbf{\bar{m}}_{s}^{T,\lambda,\lambda'} }\,=\, \psi(T) \,.  $$
On the other hand, the measure $ \mathscr{T}^{T,\lambda'}_{t-s}(\eta_s,\cdot)$ is equal to $\delta_T$, and so $\int_{[0,T]}\psi(b)   \mathscr{T}^{T,\lambda'}_{t-s}(\eta_s,db)=\psi(T)$. Thus the last equality in~(\ref{EtaBayes}) holds  when $s \geq \mathbf{S}$.  Next, in the event that $s <\mathbf{S} $, we split up our expression into two parts:
\begin{align*}
 \frac{E^{T,\lambda }_{\vartheta}\Big[\mathbf{\bar{m}}_{t}^{T,\lambda,\lambda'}\,\psi(\eta_t)\,\Big|\, F_s \Big] }{\mathbf{\bar{m}}_{s}^{T,\lambda,\lambda'}}  \,=\,\underbrace{\frac{E^{T,\lambda }_{\vartheta}\Big[\mathbf{\bar{m}}_{t}^{T,\lambda,\lambda'}\psi(\eta_t)1_{\mathbf{S} > t} \, \Big|\, F_s \Big]}{ \mathbf{\bar{m}}_{s}^{T,\lambda,\lambda'} }}_{\textup{(I)}} \,+\,\underbrace{\frac{E^{T,\lambda }_{\vartheta}\Big[\mathbf{\bar{m}}_{t}^{T,\lambda,\lambda'}\psi(\eta_t) 1_{\mathbf{S}\leq t} \,\Big|\, F_s \Big]}{ \mathbf{\bar{m}}_{s}^{T,\lambda,\lambda'} }}_{\textup{(II)}} \,.\nonumber 
\end{align*}
For (I) we can apply the Markov property to get
\begin{align}\label{EtaBayes1}
 \textup{(I)} \,=\,&\,\Big(\frac{\lambda'}{\lambda}\Big)^{t-s }\,\frac{ \nu\big( (T-\eta_s) \lambda   \big) }{ \nu\big((T-\eta_s) \lambda'   \big)  } \,\int_{[0,T)}\, \frac{ \nu\big( (T-b)  \lambda'  \big) }{ \nu\big( (T-b)\,\lambda      \big)  } \,\psi(b) \,\mathscr{T}_{t-s}^{T,\lambda}(\eta_s, db) \nonumber  \\
\,=\,&\, \Big(\frac{\lambda'}{\lambda}\Big)^{t-s }\,\frac{ \nu\big( (T-\eta_s) \lambda   \big) }{ \nu\big((T-\eta_s)\lambda'    \big)  }  \,\int_{0}^T\, \psi(b) \,\frac{ \nu\big((T-b) \lambda'\big)}{ \nu\big((T-\eta_s)\lambda \big)}\,\frac{(b-\eta_s)^{t-s-1} \lambda^{t-s}  }{ \Gamma(t-s) } \,db\nonumber  \\
  \,=\,&\,    \int_{0}^T\,  \psi(b)\, \frac{ \nu\big( (T-b) \lambda'   \big) }{ \nu\big((T-\eta_s)\lambda' \big)  } \frac{(b-\eta_s)^{t-s-1} \big(\lambda'\big)^{t-s}  }{ \Gamma(t-s) } \, db  \,=\,\int_{[0,T)}\,\psi(b)  \, \mathscr{T}^{T,\lambda'}_{t-s}(\eta_s,db)\,.
\end{align}
Next, for (II) we compute as follows:
\begin{align} \label{EtaBayes2}
 \textup{(II)}  &\,=\,\psi(T)\,\frac{ \nu\big( (T-\eta_s) \lambda   \big) }{ \nu\big( (T-\eta_s)\lambda'    \big)  } \,E^{T,\lambda }_{\vartheta}\bigg[\, \Big(\frac{\lambda'}{\lambda}\Big)^{\mathbf{S}-s  } \, 1_{\mathbf{S}\leq t} \,\Big|\, F_s \,\bigg] \nonumber \\
  &\,=\,\psi(T)\,\frac{ \nu\big((T-\eta_s)  \lambda   \big) }{ \nu\big( (T-\eta_s) \lambda'   \big)  }\, \frac{1  }{ \nu\big((T-\eta_s) \lambda\big)}\, \int_{0}^{t-s}\, \Big(\frac{\lambda'}{\lambda}\Big)^r \,\frac{(T-\eta_s)^r \lambda^r  }{ \Gamma(r+1) }   \,     dr \nonumber  \\
  &\,=\,\psi(T)\,\frac{ 1}{ \nu\big( (T-\eta_s) \lambda'   \big)  }  \,\int_{0}^{t-s}\,  \frac{(T-\eta_s)^r (\lambda')^r  }{ \Gamma(r+1) }    \,    dr  
  \,=\,\psi(T)  \, \mathscr{T}^{T,\lambda'}_{t-s}\big(\eta_s,\{T\} \big)\,,
\end{align}
where we have used that the conditional probability density for  $\mathbf{S}-s$ given $F_s$ in the event $s<\mathbf{S}$ is given by $ \Pi^{T,\lambda}_{\vartheta}\big[\mathbf{S}-s=r\,|\,F_s  
 \big] =\frac{(T-\eta_s)^r \lambda^r  }{\nu((T-\eta_s) \lambda ) \Gamma(r+1) } $, which  follows from Proposition~\ref{PropMartEta} and~(\ref{DenForTerm2}).  Adding~(\ref{EtaBayes1}) and~(\ref{EtaBayes2}) yields~(\ref{EtaBayes}).
\end{proof}

\subsection{Proof of Lemma~\ref{LemmaEscape}} \label{SubSecSomeTimes}

\begin{proof} Recall that the renewal density $\mathbf{G}^{T,\lambda}_a$ has the form in  Corollary~\ref{CorRenewalForm} and the jump rate measure $\mathscr{J}^{T,\lambda}(a,\cdot)$ is defined as in~(\ref{JUMPFORM}).  The distributional measure for $\eta_{ \varpi^{\varepsilon} }$ can be written in terms of  $\mathbf{G}^{T,\lambda}_0$ and $\mathscr{J}^{t,\lambda}(a,\cdot)$ as 
$$ \mathcal{P}\big[ \,\eta_{ \varpi^{\varepsilon } } \in E \,\big]\,=\,  \int_{0}^{\varepsilon}\,\mathbf{G}^{T,\lambda}_0(a)\,\mathscr{J}^{T,\lambda}\big(a, E\cap [\varepsilon,T] \big)\,da \,.   $$
Applying the above with $E=\{T\}$  yields that $
\mathcal{P}\big[\eta_{ \varpi^{\varepsilon} } = T \big]
= \frac{ \nu(\varepsilon\lambda  )}{ \nu(T\lambda )} 
$,
and for $b\in [0,T)$ 
\begin{align*}
\frac{\mathcal{P}\big[\eta_{ \varpi^{\varepsilon} } \in db\big]}{db}\,=\,& \, 1_{(\varepsilon,T)}(b) \,\int_{0}^{\varepsilon}\,\mathbf{G}^{T,\lambda}_0(a)\,\frac{\mathscr{J}^{T,\lambda}(a,db)}{db}\,da 
\,=\,  1_{ (\varepsilon,T)}(b) \,\frac{ \nu\big((T-b)\lambda \big)}{ \nu(T\lambda )} \,\int_{0}^{\varepsilon} \,\frac{\lambda \nu'(a\lambda )}{b-a}\,da \,.
\end{align*}
Thus, the law of the random variable $\eta_{ \varpi^{\varepsilon} } $ has the stated form. 
\end{proof}

\section{Hausdorff analysis for the set of visitation times to the  origin } \label{SectionHausdorff}
Recall that $\mathscr{O}(\omega)$ for $\omega\in \boldsymbol{\Omega}$  denotes the set of times $t\in [0,\infty)$  such that  $X_t(\omega)=0 $.  Also, recall that the outer measure $H_{h}$ on $\R$ for a given value  $h >0$ is defined as in~(\ref{LogHausMeas}), and we use the family $\{ H_{h}\}_{h\in (0,\infty)}$ to define the \textit{log-Hausdorff exponent} of a set $S\subset \R$ as in~(\ref{LogHausExp}). In this section, we will prove  the pair of propositions below, which together imply Theorem~\ref{ThmHAUSDORF}.  Note that Theorem~\ref{ThmHAUSDORF} trivially implies (i) \& (ii) of Proposition~\ref{PropZeroSetBasics}, and (iii) of Proposition~\ref{PropZeroSetBasics}, which states that   the Borel measure $\vartheta(\omega,\cdot)$ on $[0,\infty)$ with distribution function $t\mapsto \mathbf{L}_t(\omega)$ is  $\mathbf{P}_{\mu}^{T,\lambda}$ almost surely supported on $\mathscr{O}(\omega)$, follows from the second bullet point in Proposition~\ref{PropSubMartIII}  because we define  $\mathbf{L}:=-\mathring{\mathcal{A}}^{T,\lambda}$.

\begin{proposition}\label{PropHausdorfupperBound} Fix some $T,\lambda>0$ and a Borel probability measure $\mu$ on $\R^2$. The value   $H_{1}\big(\mathscr{O}(\omega) \big)$ is finite for  $\mathbf{P}_{\mu}^{T,\lambda}$ almost every  $\omega\in \boldsymbol{\Omega} $.  Thus, the random  set $\mathscr{O}(\omega)$  has  log-Hausdorff exponent $\leq 1$ almost surely under $\mathbf{P}_{\mu}^{T,\lambda}$.
\end{proposition}
The next proposition states that  the random measure  $\vartheta(\omega,\cdot)$ almost surely  $\mathbf{P}_{\mu }^{T,\lambda}$ has finite energy 
corresponding to the dimension function $d_h(a)=(1+\log^+\frac{1}{a})^{-h}     $  for all $h\in (0,1)$.  
\begin{proposition}\label{PropHausdorflowerBound}  Fix some $T,\lambda>0$ and a Borel probability measure $\mu$ on $\R^2$.
  For $\mathbf{P}_{\mu }^{T,\lambda}$ almost every realization of the random  Borel measure $\vartheta(\omega,\cdot)$, we have
$$ \int_{[0,\infty)}\, \int_{[0,\infty)} \,\frac{1}{d_{h}\big(|t-t'|\big)} \,\vartheta(\omega, dt)\,\vartheta(\omega, dt') \,<\, \infty $$
for all $h\in (0,1)$.  
\end{proposition}
Combining Proposition~\ref{PropHausdorfupperBound} with the following corollary of Proposition~\ref{PropHausdorflowerBound} gives us Theorem~\ref{ThmHAUSDORF}.
\begin{corollary}\label{CorollaryHausdorflowerBound} Fix some $T,\lambda>0$ and a Borel probability measure $\mu$ on $\R^2$.  For $\mathbf{P}_{\mu}^{T,\lambda}$ almost every $\omega\in \mathcal{O}$, the set $\mathscr{O}(\omega)$ has log-Hausdorff exponent $\geq 1$.
\end{corollary}
\begin{proof} As mentioned above, the measure  $\vartheta(\omega,\cdot)$ is supported on $\mathscr{O}(\omega)$ for  $\mathbf{P}_{\mu}^{T,\lambda}$ almost every $\omega\in \boldsymbol{\Omega}$.  The total mass of $\vartheta(\omega,\cdot)$ is $\mathbf{L}_{T}(\omega)$, which is $\mathbf{P}_{\mu}^{T,\lambda}$ almost surely positive in the event $\mathcal{O}$ by  Proposition~\ref{PropLocalTimeProp}.  It follows from a standard energy method argument that $\mathscr{O}(\omega)$  has  log-Hausdorff exponent $\geq 1$ for  $\mathbf{P}_{\mu}^{T,\lambda}$ almost every $\omega\in \mathcal{O}$.
\end{proof}

\subsection{Proof of Proposition~\ref{PropHausdorfupperBound}}

The following corollary of Lemma~\ref{LemmaEscape} provides a lower bound for the probability that a parameter $(\lambda,T)$ Volterra jump  process starting from $0$ never takes a value in the interval $[\varepsilon, T)$ for small $\varepsilon>0$, that is, the process jumps directly to the terminal state $T$ from some point in $[0,\varepsilon)$.
\begin{corollary}\label{CorollaryEscape} For $T,\lambda>0$ let  $\{ \eta_s \}_{s\in [0,\infty)}$ be a parameter $(\lambda,T)$ Volterra jump process   with respect to a filtration $\{ F_s\}_{s\in [0,\infty)}$ on a probability space  $(\Sigma, \mathcal{F},\mathcal{P}_{T,\lambda})$.  Assume that $\eta_0=0$ almost surely.  For $\varepsilon>0$  define the stopping time $ \varpi^{\varepsilon}:=\inf\{s\in [0,\infty) \,:\,  \eta_s  \geq  \varepsilon  \}$. For any $L>1$ there exists a  $C_{L}>0$ such that $\mathcal{P}_{T,\lambda}[\eta_{ \varpi^{\varepsilon} } =T]\geq C_{L} (1+ \log^+  \frac{1}{\varepsilon} )^{-1}$ for all $\varepsilon>0$, $T >0$, and $\lambda\geq \frac{1}{L}$ with $T\lambda \leq L $.
\end{corollary}
\begin{proof}
Since $\mathcal{P}_{T,\lambda}[\eta_{ \varpi^{\varepsilon} } =T]=\frac{\nu(\varepsilon \lambda)}{\nu(T \lambda) } $ by Lemma~\ref{LemmaEscape}, the above lower bound for  $\mathcal{P}_{T,\lambda}[\eta_{ \varpi^{\varepsilon} } =T]$ holds since 
 $\nu$ is an increasing function having the small $x>0$ asymptotics $\nu(x)\sim (  \log  \frac{1}{x} )^{-1}   $; see Lemma~\ref{LemmaEFunAsy}.    
\end{proof}

Given  some  $\varepsilon >0$ and a Volterra jump process $\eta$,  define the sequence of stopping times $\{ \varpi_{n}^{\varepsilon} \}_{n\in \mathbb{N}_0}$  such that $\varpi_{0}^{\varepsilon}=0$ and for $n\in \mathbb{N}$
\begin{align}\label{DefVarpis}
    \varpi_{n}^{\varepsilon} \,=\, \inf\big\{ s\in \big[\varpi_{n-1}^{\varepsilon},\infty\big) \,:\,\eta_s\geq \eta_{ \varpi_{n-1}^{\varepsilon} }+\varepsilon\big\}\,,
\end{align}
with the usual interpretation that $\inf \emptyset =\infty$.  For the stopping time $\mathbf{S}=\inf\{s\in [0,\infty)\,:\,  \eta_s=T \}$, let $\mathscr{N}_{\varepsilon}$ denote the largest $n$ such that $\varpi_{n}^{\varepsilon} <\mathbf{S}$. The next lemma bounds the expectation of $\mathscr{N}_{\varepsilon}$ when $\varepsilon$ is small.
\begin{lemma}\label{LemmaEscapeTimes} Fix some $T,\lambda>0$.  Let  $\{ \eta_s \}_{s\in [0,\infty)}$ be a parameter $(\lambda,T)$ Volterra jump process with respect to a filtration $\{  F_s\}_{s\in [0,\infty)}$  on a probability space  $(\Sigma, \mathcal{F},\mathcal{P}_{T,\lambda})$. For $\varepsilon>0$  define the sequence of stopping times $\{ \varpi_{n}^{\varepsilon} \}_{n\in \mathbb{N}_0}$  and the random variable $\mathscr{N}_{\varepsilon}$  as above. Then there exists a  $\mathbf{c}_{T,\lambda}>0$ such that $\mathcal{E}_{T,\lambda}[  \mathscr{N}_{\varepsilon} ] \leq  \mathbf{c}_{T,\lambda}(  1+\log^+ \frac{1}{\varepsilon}   )$ for all $\varepsilon >0$.

\end{lemma}

\begin{proof} Since $  \mathscr{N}_{\varepsilon}=\sum_{n=0}^{\infty}  1_{  \varpi_{n}^{\varepsilon}<\mathbf{S} }  $, we can write
$    \mathcal{E}_{T,\lambda}[  \mathscr{N}_{\varepsilon} ]=\sum_{n=0}^{\infty} \mathcal{P}_{T,\lambda}\big[  \varpi_{n}^{\varepsilon}   <\mathbf{S}\big]
$.
For $n\in \mathbb{N}$ we can use that $ 1_{ \varpi_{n}^{\varepsilon}<\mathbf{S}  }=1_{ \varpi_{n}^{\varepsilon}<\mathbf{S}  }  1_{ \varpi_{n-1}^{\varepsilon}<\mathbf{S} }$ to write
\begin{align}\label{Heffer}
  \mathcal{P}_{T,\lambda}\big[  \varpi_{n}^{\varepsilon}   <\mathbf{S} \big] \,=\,&  \mathcal{E}_{T,\lambda}\Big[\,   \mathcal{P}_{T,\lambda}\big[\,\varpi_{n}^{\varepsilon}<\mathbf{S}  \,\big|\,   F_{ \varpi_{n-1}^{\varepsilon} }\,\big]\, 1_{ \varpi_{n-1}^{\varepsilon}<\mathbf{S}  } \,\Big] \,.
  \end{align}
For the nested conditional probability above, we can apply  the strong Markov property to get that
 \begin{align*}
   \mathcal{P}_{T,\lambda}\big[\,\varpi_{n}^{\varepsilon}<\mathbf{S} \,\big|\,   F_{ \varpi_{n-1}^{\varepsilon} }\,\big]\,=\, \mathcal{P}_{T-\varpi_{n-1}^{\varepsilon},\lambda}\big[\,\varpi^{\varepsilon}<\mathbf{S}\, \big] \,=\, 1\,-\,\mathcal{P}_{T-\varpi_{n-1}^{\varepsilon},\lambda}\big[\,\eta_{\varpi^{\varepsilon}  }= T\,\big] 
    \,\leq\, 1- \frac{C_{T,\lambda}  }{ 1+\log^+ \frac{1}{\varepsilon}  } \,, \nonumber
\end{align*}
where the second equality uses that $\varpi^{\varepsilon}<\mathbf{S} $ if and only if $\eta_{\varpi^{\varepsilon}  }\neq T$, and the inequality holds for some constant $C_{T,\lambda}>0$ and all $\varepsilon>0$ by Corollary~\ref{CorollaryEscape}. Applying this bound in~(\ref{Heffer}) yields the  inequality 
 $
  \mathcal{P}_{T,\lambda}[  \varpi_{n}^{\varepsilon}   < \mathbf{S} ] 
    \leq  \big(1-\frac{C_{T,\lambda}  }{ 1+\log^+ \frac{1}{\varepsilon} } \big)\mathcal{P}_{T,\lambda}\big[   \varpi_{n-1}^{\varepsilon}<\mathbf{S}  \big] $, and thus by induction $
   \mathcal{P}_{T,\lambda}[  \varpi_{n}^{\varepsilon}   <\infty ] \leq  \big(1-\frac{C_{T,\lambda}  }{ 1+\log^+ \frac{1}{\varepsilon}  } \big)^n
$ for all $n\in \mathbb{N}_0$. Finally, applying this bound to the terms in the series expression  for $ \mathcal{E}_{T,\lambda}[  \mathscr{N}_{\varepsilon} ]$ above results in a geometric series that sums to a multiple of $1+\log^+ \frac{1}{\varepsilon}$.
\end{proof}

\begin{proof}[Proof of Proposition~\ref{PropHausdorfupperBound}] Recall that $\{\boldsymbol{\eta}_s  \}_{s\in [0,\infty)}$ denotes the right-continuous process inverse of the local time  $\{\mathbf{L}_t\}_{t\in [0,T]}$.  By Theorem~\ref{ThmLocalTimeToLEVY} the process   $\boldsymbol{\eta} $ is a parameter $(\lambda,T) $ Volterra jump process with respect to the filtration $F_s:=\mathscr{F}_{\boldsymbol{\eta}_s}^{T,\mu} $.  Given some $\varepsilon>0$,  let the sequence of $F$-stopping times $\{\varpi_{n}^{\varepsilon}\}_{n\in\mathbb{N}_0}$ be defined in terms of $\boldsymbol{\eta} $  as in~(\ref{DefVarpis}), and let $\mathscr{N}_{\varepsilon}$ denote the largest $n\in \mathbb{N}_0$ such that $\varpi_{n}^{\varepsilon} <\mathbf{S}:= \mathbf{L}_T$.  The random times  $\boldsymbol{\eta}_{  \varpi_{n}^{\varepsilon}}$ for each $0\leq n < \mathscr{N}_{\varepsilon}$ satisfy
\begin{align*}
    \boldsymbol{\eta}_{  \varpi_{n+1}^{\varepsilon}}\,=\,\inf\Big\{ t\in \big[\boldsymbol{\eta}_{  \varpi_{n}^{\varepsilon}}+\varepsilon,\infty\big)  \,:\, \mathbf{L}_t> \mathbf{L}_{\boldsymbol{\eta}_{  \varpi_{n}^{\varepsilon} } +\varepsilon }\Big\}\,.
\end{align*}
  We will argue below that the random set $\mathscr{O}\equiv \mathscr{O}(\omega)$ satisfies
\begin{align}\label{OriginSetCovering}
\mathscr{O}\,\subset \bigcup_{n=0}^{ \mathscr{N}_{\varepsilon} } \,\big[ \boldsymbol{\eta}_{  \varpi_{n}^{\varepsilon}}, \,\boldsymbol{\eta}_{  \varpi_{n}^{\varepsilon} } \,+\,\varepsilon \big] \hspace{.5cm}\text{a.s.\,\,$\mathbf{P}_{\mu}^{T,\lambda}$\,.} 
\end{align}
 Since each interval $ \big[ \boldsymbol{\eta}_{  \varpi_{n}^{\varepsilon}}, \boldsymbol{\eta}_{  \varpi_{n}^{\varepsilon} } +\varepsilon \big]$ has length $ \varepsilon$, it follows that
$ H_{ \varepsilon,1}( \mathscr{O}) \leq   (1+\log^+  \frac{1}{\varepsilon}      )^{-1} (1+\mathscr{N}_{\varepsilon}) $   almost surely $\mathbf{P}_{\mu}^{T,\lambda}$.  Applying this inequality  after the monotone convergence theorem   yields the relations below.
$$   \mathbf{E}_{\mu }^{T,\lambda} \big[ H_{1}( \mathscr{O})\big]\,=\lim_{\varepsilon\searrow 0} \, \mathbf{E}_{\mu }^{T,\lambda} \big[ H_{ \varepsilon,1}( \mathscr{O}) \big] \,\leq \,\limsup_{ \varepsilon\searrow 0 }  \, \frac{1+\mathbf{E}_{\mu }^{T,\lambda} [\mathscr{N}_{\varepsilon} ]}{1+\log^+  \frac{1}{\varepsilon}       }  $$
However, by Lemma~\ref{LemmaEscapeTimes}, the expectation $\mathbf{E}_{\mu }^{T,\lambda} [\mathscr{N}_{\varepsilon} ]$ is bounded by a constant multiple of $ 1+\log^+  \frac{1}{\varepsilon} $ for all $\varepsilon>0$.  Therefore, $\mathbf{E}_{\mu }^{T,\lambda} \big[ H_{1}( \mathscr{O})\big]$ is finite, and so $ H_{1}( \mathscr{O})$ is $\mathbf{P}_{\mu }^{T,\lambda} $-a.s.\ finite.  \vspace{.1cm}

To prove~(\ref{OriginSetCovering}), it suffices to show that for any $n\in \mathbb{N}$    the random sets $\mathscr{O}$ and $\big(\boldsymbol{\eta}_{\varpi_{n}^{\varepsilon}} +\varepsilon    , \boldsymbol{\eta}_{\varpi_{n+1}^{\varepsilon}} \big)  $  are $\mathbf{P}_{\mu}^{T,\lambda}$ almost surely  disjoint.  Define $\tau_{n}:=\inf\big\{t\in [\boldsymbol{\eta}_{\varpi_{n}^{\varepsilon}}+\varepsilon ,T]\,:\,X_t=0   \big\}  $. Then  $\mathscr{O}$ and $\big( \boldsymbol{\eta}_{\varpi_{n}^{\varepsilon}} +\varepsilon    , \boldsymbol{\eta}_{\varpi_{n+1}^{\varepsilon}} \big)  $ intersect  only when $\tau_{n} < \boldsymbol{\eta}_{\varpi_{n+1}^{\varepsilon}} $.   However, the  strong Markov property with the stopping time $\boldsymbol{\eta}_{\varpi_{n}^{\varepsilon}} +\varepsilon$ combined with  Proposition~\ref{PropLocalTimeProp} yields that $\tau_n= \boldsymbol{\eta}_{\varpi_{n+1}^{\varepsilon}} $ holds  almost surely $\mathbf{P}_{\mu}^{T,\lambda}$.
\end{proof}

\subsection{Proof of Proposition~\ref{PropHausdorflowerBound}}

\begin{proof} Let $\{\boldsymbol{\eta}_s \}_{s\in [0,\infty)}$ denote the right-continuous process inverse of the local time  $\{ \mathbf{L}_t  \}_{t\in [0,T]}$.  Since $\vartheta(\omega,  \cdot )$ for $\omega \in \boldsymbol{\Omega}$ is defined as the Borel measure on $[0,\infty)$ with cumulative distribution function $t\mapsto \mathbf{L}_t(\omega)$, we have that
\begin{align} \label{TimePass}
\int_{[0,T]}\,g(t)\, \vartheta(\omega, dt)\,=\, \int_0^{\infty}\,g\big(\boldsymbol{\eta}_s(\omega)\big) \, 1_{\boldsymbol{\eta}_s(\omega)< T }\,ds \,, 
\end{align}
 for any nonnegative Borel measurable function $g$ on $[0,T]$.  Using~(\ref{TimePass}) and symmetry, we can write
\begin{align*}
\int_{[0,T]}\,\int_{ [0,T] }\,\frac{1}{d_{h}\big(|t'-t|\big)} \,\vartheta(\omega, dt')\,\vartheta(\omega, dt)   \,=\, 2\int_{0}^{\infty}\,\int_{r}^{\infty} \,\frac{1}{d_{h}\big(|\boldsymbol{\eta}_{s}(\omega)-\boldsymbol{\eta}_r(\omega)|\big)} \,1_{\boldsymbol{\eta}_{s}(\omega) < T }\, ds\,1_{\boldsymbol{\eta}_r(\omega)  < T }\,dr\, .
\end{align*}
 Thus, through first swapping the order of the expectation and the outer integration and then inserting a nested conditional expectation, we have the first equality below.
\begin{align}\label{EnergyCalculation}
\mathbf{E}_{\mu}^{T,\lambda}\bigg[\, \int_{[0,T]}\,\int_{ [0,T] }\,&\,\frac{1}{d_{h}\big(|t'-t|\big)} \,\vartheta(dt')\,\vartheta(dt) \, \bigg]  \nonumber  \\
\,=\,&\, 2\, \int_{0}^{\infty}\,\mathbf{E}_{\mu}^{T,\lambda}\bigg[ \,\mathbf{E}_{\mu}^{T,\lambda}\bigg[\,  \int_{r}^{\infty}\, \frac{1}{d_{h}\big(|\boldsymbol{\eta}_{s}-\boldsymbol{\eta}_r|\big)} \, 1_{\boldsymbol{\eta}_{s} < T }\,ds\,\bigg|\,\mathscr{F}_{\boldsymbol{\eta}_r}^{T,\mu} \,\bigg] \,  1_{\boldsymbol{\eta}_r < T } \, \bigg] \,dr   \nonumber
\\
\,=\,&\, 2 \,\int_{0}^{\infty}\,\mathbf{E}_{\mu}^{T,\lambda}\bigg[ \,E^{T-\boldsymbol{\eta}_r ,\lambda}_{0}\bigg[  \,\int_{0}^{\infty}\, \frac{1}{d_{h}(\eta_{s})}\,  1_{ \eta_{s} < T-\boldsymbol{\eta}_r  }\,ds\,\bigg]  \, 1_{\boldsymbol{\eta}_r < T  }\,\bigg]\, dr   \nonumber
 \\
\,=\,&\, 2\, \int_{0}^{\infty}\,\mathbf{E}_{\mu}^{T,\lambda}\bigg[ \,\bigg(\int_{0}^{T-\boldsymbol{\eta}_r}\,\frac{1}{d_{h}(a)} \,\mathbf{G}^{T-\boldsymbol{\eta}_r , \lambda }_0(a)\, da\bigg)\, 1_{\boldsymbol{\eta}_r < T } \,\bigg] \, dr
\end{align}
As in Section~\ref{SubsubsectionNuPathMeasure}, $E^{T ,\lambda}_a$ denotes the expectation with respect to the probability measure $\Pi^{T,\lambda}_a$ on the path space $\boldsymbol{\bar{\Omega}}$, and $\{\eta_t\}_{t\in [0,\infty)} $ denotes the coordinate  process on $\boldsymbol{\bar{\Omega}}$. By Theorem~\ref{ThmLocalTimeToLEVY}, $\boldsymbol{\eta}$ is a parameter $(\lambda,T) $ Volterra jump process under $\mathbf{P}_{\mu}^{T,\lambda}$    with respect to the filtration $F_s:=\mathscr{F}_{\boldsymbol{\eta}_s}^{T,\mu} $.  The second equality in~(\ref{EnergyCalculation}) uses that  $X_{\boldsymbol{\eta}_s}=0$ when $\boldsymbol{\eta}_s<T$ and  applies the strong Markov property. The third equality holds since  $\mathbf{G}^{T, \lambda }_0:[0,T)\rightarrow [0,\infty] $ is the density of the renewal measure~(\ref{RenewalMeasure}).

Bringing in the definition of $d_{h}$ and the expression for $\mathbf{G}^{t,\lambda}_0$ in Corollary~\ref{CorRenewalForm},   we have the following equality for any $t\in [0,T]$:
\begin{align}\label{SubEnergy}
    \int_{0}^{t}\,\frac{1}{d_{h}(a)} \, \mathbf{G}^{t,\lambda}_0(a)\,da \,=\,&  \int_{0}^{t}\,\lambda\,\nu'(a\lambda)\,\frac{ \nu\big((t-a)\lambda\big)  }{ \nu(t\lambda)   }  \,\Big(1+\log^+\frac{1}{a}    \Big)^h \, da \nonumber  \\
     \,\preceq \,& \,\int_{0}^{T}\,\frac{1}{a\big(1+\log^+ \frac{1}{a}  \big)^{2-h}  }\,da\,= \,\frac{ 1 }{(1-h)\big(1+\log^+\frac{1}{T} \big)^{1-h}  }\,+\,(\log T)\,1_{T>1}\,.
\end{align}
Since $\nu$ is an increasing function and $\nu'(x)\sim \frac{1}{x}\log^{-2}\frac{1}{x}   $ as $x\searrow 0$ by (iii) of Lemma~\ref{LemmaEFunAsy}, we have that $ \frac{ \nu((t-a)\lambda)  }{ \nu(t\lambda)   } \leq 1 $  and that $\nu'(x)$ is bounded by a multiple of $\frac{1}{x }(1+\log^+ \frac{1}{x}  )^{-2} $ for all $x\in (0,T\lambda]$.  Thus, the first inequality above holds for all $t\in [0,T]$.  We can bound~(\ref{EnergyCalculation}) by the product of $(\ref{SubEnergy}) $ and $ \mathbf{E}_{\mu}^{T,\lambda}[\, \mathbf{L}_T]$ because
\begin{align*}
 \int_{0}^{\infty}\,\mathbf{E}_{\mu}^{T,\lambda}\big[\, 1_{\boldsymbol{\eta}_s < T } \, \big]\,ds  \,  
 \,=\, \mathbf{E}_{\mu}^{T,\lambda}\bigg[\,\int_{0}^{\infty}\, 1_{\boldsymbol{\eta}_s < T }\, ds\, \bigg]  \,  
 \,=\,  \mathbf{E}_{\mu}^{T,\lambda}\big[\, \mathbf{L}_T\,\big] \,.
\end{align*}
Since $\mathbf{E}_{\mu}^{T,\lambda}[ \mathbf{L}_T ]  <\infty$ (and, in fact, $\mathbf{L}_T $ has finite exponential moments in consequence of Proposition~\ref{PropStochPre}), we have shown that the  $d_h$-energy of the random measure $\vartheta(\omega,\cdot)$ has finite expectation for any $h\in (0,1)$, and this implies our desired result. 
\end{proof}

\section{Bounds and approximations for special functions}\label{SectionBasics}
We will now provide some bounds and approximations for special  functions related to the transition density kernel $d^{T,\lambda}_{s,t}(x,y)$ introduced in Section~\ref{SubSecTransProb}.  In particular, we treat  $H_T^{\lambda}(x)$ in Section~\ref{SubsectHFunct}, the drift function $b_t^{\lambda}(x)$ in Section~\ref{SubsecDriftFunc}, and the integral kernel  $h_t^{\lambda}(x,y)$  in Section~\ref{SubSectHKern}.  As preliminaries, we discuss some properties of the Volterra  function $\nu$ in  Section~\ref{SubsecFractExp} and the exponential integral function $E(x)$ in Section~\ref{SubsecGammaFun}.

\subsection{The exponential integral function}\label{SubsecGammaFun}
For $x>0$ the so-called \textit{exponential integral} is defined by  $E(x):=\int_x^{\infty} \frac{e^{-y   }}{y }dy$. Note that $E$ is equal to the  upper incomplete gamma function $\Gamma(s,x)=\int_{x}^{\infty}y^{s-1}e^{-y}dy$ in the degenerate  borderline case $s=0$. Naturally, we can extend $E$ to a function on $\C\backslash \{0\} $  that is analytic on $\C\backslash (-\infty, 0]$ by using
contour integrals that do not cross the negative real axis.  Also, we  define the  \textit{complementary exponential integral}  by $\widetilde{E}(x):= \int_{0}^{x}\frac{1-e^{-y}  }{  y}dy$, which is entire, and the  function   $\boldsymbol{E}: \C\backslash\{0\}\rightarrow \C $ 
 \begin{align*}
 \boldsymbol{E}(x)\,:=\,e^x\,E(x)\,,
 \end{align*}
which is equal to the Tricomi confluent hypergeometric function $U(a,b,x)$ with $(a,b)=(1,1)$, and thus satisfies Kummer's differential equation $x\frac{d^2 \boldsymbol{E} }{ dx^2 }+(1-x)\frac{d \boldsymbol{E} }{ dx } -\boldsymbol{E}=0 $.  The following lemma collects some basic properties of $E$  and  $ \boldsymbol{E}$.  The identity in (i) below can be found in many places, for instance~\cite[Eq.\ 5.1.39]{Abramowitz}, and (ii)--(iv) are readily verified starting from it.
\begin{lemma}\label{LemEM} Let the functions $E$, $\widetilde{E}$, and $\boldsymbol{E}$ be defined as above.
\begin{enumerate}[(i)]
    \item For any $x\in \C\backslash\{0\}$, we have $ E(x)=-\log x-\gamma_{\mathsmaller{\textup{EM}}} +\widetilde{E}(x) $.

  \item  As $x\rightarrow 0$ we have   $E(x)= -\log x-\gamma_{\mathsmaller{\textup{EM}}} +\mathit{O}(x) $,  and thus $\boldsymbol{E}(x)= -\log x-\gamma_{\mathsmaller{\textup{EM}}} +\mathit{O}(x\log x) $.

 \item For any $x\in  \C\backslash \{0\}$, we have $ \boldsymbol{E}(x)=-\log x-\gamma_{\mathsmaller{\textup{EM}}} +\int_{0}^x \boldsymbol{E}(y)dy  $.

    \item   There exist  $c, C>0$ such that for all $x> 0$
\begin{align*}
  c \,\bigg( \log^+ \frac{1}{x}\,+\,\frac{e^{-x}}{1+x} \bigg)  \,\leq \, E(x)\,\leq \, C\, \bigg( \log^+ \frac{1}{x}\,+\,\frac{e^{-x}}{1+x} \bigg)\,.
\end{align*}
    
\end{enumerate}

\end{lemma}

\subsection{The  Volterra function}\label{SubsecFractExp}

Recall that  $\nu: \C\rightarrow \C$ is defined by $\nu(x)=\int_0^{\infty} \frac{x^r}{\Gamma(r+1) }dr  $, which has a branch cut along the negative real axis.  This function   appeared in Volterra's 1916 article~\cite{Volterra0}, in which it was used to solve certain integral equations involving logarithmic convolution kernels; see~\cite[Ch.\ IX]{Volterra}, wherein the notation $\lambda(x,y)=\int_0^{\infty}\frac{(x-y)^{\eta}}{\Gamma(\eta+1)}d\eta$ is used. As previously mentioned, the function $\nu$  and a family of generalizations of it are referred to as the \textit{Volterra functions}.  These  are not a standard topic within function handbooks. For example, the Volterra functions are absent from  Abramowitz and Stegun's handbook~\cite{Abramowitz}.     An assortment of identities involving the Volterra functions can be found in the third volume of the Bateman Manuscript Project~\cite[pp.\ 217--225]{Bateman}, and there is some discussion of them in the   fractional calculus book by Samko, Kilbas, and Marichev~\cite[Sec.\ 32]{SKM}.  A more thorough and focused effort to   collect results on the  Volterra functions  was undertaken by Apelblat in the monographs~\cite{Apelblat0,Apelblat}, the first of which begins with a historical summary.  Volterra functions have been discussed in the context of related two-dimensional  models before;  for instance, in~\cite{Llewellyn} Llewellyn Smith analyzed the large-time behavior of the solution to the two-dimensional heat equation with a special cylindrical boundary condition, and in~\cite{CCF}  Carlone, Correggi, and Figari studied the time evolution of a quantum particle governed by a two-dimensional  Schr\"odinger Hamiltonian with a many-center point potential whose coupling parameters are allowed to vary with time.


\subsubsection{Basic properties of the Volterra function}
 The integral $\int_0^{\infty} \frac{x^r}{\Gamma(r+1) }dr$ defining the Volterra function  resembles  the Maclaurin series $\sum_{n=0}^{\infty}\frac{x^n}{n!} $ for the natural exponential function $e^x$, with the unit-shifted gamma function replacing the factorial.  In this sense at least, $\nu$ has the form of a fractional analog of the natural exponential function.  For $s>0$ and an appropriate function $ f:[0,\infty)\rightarrow \R $,   recall that  the \textit{$s$-fractional integral} of $f$ is defined by
$$I_s f\,(x)\,:=\,\frac{1}{\Gamma(s)}\,\int_0^x\,(x-y)^{s-1}\,f(y)\,dy\,, $$   
and we set $I_0f:=f$. Using the identity $\int_{0}^{x}(x-y)^{\alpha-1}y^{\beta-1}dy=x^{\alpha+\beta-1}\frac{ \Gamma(\alpha)\Gamma(\beta)  }{\Gamma(\alpha+\beta) }   $ for $\alpha,\beta>0$, it is easy to show that the family of operators $\{I_s\}_{s\in [0,\infty)} $ forms a semigroup, meaning  $I_{s+t}=I_sI_t$. 
The proposition below collects a few basic identities for $\nu$, which can be found in~\cite[\S 18.3]{Bateman} or~\cite[Ch.\ 1]{Apelblat}.  Parts (i) \& (ii) are not  difficult to verify.  A proof of (iii) from (ii) through inverting the Laplace transform is given in Hardy's book~\cite[p.\ 196]{Hardy} on Ramanujan's work, along with a more general result.  We provide a different derivation of (iii)  in Appendix~\ref{AppendixSectNu}.
\begin{proposition} \label{PropEFunctForm}
Let  the function  $\nu$ be defined as before.
\begin{enumerate}[(i)]
    \item We have $ I_s\nu(x) = \nu(x)- \int_{0}^s\frac{ x^r}{\Gamma(r+1) }dr$ for all $s\geq 0$, and in particular that $
\frac{\partial}{\partial s} I_s\nu(x) =- \frac{x^{s}}{\Gamma(s+1) }   $.

   \item The Laplace transform of $\nu(x)$ is finite and equal to  $\bar{\nu}(s)=\frac{1}{s\log s}$ when $s>1$.

\item We have $\nu(x)=e^x-N(x)$ for  Ramanujan's integral $N(x):= \int_{0}^{\infty} \frac{e^{-ax} }{\pi^2+\log^2 a   }\frac{1}{a} da $ for all $x\geq 0$.

\end{enumerate}

\begin{remark}For the Ramanujan identity in (iii) above, note that $N(x) $ is the Laplace transform of the probability density for a centered $\log$-Cauchy distribution with scale parameter $\pi$.  In particular, $0<N(x)\leq 1$.
\end{remark}

\end{proposition}

The representation of the function $\nu$  in (iii) of Proposition~\ref{PropEFunctForm} can  be used to find its asymptotic behavior for  both small and large $x>0$. We only present the first few leading terms,  but a more thorough discussion of the asymptotics for $\nu$ can be found in~\cite{Llewellyn} and of more general Volterra functions in~\cite{Garrappa}.  In the case of large $x$, the asymptotics below are not merely an application of Watson's lemma because the integrand defining $N(x)$  is too singular near zero. Formally, parts (iii) \& (iv) of the lemma below, which concern the first- and second-order derivatives of $\nu$, follow from differentiating the leading terms in the asymptotic expressions in parts (i) \& (ii). We prove the lemma below in Appendix~\ref{AppendixSectNu}.  
\begin{lemma}\label{LemmaEFunAsy} Let the function $\nu$ be defined as above. \vspace{-.2cm}

\begin{enumerate}[(i)]

\item As $x\searrow 0$ we have the asymptotics $  \nu(x)= \frac{1}{\log \frac{1}{x}  }+ \frac{\gamma_{\mathsmaller{\textup{EM}}}}{\log^2 \frac{1}{x}  }+\frac{\gamma_{\mathsmaller{\textup{EM}}}^2-\frac{ \pi^2 }{6} }{\log^3 \frac{1}{x}  }+ \mathit{O}\big(\frac{1}{\log^4 \frac{1}{x}  } \big) $. 

\item  As $x\nearrow \infty$ we have the asymptotics
$  \nu(x)= e^{x}-\frac{1}{\log x  }+ \frac{\gamma_{\mathsmaller{\textup{EM}}}}{\log^2 x   }-\frac{\gamma_{\mathsmaller{\textup{EM}}}^2-\frac{ \pi^2 }{6} }{\log^3 x  }+ \mathit{O}\big(\frac{1}{\log^4 x } \big)  $.

\item As $x\searrow 0$ we have $ \nu'(x)\sim \frac{1}{x\log^2\frac{1}{x}  } $ and $ \nu''(x)\sim  -\frac{1}{x^2\log^2\frac{1}{x}  } $.

\item  As  $x\nearrow \infty$ we have  $ \nu'(x)\sim e^x\sim \nu''(x)$.

\end{enumerate}
\end{lemma}

\begin{remark}\label{Remark1st}
From (iii) we see that $\nu'$ is not an increasing function. It is, however,  convex.   
\end{remark}

\begin{remark}\label{Remark2nd} It is useful to observe that $\frac{d}{dx}\big[\frac{x\nu'(x)  }{\nu(x)  }   \big] $ is positive for all $x>0$. This can be inferred from   $(\lambda\frac{d}{d\lambda})^2\log \nu(\lambda)  $ being the variance of a nonnegative random variable whose distribution has probability density $\mathcal{D}_{\lambda}(t)= \frac{ 1 }{ \nu(\lambda)  }\frac{ \lambda^{t}  }{\Gamma(t+1) } $; see the brief discussion  of distributions of this form in Appendix~\ref{AppendixSectPoisson}.
\end{remark}

\subsubsection{Some convolution formulas involving the Volterra function}

We will present three lemmas concerning convolutions with the function $f_{\lambda}(x)=\lambda \nu'(\lambda x)$.  Integration by parts is not helpful as a first step towards deriving these identities because  the functions involved are singular at the boundary points.  Note that the integrals appearing in the next lemma are well-defined as a consequence of the small $x$ asymptotics for $\nu'(s)$ and $\nu''(x)$ in (iii) of  Lemma~\ref{LemmaEFunAsy}.  Part (i) of the lemma  is equivalent to~\cite[Lem.\ 32.1]{SKM} in the case of their parameters $\alpha, h$ taking the values $\alpha=1$ and $h=\log\lambda$, and part (ii) has essentially the same proof.
\begin{lemma}\label{LemmaELog} The identities below hold for all $x,\lambda>0$.
\begin{enumerate}[(i)]
\item   $\displaystyle \lambda\, \int_0^{x}\,\log(x-y) \,\nu'( y\lambda)\,dy\,=\, -1\, -\,\big(\log\lambda +\gamma_{\mathsmaller{\textup{EM}}}\big)\, \nu( x\lambda ) $
\item  $\displaystyle \lambda\, \int_0^{1} \,\log(1-y)\,\nu''( y \lambda)\,dy\,=\,  -\big(\log\lambda+\gamma_{\mathsmaller{\textup{EM}}}\big) \,\nu'(\lambda  ) $
\end{enumerate}

\end{lemma}

The next two technical lemmas involve the functions $E$, $\widetilde{E}$, $\mathbf{E}$  defined  in Section~\ref{SubsecGammaFun}.  We apply the next result in the proofs of Lemmas~\ref{LemmaEUpsilonBold} \&~\ref{LemmaUpsilonApprox} below. 

\begin{lemma}\label{LemmaUpsilon} The identity below holds for all $x,\lambda>0$ and $z\in \C\backslash \{0\}$.
\begin{align*}
\lambda\, \int_0^{x} \,E \Big(\frac{z}{x-y}\Big)\,\nu'( y\lambda)\,dy\,=\,& -1\, -\,\big(\log ( z \lambda)\,+\,2\,\gamma_{\mathsmaller{\textup{EM}}}\big) \,\nu(x \lambda )\,+\,\lambda \,\int_0^{x} \,\widetilde{E} \Big(\frac{z}{x-y}\Big)\, \nu'( y \lambda)\,dy 
\end{align*}

\end{lemma}
\begin{proof}  Since $E(z)=-\log z-\gamma_{\mathsmaller{\textup{EM}}}+\widetilde{E}(z)  $ by (i) of Lemma~\ref{LemEM}, we can write the left side  above as 
\begin{align*} \lambda \,\int_0^{x}\,\log(x-y)\, \nu'( y\lambda)\,dy \, - \,\lambda\,\log z \int_0^{x}  \,\nu'( y\lambda)\,dy 
\, - \,\lambda\,\gamma_{\mathsmaller{\textup{EM}}} \, \int_0^{x}\,\nu'( y\lambda)\,dy  \,+\,\lambda \,\int_0^{x}\,\widetilde{E} \Big(\frac{z}{x-y}\Big)\, \nu'( y\lambda) \,dy \,.
\end{align*}
 Applying  (i) of Lemma~\ref{LemmaELog} to the first term   and using that $\lambda\int_0^{x} \nu'( y \lambda)dy=\nu(x \lambda)$, we arrive at the result. 
\end{proof}

We will  apply the following lemma  in the special cases $z=1$ and $z=-1$ in the proof of Proposition~\ref{PropRealUnnormEigen}.
\begin{lemma}\label{LemmaEUpsilonBold} The convolution identity below  holds for all $x,\lambda>0$ and $z\in \C\backslash \{0\}  $.
$$ \lambda\,\int_0^{x} \,\boldsymbol{E}\big((x-y)z\big) \, \nu'( y \lambda)\,dy\,=\,e^x\,+\,\log( z\lambda)\,\nu( x \lambda)\,+\,z\,\log( z\lambda) \,\int_0^x\, e^{(x-a)z}\,\nu(a \lambda  )\,da $$
In particular, when $z\lambda =1$, the right side above reduces to $e^x$.
\end{lemma}

\begin{proof} By Lemma~\ref{LemEM}, we have the identity $\boldsymbol{E}(z)=-\log z-\gamma_{\mathsmaller{\textup{EM}}}+\int_{0}^{z}\boldsymbol{E}(a)da  $ for $z\in \C\backslash \{0\}$, and hence  
\begin{align*}\lambda \,\int_0^{x} \,\boldsymbol{E}\big((x-y)z\big)\,\nu'( y\lambda)\,dy \,=\,& \, -\lambda\,\int_0^{x}\,\log(x-y)\, \nu'( y\lambda)\,dy \,-\,\lambda \,\log z \,\int_0^{x} \,\nu'( y\lambda)\,dy \nonumber \\  &\, - \,\lambda \, \gamma_{\mathsmaller{\textup{EM}}} \,\int_0^{x}\,\nu'( y\lambda)\,dy\,+\,\lambda\, \int_0^{x}\, \nu'(y\lambda ) \,\int_{0}^{(x-y)z}\,\boldsymbol{E}(a)\,da\,dy\,.
\end{align*}
Applying (i) of Lemma~\ref{LemmaELog} to the first term on the right side above and integrating by parts in the fourth term, we can express the right side above as
\begin{align} \label{EGammaConv}
 1\, +&\,\log( z\lambda)\,\nu( x \lambda)   \,+\, \bigg[\, \nu( y\lambda )\, \int_{0}^{z(x-y)}\,\boldsymbol{E}(a)\,da \,\bigg]_{y=0}^{y=x} \,+\,z\int_0^{x}  \,\boldsymbol{E}\big((x-y)z\big)\,\nu( y\lambda)\, dy \nonumber \\    \,=\,&\, 1\, +\,\log( z\lambda)\, \nu( x \lambda ) 
  \,+\,z\,\int_0^{x}\,  \boldsymbol{E}(yz)\,\nu\big( (x-y)\lambda\big)\, dy\,,
\end{align}
where the lower boundary term  vanishes since $\nu(0)=0$, and we have changed integration variable $y\rightarrow x-y$.  Define the function $F_{\lambda,z}:(0,\infty)\rightarrow  \C$   by  
$$F_{\lambda,z}(x)\,:=\,\lambda \,\int_0^{x}\,\boldsymbol{E}\big((x-y)z\big)\, \nu'(y\lambda )\,dy\,=\,\lambda\,\int_0^{x} \,\boldsymbol{E}(yz)  \,\nu'\big( (x-y)\lambda\big)\,dy\,. $$
Since $F_{\lambda,z}(x)$ is equal to~(\ref{EGammaConv}), we have  $\lim_{x\searrow 0}F_{\lambda,z}(x)=1$ and that $F_{\lambda,z}(x)$ satisfies  the differential equation
\begin{align*}
F_{\lambda,z}'(x)\,=\,&\,\frac{d}{dx}\bigg[\, 1\, +\,\log( z\lambda)\, \nu( x \lambda ) 
  \,+\,z\,\int_0^{x} \, \boldsymbol{E}( yz)\,\nu\big((x- y)\lambda\big)\, dy \, \bigg]\nonumber \\
  \,=\,&\,\lambda \,\log( z\lambda)\, \nu'( x\lambda ) 
  \,+\,z\,\boldsymbol{E}(xz)\,\underbrace{\nu(0)}_{=\,0} \,+\,z\,\lambda \,\int_0^{x}\, \boldsymbol{E}( yz) \,\nu'\big( (x- y)\lambda\big)\, dy \nonumber  \\  \,=\,&\,\lambda \,\log( z \lambda)\, \nu'( x \lambda) 
  \,+\,z\,F_{\lambda,z}(x) \,.\nonumber 
\end{align*}
It follows that the function $A_{\lambda,z}(x):= e^{-xz }F_{\lambda,z}(x)$ has derivative $A_{\lambda,z}'(x)=\lambda \log( z\lambda)\, e^{-xz }\nu'( x\lambda )$, and since $A_{\lambda,z}(0)=F_{\lambda,z}(0)=1$, we can  write
\begin{align*}
A_{\lambda,z}(x)\,=\,1\,+\,\lambda \,\log( z\lambda)\int_0^x \,e^{- az}\,\nu'( a \lambda)\,da  \,=\,1\,+\,\log(z \lambda )\,e^{-x z}\,\nu( x \lambda )\,+\,z\,\log( z \lambda)\,\int_0^x \,e^{- za}\,\nu( a \lambda )\,da\,,
\end{align*}
in which we have integrated by parts.
Multiplying through by $e^{xz}$ completes the proof.
\end{proof}

\subsection{The function $   \boldsymbol{H_{ T}^{\lambda}(x)} $}\label{SubsectHFunct}

For $T,\lambda >0$ recall that we define $H_{T}^{\lambda}: \R^2\rightarrow [0,\infty]$  as in~(\ref{DefH}).   We can alternatively express $H_{T}^{\lambda}(x)$ in the form following the second equality below
\begin{align} \label{Kfun} 
 H_{ T}^{\lambda}(x)\,=\,  \int_{0}^T \,\frac{  e^{-\frac{|x|^2}{2r}}  }{ r  } \,\nu\big((T- r)\lambda\big)\,dr  \,=\,\lambda \,\int_{0}^T \,  E\bigg(\frac{|x|^2}{2(T-r) }\bigg)\,\nu'\big(r\lambda \big)\,dr \,,
\end{align}
in which we have used that $E\big(\frac{A}{T}\big)=\int_0^T \frac{ e^{-\frac{A}{r}}  }{ r  } dr$ after integrating by parts and changing integration variable.  Note that  the   scale invariance $H_{ T}^{\lambda}(x)= H_{\frac{T}{a} }^{a\lambda}\big(\frac{x}{\sqrt{a} }\big)$ holds for all $a>0$  since
\begin{align}\label{H2N}
    H_{ T}^{\lambda}(x) \,=\,N_{\alpha,\beta} \hspace{.7cm}\text{for}\hspace{.3cm} N_{\alpha,\beta} \, := \, \int_{0}^1\, \frac{e^{-\frac{\alpha}{r} }}{r}\, \nu \big( (1-r)\beta \big)\, dr \hspace{.3cm}\text{with} \hspace{.3cm} \alpha\,=\,\frac{|x|^2}{2T}\,, \hspace{.2cm}  \beta=T\lambda\,.
\end{align}
 The value $H_T^{\lambda}(x)$ is increasing in the parameters  $T$, $\lambda$ but decreasing with $|x|$.

\subsubsection{Near-origin approximations for $ \boldsymbol{H_{ T}^{\lambda}(x)} $ and its derivatives}  \label{SubsubHApprox}

In Proposition~\ref{PropK} below, we present approximations for $H_{ T}^{\lambda}(x)$ and its partial derivatives when $|x|\ll 1$ and the parameters $T,\lambda>0$ are not too big or too small.  For the proof, we will use the following lemma in the regime where $ 0 < \lambda  \ll 1 $ and $x\gg 1$.
\begin{lemma}\label{LemmaUpsilonApprox} For any $L>1$ there exists a  $C_L>0$ such that for all $x,\lambda >0$ with  $\frac{1}{L}<\lambda x< L$
\begin{align}\label{Tris}
\bigg|\,\lambda  \,\int_0^{x}\,E \Big(\frac{1}{x-y}\Big)\,\nu'( y \lambda)\,dy\,+\,1\, +\,\big(\log\lambda \,+\,2\,\gamma_{\mathsmaller{\textup{EM}}}\big) \,\nu( x \lambda)\,\bigg|\,\leq \,C_L\,\frac{1+\log^+ x}{x} \,.
\end{align}
\end{lemma}

\begin{proof} By Lemma~\ref{LemmaUpsilon} with $z=1$, the expression that we wish to  bound is equal to $\lambda \int_0^{x} \widetilde{E} \big(\frac{1}{x- y} \big) \nu'( y\lambda)dy$.  Since $\widetilde{E}$ is an increasing function on $[0,\infty)$ such that  $\widetilde{E} (a)\sim a$ for $0<a\ll 1$  and   $\widetilde{E} (a)\sim \log a$ for $a\gg 1$,  there exists a constant $c>0$ such that for all $a>0$
$$  \widetilde{E} (a)\,\leq  \,  c \min \big(a, 1+\log^+ a\big) \,. $$
In particular, we can choose $c$ large enough so that $\int_0^a \widetilde{E}\big(\frac{1}{y}\big)dy\leq c(1+\log^+ a   ) $.  We  now obtain our bound as follows:
\begin{align*}
\lambda \,\int_{0}^{x }\,\widetilde{E} \Big(\frac{1}{x- y} \Big)\,  \nu'( y\lambda)\,dy \,= \,&\,\lambda \,\int_{0}^{\frac{x}{2} }\,\widetilde{E} \Big(\frac{1}{x- y} \Big) \, \nu'(y\lambda )\,dy\,+\, \lambda\, \int_{\frac{x}{2} }^{x }\,\widetilde{E} \Big(\frac{1}{x- y} \Big)\,  \nu'( y\lambda)\,dy  \\ \,\leq\,&\, \lambda \,\int^{\frac{x}{2} }_{0} \,\frac{ c}{ x- y } \, \nu'( y\lambda) \,dy \,+\,\lambda\, \bigg(\sup_{\frac{1}{2L}\leq a\leq L }\,\nu'(a)\bigg)\,\int_{\frac{x}{2} }^{x}\,\widetilde{E} \Big(\frac{1}{x- y} \Big)\,dy \\ \,\leq \,& \,\lambda \,\frac{ 2c}{ x  }\,\int^{\frac{x}{2} }_{0}  \,\nu'( y\lambda)\, dy \,+\, c\, \lambda \,\big(1+\log^+ x\big)  \sup_{\frac{1}{2L}\leq a\leq L }\,\nu'(a) \\  \,\leq\,&\,\frac{2c }{x} \,\nu(L)\,+\, \frac{cL}{x}\,\big(1+\log^+ x\big) \sup_{\frac{1}{2L}\leq a\leq L }\,\nu'(a) \,,
\end{align*}
where we have used that $\lambda\leq \frac{L}{x}$ and  $\lambda  \int^{x }_{0} \nu'( y\lambda) dy =\nu(x\lambda )\leq \nu(L)$.  The above implies our desired inequality.
\end{proof}

Recall that $\bar{H}_{T}^{\lambda}: [0,\infty)\rightarrow [0,\infty]$ is related to the radially symmetric function $H_{T}^{\lambda}$ through  $\bar{H}_{T}^{\lambda}(|x|)= H_{T}^{\lambda}(x)$, and so  $\nabla_x H_{T}^{\lambda}(x)=\frac{x}{|x|}\bar{H}_{T}^{\lambda}(a)\big|_{a=|x|}  $ for $x\neq 0$.  Part (i) of the proposition below implies that for fixed $\lambda, T>0$ the following asymptotics  holds   for small  $a> 0$ 
$$1\,+\,\bar{H}_{T}^{\lambda}(a)\, \stackrel{a\rightarrow 0}{=}\,-2\,\nu( T\lambda)\,\Big( \log a\,+\,  \frac{1}{2}\log \frac{\lambda }{2 }   \,+\,\gamma_{\mathsmaller{\textup{EM}}}\Big)\,+\,\mathit{o}(1)\,. $$
The terms in the approximations for the derivatives of $\bar{H}_{T}^{\lambda}(a)$ in parts (ii)--(v) can be heuristically derived  by differentiating the above.
\begin{proposition}\label{PropK}  For any $L\in (1,\infty) $, there exists a $C_{L}>0$ such that the inequalities below hold for all $T,\lambda \in [\frac{1}{L},  L] $  and $a>0$.
\begin{enumerate}[(i)]
\item  $ \displaystyle \bigg|1+\bar{H}_{T}^{\lambda}(a)\,+\, 2\,\nu( T\lambda)\,\Big(\log a\,+\,\frac{1}{2}\log \frac{\lambda }{2 }   \,+\,\gamma_{\mathsmaller{\textup{EM}}}\Big)\bigg| \,\leq \, C_{L}\, a^2\, \Big(\, 1\,+\,\log^+ \frac{1}{a}  \,\Big) $

\item $ \displaystyle \bigg|\frac{\partial}{\partial T}\bar{H}_{T}^{\lambda}(a)\,+\, 2\,\lambda \,\nu'(  T\lambda) \,\Big(\log a \,+\,\frac{1}{2}\log \frac{\lambda }{2 }  \, +\,\gamma_{\mathsmaller{\textup{EM}}}\Big)\bigg|\,\leq \,  C_{L}\,a^2 \,  \Big(\, 1\,+\,\log^+ \frac{1}{a}  \,\Big) $

\item $ \displaystyle \bigg|\frac{\partial}{\partial a}\bar{H}_{T}^{\lambda}(a)\,+\,\nu(T\lambda )\, \frac{2}{a}\bigg|\,\leq \,  C_{L}\,a \,\Big(\, 1\,+\,\log^+ \frac{1}{a}  \,\Big) $

\item  $ \displaystyle \bigg|\lambda \,\frac{\partial}{\partial \lambda}\bar{H}_{T}^{\lambda}(a)\,+\,2\, T\,\lambda \,\nu'( T\lambda)\,\Big(\log a\,+\,\frac{1}{2}\log \frac{\lambda }{2 }   \,+\,\gamma_{\mathsmaller{\textup{EM}}}\Big)\,+\,\nu(T\lambda )   \bigg|\,\leq \,  C_{L}\,a^2 \, \Big(\, 1\,+\,\log^+ \frac{1}{a}\,  \Big) $

\item $ \displaystyle \bigg|\lambda \,\frac{\partial^2}{\partial a\partial \lambda }\bar{H}_{T}^{\lambda}(a)\,+\, \nu'\big( T\lambda )\,\frac{2T\lambda}{a}\bigg|\,\leq \,  C_{L}\,a  \,\Big(\, 1+\log^+ \frac{1}{a}  \,\Big)   $

\end{enumerate}

\end{proposition}

\begin{proof} Part (i):  With~(\ref{Kfun}) and a change of integration variable to $t= \frac{2}{a^2}r $,  we have that
\begin{align*}
\bar{H}^{\lambda}_T(a) \,=\, \frac{\lambda \, a^2}{ 2
}\,\int_0^{\frac{2T }{ a^2 } }\,E\bigg(\frac{1}{\frac{2T}{a^2} -t }  \bigg) \,\nu'\bigg(t\frac{\lambda\,  a^2}{ 2
}\bigg)\,dt \,.
\end{align*}
Since $ T,\lambda  \in \big[\frac{1}{L}, L\big]$, the variables  $\mathbf{x}=\frac{2T}{a^2}$ and $\boldsymbol{\lambda}= \frac{ \lambda a^2  }{ 2 }$ satisfy  $\mathbf{x}\boldsymbol{\lambda}  \in \big[\frac{1}{L^2}, L^2\big]$. Thus, by Lemma~\ref{LemmaUpsilonApprox} there exists a $\mathbf{C}_L>0$ such that 
\begin{align*}
\bigg| 1+\bar{H}^{\lambda}_T(a)\, + \,2\,\nu( T\lambda  )\,\Big(\log a\,+\,  \frac{1}{2}\log \frac{ \lambda  }{ 2  }  \,+\, \gamma_{\mathsmaller{\textup{EM}}}  \Big)   \bigg| \,\leq \,   \mathbf{C}_L \,\frac{1+\log^+ \frac{2T}{a^2}  }{ \frac{2T}{a^2}   } \,.
\end{align*}
The above is bounded by a multiple of  $a^2 \big( 1+\log^+ \frac{1}{a}  \big)$ since $T\in  \big[\frac{1}{L},  L\big]$. \vspace{.2cm}

\noindent Part (ii):  Since $\bar{H}_{T}^{\lambda}(a)= \lambda  \int_0^T\int_0^s \frac{ e^{-\frac{a^2}{2r }} }{r 
 }\nu'\big((s-r)\lambda\big)drds$, the fundamental theorem of calculus gives us that
 \begin{align}\label{PartialInT}
     \frac{\partial}{\partial T}\bar{H}_{T}^{\lambda}(a) \,=\, \lambda \,\int_0^{T} \,\frac{e^{-\frac{a^2}{2r}}}{r} \,\nu'\big( (T-r)\lambda\big)\,dr\,=\, \lambda \,\int_0^{1} \,\frac{e^{-\frac{\alpha}{s}}}{s}\, \nu'\big( (1-s)\beta\big)\,ds\,,
 \end{align}
 for $\alpha:=\frac{ a^2 }{2T }$ and $\beta:=T\lambda $.   Using~(\ref{PartialInT}) and that $ E(\alpha)= \int_0^{1} \frac{  e^{-\frac{\alpha}{s}} }{  s } ds$, we arrive at the first equality below.
 \begin{align*} \frac{\partial }{\partial T}\bar{H}_{T}^{\lambda}(a)\,-\, \lambda \,\nu'(\beta ) \,E (\alpha)  \,=\,&\, -\lambda \,\int_0^{1}\, \frac{e^{-\frac{\alpha}{s}}}{s} \,\Big(\, \nu'(\beta) \,- \,\nu'\big((1-s)\beta\big)   \,\Big)\, ds \nonumber  \\
 \,=\,&\, -\lambda \,\int_0^{1} \,\frac{e^{-\frac{\alpha}{s}}}{s} \,\int_{1-s }^{1  } \,\beta\,\nu'' ( v \beta)\,dv\,  ds  \nonumber \\
 \,=\,&\, -\lambda \,\int_0^{1} \,\beta \,\nu'' ( v\beta)\, \int_{1-v }^{1  }\, \frac{e^{-\frac{\alpha}{s}}}{s} \,  ds \, dv \nonumber 
 \\
 \,=\,&\, \underbrace{\lambda \, \int_0^{1}\,\beta \,\nu'' ( v\beta)\, \log (1-v)\, dv}_{ \textup{(I)}}\,+\,\underbrace{\lambda\, \int_0^{1} \,\beta  \,\nu'' ( v \beta)\, \int_{1-v }^{1  } \,\frac{1-e^{-\frac{\alpha}{s}}  }{s} \,  ds  \, dv}_{\text{(II)}} 
 \end{align*}
The fourth equality above merely uses that 
$$ -\int_{1-v }^{1  } \,\frac{e^{-\frac{\alpha}{s}}}{s} \,  ds \,=\,- \int_{1-v }^{1  }\, \frac{1 }{s}\,   ds\,+\, \int_{1-v }^{1  } \,\frac{1-e^{-\frac{\alpha}{s}} }{s} \,  ds \,=\, \log (1-v) \,+\, \int_{1-v }^{1  }\, \frac{1-e^{-\frac{\alpha}{s}} }{s} \,  ds\,. $$
The term (I) is equal to $ -\lambda \big(\log \beta  +\gamma_{\mathsmaller{\textup{EM}}}\big)\nu'(\beta)$ by part (ii) of Lemma~\ref{LemmaELog}.  Recall that the exponential integral function $E$ is decreasing and satisfies   $E(\alpha )=-\log \alpha-\gamma_{\mathsmaller{\textup{EM}}}+\mathit{O}(\alpha)$ as $\alpha\rightarrow 0$ by (ii) of Lemma~\ref{LemEM}.  Thus, there is a constant $c_L>0$ such that for all $\alpha,\beta>0$ with $ \frac{1}{L^2}\leq  \beta\leq L^2$
\begin{align*}
\bigg| \frac{\partial}{\partial T}\bar{H}_{T}^{\lambda}(a)\,+\,\big( \log (\alpha \beta)+2\,\gamma_{\mathsmaller{\textup{EM}}}    \big)\,\lambda \,\nu'(\beta)   \bigg|\,\leq  \, \big|\textup{(II)}\big|\,+\, \lambda \,c_L\, \alpha \,.
\end{align*}
It suffices to show that (II) is bounded in absolute value by a multiple of $\alpha \big(1+\log^+ \frac{1}{\alpha}\big)$.  Since $1-e^{-t}\leq \min(t,1)  $ for $t\geq 0$, the inner integral in the expression (II) can be bounded as follows:
\begin{align*}
\int_{1-v }^{1  } \,\frac{1-e^{-\frac{\alpha}{s}}  }{s}  \, ds\,\leq \,  \int_{1-v }^{1  }\,\frac{ \textup{min}(\frac{\alpha}{s},1)  }{s} \,  ds   \,\leq \,  \frac{ \alpha  \,v}{1-v }\, 1_{\alpha\leq 1-v  } \,+\,\log \frac{1}{1-v}\,1_{\alpha >1-v }  \,.
\end{align*}
Thus we have 
\begin{align*}
 \big|\textup{(II)}\big|\,\leq \,  \lambda \,\alpha \, \int_0^{\textup{max}(1-\alpha,0)} \,\beta \, v\, \frac{\big|\nu'' ( v \beta)\big|}{1-v} \,   dv \,+\, \lambda \,\int^1_{\textup{max}(1-\alpha,0)} \,\beta \,\big|\nu'' ( v \beta)\big| \,\log\frac{1}{1-v}\,dv\,.
\end{align*}
The above is bounded by a constant multiple of $\alpha \big(1+\log^+ \frac{1}{\alpha}\big)$ for all $\alpha,\beta >0$ with $ \beta\in \big[\frac{1}{L},L\big]  $ because $\nu''$ is a continuous function on $(0,\infty)$ having the asymptotic form $\nu''(a)\sim - ( a\log a )^{-2}$ as $a\searrow 0$. \vspace{.2cm}

\noindent Part (iii): Since  $
   \bar{H}_{ T}^{\lambda}(a)=\int_0^T  \frac{e^{-\frac{a^2}{2r} }}{r}\nu\big( (T-r) \lambda\big)dr$, we can write $ \frac{\partial}{\partial a} \bar{H}^{\lambda}_T(a) =  -a\int_0^T  \frac{e^{-\frac{a^2}{2r} }}{r^2}\nu\big( (T-r) \lambda\big)dr$. Therefore, we have the following:
\begin{align}\label{Nosfer}
 \frac{\partial}{\partial a} \bar{H}^{\lambda}_T(a)\,+\,\nu(T\lambda)\,\frac{2}{a}\,=\,a\,\int_0^T\,  \frac{e^{-\frac{a^2}{2r} }}{r^2}\,\Big(\,\nu( T\lambda)-\nu\big( (T-r)\lambda \big) \,\Big)\, dr \,+\,\nu(T\lambda )\,\frac{2}{a}\,\Big(1- e^{-\frac{a^2}{2T} }  \Big)   \,,
\end{align}
where we have used that $\int_0^T  \frac{e^{-\frac{a^2}{2r} }}{r^2}dr = \frac{2}{a^2}e^{-\frac{a^2}{2T} } $.  The second term on the right side of~(\ref{Nosfer}) is negligible because it is bounded by $\nu(T\lambda )\frac{a}{T}  $, as  $1-e^{-x} \leq x$ for $x\geq 0$.  The first term on the right side of~(\ref{Nosfer}) is bounded by 
\begin{align*}
\mathbf{c}_L \,a\,\int_0^T \, \frac{e^{-\frac{a^2}{2r} }}{r}\, dr  \,=\, \mathbf{c}_L \,a\,E\bigg(\frac{a^2 }{ 2T }  \bigg)
 \hspace{.7cm}\text{for}\hspace{.7cm}\mathbf{c}_L\, :=\,\sup_{\substack{ \frac{1}{L} \leq  \lambda , T\leq L  \\  r\in (0,T)   }  } \frac{\nu(T\lambda )-\nu\big( (T-r)\lambda\big) }{r}\,.
\end{align*}
The result then follows since $E(x)\sim \log \frac{1}{x}$ for small $x>0$.\vspace{.2cm}

\noindent Part (iv):  Recall from~(\ref{H2N}) that $ \bar{H}^{\lambda}_T(a)=N_{\alpha,\beta}  $ for $\alpha=\frac{a^2}{2T}$ and $\beta=T\lambda $. The chain rule gives us that
$$  \frac{\partial}{\partial T} \bar{H}^{\lambda}_T(a)\,=\,-\frac{ a^2 }{ 2T^2 } \,\frac{\partial}{\partial \alpha }N_{\alpha, \beta}\,+\, \lambda \,\frac{\partial}{\partial \beta }N_{\alpha, \beta}\,=\,-\frac{a}{2T} \,\frac{\partial}{\partial a} \bar{H}^{\lambda}_T(a) \,+\,\frac{\lambda}{T} \,\frac{\partial}{\partial \lambda } \bar{H}^{\lambda}_T(a) \,.$$
Multiplying through by $T$ and rearranging yields that $\lambda \frac{\partial}{\partial \lambda } \bar{H}^{\lambda}_T(a)\,=\,T\frac{\partial}{\partial T} \bar{H}^{\lambda}_T(a)+ \frac{a}{2}  \frac{\partial}{\partial a} \bar{H}^{\lambda}_T(a) $.  We can then apply parts (ii) and (iii) to get the desired approximation for $\lambda \frac{\partial}{\partial \lambda} \bar{H}^{\lambda}_T(a)$. \vspace{.2cm}

\noindent Part (v): Since $ \lambda \frac{\partial^2}{\partial a \partial \lambda } \bar{H}^{\lambda}_T(a) =  -a\int_0^T  \frac{e^{-\frac{a^2}{2r} }}{r^2} \mathbf{f}\big( (T-r)\lambda \big)dr$ for the function $\mathbf{f}(x)=x\nu'(x)$, we  can apply the same argument as in part (iii) to approximate  $ \lambda \frac{\partial^2}{\partial a \partial \lambda } \bar{H}^{\lambda}_T(a)$.
\end{proof}

\subsubsection{Upper and lower  bounds  for $   \boldsymbol{H_{ T}^{\lambda}(x)} $ }

We need bounds for $H_{ T}^{\lambda}(x)$ in a larger range of the variables $x $, $T$, $\lambda$ than is allowed by Proposition~\ref{PropK}.

\begin{lemma} \label{LemmaUpsilonUpDown}  For any $L>0$ there exist  $c_L,C_L>0$ such that the inequalities below hold for all $x\in \R^2$ and $T,\lambda>0 $ with $T\lambda\leq  L  $.

\begin{enumerate}[(i)] \item
$ \displaystyle \frac{c_L  }{ 1+\log^+\frac{1}{T\lambda}  }\,\Bigg(\,    \log^+ \frac{2T}{|x|^2} \,+\,\frac{e^{ -\frac{|x|^2}{T}  }}{ 1+ \frac{|x|^2}{2T} } \,\Bigg)  \,\leq \,H_{ T}^{\lambda}(x)\,\leq  \, \frac{C_L  }{ 1+\log^+\frac{1}{T\lambda}  }\,\Bigg(\,    \log^+ \frac{2T}{|x|^2} \,+\,\frac{e^{ -\frac{|x|^2}{2T}  }}{ 1+ \frac{|x|^2}{2T} }\, \Bigg)      $

\item  $ \displaystyle \big| \nabla_x \,H_{ T}^{\lambda}(x)\big|\,\leq \, C_L\, \frac{1 }{ |x| } \,\frac{e^{- \frac{|x|^2}{2T} }  }{ 1+\log^+\frac{1}{ T\lambda} }  $

\end{enumerate}

\end{lemma}

\begin{proof}  Part (i): Recall that $H_{ T}^{\lambda}(x)=N_{\alpha,\beta}$ for $N_{\alpha,\beta}:=\int_{0}^1 \frac{e^{-\frac{\alpha}{r} }}{r} \nu \big( (1-r) \beta\big) dr$ and $\alpha=\frac{ |x|^2 }{ 2T }$, $\beta=T\lambda$. The second and fourth inequalities below hold since $\nu$ is an increasing function and $E(\alpha)=\int_{0}^1 \frac{e^{-\frac{\alpha}{r} }}{r} dr$ is decreasing.
\begin{align}\label{NAlphaBeta}
 \, c_L\, \frac{\log^+ \frac{1}{\alpha} \,+\,\frac{e^{ -2\alpha  }}{ 1+ \alpha }  }{ 1+\log^+\frac{1}{\beta}  }\,\leq \,E(2\alpha   )\,\nu\Big(\frac{\beta}{2}\Big)\, & \,\leq \,  \int_{0}^{\frac{1}{2}}\,  \frac{e^{-\frac{\alpha}{r} }}{r}\,\nu\big((1-r)\beta\big) \,dr\nonumber \\  & \, \leq \, N_{\alpha,\beta}  \,\leq \, E (\alpha)\,\nu(\beta) \, \leq  \, C_L\, \frac{\log^+ \frac{1}{\alpha} \,+\,\frac{e^{ -\alpha  }}{ 1+ \alpha }  }{ 1+\log^+\frac{1}{\beta}  } 
  \end{align}
 The first and last inequalities above hold for some $c_L,C_L>0$ and all $\alpha,\beta>0$ with $ \beta<L$ as a consequence of (iv) of Lemma~\ref{LemEM} and that $\nu(\beta)\sim (\log\frac{1}{\beta})^{-1}$ for small $\beta>0$ by (i) of Lemma~\ref{LemmaEFunAsy}.  Plugging in  $\alpha=\frac{ |x|^2 }{ 2T }$ and $\beta=T\lambda$ yields the claim. \vspace{.3cm}

\noindent Part (ii): Using that $\big| \nabla_x H_{ T}^{\lambda}(x)\big|=-\frac{\partial}{\partial a}\bar{H}_{ T}^{\lambda}(a)$ for $a=|x|$ and that $\frac{\partial}{\partial a}\bar{H}_{ T}^{\lambda}(a)=\frac{a}{T}\frac{\partial}{\partial \alpha}N_{\alpha, \beta} $, we can compute 
\begin{align*}
\big| \nabla_x \,H_{ T}^{\lambda}(x)\big|\,=\, \frac{a}{T}\,\int_0^1 \, \frac{e^{-\frac{\alpha}{r} }}{r^2}\,\nu\big((1-r)  \beta \big)\,dr 
\, \leq  \,&\,\frac{a}{T}\,\nu(\beta )\,\int_0^1\,  \frac{e^{-\frac{\alpha}{r} }}{r^2}\,dr \nonumber  
\,=\,\frac{a }{\alpha T }\, e^{- \alpha   }\, \nu(\beta)  \,, 
\end{align*}
and we again use that $\nu(\beta)\sim (\log\frac{1}{\beta})^{-1}$ for small $\beta>0$.
\end{proof}

Lemma~\ref{LemmaUpsilonUpDown} has the following corollary for the functions
  $\mathfrak{p}_T^{\lambda}(x):=\frac{1}{1+H^{\lambda}_T(x)  }$  and $\mathfrak{q}_T^{\lambda}(x):=1-\mathfrak{p}_T^{\lambda}(x)$.
\begin{corollary}\label{CorollaryThePandQ}  For any $L>0$ there exists a  $C_L>0$ such that the inequalities below hold for all $x\in \R^2$ and $T,\lambda>0$ with $T\lambda \leq L$.
\begin{enumerate}[(i)]
\item  $\displaystyle \mathfrak{p}_T^{\lambda}(x)\,\leq \, 1_{ |x|^2\geq 2T^2\lambda  }\,+\, C_L\,\frac{1+\log^+\frac{1}{T\lambda }  }{1+ \log^{+} \frac{2T}{|x|^2}    } \,  1_{ |x|^2 < 2T^2\lambda  }$

    \item  $\displaystyle \mathfrak{q}_T^{\lambda}(x)\,\leq \, C_L\,\frac{ \frac{T}{|x|^2}\,e^{-\frac{|x|^2}{2T}}  }{ 1+\log^+\frac{1}{T\lambda}   } \, 1_{ |x|^2\geq 2T^2\lambda  }\,+\, \frac{1+\log^+\frac{1}{T\lambda }  }{1+ \log^{+} \frac{2T}{|x|^2}    } \,  1_{ |x|^2 < 2T^2\lambda  }$

\end{enumerate}

\end{corollary}

\subsection{The drift function $\boldsymbol{b_{t}^{\lambda}(x)}$  } \label{SubsecDriftFunc}
Recall that the drift function  is defined by $ b_{T}^{\lambda}(x)=\frac{ \nabla_x H_T^{\lambda}(x)  }{  1+ H^{\lambda}_T(x) } $ for $x\neq 0$.
\begin{lemma} \label{LemDrift}
For any $L>0$ there exists a $C_L>0$ such that for all $x\in \R^2$ and $T,\lambda >0$ with $T\lambda \leq L$
\begin{align*} 
\big| b_{T}^{\lambda}(x)\big| \, \leq  \,\frac{C_L }{ |x| } \,\Bigg(\, \frac{e^{- \frac{|x|^2}{2T} } }{ 1 +\log^+\frac{1}{T\lambda}}\,  1_{ |x|^2 \geq 2T^2\lambda  }  \,+\,\frac{1}{1+ \log^+\frac{2T}{|x|^2}   }\, 1_{|x|^2  <   2T^2 \lambda   } \,\Bigg)\,.
\end{align*}
\end{lemma}
\begin{proof} Part (ii) of Lemma~\ref{LemmaUpsilonUpDown} gives us the inequality below for some constant $c_L>0$ and all $x\in \R^2$ and $T,\lambda>0$ with $T\lambda \leq L$.
\begin{align*}
  \big|b_{T}^{\lambda}(x)\big|\,=\, \frac{ \big|\nabla_x H_T^{\lambda}(x) \big| }{ 1+ H^{\lambda}_T(x) }  \,=\, \mathfrak{p}_T^{\lambda}(x)\, \big|\nabla_x H_T^{\lambda}(x) \big|  \, \leq \, \mathfrak{p}_T^{\lambda}(x)\,\frac{ c_L }{ |x| } \,\frac{e^{- \frac{|x|^2}{2T} }  }{ 1+\log^+\frac{1}{ T\lambda} }  
\end{align*}
The result can then be deduced by bounding  $\mathfrak{p}_T^{\lambda}(x)$ using (i) of Corollary~\ref{CorollaryThePandQ}.
\end{proof}

The radial symmetry of the magnitude of $b_{T}^{\lambda}(x)$ implies that  there is a function $\bar{b}_{T}^{\lambda}:(0,\infty)\rightarrow (0,\infty)$ with $\big|b_{T}^{\lambda}(x)\big|=\bar{b}_{T}^{\lambda}\big(|x|\big)$.  The next proposition provides an upper bound for  $\frac{\partial}{\partial \lambda}\overline{b}_{T}^{\lambda}(x)$. Its proof relies  
on the  lemma below, which we prove in Section~\ref{Subsectionlem exp}.
\begin{proposition}\label{PropBulkForDrift} The partial derivative of $\bar{b}_{T}^{\lambda}(a) $ with respect to $\lambda$ is strictly positive for all $T,a,\lambda>0$.  Given any $L>1$ there exists a $C_L>0$ such that for all  $T,a,\lambda >0$ with $T\lambda \leq L$
\begin{align*}
 \lambda\,\frac{\partial}{\partial \lambda}\bar{b}_{T}^{\lambda}(a) \,\leq \,  \frac{C_L}{a} \,\Bigg( \,\frac{  e^{-\frac{a^2}{2T}}}{\big(1+\log^+\frac{1}{T\lambda }\big)^2    }\, 1_{a^2\geq 2T^2\lambda  }\,+\,  \frac{  1}{\big(1+\log^+\frac{  2T }{ a^2 }\big)^2    }\, 1_{a^2 < 2T^2\lambda  }   \,\Bigg) \,.
\end{align*}
\end{proposition}

\begin{lemma} \label{lem exp} Given $\alpha,\beta>0$, let $X$and $Y$ be nonnegative random variables on a probability space $\big(\Omega,\mathcal{F},\mathcal{P}_{\alpha,\beta} \big)$ such that $X$ and $Y$ have joint density
\begin{align*}
\mathcal{P}_{\alpha,\beta}\big[\, X\in dR,\, Y\in dv   \,\big] \,=\,   \frac{1}{N_{\alpha,\beta} }\,\frac{ e^{-\alpha R } }{ R }\, \frac{ \big(1-\frac{1}{R}\big)^v \beta^v  }{\Gamma(v+1)  }\,1_{[1,\infty)}(R) \,dR\,dv
\end{align*}
for
$N_{\alpha,\beta}:=\int_1^{\infty}\frac{ e^{-\alpha R } }{ R } \nu\big( (1-\frac{1}{R}) \beta  \big)dR$. For any $L>0$, there exists a $C_L>0$ such that the inequalities below hold for all $\alpha>0$ and $\beta\in (0, L]$.
\begin{enumerate}[(i)]
\item $  \displaystyle   \mathcal{E}_{\alpha, \beta}[\,X\,Y\,] \, \leq \,C_L\frac{1}{ N_{\alpha,\beta}}\,\frac{ e^{-\alpha  } }{ \alpha} \, \frac{1}{\big(1+\log^+\frac{1}{\beta}\big)^2 } $

\item  $    \displaystyle \mathcal{E}_{\alpha,\beta}[\,X\,Y\,]\,-\,  \mathcal{E}_{\alpha,\beta}[\,X\,]\,\mathcal{E}_{\alpha,\beta}[\,Y\,]  \, \leq \, C_L\,\frac{ 1}{N_{\alpha,\beta}^2}\,\frac{e^{-2\alpha}}{\alpha }\,\frac{1}{ \big(1+\log^+ \frac{1}{\beta}\big)^3}$

\end{enumerate}

 \end{lemma}

\begin{proof}[Proof of Proposition~\ref{PropBulkForDrift}] Since $\bar{b}_{T}^{\lambda}(a)=-\frac{ \frac{\partial}{\partial a}\bar{H}_{T}^{\lambda}(a) }{1+ \bar{H}_{T}^{\lambda}(a)   } $, we can express $ \lambda\,\frac{\partial}{\partial \lambda}\bar{b}_{T}^{\lambda}(a)$ in the form
\begin{align*} 
   \frac{ \bar{H}_{T}^{\lambda}(a)  }{ \big(1+ \bar{H}_{T}^{\lambda}(a)  \big)^2 }\, \frac{ -\lambda   \, \frac{\partial^2}{\partial \lambda \partial a} \bar{H}_{T}^{\lambda}(a) 
  }{\bar{H}_{T}^{\lambda}(a)  } \,+\,  \bigg(\frac{\bar{H}_{T}^{\lambda}(a)  }{ 1+ \bar{H}_{T}^{\lambda}(a)   } \bigg)^2\,\Bigg(\,\frac{ -\lambda    \,\frac{\partial^2}{\partial \lambda \partial a} \bar{H}_{T}^{\lambda}(a) 
  }{\bar{H}_{T}^{\lambda}(a)  }\,-\, \frac{ \lambda  \,  \frac{\partial}{\partial \lambda } \bar{H}_{T}^{\lambda}(a) 
  }{\bar{H}_{T}^{\lambda}(a)  }\frac{    - \frac{\partial}{\partial a} \bar{H}_{T}^{\lambda}(a) 
  }{\bar{H}_{T}^{\lambda}(a)  }\,\Bigg) \,.
\end{align*}
For $\alpha=\frac{a^2}{2T}$ and $ 
\beta=T\lambda $, let $X$ and $Y$ be random variables on a probability space $\big(\Omega,\mathcal{F},\mathcal{P}_{\alpha,\beta} \big)$ having the joint distribution in Lemma~\ref{lem exp}. Then the above is equal to 
\begin{align*} 
    \frac{a}{T} \,\frac{ N_{\alpha,\beta}  }{ \big(1+ N_{\alpha,\beta} \big)^2 }\,  \mathcal{E}_{\alpha,\beta}[\,X\,Y\,]\, + \, \frac{a}{T}\, \bigg(\frac{N_{\alpha,\beta} }{ 1+ N_{\alpha,\beta}  } \bigg)^2\,\Big(\,  \mathcal{E}_{\alpha,\beta}[\,X\,Y\,]\,-\,  \mathcal{E}_{\alpha,\beta}[\,X\,]\,\mathcal{E}_{\alpha,\beta}[\,Y\,]\,\Big) \,.
\end{align*}
Both terms above are strictly positive  because the  random variables $X,Y$ are positive-valued and  positively correlated, and hence  $\frac{\partial}{\partial \lambda}\bar{b}_{T}^{\lambda}(a) >0$. Now, applying Lemma~\ref{lem exp}, the above is bounded by a multiple of 
\begin{align*} 
    \frac{a}{T\alpha} \,\frac{  e^{-\alpha  }  }{ \big(1+ N_{\alpha,\beta} \big)^2 } \,  \frac{1}{\big(1+\log^+\frac{1}{\beta}\big)^2 }\,+\, \frac{a}{T\alpha}\, \frac{e^{-2\alpha  } }{ \big(1+ N_{\alpha,\beta} \big)^2 } \, \frac{1}{\big(1+\log^+\frac{1}{\beta}\big)^3 } \,\leq \,   \frac{2a}{T\alpha} \,\frac{  e^{-\alpha  }  }{ \big(1+ N_{\alpha,\beta} \big)^2 }  \, \frac{1}{\big(1+\log^+\frac{1}{\beta}\big)^2 }\ \,.
\end{align*}
However, there is a constant $c>0$ such that $\frac{1}{1+N_{\alpha,\beta}}\leq 1_{\alpha\geq \beta} +c\frac{1+\log^+ \frac{1}{\beta}}{1+\log^+\frac{1}{\alpha}   }   1_{\alpha < \beta} $ for all $\alpha>0$ and $\beta\in (0,L] $, which  is an equivalent statement to (i) of Corollary~\ref{CorollaryThePandQ} by the observation~(\ref{H2N}).  Plugging back in with $\alpha=\frac{a^2}{2T}$ and $ 
\beta=T\lambda $ yields the result.
\end{proof}

\subsection{The kernel $\boldsymbol{h_{ t}^{\lambda}(x,y)}$}\label{SubSectHKern}
For $t,\lambda>0$ recall that we define $h_{t}^{\lambda}:\R^2\times \R^2\rightarrow [0,\infty]$ as in~(\ref{DefLittleH}).  We can express $h_{t}^{\lambda}$ in the form
\begin{align}\label{hKern}
h_{ t}^{\lambda}(x,y)\,=\,\frac{\lambda}{2\pi} \,\int_{0}^1 \,\frac{1}{1-r}\,\upsilon\bigg(\frac{|x|^2}{2t(1-r)}, \frac{|y|^2}{2t(1-r)} \bigg)\, \nu' \big(t\lambda  r\big)\,dr
\end{align}
for  $\upsilon:[0,\infty)\times [0,\infty)\rightarrow [0,\infty)     $ given by
$\upsilon(\alpha,\beta)= \int_0^1 \frac{e^{-\frac{\alpha}{ s}  }  }{ s }\frac{ e^{-\frac{\beta}{ 1-s} }  }{ 1-s  }ds$.   We have the obvious  symmetry  $\upsilon(\alpha,\beta)= \upsilon(\beta,\alpha)$, and $\upsilon$ is related to the exponential integral function through $\frac{1}{\pi}\int_{\R^2}\upsilon\big(\alpha,|y|^2\big)dy =E(\alpha) $. The proposition below provides an upper bound for $ h_{ t}^{\lambda}(x,y)$ in which the variables $x$ and $y$ are decoupled.
\begin{proposition} \label{hFunction}    For any $L>0$ there exists a  $C_L>0$ such that for all  $x,y\in \R^2$ and $t,\lambda >0$  with $ t\lambda \leq L  $ we have
    $$  h_{t}^{\lambda}(x,y)\,\leq \,\frac{1}{t}\,\frac{C_L}{ 1+\log^+ \frac{1}{\lambda t} }\,\bigg(\, e^{-\frac{|x|^2}{2t}}+\log^{+}\frac{2t}{|x|^2}  \,\bigg)\,\bigg(\, e^{-\frac{|y|^2}{2t}}+\log^{+}\frac{2t}{|y|^2} \, \bigg) \,.  $$

\end{proposition}

The following lemma isolates the main steps in the proof of Proposition~\ref{hFunction} below.
\begin{lemma}\label{LemmaLittleUpsilon} Let the function $\upsilon$ be defined as above and set $\theta_{\lambda}(\alpha):=\lambda \int_0^1\frac{ e^{-\frac{\alpha}{1-r  }}  }{1-r   }\frac{\nu'( r \lambda ) }{\nu(\lambda)  }\,dr $ for $\alpha,\lambda>0$.

\begin{enumerate}[(i)] 
    \item  For all $\alpha,\beta\geq 0$ we have
$\upsilon(\alpha,\beta)\leq   2 e^{-\beta  } E(\alpha)+2 e^{-\alpha  } E(\beta)$.

    \item  For any $L>0$ there exists a  $C_L>0$ such that $\theta_{\lambda}(\alpha)\leq C_{L}\big(e^{-\alpha}  +  E( \alpha)\big)$ for all $\lambda\in (0,L]$ and $\alpha>0$.

\end{enumerate}

\end{lemma}

\begin{proof} Part (i): After writing  $\upsilon(\alpha,\beta)$ as the sum $\int_0^{\frac{1}{2}}  \frac{  e^{-\frac{\alpha}{ 
r}   } }{r }\frac{  e^{-\frac{ \beta}{ 1-r}   } }{ 1-r  }dr +\int_0^{\frac{1}{2}}  \frac{  e^{-\frac{\beta}{ 
r}   } }{r }\frac{  e^{-\frac{ \alpha}{ 1-r}   } }{ 1-r  }dr$, we observe that
\begin{align*}
 \upsilon(\alpha,\beta) \,\leq \, 2 \,e^{-\beta  } \,\int_0^{\frac{1}{2}}  \,\frac{  e^{-\frac{\alpha}{ 
r}   } }{r }\, dr \,+\,2\, e^{-\alpha   } \int_0^{\frac{1}{2}} \, \frac{  e^{-\frac{\beta}{ 
r}   } }{r }\, dr  
\,\leq \, 2 \,e^{-\beta  } \,\int_0^{1}\,  \frac{  e^{-\frac{\alpha}{ 
r}   } }{r }\, dr \,+\,2 \,e^{-\alpha  }\, \int_0^{1} \, \frac{  e^{-\frac{\beta}{ 
r}   } }{r } \, dr  \,,
\end{align*}
where the right side is equal to $2 e^{-\beta   } E(\alpha)+2 e^{-\alpha  } E(\beta) $.  \vspace{.2cm}

\noindent Part (ii): We observe that
\begin{align*}
 \theta_{\lambda}(\alpha) \,= \,&\,\lambda\,\int_0^{\frac{1}{2}}\, \frac{ e^{-\frac{\alpha}{1-r  }}  }{1-r   }\,\frac{ \nu'( r \lambda) }{\nu(\lambda)  } \,dr\,+\,\lambda \,\int_{\frac{1}{2}}^{1}\,\frac{ e^{-\frac{\alpha}{1-r  }}  }{1-r   }\,\frac{\nu'( r \lambda ) }{\nu(\lambda)  } \,dr \nonumber \\
    \leq \,&\,2\,\lambda\,e^{-\alpha}\,  \int_0^{\frac{1}{2}}\,\frac{ \nu'(r \lambda  ) }{\nu(\lambda)  }\,dr  \,+\,\Bigg( \sup_{l\in [\frac{\lambda}{2},L ]}\,\frac{\lambda \,\nu'(l ) }{\nu(\lambda)  }\Bigg)\,\int_{\frac{1}{2}}^{1}\,\frac{ e^{-\frac{\alpha}{1-r  }}  }{1-r   }\,dr \nonumber \\
     \leq \,&\,2\,e^{-\alpha}  \,+\,\Bigg( \sup_{l\in [\frac{\lambda}{2},L ]}\,\frac{\lambda \nu'(l ) }{\nu(\lambda)  }\Bigg)\int_{0}^{1}\,\frac{ e^{-\frac{\alpha}{r  }}  }{r   } \,dr\,
     = \,2\,e^{-\alpha} \,  \,+\, \underbracket{\frac{ \lambda \max\big( \nu'(\frac{\lambda}{ 2}),   \nu'(L)\big) }{ \nu(\lambda)  }}\,E( \alpha ) \,,\nonumber 
\end{align*}
where the last equality uses that the maximum of  $\nu'$ over the interval $ \big[\frac{\lambda}{2},L\big]$ must occur at one of the boundary points because  $\nu'$ is convex by Remark~\ref{Remark1st}.  We only need to show that the bracketed term is bounded for $\lambda\in (0,L]$. Let $\lambda^*>0$ denote the point at which $\nu'$ attains its global minimum.    If $ \frac{\lambda}{2} \leq \lambda^{*}$, then we bound the bracketed term as follows:
\begin{align*}
   \frac{\lambda \max\big( \nu'(\frac{\lambda}{ 2}),   \nu'(L)\big) }{ \nu(\lambda)  }\,\leq\, \frac{  \nu'\big(\frac{\lambda}{2}\big) }{ \int_0^1\nu'(r \lambda )dr }\,+\, \frac{  \nu'(L) }{ \int_0^1\nu'( r\lambda)dr }\,\leq \, 2\,+ \, \frac{  \nu'(L ) }{ \nu'(\lambda^*) }\,,
\end{align*}
in which we wrote $\nu(\lambda)=\lambda\int_0^{1}\nu'(r\lambda)dr $  and then used that $\nu'\big(\frac{\lambda}{2}\big) \leq  2\int_0^{1/2}\nu'( r\lambda)dr $ since $\nu'$ is decreasing on the interval $\big[0,\frac{\lambda}{2}\big]$. On the other hand,  if $\frac{\lambda }{2}\geq \lambda^{*}$, then 
\begin{align*}
    \frac{\lambda \max\big( \nu'(\frac{\lambda}{2}),   \nu'(L)\big) }{ \nu(\lambda)  }\, =\, \frac{  \nu'(L ) }{ \int_0^1\nu'( r\lambda)dr }\,\leq \, \frac{  \nu'(L ) }{ \nu'(\lambda^*) } 
\end{align*}
since $\nu'$ is increasing over the interval $[\lambda^{*},\infty) $.
\end{proof}

\vspace{.3cm}

\begin{proof}[Proof of Proposition~\ref{hFunction}] Starting with the expression for    $h_{t}^{\lambda}(x,y)$ in~(\ref{hKern}), we can apply (i) of Lemma~\ref{LemmaLittleUpsilon} to get the first inequality below.
\begin{align*}
  h_{t}^{\lambda}(x,y)\,\leq \, &\,   
 \frac{\lambda}{\pi}\,  \int_{0}^1 \,\frac{1}{1-r}\,\Bigg(\, e^{-\frac{ |y|^2}{ 2t(1-r)}   } \,E\bigg(\frac{|x|^2}{2t(1-r)}\bigg)\,+\, e^{-\frac{ |x|^2}{ 2t(1-r)}   } \,E\bigg(\frac{|y|^2}{2t(1-r)}\bigg)\,\Bigg) \, \nu'( tr\lambda  )\,dr \\  \leq \,&\, \frac{\lambda}{\pi} \,E\bigg(\frac{|x|^2}{2t}\bigg)\,\int_{0}^1 \,\frac{ e^{-\frac{ |y|^2}{2 t(1-r)}   }}{1-r}\, \nu'( t r\lambda  ) \,dr \,+\, \frac{\lambda}{\pi}\, E\bigg(\frac{|y|^2}{2t}\bigg)\,\int_{0}^1 \, \frac{e^{-\frac{ |x|^2}{2 t(1-r)}   }}{1-r} \,\nu'( t r\lambda  )\,dr
 \\    \, \prec\,&\, \frac{1}{ t}\,\nu( t \lambda)\,E\bigg(\frac{|x|^2}{2t}\bigg)\,\bigg( \,e^{-\frac{|y|^2}{2t}  } +  \,E\bigg(\frac{|y|^2}{2t}\bigg)\,\bigg) \,+\, \frac{ 1}{ t}\,\nu( t\lambda)\,E\bigg(\frac{|y|^2}{2t}\bigg)\,\bigg( e^{-\frac{|x|^2}{2t}  } + \,E\bigg(\frac{|x|^2}{2t}\bigg)\bigg) 
\\  \leq \,&\, \frac{1}{t}\,\nu( t \lambda)\,\bigg(\,  e^{-\frac{|x|^2}{2t}  }\,+\,\ E\bigg(\frac{|x|^2}{2t}\bigg) \,\bigg)\,\bigg(\, e^{-\frac{|y|^2}{2t}  }\,+\,E\bigg(\frac{|y|^2}{2t}\bigg)\,  \bigg)
\end{align*}
The second inequality above holds since $E$ is decreasing, and the third inequality applies part (ii) of Lemma~\ref{LemmaLittleUpsilon}.  Finally, we arrive at the result using  the upper bound for $E(x)$ in (iv) of Lemma~\ref{LemEM}  and that $\nu(x)$  satisfies $\nu(x)\sim \frac{1}{\log \frac{1}{x}}$ for $0<x\ll 1$ by (i) of Lemma~\ref{LemmaEFunAsy}.
\end{proof}

\section{The lemma proofs} \label{SectionMiscProofs}

\subsection{Proof of Lemma~\ref{LemmaLeave}} \label{SubsectionLemmaLeave}

If $\{X_t\}_{t\in [0,\infty)}$ is a  two-dimensional Brownian motion starting from $x\in \R^2$, then the process  $\{\mathbf{S}_t\}_{t\in [0,\infty)}$ defined by  $\mathbf{S}_t:=\frac{1}{2}|X_t|^2$ is a submartingale whose Doob-Meyer decomposition has martingale component $\mathbf{M}_t=\frac{1}{2}|x|^2+\int_0^t X_r\cdot dX_r$ and increasing component    $\mathbf{A}_t=t$.
The linearity of $\mathbf{A}$ in time  allows us to derive a simple formula for the expectation of the stopping time  $\varrho^{\uparrow,\varepsilon}=\inf\{t\in [0,\infty)\,:\,|X_t|=\varepsilon \}  $ given some $\varepsilon\in (|x|,\infty)$  through a standard argument:
\begin{align}\label{BasicVarrho}
\mathbf{E}_x\big[\,\varrho^{\uparrow,\varepsilon}\,\big]\,=\,\mathbf{E}_x\big[\,\mathbf{A}_{\varrho^{\uparrow,\varepsilon}}\,\big]\,=\,\mathbf{E}_x\big[\,\mathbf{S}_{\varrho^{\uparrow,\varepsilon}}\,\big]\,-\,\frac{1}{2}\,|x|^2\,=\,\frac{1}{2}\,\big(\,\varepsilon^2-|x|^2\,\big) \,,
\end{align}
where we used that $ \mathbf{M}_{\varrho^{\uparrow,\varepsilon}}=\mathbf{S}_{\varrho^{\uparrow,\varepsilon}}-\mathbf{A}_{\varrho^{\uparrow,\varepsilon}}$ has mean $\frac{1}{2}|x|^2$ as a consequence of the optional stopping theorem.
The proposition below presents a $\mathbf{P}_{\mu}^{T,\lambda}$-submartingale whose increasing Doob-Meyer component satisfies $\frac{d}{dt}\mathbf{A}^{T,\lambda }_t\geq 1$.  Applying a similar argument as above yields  an upper bound for $\mathbf{E}_x^{T,\lambda}\big[\varrho^{\uparrow,\varepsilon}\big]$.  

For $T,\lambda>0$ define the   function $\eta_{T}^{\lambda}:[0,\infty)\times \R^2\rightarrow [0,\infty)$  by
$$\eta_{T}^{\lambda}(t,x)\,=\,2\,\int_0^{|x|} \, \int_0^b\, \frac{a}{b}\,\frac{\big( 1+\bar{H}_{T}^{\lambda}(a)\big)^2     }{\big( 1+\bar{H}_{T-t}^{\lambda}(b)\big)^2      }\,da\, db \,,  $$
which is increasing in both $t$ and $|x|$. Notice that $\eta_{T}^{\lambda}(t,x)=\eta_{T}^{\lambda}(T,x)$ for $t\geq T$, as $\bar{H}_{r}^{\lambda}:=0$ when $r<0$. Since $\bar{H}_r^{\lambda}(a)\sim 2\nu(r\lambda)\log \frac{1}{a}$ as $a\searrow 0$ for $r>0$  by (i) of Proposition~\ref{PropK},  the value $\eta_{T}^{\lambda}(t,x)$ has the following small $|x|$ asymptotic form  for any  $t\in [0,T)$:
\begin{align}\label{EtaAppox}
    \eta_{T}^{\lambda}(t,x)\,\stackrel{ x\rightarrow 0}{\sim}\,\frac{1}{2}\, \frac{ \nu^2(T\lambda) }{ \nu^2\big((T-t)\lambda \big) } 
 \,|x|^2 \, \,.
\end{align}
 We use the next lemma to check the square-integrability of the martingale $\mathbf{M}^{T,\lambda }$ appearing in Proposition~\ref{PropSquaredBessel}. 
\begin{lemma} \label{LemmaPreSquaredBessel} For any $L>0$ there exists a  $C_L>0$ such that for all $x\in \R^2$, $t\in [0,\infty)$, and $T,\lambda\in (0,L]$  we have
\begin{align*}
 \int_0^t \,\int_{\R^2 }\,   \mathlarger{d}_{0,s}^{T,\lambda}(x,y)\,\Big|  \nabla_y \,\eta_{T}^{\lambda}(s,y)\Big|^2\,dy\,ds \,\leq \, C_L\,\Big(\, t\,\big(1+|x|^2\big)\,+\,t^2  \, \Big)\,.
\end{align*}
    
\end{lemma}
\begin{proof}We can bound the norm of $ \nabla_y \eta_{T}^{\lambda}(t,y) $  as follows:\begin{align*} \Big| \nabla_y \,\eta_{T}^{\lambda}(t,y)\Big|^2\,=\,\left|\frac{y}{|y|}\,\frac{2\int_0^{|y|} a\big( 1+\bar{H}_{T}^{\lambda}(a)\big)^2    da  }{|y| \big( 1+\bar{H}_{T-t}^{\lambda}(|y|)\big)^2   }\right|^2   \,\leq \,4\,\bigg(\,\int_0^{|y|}\, \big( 1+\bar{H}_{T}^{\lambda}(a)\big)^2  \,  da \,\bigg)^2   \,\preceq \,\,\big(1+|y|^2\big)\,,\end{align*}
where the second inequality can be shown using (i) of Lemma~\ref{LemmaUpsilonUpDown}.  The result then follows  from the inequality $\int_{\R^2 }   \mathlarger{d}_{0,t}^{T,\lambda}(x,y)\big(1+|y|^2\big)dy\leq 1+|x|^2+2t$, which can be shown using the forward Kolmogorov equation~(\ref{KolmogorovForJ}) and integration by parts. \end{proof}

\begin{proposition}\label{PropSquaredBessel} Fix some $T,\lambda > 0 $  and  a Borel probability measure $\mu$ on $ \R^2$ with $\int_{\R^2}|x|^2\mu(dx)<\infty$. Define the process $\{\mathbf{S}_{t}^{T,\lambda }\}_{t\in [0,\infty)}$  by $\mathbf{S}_{t}^{T,\lambda}=\eta_{T}^{\lambda}(t,X_t)$.  Then $\mathbf{S}^{T,\lambda}$ is a continuous $\mathbf{P}_{\mu}^{T,\lambda}$-submartingale with respect to $\{\mathscr{F}_{t}^{T,\mu}\}_{t\in [0,\infty)}$ for which the martingale $\mathbf{M}^{T,\lambda }$ and increasing $\mathbf{A}^{T,\lambda }$ components in its Doob-Meyer decomposition satisfy the following:
\begin{itemize}
    \item $
\mathbf{M}_{t}^{T,\lambda }:=\eta_{T}^{\lambda}(0,X_0)+\int_0^t \big(\nabla_x\eta_{T}^{\lambda}\big)(s, X_{s}) \cdot dW_{s}^{T,\lambda}
$ is in $L^2(\mathbf{P}_{\mu}^{T,\lambda})  $.

    \item $
\frac{\partial}{\partial t}\mathbf{A}^{T,\lambda }_t \geq  1
$
for all $t\in [0,\infty)$ almost surely under $\mathbf{P}_{\mu}^{T,\lambda}$. 
\end{itemize}

\end{proposition}

\begin{proof} To check that the It\^o integral in the first bullet point defines a square-integrable martingale, we observe that 
\begin{align*}
\mathbf{E}_{\mu}^{T,\lambda}\bigg[\, \int_0^t\, \Big|\big(\nabla_x\,\eta_{T}^{\lambda}\big)(s, X_{s})\Big|^2\,ds \,   \bigg]\,=\,&\, \int_{\R^2}\, \int_0^t \,\int_{\R^2 } \,  \mathlarger{d}_{0,s}^{T,\lambda}(x,y)\,\Big|  \nabla_y\, \eta_{T}^{\lambda}(s,y)\Big|^2\,dy\,ds\,\mu(dx) \nonumber \\ \,\leq \,&\,t\,C_T^{\lambda}\, \int_{\R^2}\,\big( 1+|x|^2  \big) \,\mu(dx) \,+\,t^2\,C_T^{\lambda} \,,
 \end{align*}
where the inequality holds for some constant $C_T^{\lambda}>0$ and all $t>0$ by Lemma~\ref{LemmaPreSquaredBessel}.  The above is finite by our assumption on $\mu$.

Since $ \mathbf{S}_{t}^{T,\lambda}=\eta_{T}^{\lambda}(t,X_t) $, using Proposition~\ref{PropStochPre} and It\^o's chain rule, we can write
\begin{align}\label{ItoZ}
    d\mathbf{S}_{t}^{T,\lambda}\,=\,\Big(\frac{\partial}{\partial t}\,\eta_{T}^{\lambda} \Big)  (t,X_t)\,dt\,+\,\big(\nabla_x \,\eta_{T}^{\lambda}\big)(t,X_t) \cdot \Big(  dW^{T,\lambda}_t\,+\,b^{\lambda}_{T-t}(X_t)\,dt   \Big)\,+\,\frac{1}{2}\big(\Delta_x \,\eta_{T}^{\lambda}\big)(t,X_t)\,dt\,.
  \end{align}  
 For the second term on the right side of~(\ref{ItoZ}), we compute that for $t\geq 0$ and $y\in \R^2$
\begin{align}\label{NablaEta}
 \big(\nabla_y \,\eta_{T}^{\lambda}\big)(t,y) \cdot b^{\lambda}_{T-t}(y)\,=\,&\, \Bigg(\,\frac{y}{|y|}\,\frac{2\int_0^{|y|} a\big( 1+\bar{H}_{T}^{\lambda}(a)\big)^2    da  }{|y|\big( 1+\bar{H}_{T-t}^{\lambda}(|y|)\,\big)^2   }\,\Bigg) \cdot \frac{  \big(\nabla_y H_{T-t}^{\lambda}\big)(y)  }{1+H_{T-t}^{\lambda}(y)  } \nonumber \\
 \,=\,&\,2\,\frac{\int_0^{|y|} a\big( 1+\bar{H}_{T}^{\lambda}(a)\big)^2    da  }{|y| \big( 1+\bar{H}_{T-t}^{\lambda}(|y|)\big)^3   }  \,\frac{\partial}{\partial z}\bar{H}_{T-t}^{\lambda}(z)\,\Bigg|_{z=|y|}  \,,
\end{align}
where we have used that $\frac{\partial}{\partial z}\bar{H}_{t}^{\lambda}(z)\big|_{z=|y|}=\frac{y}{|y|}\cdot (\nabla H_{t}^{\lambda})(y) $.  For $f:[0,\infty)\rightarrow \R$ and $y\in \R^2$, the radial Laplacian form $\Delta_yf(|y|)=\big(f''(z)+\frac{1}{z} f'(z) \big)\big|_{z=|y|} $ yields that
\begin{align}\label{IneqEta1}
  \frac{1}{2} \, \big(\Delta_y \,\eta_{T}^{\lambda}\big)(t,y)\,=\,&\,\bigg( \frac{\partial^2}{\partial z^2}+\frac{1}{z}\,\frac{\partial}{\partial z}   \bigg) \,\int_0^{z}\,  \int_0^b\, \frac{a}{b}\,\frac{\big( 1+\bar{H}_{T}^{\lambda}(a)\big)^2     }{\big( 1+\bar{H}_{T-t}^{\lambda}(b)\big)^2      }\,da \,db \,\Bigg|_{z=|y|} \nonumber \\
    \,=\,&\, - 2 \, \frac{\int_0^{|y|} a\big( 1+\bar{H}_{T}^{\lambda}(a)\big)^2    da   }{|y| \big( 1+\bar{H}_{T-t}^{\lambda}(|y|)\big)^3   } \,\frac{\partial}{\partial z}\bar{H}_{T-t}^{\lambda}(z)\,\Bigg|_{z=|y|}  \,+\,\underbrace{\frac{ \big( 1+\bar{H}_{T}^{\lambda}(|y|)\big)^2      }{\big( 1+\bar{H}_{T-t}^{\lambda}(|y|)\big)^2   }}_{\geq \,1}\nonumber \\ \,\geq \,&\, -\big(\nabla_y\, \eta_{T}^{\lambda}\big)(t,y) \cdot b^{\lambda}_{T-t}(y)\,+\,1\,,
\end{align}
where the inequality uses~(\ref{NablaEta}) and that $\bar{H}_{t}^{\lambda}(y)$ is increasing in $t$.  We also have that 
\begin{align}\label{IneqEta2}
    \frac{\partial}{\partial t}\eta_{T}^{\lambda} (t,y)\,=\,-4\,\int_0^{|y|} \, \int_0^b \,\frac{a}{b} \,\frac{\big( 1+\bar{H}_{T}^{\lambda}(a)\big)^2     }{\big( 1+\bar{H}_{T-t}^{\lambda}(b)\big)^3     }\,\underbrace{\frac{\partial}{\partial t}\bar{H}_{T-t}^{\lambda}(b)}_{\leq \, 0} \, da\, db \,\geq \,0\,.
\end{align}
Applying the relations~(\ref{NablaEta})--(\ref{IneqEta2}) in the SDE~(\ref{ItoZ}) yields that $  d\mathbf{S}_{t}^{T,\lambda}\geq \big(\nabla_x\eta_{T}^{\lambda}\big)(t,X_t) \cdot   dW^{T,\lambda}_t
 +dt $, which implies our desired result.
\end{proof}

\begin{proof}[Proof of Lemma~\ref{LemmaLeave}]  Part (i): Recall from Proposition~\ref{PropSquaredBessel}  that  $\mathbf{S}^{T,\lambda }_t=\eta_{T}^{\lambda}(t,X_t)$ is a $\mathbf{P}^{T,\lambda}_{\mu}$-submartingale with Doob-Meyer decomposition $ \mathbf{M}^{T,\lambda}+ \mathbf{A}^{T,\lambda}$.  For $t\geq 0$ the quadratic variation of $\mathbf{M}^{T,\lambda}$ has the form
\begin{align}\label{BoundQVPre}
 \big\langle \mathbf{M}^{T ,\lambda} \big\rangle_{t}\,=\,\int_0^t \,\Big|\big(\nabla_x\,\eta^{\lambda}_{T-s}\big)(s,X_s)  \Big|^2\, ds\,=\, \int_0^t\, \Bigg( \underbracket{\frac{2\int_0^{|X_s|} a\big( 1+\bar{H}_{T}^{\lambda}(a)\big)^2    da  }{|X_s| \big( 1+\bar{H}_{T-t}^{\lambda}(X_s)\big)^2   }} \Bigg)^2\, ds \,,
 \end{align}
 and thus for $t\leq \frac{T}{2}$ we have that
 \begin{align}\label{BoundQV}
     \big\langle \mathbf{M}^{T ,\lambda} \big\rangle_{t} \,\leq \,t\,\mathbf{c}_{L}^2\,\sup_{s\in [0,t]}\,|X_s|^2 \hspace{1cm}\text{for } \,\mathbf{c}_{L}\,:=\, \sup_{\substack{ b>0 \\ \lambda, T\leq L }} \, \frac{ 2\int_0^b  a\big( 1+H_{T}^{\lambda}(a)\big)^2  da    }{b^2\big( 1+H_{\frac{T}{2}}^{\lambda}(b)\big)^2 } \,.
 \end{align}
 This bound arises from multiplying and dividing the underbracketed expression in (\ref{BoundQVPre}) by $|X_s|$, which is why $b^2$ appears in the denominator of the expression for $\mathbf{c}_{L}$. The bounds for $H_{T}^{\lambda}$ in  (i) of Lemma~\ref{LemmaUpsilonUpDown} imply that  the constant $\mathbf{c}_{L}$ is finite.  For $t>0$ define $\varrho^{\uparrow, \varepsilon}_{t}:=t\wedge\varrho^{\uparrow, \varepsilon}$. 
 Using that $t\leq \mathbf{A}_{t}^{T,\lambda}$ by  Proposition~\ref{PropSquaredBessel} and writing $\mathbf{A}^{T,\lambda}=\mathbf{S}^{T,\lambda}-\mathbf{M}^{T,\lambda}$, we get the first two relations below.
\begin{align}\label{RhoPreArrange}
     \mathbf{E}_{x}^{T,\lambda}\Big[ \,\big(\varrho^{\uparrow, \varepsilon}_{\frac{T}{2}} \big)^m \,\Big]^{\frac{1}{m}}\,\leq \,  \mathbf{E}_{x}^{T,\lambda}\bigg[ \,\Big(\mathbf{A}_{\varrho^{\uparrow, \varepsilon}_{\frac{T}{2}} }^{T,\lambda}\Big)^m \,\bigg]^{\frac{1}{m}}\,=\,&\, \mathbf{E}_{x}^{T,\lambda}\bigg[ \,\Big|  \mathbf{S}_{\varrho^{\uparrow, \varepsilon}_{\frac{T}{2}} }^{T,\lambda}\,-\,\mathbf{M}_{\varrho^{\uparrow, \varepsilon}_{\frac{T}{2}} }^{T,\lambda}  \Big|^m \,\bigg]^{\frac{1}{m}}\nonumber \\ \,\leq \,& \,\mathbf{E}_{x}^{T,\lambda}\bigg[ \,\Big|  \mathbf{S}_{\varrho^{\uparrow, \varepsilon}_{\frac{T}{2}} }^{T,\lambda}\Big|^m \,\bigg]^{\frac{1}{m}}\,+\, \mathbf{E}_{x}^{T,\lambda}\bigg[\, \Big|  \mathbf{M}_{\varrho^{\uparrow, \varepsilon}_{\frac{T}{2}} }^{T,\lambda}  \Big|^{m}\, \bigg]^{\frac{1}{m}}\nonumber \\ \,\leq \,& \,\eta_{T}^{\lambda}\Big(\frac{T}{2},y\Big)\,\Big|_{|y|=\varepsilon} \,+\,c_m \,\mathbf{E}_{x}^{T,\lambda}\Big[ \,  \big\langle \mathbf{M}^{T ,\lambda} \big\rangle_{\varrho^{\uparrow, \varepsilon}_{\frac{T}{2}} }^{\frac{m}{2}} \,\Big]^{\frac{1}{m}}\nonumber 
     \\ \,\preceq \,& \,\varepsilon^2\,+\,\varepsilon \, \mathbf{E}_{x}^{T,\lambda}\Big[ \, \big(\varrho^{\uparrow, \varepsilon}_{\frac{T}{2}}  \big)^{\frac{m}{2}} \,\Big]^{\frac{1}{m}}
\end{align}
For the third inequality, we  have  applied Burkholder-Davis-Gundy to the second term, where  $c_m>0$ is a universal constant, and for the first term we have used that $\varrho^{\uparrow, \varepsilon}_{T/2} \leq \frac{T}{2}$ and $|X_{\varrho^{\uparrow, \varepsilon}_{T/2}}|\leq \varepsilon $, along with that $\eta_{T}^{\lambda}(t,x)$ is increasing in both $t$ and $|x|$.  The fourth inequality holds by~(\ref{EtaAppox}) and~(\ref{BoundQV}). Applying Jensen's inequality and  that $(a+b)^m\leq 2^m(a^m+b^m)$ for $a,b\geq 0$ to~(\ref{RhoPreArrange}), we can write
\begin{align*}
     \mathbf{E}_{x}^{T,\lambda}\Big[ \,\big(\varrho^{\uparrow, \varepsilon}_{\frac{T}{2}} \big)^m \, \Big]\,\preceq \,\varepsilon^{2m} \,+\,\varepsilon^{m} \, \mathbf{E}_{x}^{T,\lambda}\Big[ \, \big(\varrho^{\uparrow, \varepsilon}_{\frac{T}{2}} \big)^m \,\Big]^{\frac{1}{2}}\,.
\end{align*}
The above implies that  $\mathbf{E}_{x}^{T,\lambda}\big[ \big(\varrho^{\uparrow, \varepsilon}_{T/2}  \big)^m \big]\leq  \mathbf{C}_L\varepsilon^{2m}  $ for some $\mathbf{C}_L>0$.  Applying Chebyshev's inequality  yields that $\mathbf{P}_{\mu}^{T,\lambda}\big[ \varrho^{\uparrow, \varepsilon} \geq \frac{T}{2} \big] \leq \mathbf{C}_L \big(\frac{2}{T}\varepsilon^2\big)^m $, and hence that
 \begin{align*}
      \mathbf{E}_{x}^{T,\lambda}\Big[\,\big(\varrho^{\uparrow, \varepsilon}_T \big)^m \,\Big]\,\leq \,\mathbf{E}_{x}^{T,\lambda}\Big[\,\big(\varrho^{\uparrow, \varepsilon}_{\frac{T}{2}}\big)^m \,\Big]\,+\,T^m\,\mathbf{P}_{x}^{T,\lambda}\Big[ \,\varrho^{\uparrow, \varepsilon}\geq  \frac{T}{2} \, \Big]\,\leq \,\big(1+2^m\big)\,\mathbf{C}_L\,\varepsilon^{2m} \,.
 \end{align*}
Thus, the $m^{\textup{th}}$ moment of $\varrho^{\uparrow, \varepsilon}_T$ under $\mathbf{P}_{x}^{T,\lambda}$ is bounded by a constant multiple of $\varepsilon^{2m}$.  

 To extend  our result from $ \varrho^{\uparrow, \varepsilon}_T$ to $ \varrho^{\uparrow, \varepsilon}$, we observe that by writing $ \varrho^{\uparrow, \varepsilon}=\varrho^{\uparrow, \varepsilon}_T+1_{ \varrho^{\uparrow, \varepsilon} > T }\big(\varrho^{\uparrow, \varepsilon}-T\big)$ and again using the inequality $(a+b)^m\leq 2^m(a^m+b^m)$
 \begin{align*}
    \mathbf{E}_{x}^{T,\lambda}\Big[ \,\big( \varrho^{\uparrow, \varepsilon}\big)^m \, \Big] \,\leq \,&\,    2^m\,\mathbf{E}_{x}^{T,\lambda}\Big[\, \big(\varrho^{\uparrow, \varepsilon}_T\big)^m \,\Big]\,+\,  2^m\,\mathbf{E}_{x}^{T,\lambda}\Big[\,1_{ \varrho^{\uparrow, \varepsilon} > T }\, \big(\varrho^{\uparrow, \varepsilon}-T\big)^m  \,\Big]  \nonumber   \\
    \,=\, & \,  2^m\, \mathbf{E}_{x}^{T,\lambda}\Big[ \,\big(\varrho^{\uparrow, \varepsilon}_T\big)^m \,\Big]\,+\,  2^m\,\mathbf{E}_{x}^{T,\lambda}\bigg[\,1_{ \varrho^{\uparrow, \varepsilon} > T }\,\underbrace{\mathbf{E}_{x}^{T,\lambda}\Big[\, \big(\varrho^{\uparrow, \varepsilon}-T\big)^m\,\Big|\,\mathscr{F}_{T}^{T,\mu} \,\Big]}_{ \mathbf{E}_{X_T}\big[ (\varrho^{\uparrow, \varepsilon})^m  \big]  } \,\bigg]   \,.
 \end{align*}
 Applying the strong Markov property to the nested conditional expectation within the second term, we find that it is equal to the expectation $\mathbf{E}_{X_T}^{0,\lambda}\equiv\mathbf{E}_{X_T}$ of $(\varrho^{\uparrow, \varepsilon})^m$ with respect to two-dimensional Wiener measure $\mathbf{P}_{X_T}$.
This reduces the problem to that for a two-dimensional Brownian motion, for which $\mathbf{E}_{y}\big[ (\varrho^{\uparrow, \varepsilon})^m  \big] \preceq \varepsilon^{2m}$ holds.\vspace{.3cm}

\noindent Part (ii): Since the coordinate process $X$ obeys the SDE $dX_t=b_{T-t}^{\lambda}(X_t)dt+dW_t^{T,\lambda}$ for the drift vector $b_{T}^{\lambda}(x)=-\frac{x}{|x|}\mathbf{b}_{T}^{\lambda}(x)$, which points towards the origin and has magnitude $\mathbf{b}_{T}^{\lambda}(x)$ that increases with the parameter $\lambda$ (see Proposition~\ref{PropBulkForDrift}), the expectation $\mathbf{E}_{0}^{T,\lambda}[ \varrho^{\uparrow, \varepsilon} ]$ is increasing with $\lambda$.  It follows that the expectation of $ \varrho^{\uparrow, \varepsilon}$ under  $\mathbf{P}_{0}^{T,\lambda}$ is greater than that under the Wiener measure $\mathbf{P}_{0}^{T,0}\equiv \mathbf{P}_{0} $:
$$\mathbf{E}_{0}^{T,\lambda}\big[ \,\varrho^{\uparrow, \varepsilon} \,\big] \,\geq \,\mathbf{E}_{0}\big[\, \varrho^{\uparrow, \varepsilon} \,\big] \,\stackrel{(\ref{BasicVarrho})}{=}\,\frac{\varepsilon^2}{2}\,. $$
 Thus, it suffices for us to bound the difference $\mathbf{E}_{0}^{T,\lambda}[ \varrho^{\uparrow, \varepsilon} ]-\frac{1}{2}\varepsilon^2$ from above.  
 Since $t \leq \mathbf{A}_{t }^{T,\lambda}  $ by  Proposition~\ref{PropSquaredBessel}, we have the inequality below. 
\begin{align*}
    \mathbf{E}_{0}^{T,\lambda }\big[\, \varrho^{\uparrow, \varepsilon}\, \big] \,\leq \,   \mathbf{E}_{0}^{T,\lambda }\Big[\,\mathbf{A}_{\varrho^{\uparrow, \varepsilon} }^{T,\lambda} \, \Big] \,=\, \mathbf{E}_{0}^{T,\lambda}\Big[\, \mathbf{S}_{\varrho^{\uparrow, \varepsilon} }^{T,\lambda}  \,\Big]\,=\,\mathbf{E}_{0}^{T,\lambda}\Big[\,\eta_{T}^{\lambda}\big(\varrho^{\uparrow, \varepsilon}, X_{\varrho^{\uparrow, \varepsilon}}\big) \,\Big]
\end{align*}
The first equality above applies the optional stopping theorem to the martingale $\mathbf{M}^{T,\lambda}=\mathbf{S}^{T,\lambda}-\mathbf{A}^{T,\lambda} $, which has mean zero when $X_0=0$.  Hence we have the upper bound
\begin{align*}
     \mathbf{E}_{0}^{T,\lambda}\big[ \,\varrho^{\uparrow, \varepsilon} \,\big]\,-\, \frac{\varepsilon^2 }{2}
   \leq   \,\underbrace{\mathbf{E}_{0}^{T,\lambda}\bigg[\,\bigg(\,\eta_{T}^{\lambda}\big(\varrho^{\uparrow, \varepsilon}, X_{\varrho^{\uparrow, \varepsilon}}\big) \,-\, \frac{\varepsilon^2}{2}\,\bigg) \,1_{ \varrho^{\uparrow, \varepsilon}\leq \frac{T}{2}   } \,\bigg]}_{\textup{(I)}}   \,+\,  \underbrace{\mathbf{E}_{0}^{T,\lambda}\Big[\,\eta_{T}^{\lambda}\big(\varrho^{\uparrow, \varepsilon}, X_{\varrho^{\uparrow, \varepsilon}}\big)\, 1_{ \varrho^{\uparrow, \varepsilon} > \frac{T}{2}  }\,\Big]}_{\textup{(II)}} \,.
\end{align*}
  Since $\eta_{T}^{\lambda}(0, x)=\frac{1}{2}|x|^2$, the fundamental theorem of calculus yields $
    \eta_{T}^{\lambda}(t, x)- \frac{|x|^2}{2}= \int_0^{t}  \frac{\partial}{\partial r}\eta_{T}^{\lambda}(r, x) \,dr
$, and so
\begin{align*}
\textup{(I)} \,= \,&\, \mathbf{E}_{0}^{T,\lambda}\left[\,1_{ \varrho^{\uparrow, \varepsilon}\leq \frac{T}{2}  }\,\int_0^{\varrho^{\uparrow, \varepsilon}  }\,\Big(\frac{\partial}{\partial r}\eta_{T}^{\lambda}\Big)\big(r, X_{\varrho^{\uparrow, \varepsilon}}\big)\, dr \,\right] \,\leq \, \mathbf{c}_L\,\varepsilon^2 \,\mathbf{E}_{0}^{T,\lambda}\big[\,\varrho^{\uparrow, \varepsilon} \, \big] \,\stackrel{ (i) }{\preceq}\, \varepsilon^4 \,,
\end{align*}
where the first inequality holds for $ \mathbf{c}_L>0$   defined as below, since   $|X_{ \varrho^{\uparrow, \varepsilon}  }|= \varepsilon \leq \frac{1}{2}$.
\begin{align*}
    \mathbf{c}_L\, :=\,  \sup_{\substack{ |x|\leq \frac{1}{2} \\  t\in [0,\frac{T}{2}] \\  \frac{1}{L}  \leq  T,\lambda \leq L   }   }\,\frac{1}{|x|^2}\, \frac{\partial}{\partial t}\eta_{T}^{\lambda}(t, x)  \,\leq \, 2\,\Bigg( \sup_{\substack{ |x|\leq \frac{1}{2}  \\  t\in [\frac{T}{2},T] \\  \frac{1}{L}  \leq  T,\lambda \leq L   }   } \,\frac{ \frac{\partial}{\partial t}\bar{H}_{t}^{\lambda}(x)  }{ 1+\bar{H}_{t}^{\lambda}(x) } \Bigg)\, \Bigg(\sup_{\substack{ |x|\leq \frac{1}{2} \\  \frac{1}{L}  \leq  T,\lambda \leq L   }  }
 \frac{\eta_{T}^{\lambda}(T, x)}{|x|^2}\Bigg) 
\end{align*}
The above inequality can be seen from the expression for $ \frac{\partial}{\partial t}\eta_{T}^{\lambda}(t, x)$ in~(\ref{IneqEta2}), and the right side  is finite as a consequence of Proposition~\ref{PropK}.  This gives us our needed bound for (I).
Next, we can bound (II) as follows since $|X_{\varrho^{\uparrow, \varepsilon}}|= \varepsilon\leq \frac{1}{2}$:
 \begin{align}\label{(II)Bound}
\textup{(II)} \,\leq \,\varepsilon^2\,\log \frac{1}{\varepsilon} \, \Bigg(\sup_{ \substack{ |x| \leq \frac{1}{2}\\ t\in (0,\infty)  \\  \frac{1}{L}  \leq  T,\lambda \leq L  }    } \,\frac{ \eta^{\lambda}_T(t,x)   }{ |x|^2\log \frac{1}{|x|}  } \Bigg)\, \mathbf{P}_{0}^{T,\lambda}\Big[ \,\varrho^{\uparrow, \varepsilon} > \frac{T}{2} \,  \Big] \,\preceq\,&\,\varepsilon^4\, \log^2\frac{1}{\varepsilon} \,.
\end{align}
The supremum above is finite since
\begin{align*}
    \sup_{t\in [0,\infty)}\,\eta_{T}^{\lambda}(t,x)\,= \, \eta_{T}^{\lambda}(T,x)  \,= \, 2\,\int_0^{|x|} \, \int_0^b\, \frac{a}{b} \,\big( 1+\bar{H}_{T}^{\lambda}(a)\big)^2  \,   da\, db\,\preceq \,|x|^2\,\log^2\frac{1}{|x|} \,,
\end{align*}
where the  inequality applies Proposition~\ref{PropK} again. The second inequality in~(\ref{(II)Bound}) holds since $\mathbf{P}_{0}^{T,\lambda}[ \varrho^{\uparrow, \varepsilon} > \frac{T}{2}   ]\preceq \varepsilon^2$ by   Chebyshev's inequality and part (i).
\end{proof}

\subsection{Proof of Lemma~\ref{LemmaKbounds1} } \label{SubsectionLemmaKbounds1}

\begin{proof} We will refer to the set of $T,\lambda>0$ such that $T\lambda\in \big[\frac{1}{L},L\big] $ as the \textit{given range} of the variables $T,\lambda$.  For real-valued functions $f$ and $g$ on $\R^2\times (0,\infty)\times (0,\infty) $, we will write $f\preceq g$ if there is a constant $c_L>0$ such that $f(x, T,\lambda)\leq c_Lg(x,T,\lambda)$ for all $x\in \R^2$ and $T,\lambda$ in the given range.  Since $V_t^{\lambda}(y)= \frac{  b_{t}^{\lambda}(y)  }{  1+H_{t}^{\lambda}( y)  }$, the inner integral in~(\ref{SupSquar})  is equal to  
\begin{align*}
\int_{\R^2} \, & \mathlarger{d}_{0,s}^{T,\lambda}(x,y)\, \frac{  |b_{T-s}^{\lambda}(y)|^2  }{\big(  1+H_{T-s}^{\lambda}( y)  \big)^2}\, dy \nonumber  \\
&\,=  \,  \frac{1}{1+ H_{T}^{\lambda}(x)}\, \int_{\R^2}  \,\Big(\, g_s(x-y) + h_{s}^{\lambda}(x, y)\, \Big) \frac{\big| b_{T-s}^{\lambda}(y)\big|^2}{1+ H_{T-s}^{\lambda}(y)}\, dy \nonumber \\
&\,=  \,  \underbrace{\frac{1}{1+ H_{T}^{\lambda}(x)} }_{\leq \,1} \,\underbrace{\int_{\R^2} \, g_s(x-y) \,\frac{\big| b_{T-s}^{\lambda}(y)\big|^2}{1+ H_{T-s}^{\lambda}(y)}\, dy}_{ J_s^{T,\lambda}(x)  } \,+ \,  \underbrace{\frac{1}{1+ H_{T}^{\lambda}(x)}\, \int_{\R^2}  \,h_{s}^{\lambda}(x, y)\, \frac{\big| b_{T-s}^{\lambda}(y)\big|^2}{1+ H_{T-s}^{\lambda}(y)}\,dy}_{   K_s^{T,\lambda}(x)   } \,. 
\end{align*}
In the analysis below, we will show that $ \sup_{x\in \R^2} \int_0^T K_s^{T,\lambda}(x)ds  $ is uniformly bounded for all $\lambda,T>0 $ in the given range. A similar analysis, which we omit, can be used to bound $ \sup_{x\in \R^2} \int_0^T J_s^{T,\lambda}(x)ds  $. 
By  Proposition~\ref{hFunction}, there exists a $C_L>0$ such that for all $x,y\in \R^2$ and $s,\lambda>0$ with $s\lambda \leq L$
\begin{align*}
h^{\lambda}_s(x,y)\,\leq \,\frac{  C_L}{s}\,\frac{1}{   1+ \log^+ \frac{1}{s\lambda }  }\,\bigg(\, e^{-\frac{|x|^2}{2s}}+\log^{+}\frac{2s}{|x|^2}  \,\bigg)\,\bigg(\, e^{-\frac{|y|^2}{2s}}+\log^{+}\frac{2s}{|y|^2}\,  \bigg)\,,
\end{align*}
and  $\frac{1}{1+H^{\lambda}_T(x)}  \leq C_L\frac{1+\log^+\frac{1}{T\lambda}    }{1+\log^+ \frac{2T}{|x|^2}   }$ in consequence of (i) of Corollary~\ref{CorollaryThePandQ}.  It follows that for all $x\in \R^2$, $T,\lambda>0$ in the given range,  and $s\in [0,T]$
\begin{align}
  K_s^{T,\lambda}(x)\,\leq \,&\,\frac{C_L^2}{s}\,\frac{   1+ \log^+ \frac{1}{T\lambda }   }{   1+ \log^+ \frac{1}{s\lambda}   } \,\underbrace{\frac{ 1+\log^{+}\frac{2s}{|x|^2}  }{    1+\log^{+}\frac{2T}{|x|^2} }}_{\leq \, 1  }\,\int_{\R^2} \, \bigg(\, e^{-\frac{|y|^2}{2s}}+\log^{+}\frac{2s}{|y|^2} \, \bigg) \,\frac{\big| b_{T-s}^{\lambda}(y)\big|^2}{1+ H_{T-s}^{\lambda}(y)}\,dy \nonumber \\
  \,\leq \,&\,\frac{\mathbf{C}_L}{s\,\big( 1+ \log^+ \frac{1}{s\lambda } \big) }\,\int_{\R^2}  \,\bigg(\, e^{-\frac{|y|^2}{2s}}+\log^{+}\frac{2s}{|y|^2}  \,\bigg) \frac{\big| b_{T-s}^{\lambda}(y)\big|^2}{1+ H_{T-s}^{\lambda}(y)}\,dy\,, \label{KayX}
\end{align}
where the second inequality holds with $\mathbf{C}_L:= (1+\log L)C_L^2 $   since $T\lambda>\frac{1}{L} $.   By combining (i) of  Corollary~\ref{CorollaryThePandQ} and Lemma~\ref{LemDrift},  we get that
\begin{align}\label{CombineLemma}
\frac{\big| b_{t}^{\lambda}(y)\big|^2}{1+ H_{t}^{\lambda}(y) } \,\preceq \, \frac{ 1+\log^+\frac{1}{t\lambda} }{|y|^2\big(1+\log^+ \frac{2t}{|y|^2}  \big)^3 }\,1_{|y|^2 <  2t^2\lambda    }   \,+\,  \frac{e^{-\frac{|y|^2}{t}}    }{|y|^2\big(1+\log^+\frac{1}{t\lambda}\big)^2    }\,1_{|y|^2 \geq  2t^2 \lambda  } \,.
\end{align}
Applying~(\ref{CombineLemma}) in~(\ref{KayX}) yields that $ K_s^{T,\lambda}(x) \preceq  \hat{K}_s^{T,\lambda} + \check{K}_s^{T,\lambda} $ for
\begin{align*}
      \hat{K}_s^{T,\lambda}
        \, := \,&\,\frac{1}{s\big( 1+ \log^+ \frac{1}{s\lambda } \big)}\, \int_{|y|^2< 2(T-s)^2 \lambda}   \,\frac{\big( e^{-\frac{|y|^2}{2s}}+\log^{+}\frac{2s}{|y|^2}  \big) \big(1+\log^+\frac{1}{(T-s)\lambda}  \big)  }{|y|^2\big(1+\log^+ \frac{2(T-s)}{|y|^2}  \big)^3 }\,dy \,,\\
       \check{K}_s^{T,\lambda}\,
        \, := \,&\, \frac{1}{s\big( 1+ \log^+ \frac{1}{s\lambda } \big)}\,\int_{|y|^2\geq  2(T-s)^2 \lambda } \,\frac{\big( e^{-\frac{|y|^2}{2s}}+\log^{+}\frac{2s}{|y|^2}  \big)\,e^{-\frac{|y|^2}{T-s}}    }{|y|^2\big(1+\log^+\frac{1}{(T-s)\lambda}\big)^2    } \,dy\,.
    \end{align*}
Through going to the radial variable $R=\frac{|y|^2}{2(T-s)}$, we can express $\hat{K}_s^{T,\lambda}$ and $\check{K}_s^{\,T,\lambda}$ as
\begin{align*}
      \hat{K}_s^{T,\lambda}
        \, = \,&\,\frac{1}{s\big( 1+ \log^+\frac{1}{s\lambda } \big)} \,\int_{0}^{(T-s) \lambda} \,  \frac{\big( e^{-\frac{T-s}{s}R}+\log^{+}\frac{s}{(T-s)R}  \big)\big(1+\log^+\frac{1}{(T-s)\lambda} \big)  }{R\big(1+\log^+ \frac{1}{R}  \big)^3 }\,dR \,, \\
       \check{K}_s^{T,\lambda}\,
        \, = \,&\,\frac{1}{s\big( 1+\log^+ \frac{1}{s\lambda } \big)}\, \int^{\infty}_{(T-s) \lambda} \, \frac{\big( e^{-\frac{T-s}{s}R}+\log^{+}\frac{s}{(T-s)R}  \big)e^{-2R }    }{R\big(1+\log^+\frac{1}{(T-s)\lambda}\big)^2    }\, dR\,.
    \end{align*}
The integral of $ \hat{K}_s^{T,\lambda}$ over $s\in [0,T]$ has the following bound
\begin{align}  \label{Yabuz3}
\int_0^T\,   \hat{K}_s^{T,\lambda}\,ds \, =\,& \, \int_0^T \,\int_{0  }^{(T-s)\lambda} \, \frac{\big(e^{-\frac{T-s}{s}R}+\log^{+}\frac{s}{(T-s)R}  \big) \big(1+\log^+\frac{1}{(T-s)\lambda}  \big)  }{s\big( 1+ \log^+ \frac{1}{s\lambda } \big)R\big(1+\log^+ \frac{1}{R}  \big)^3 }\,dR \, ds \nonumber  \\
 \, =\,& \,\int_{0  }^{T\lambda} \, \frac{1 }{R\big(1+\log^+ \frac{1 }{R}  \big)^3 } \,\int_0^{1-\frac{R}{T\lambda}} \, \frac{\big( e^{-\frac{1-a}{a}R}+\log^{+}\frac{a}{(1-a)R}  \big) \big(1+\log^+\frac{1}{(1-a)T\lambda} \big) }{a\big( 1+ \log^+ \frac{1}{T\lambda a} \big)}\,  da\, dR  \nonumber  \\
 \, \leq \,& \,\int_{0  }^{L} \, \frac{e^R }{R\big(1+\log^+ \frac{1 }{R}  \big)^3 }\, \underbrace{\int_0^{1}\,  \frac{\big( e^{-\frac{R}{a}}+\log^{+}\frac{a}{(1-a)R}  \big) \big(1+\log^+\frac{L}{1-a} \big)}{a\big( 1+ \log^+ \frac{1}{L a} \big)}\,   da}_{\textup{(A)}}\, dR     \,.
\end{align}
The second equality swaps the order of integration and changes to the integration variable $a=\frac{s}{T}$.  For the inequality, we have used  that $e^R>1$, $\frac{1}{L}\leq T\lambda\leq L$, and  $\log^+$ is increasing.  We will argue below that 
\begin{align}\label{PreLook}
\textup{(A)}\,\preceq  \,\Big(\,1+\log^+\frac{1}{R} \,\Big)\,\Big(1+\log\Big(1+\log^+\frac{1}{R}\Big)\Big)\,.
\end{align}
Notice that by applying~(\ref{PreLook}) in~(\ref{Yabuz3}) we find that $\int_0^T   \hat{K}_s^{T,\lambda}ds $ is uniformly bounded for the given range of $\lambda,T$.  Towards showing~(\ref{PreLook}), we first observe that
\begin{align}\label{Intermed}
 \int_0^{\frac{1}{2}} \,\frac{ \log^{+}\frac{2a}{R} }{a\big( 1+ \log^+ \frac{1}{L a} \big)}\, da \,=\,1_{R< 1 }\,\int_{\frac{R}{2} }^{\frac{1}{2}} \,\frac{ \log^+\frac{2a}{R} }{a\big( 1+ \log^+ \frac{1}{L a} \big)}\, da  
 \,\leq  \,&\,1_{R< 1 }\,\log^+\frac{1}{R}\,\int_{\frac{LR}{2} }^{\frac{L}{2}} \,\frac{ 1 }{a\big( 1+ \log^+ \frac{1}{a} \big)}\, da   \nonumber
   \\
 \,\leq\,&\,c_L\, 1_{R< 1 }\,\Big(\log^+\frac{1}{R}\Big)\,\Big(1+\log\Big(1+\log^+\frac{1}{R}\Big)\Big)\,,
\end{align}
where the last inequality can be shown to hold for some $c_L>0$ using that $\int \frac{1}{a (1+\log \frac{1}{a})  }da= \log\big(1+\log \frac{1}{a}\big)$.  The expression (A)  is bounded by
\begin{align*}
 \textup{(A)} \,\leq \,&\, \big(1+\log^+(2L) \big)\,\Bigg(\int_0^{\frac{1}{2} } \frac{ e^{-\frac{R}{a}} }{a}\, da \,+\, \int_0^{\frac{1}{2} } \frac{ \log^{+}\frac{2a}{R} }{a\big( 1+ \log^+ \frac{1}{L a} \big)}\, da \Bigg)\nonumber  \\
 &\,+\, 2\int_{\frac{1}{2} }^1 \bigg(1+ \log^+\frac{1}{(1-a)R}\bigg)\bigg(1+\log^+\frac{1}{1-a}+\log^+L \bigg) \, da  \nonumber 
  \\  \,\leq \,&\, \big(1+\log^+ (2L) \big) \bigg( E(R)\,+\,c_L 1_{R< 1 }\Big(\log^+\frac{1}{R}\Big)\Big(1+\log^+\Big(1+\log^+\frac{1}{R}\Big) \Big) \bigg) \nonumber  \\
 &\,+\, 2\int_{\frac{1}{2} }^1 \bigg(1+ \log^+\frac{1}{R}+ \log\frac{1}{1-a} \bigg)\bigg(1+\log\frac{1}{1-a}+\log^+L \bigg) \, da\,,
\end{align*}
in which we have used that $E(R)=\int_0^1\frac{e^{-\frac{R}{a}}  }{a}da$ and~(\ref{Intermed}). The above is easily seen to be bounded by the right side of~(\ref{PreLook}) since  $E(R)$ is bounded by a multiple of $1+\log^+\frac{1}{R}$.\vspace{.1cm}

For the integral of  $\check{K}_s^{T,\lambda}$ over $[0,T]$, we observe that
\begin{align}\label{Heyble}  
\int_0^T \check{K}_s^{T,\lambda}\,ds\, =\,& \, \int_0^T\int^{\infty  }_{(T-s)\lambda}   \frac{\big( e^{-\frac{T-s}{s}R}+\log^{+}\frac{s}{(T-s)R}  \big)e^{-2R }    }{s R\big( 1+ \log^+ \frac{1}{s\lambda } \big)\big(1+\log^+\underbracket{\frac{1}{(T-s)\lambda}}\big)^2    }\,dR\,  ds \nonumber  \\
\, \leq \,& \, \int_0^T\int^{\infty  }_{(T-s)\lambda}   \frac{\big( e^{-\frac{T-s}{s}R}+\log^{+}\frac{s}{(T-s)R}  \big)e^{-2R }    }{s R\big( 1+ \log^+ \frac{1}{s\lambda } \big)\big(1+\log^+\underbracket{\,\frac{1}{R}\,}\big)^2    }\,dR\,  ds \nonumber   \\
 \, =\,& \, \int_0^{\infty}  \frac{  e^{- 2R     }   }{R\big(1+\log^+ \frac{1}{R}  \big)^2 }   \int^{1  }_{\textup{max}(0,1-\frac{ R }{ T\lambda  }) }\frac{ e^{-\frac{1-a}{a}R}+\log^{+}\frac{a}{(1-a)R}  }{a\big( 1+ \log^+ \frac{1}{T\lambda a} \big)} \,dR  \nonumber   \\
 \, \leq \,& \, \int_0^{\infty}  \frac{  e^{- 2R     }   }{R\big(1+\log^+ \frac{1}{R}  \big)^2 }   \underbrace{\int^{1  }_{\textup{max}(0,1-LR ) }\frac{ e^{-\frac{1-a}{a}R}+\log^{+}\frac{a}{(1-a)R}  }{a\big( 1+ \log^+ \frac{1}{L a} \big)} \,dR}_{ \textup{(B)} } 
 \,.
\end{align}
The first inequality  involves only the underbracketed terms, using that $R\geq (T-s)\lambda$, and  the second inequality invokes  the constraint $\frac{1}{L}\leq T\lambda \leq L $.   The second equality swaps the order of the integration and changes the integration variable to $a=\frac{s}{T}$.  In the analysis below,  we will show that 
\begin{align}\label{PreLook2}
\textup{(B)}\,\preceq \, R \Big(1\,+\,\log^{+}\frac{1}{R} \Big) 1_{R\leq  \frac{1}{2L} }  \,+\,\Big(1+\log^+\frac{1}{R}\Big)\,1_{R >  \frac{1}{2L} }    \,,
\end{align}
and applying~(\ref{PreLook2}) in~(\ref{Heyble}) yields that $\int_0^T \check{K}_s^{T,\lambda}ds$ is bounded for the given range of $T,\lambda$.

When $R\leq  \frac{1 }{2L}$ we can bound the integral (B)  by
\begin{align} \label{Umpa1}
\int^{1  }_{1-L R }\frac{1}{a}\bigg( 1+\log^{+}\frac{1}{(1-a)R}  \bigg)\, da  \nonumber \,\leq\, & \,2\,\int^{1  }_{1-L R }\bigg( 1\,+\,\log^{+}\frac{1}{R} +\log\frac{1}{1-a}   \bigg)\, da  \nonumber 
\\ \,= \,&\,  2\,L\,R \,\Big(1\,+\,\log^{+}\frac{1}{R} \Big)\,+\,   2\int^{1  }_{1-L R  }\log\frac{1}{1-a}  \,  da  \nonumber
\\ \,= \,&\,  2\,L\,R\, \Big(1\,+\,\log^{+}\frac{1}{R} \Big)\,+\,   2\,R\, L\,\log\frac{1}{LR}   \, \preceq \, R\,\Big(1+\log^+\frac{1}{R}\Big)  \,.
\end{align}
  When $R >  \frac{1}{2L}$ we can bound (B) by
\begin{align}\label{Umpa2}
 \int^{1  }_{0}\frac{ e^{-\frac{1-a}{a}R}+\log^{+}\frac{a}{(1-a)R}  }{a\big( 1+ \log^+ \frac{1}{L a} \big)}\, da \nonumber     \,\leq \,&\, \int^{\frac{1}{2}  }_{0}\frac{ e^{-\frac{R}{2a}}+\log^{+}(4La)  }{a\big( 1+ \log^+ \frac{1}{L a} \big)}\, da\, +\,2\int^{1  }_{\frac{1}{2} } \bigg( 1+\log^{+}\frac{1}{(1-a)R}  \bigg)\, da\nonumber  \\   \,\leq \, &\,E\Big( \frac{R}{2} \Big) \,+\, \,\int^{\frac{1}{2}  }_{\frac{1}{4L}  }\frac{\log(4La)   }{a\big( 1+ \log^+ \frac{1}{L  a} \big)}\, da\, +\,2\int^{1  }_{\frac{1}{2} } \bigg( 1+\log^{+}\frac{1}{R} +\log\frac{1}{1-a} \bigg)\, da\nonumber  \\
 \,\preceq \,&\, 1+\log^+\frac{1}{R}\,,
\end{align}
in which the last inequality uses  that the exponential integral function $E(R)$ is bounded by a multiple of $1+ \log^+ \frac{1}{R}$.  Combining our results in~(\ref{Umpa1}) and (\ref{Umpa2}) gives us~(\ref{PreLook2}).  This completes our bound of $ \sup_{x\in \R^2} \int_0^T K_s^{T,\lambda}(x)ds  $. \vspace{.1cm}

  Nearly the same analysis as used to handle the case of $V_{t}^{ \lambda }$ can be  applied to $\mathring{V}_{t}^{ \lambda }:=\lambda \frac{ \partial }{\partial \lambda }b_t^{\lambda}$, where in place of~(\ref{CombineLemma})   we have a slightly stronger inequality of the form 
$$  \big(1+ H_{t}^{\lambda}(y) \big)\,\big| \mathring{V}_{t}^{ \lambda }(y)\big|^2 \,\preceq \, \frac{1 }{|y|^2\big(1+\log^+ \frac{2t}{|y|^2}  \big)^3}\,1_{|y|^2 <  2t^2\lambda    }   \,+\,  \frac{e^{-\frac{|y|^2}{t}}    }{|y|^2\big(1+\log^+\frac{1}{t\lambda}\big)^4    }\,1_{|y|^2 \geq  2t^2 \lambda  } \,,    $$
which follows from Lemma~\ref{LemmaUpsilonUpDown} and Proposition~\ref{PropBulkForDrift}.  Furthermore, to treat the case $ V_t^{\lambda,\lambda'}$ we can use that $ V_t^{\lambda,\lambda'}(y)=\int_{\lambda}^{\lambda'} \frac{1}{\ell} \mathring{V}_{t}^{\ell }(y)d\ell$ to reduce the problem to $\mathring{V}_{t}^{\lambda }$.
\end{proof}

\subsection{Proof of Lemma~\ref{LemmaRIncrease}} \label{SubsectionLemmaRIncrease}

\begin{proof} Part (i): Recall that $\bar{b}_{T}^{\lambda}(a):=\big| b_T^{\lambda}(x)  \big|$ for $a=|x|$. Since $\bar{b}_T^{\lambda}(a)=-\frac{ \frac{\partial}{\partial a} \bar{H}_{T}^{\lambda}( a) }{1+\bar{H}_{T}^{\lambda}( a)  }$, we can write
\begin{align*}
   \frac{\partial}{\partial a} \bar{R}^{\lambda,\lambda'}_{T}(a) \,=\,- \bar{R}^{\lambda,\lambda'}_{T}(a)\,\Big(\bar{b}_T^{\lambda'}(a)\,-\, \bar{b}_T^{\lambda}(a) \Big)\,.
\end{align*}
In consequence of Proposition~\ref{PropBulkForDrift},  $\bar{b}_T^{\lambda}(a)$ is increasing in $\lambda$, and so $\frac{\partial}{\partial a} \bar{R}^{\lambda,\lambda'}_{T}(a)$ has the same sign as $\lambda-\lambda'$.\vspace{.2cm}

\noindent Part (ii): Since $\bar{H}_{rT}^{\lambda}( a)=\bar{H}_{T }^{\,r\lambda}\big(\frac{a}{\sqrt{r}}\big)$ for $r>0$, we also have the scale invariance $\bar{R}^{\lambda,\lambda'}_{rT}(a)=\bar{R}^{\,r\lambda,r\lambda'}_{ T}\big(\frac{a}{\sqrt{r}} \big)$, which implies that
\begin{align*}
\frac{\partial}{\partial T}  \bar{R}^{\lambda,\lambda'}_{T}(a)\,=\,&\, \frac{\lambda'}{T}\,\frac{\partial}{\partial \lambda' } \bar{R}^{\lambda,\lambda'}_{T}(a) \,+\,\frac{\lambda}{T}\,\frac{\partial}{\partial \lambda } \bar{R}^{\lambda,\lambda'}_{T}(a)\, -\,\frac{a}{2T}\,\frac{\partial}{\partial a} \bar{R}^{\lambda,\lambda'}_{T}(a) 
\nonumber   \\
\,=\,&\,\bar{R}^{\lambda,\lambda'}_{T}(a) \,\Bigg(\underbracket{\frac{\lambda' \frac{\partial}{\partial \lambda'}\bar{H}_{T}^{\lambda'}( a) }{1+\bar{H}_{T}^{\lambda'}( a)  }  \,   -\,\frac{\lambda \frac{\partial}{\partial \lambda}\bar{H}_{T}^{\lambda}( a) }{1+\bar{H}_{T}^{\lambda}( a)  } }\Bigg)   \, -\,\frac{a}{2T}\underbracket{\frac{\partial}{\partial a} \bar{R}^{\lambda,\lambda'}_{T}(a)}\,. \nonumber 
\end{align*}
The second underbracketed term has the same sign as $\lambda-\lambda'$ by part (i).  The first underbracketed term has the same sign as $\lambda'-\lambda$ because $\frac{\ell \frac{\partial}{\partial \ell}\bar{H}_{T}^{\ell}( a) }{1+\bar{H}_{T}^{\ell}( a)  } $ is a product of the positive terms $\frac{\bar{H}_{T}^{\ell}( a) }{1+\bar{H}_{T}^{\ell}( a)  } $ and $\frac{\ell \frac{\partial}{\partial \ell}\bar{H}_{T}^{\ell}( a) }{\bar{H}_{T}^{\ell}( a)  } $, which are both increasing in $\ell$. For the former term, this  holds simply because $\bar{H}_{T}^{\ell}( a)$ is increasing in $\ell$, and for the latter term we observe that $ \ell \frac{\partial}{\partial \ell}\big[\frac{\ell \frac{\partial}{\partial \ell}\bar{H}_{T}^{\ell}( a) }{\bar{H}_{T}^{\ell}( a)  }\big] $ is the variance of a nonnegative random variable having probability density 
$ p(v) := \frac{1}{\bar{H}_{T}^{\ell}( a) } \int_0^T \frac{  e^{-\frac{a^2}{2t}  }  }{ t } \frac{ \ell^v(T-t)^v }{ \Gamma(v+1) }  dt  $. In fact, by Proposition~\ref{PropLocalTimeProp}, this is the marginal density for $L_T$ under $\mathbf{P}_{x}^{T,\lambda}$ for $x\in \R^2$ with $|x|=a$; see also Remark~\ref{RemarkMargLoc}.
\end{proof}

\subsection{Proof of Lemma~\ref{UpcroossingInequalityPre}} \label{SubsectionUpcroossingInequalityPre}

\begin{proof} Recall from Proposition~\ref{PropSubMart} that the process $\mathcal{S}^{T,\lambda}$  defined by $\mathcal{S}_{t}^{T,\lambda}:=\mathfrak{p}_{T-t}^{\lambda}(X_t)$ is a submartingale, where  $\mathfrak{p}_{t}^{\lambda}(x):=\frac{  1}{1+H_{t}^{\lambda}(x)   }$.  Since $H_{t}^{\lambda}(x)$ is increasing in $t$, the value $\mathfrak{p}_{T-t}^{\lambda}(x)$ is also increasing in $t$. Hence, for every $n\in \mathbb{N}$ with $\varrho_{n}^{\uparrow ,\varepsilon}<T$ we have
$$ \mathcal{S}_{\varrho_{n}^{\uparrow,\varepsilon}}^{T,\lambda} \,\geq \,  \mathfrak{p}_{T}^{\lambda}\big(X_{ \varrho_{n}^{\uparrow ,\varepsilon} }\big)   \,=\,\frac{  1}{1+\bar{H}_{T}^{\lambda}(\varepsilon)   } \,.  $$
 Thus, each upcrossing of the radial process $|X_t|$ from $0$ to $\varepsilon$ coincides with at least one upcrossing of  $\mathcal{S}_{t}^{T,\lambda}$ from $0$ to $\frac{  1}{1+\bar{H}_{T}^{\lambda}(\varepsilon)   } $.  If $ \mathbf{N}_{T}^{\varepsilon}$ denotes the number of  upcrossings of $\mathcal{S}_{t}^{T,\lambda}$ from $0$ to $\frac{  1}{1+\bar{H}_{T}^{\lambda}(\varepsilon)   }$ over the interval $[0,T]$, then  $N_{T}^{\varepsilon}\leq 1+\mathbf{N}_{T}^{\varepsilon} $, where adding $1$ here covers the possibility of a final uncompleted upcrossing for $|X_t|$.  Now, by the submartingale upcrossing inequality, we have
$$     \mathbf{E}^{T,\lambda}_{x}\big[ \mathbf{N}_{T}^{\varepsilon}\big] \,\leq \, \big(  1+\bar{H}_{T}^{\lambda}(\varepsilon)  \big)\,\mathbf{E}^{T,\lambda}_{x}\Big[\mathcal{S}_{T}^{T,\lambda}\Big] \,= \,   1+\bar{H}_{T}^{\lambda}(\varepsilon) \,\leq \, C_L \frac{ 1+\log^+\frac{2T}{\varepsilon^2 }  }{1+\log^+\frac{1}{T\lambda} }   \,, $$
where the  equality uses that  $\mathcal{S}_{T}^{T,\lambda}= (1+\bar{H}_{0}^{\lambda}(X_T))^{-1} =1$, and the second inequality holds by (i) of Lemma~\ref{LemmaUpsilonUpDown}.  The factor of 2 appearing in the argument $\frac{2T}{\varepsilon^2}$ of the top logarithm can be omitted by increasing $C_L$.
\end{proof}

\subsection{Proof of Lemma~\ref{LemmaUprossingLocal}} \label{SubsectionLemmaUprossingLocal}

We will need the approximation and derivative bound stated in the following lemma.
\begin{lemma} \label{LemmaRbarFunct} For  $T,\lambda,a>0$  let $\bar{R}_{T}^{\lambda}(a)$  be defined as in~(\ref{RThing}).
\begin{enumerate}[(i)]
\item $\bar{R}_{T}^{\lambda}(a)$ is positive, decreasing in $a$, and increasing  in $T$.

\item  For any $L>1$ there exists a $C_L>0$ such that the inequality $ \big| \bar{R}_{T}^{\lambda}(a)- \bar{R}_{T}^{\lambda}(0) +\frac{1}{2\log\frac{1}{a} }\big|\leq  \frac{C_L}{\log^2 \frac{1}{a} }  $  holds for all $a\in (0,\frac{1}{2})$ and  $T,\lambda \in [\frac{1}{L},  L]$.

\item  For any $L>0$ there exists a $C_L>0$ such that $-\frac{\partial}{\partial a} \bar{R}_{T}^{\lambda}(a) \leq  \frac{C_L}{a\log^2 \frac{1}{a} } $ holds for all $a\in (0,\frac{1}{2})$ and $T,\lambda>0$ with $T\lambda\leq L$.

\end{enumerate}

\end{lemma}

\begin{proof} The function $\bar{R}^{\lambda}_T$ is positive-valued because $\bar{H}_{T}^{\lambda}( a)$ is strictly increasing in $\lambda$. Since 
$ \bar{R}^{\lambda}_T(a)= \lambda \frac{\partial}{\partial \lambda}\log\big( 1+  \bar{H}_{T}^{\lambda}( a)\big)$,  we have that $-\frac{\partial}{\partial a}  \bar{R}^{\lambda}_T(a)= \lambda \frac{\partial}{\partial \lambda}\bar{b}^{\lambda}_T(a)$, and so Proposition~\ref{PropBulkForDrift}  implies the   inequality in (iii) and that $\bar{R}^{\lambda}_T(a)$ is decreasing in $a$.  With $\alpha= \frac{a^2}{2T} $ and $\beta=T\lambda$, we can use~(\ref{H2N}) and the chain rule to write 
$$ \frac{ \partial }{\partial T  }\bar{R}^{\lambda}_T(a) \,=\, \frac{ \partial }{\partial T  }\bigg[ \frac{ \beta \frac{\partial}{\partial \beta}N_{\alpha,\beta} }{1+N_{\alpha,\beta} }\bigg] \,=\, \underbracket{\lambda\,\frac{\partial}{\partial \beta}\bigg[ \frac{ \beta \frac{\partial}{\partial \beta}N_{\alpha,\beta}  }{1+N_{\alpha,\beta} }\bigg]}\,-\,\underbrace{\frac{ a^2 }{2 T^2 }  \frac{ \partial }{\partial \alpha}\bigg[  \frac{\beta \frac{\partial}{\partial \beta}N_{\alpha,\beta}  }{1+N_{\alpha,\beta} } \bigg]}_{ =\,-\frac{\lambda a}{2T} \frac{\partial}{\partial \lambda}\bar{b}^{\lambda}_T(a) }\,. $$
The above is a positive  term subtracted by a negative one, so $\bar{R}^{\lambda}_T(a)$ is increasing in $T$.  For this, observe that the first underbracketed term is positive because $\beta\frac{\partial}{\partial \beta}\big[ \frac{ \beta \frac{\partial}{\partial \beta}N_{\alpha,\beta}  }{1+N_{\alpha,\beta} }\big]$ is the variance of a nonnegative random variable that is zero with probability $\frac{1}{ 1+N_{\alpha,\beta} }$, and in the event of being positive has conditional probability density $p(v)=\frac{1}{N_{\alpha,\beta} }\int_0^1\frac{ e^{-\alpha r} }{r}\frac{ (1-r)^v\beta^v
 }{ \Gamma(v+1) }dr  $. The approximations in  Proposition~\ref{PropK} imply that there is a $C_L>0$ such that  $\big|\frac{\partial}{\partial a} \bar{R}_{T}^{\lambda}(a) +\frac{1}{2a\log^2\frac{1}{a} }\big|\leq  \frac{C_L}{a\log^3 \frac{1}{a} }   $ holds for all $a\in (0,\frac{1}{2})$ and  $T,\lambda\in \big[\frac{1}{L},  L\big]$, and integrating yields (ii). \end{proof}

\begin{proof}[Proof of Lemma~\ref{LemmaUprossingLocal}] Part (i): We can assume that $|x|  < \varepsilon$ since $\mathbf{L}_{\varrho^{\uparrow, \varepsilon} }^{T,\lambda} =\mathbf{L}_{0 }^{T,\lambda}=0$ in the case $|x|  \geq \varepsilon$.  Recall that $\mathring{\mathcal{S}}_t^{T,\lambda}=\bar{R}_{T-t}^{\lambda}\big(|X_{t}|\big)$  and that $\mathring{\mathcal{M}}^{T,\lambda}=\mathring{\mathcal{S}}^{T,\lambda}+ \mathbf{L}^{T,\lambda}$ is a $\mathbf{P}^{T,\lambda}_{x}$-martingale by Proposition~\ref{PropSubMartIII}.  Applying the optional stopping theorem, we get that $\mathbf{E}^{T,\lambda}_{x}\big[\mathring{\mathcal{M}}^{T,\lambda}_{\varrho^{\uparrow, \varepsilon}\wedge t} \big]=\mathbf{E}^{T,\lambda}_{x}\big[\mathring{\mathcal{M}}^{T,\lambda}_{0} \big]$ for any $t\in (0,\infty)$.  Note that $  \varrho^{\uparrow, \varepsilon}<\infty $ almost surely $\mathbf{P}^{T,\lambda}_{x}$ as  a consequence of Lemma~\ref{LemmaLeave}, and hence the third equality below follows from the bounded convergence theorem and the monotone convergence theorem  for $\mathring{\mathcal{S}}^{T,\lambda}$ and $\mathbf{L}^{T,\lambda}$, respectively.
\begin{align*}
\mathbf{E}^{T,\lambda}_{x}\Big[\,\mathring{\mathcal{M}}^{T,\lambda}_{0}\, \Big]=\lim_{t\rightarrow \infty} \mathbf{E}^{T,\lambda}_{x}\Big[\,\mathring{\mathcal{M}}^{T,\lambda}_{\varrho^{\uparrow, \varepsilon}\wedge t} \,\Big]=\lim_{t\rightarrow \infty} \Big(\mathbf{E}^{T,\lambda}_{x}\Big[\,\mathring{\mathcal{S}}^{T,\lambda}_{\varrho^{\uparrow, \varepsilon}\wedge t}\, \Big]  + \mathbf{E}^{T,\lambda}_{x}\Big[\,\mathbf{L}^{T,\lambda}_{\varrho^{\uparrow, \varepsilon}\wedge t} \,\Big]  \Big)=\mathbf{E}^{T,\lambda}_{x}\Big[\,\mathring{\mathcal{S}}^{T,\lambda}_{\varrho^{\uparrow, \varepsilon}} \,\Big]  \,+\,\mathbf{E}^{T,\lambda}_{x}\Big[\,\mathbf{L}^{T,\lambda}_{\varrho^{\uparrow, \varepsilon}} \,\Big]  
\end{align*}
Since $\mathring{\mathcal{M}}^{T,\lambda}_{0}= \bar{R}_{T}^{\lambda}\big(|X_{0}|\big)  $ and  $|X_{\varrho^{\uparrow, \varepsilon}}|=\varepsilon$, we can deduce from the above that
\begin{align}\label{Tin}
 \mathbf{E}^{T,\lambda}_{x}\Big[\,\mathbf{L}^{T,\lambda}_{\varrho^{\uparrow, \varepsilon} }\,\Big] \,=\,    \mathbf{E}^{T,\lambda}_{x}\Big[\,\mathring{\mathcal{M}}^{T,\lambda}_{0} \,\Big]\,-\,\mathbf{E}^{T,\lambda}_{x}\Big[\,\mathring{\mathcal{S}}^{T,\lambda}_{\varrho^{\uparrow, \varepsilon} }\,\Big] \,=\,  \bar{R}_{T}^{\lambda}\big(|x|\big) \,-\,\mathbf{E}^{T,\lambda}_{x}\Big[ \,\bar{R}_{T-\varrho^{\uparrow, \varepsilon}}^{\lambda}(\varepsilon)\,\Big]  \,,
 \end{align}
which we can rewrite and bound as follows:
\begin{align*}
&\bar{R}_{T}^{\lambda}\big(|x|\big)-\bar{R}_{T}^{\lambda}(\varepsilon) \,+\,\mathbf{E}^{T,\lambda}_{x}\Big[\,\Big(\bar{R}_{T}^{\lambda}(\varepsilon)- \bar{R}_{T-\varrho^{\uparrow, \varepsilon}}^{\lambda}(\varepsilon)\Big)\,1_{ \varrho^{\uparrow, \varepsilon}\leq \frac{T}{2} 
 }\,\Big]\,+\,\mathbf{E}^{T,\lambda}_{x}\Big[\,\Big(\bar{R}_{T}^{\lambda}(\varepsilon)- \bar{R}_{T-\varrho^{\uparrow, \varepsilon}}^{\lambda}(\varepsilon)\Big)\,1_{ \varrho^{\uparrow, \varepsilon} > \frac{T}{2} 
 }\,\Big]\\ &\,\leq \,\bar{R}_{T}^{\lambda}(0)\,-\,\bar{R}_{T}^{\lambda}(\varepsilon) \,+\,\Bigg(\sup_{\substack{t\in [\frac{T}{2},T] \\ a\in [0,1] } }\Big|\frac{d}{dt}\bar{R}_{t}^{\lambda}(a)\Big| \Bigg)\,\mathbf{E}^{T,\lambda}_{x}[\, \varrho^{\uparrow, \varepsilon} \,]\,+\,  \bar{R}_{T}^{\lambda}(\varepsilon) \,\mathbf{P}^{T,\lambda}_{x}\Big[\,\varrho^{\uparrow, \varepsilon}\geq \frac{T}{2} \,\Big] \,\preceq \, \frac{  1}{\log \frac{1}{\varepsilon} }  \,.
\end{align*}
The first inequality above uses that  $\bar{R}_{T}^{\lambda}(a)$ is decreasing in $a$ and increasing in $T$ by (i) of Lemma~\ref{LemmaRbarFunct}. For the second inequality, we observe that the difference $\bar{R}_{T}^{\lambda}(0)-\bar{R}_{T}^{\lambda}(\varepsilon)$ is  bounded by a multiple of $(\log \frac{1}{\varepsilon})^{-1}$ by (ii) of Lemma~\ref{LemmaRbarFunct}, and the other two terms are bounded by a multiple of $\varepsilon^2$ in consequence of  Lemma~\ref{LemmaLeave}.  Thus the desired inequality~(\ref{BaseII}) holds in the case $m=1$.

Suppose for the purpose of an induction argument that the inequality~(\ref{BaseII}) holds for some $m\in \mathbb{N}$.  We can write 
$$\Big(\mathbf{L}_{\varrho^{\uparrow, \varepsilon} }^{T,\lambda}\Big)^{m+1} =(m+1)\int_0^{\varrho^{\uparrow, \varepsilon} } \Big(\mathbf{L}_{\varrho^{\uparrow, \varepsilon}}^{T,\lambda}-\mathbf{L}_{t}^{T,\lambda}\Big)^{m}\,d\mathbf{L}_{t }^{T,\lambda} \,=\,(m+1)\int_0^{\infty } \Big(\mathbf{L}_{\varrho^{\uparrow, \varepsilon}}^{T,\lambda}-\mathbf{L}_{t}^{T,\lambda}\Big)^{m}1_{\,t\leq \varrho^{\uparrow, \varepsilon}  }\,d\mathbf{L}_{t }^{T,\lambda} \,. $$  
Since the process $\mathbf{L}_{t }^{T,\lambda} $ is predictable, we have that
\begin{align}\label{MPower1}
    \mathbf{E}^{T,\lambda}_{x}\bigg[\,\Big(\mathbf{L}_{\varrho^{\uparrow, \varepsilon} }^{T,\lambda}\Big)^{m+1}\,\bigg]\,=\,(m+1)\,\mathbf{E}^{T,\lambda}_{x}\bigg[\,\int_0^{\infty }\mathbf{E}^{T,\lambda}_{x}\bigg[\,\Big(\mathbf{L}_{\varrho^{\uparrow, \varepsilon}}^{T,\lambda}-\mathbf{L}_{t}^{T,\lambda}\Big)^{m}\,\bigg|\,  \mathscr{F}_t^{T,\mu}\, \bigg]1_{\,t\leq \varrho^{\uparrow, \varepsilon}  }\,d\mathbf{L}_{t }^{T,\lambda}\,\bigg]\,. 
\end{align}
Invoking the strong Markov property and our assumption that the inequality~(\ref{BaseII}) holds for  $m\in \mathbb{N}$, we get
\begin{align}\label{MPower2}
    \mathbf{E}^{T,\lambda}_{x}\bigg[\,\Big(\mathbf{L}_{\varrho^{\uparrow, \varepsilon}}^{T,\lambda}-\mathbf{L}_{t}^{T,\lambda}\Big)^{m}\,\bigg|\,  \mathscr{F}_t^{T,\mu} \,\bigg]\,=\,\mathbf{E}^{T-t,\lambda}_{X_{t}}\bigg[\,\Big(\mathbf{L}_{\varrho^{\uparrow, \varepsilon}}^{T-t,\lambda}\Big)^{m} \,\bigg]\,\leq \, \frac{ m!\, C_L^m  }{\log^m \frac{1}{\varepsilon} }  \hspace{.7cm} \text{on} \,\, \big\{ t\leq  \varrho^{\uparrow, \varepsilon}\big\}    \,.
\end{align}
Applying~(\ref{MPower2}) in~(\ref{MPower1}) yields that
$$   \mathbf{E}^{T,\lambda}_{0}\bigg[\,\Big(\mathbf{L}_{\varrho^{\uparrow, \varepsilon} }^{T,\lambda}\Big)^{m+1}\,\bigg]\,\leq \, \frac{ (m+1)!\, C_L^m  }{\log^m \frac{1}{\varepsilon} }  \mathbf{E}^{T,\lambda}_{x}\Big[\,\mathbf{L}_{\varrho^{\uparrow, \varepsilon} }^{T,\lambda}\,\Big]\,\leq \, \frac{  (m+1)!\,C_L^{m+1}  }{\log^{m+1} \frac{1}{\varepsilon} }   \,. $$
Therefore, by induction, the inequality~(\ref{BaseII}) holds for all $m\in \mathbb{N}$. \vspace{.3cm}

\noindent Part (ii): By~(\ref{Tin}) with $x=0$, we have the following equality:
\begin{align*}
 \mathbf{E}^{T,\lambda}_{0}\Big[\,\mathbf{L}^{T,\lambda}_{\varrho^{\uparrow, \varepsilon} }\,\Big] \,-\,\frac{ 1 }{2\log \frac{1}{\varepsilon} }     \,=\,  &\,\bar{R}_{T}^{\lambda}(0) \,-\, \mathbf{E}^{T,\lambda}_{0}\Big[\,\bar{R}_{T-\varrho^{\uparrow,\varepsilon}}^{ \lambda}(\varepsilon)\,\Big]\,-\,\frac{ 1 }{2\log \frac{1}{\varepsilon} }\,=\,\textup{(A)}\,+\, \textup{(B)} \,+\, \textup{(C)}\,+\, \textup{(D)}\,,
 \end{align*}
where the terms (A)--(D) are defined by
\begin{align}
  \textup{(A)}\,:=\,&\, \mathbf{E}^{T,\lambda}_{0}\bigg[\,\bigg( \bar{R}_{T-\varrho^{\uparrow, \varepsilon}}^{\lambda}(0)\,-\, \bar{R}_{T-\varrho^{\uparrow, \varepsilon}}^{\lambda}(\varepsilon)\,-\,\frac{ 1 }{2\log \frac{1}{\varepsilon} }  \bigg) 1_{ \varrho^{\uparrow, \varepsilon}\leq \frac{T}{2} } \, \bigg]\,, \hspace{.3cm} \textup{(B)}\,:=\,- \mathbf{E}^{T,\lambda}_{0}\Big[\,\bar{R}_{T-\varrho^{\uparrow,\varepsilon}}^{ \lambda}(\varepsilon)\,1_{ \varrho^{\uparrow, \varepsilon}> \frac{T}{2}   }\,\Big] \,,\nonumber \\
 \textup{(C)}\,:=\,&\,\mathbf{P}^{T,\lambda}_{0}\Big[\,  \varrho^{\uparrow, \varepsilon} > \frac{T}{2} \,\Big]\,\bigg(\bar{R}_{T}^{\lambda}(0) \,-\,\frac{ 1 }{2\log \frac{1}{\varepsilon} } \bigg)  \,,  \hspace{.5cm} \textup{(D)}\,:=\, \mathbf{E}^{T,\lambda}_{0}\bigg[\,\Big(\bar{R}_{T}^{\lambda}(0) -\bar{R}_{T-\varrho^{\uparrow,\varepsilon}}^{ \lambda}(0)\Big)1_{ \varrho^{\uparrow, \varepsilon}\leq \frac{T}{2}   }\,\bigg]\,.  \nonumber   
\end{align}
The absolute value of (A) is bounded by a constant multiple of $(\log \frac{1}{\varepsilon} )^{-2}$ by Lemma~\ref{LemmaRbarFunct}.  For (B) we have
\begin{align*}
|  \textup{(B)}|\,\leq \,   \bar{R}_{T}^{\lambda}(0) \,\mathbf{P}^{T,\lambda}_{0}\Big[\, \varrho^{\uparrow, \varepsilon}> \frac{T}{2}   \,\Big]\,\preceq\,\varepsilon^2 \,,
\end{align*}
where the first inequality uses that $\bar{R}_{T}^{ \lambda}(a)$ is decreasing in $a$  and increasing in $T$, and the second inequality holds by (i) of Lemma~\ref{LemmaLeave} and  Chebyshev's inequality.   We can bound (C) similarly.  Finally, for (D) we observe that
\begin{align*}
|  \textup{(D)}|\,\leq \, &\, \bigg(\sup_{t\in [\frac{T}{2},T]}\Big| \frac{\partial}{\partial t} \bar{R}_{t}^{\lambda}(0) \Big|\bigg)\, \mathbf{E}^{T,\lambda}_{0}\big[\,\varrho^{\uparrow, \varepsilon}\,\big] \, \preceq\, \varepsilon^2\,,
\end{align*}
where the second inequality applies (i) of Lemma~\ref{LemmaLeave} again.
\end{proof}

\subsection{Proof of Lemma~\ref{LemmaDowncrossingNumber}} \label{SubsectionLemmaDowncrossingNumber}

\begin{proof} \noindent Part (i): In the analysis below, we will show that there exists a $\mathbf{C}_L>0$ such that for all $\varepsilon \in (0,\frac{1}{2})$, $T,\lambda \leq  L$, and  $x\in \R^2 $ with $|x|=\varepsilon$ 
\begin{align}\label{Prelimina}
    \mathbf{P}^{T,\lambda }_x\Big[ \,\varrho^{\downarrow, \varepsilon}_1 < \varrho^{\uparrow, 2\varepsilon} \, \Big]\,\leq \,\frac{ \mathbf{C}_L }{\log \frac{1}{\varepsilon} }\,.
\end{align}
Through the strong Markov property this implies that $ \mathbf{P}^{T,\lambda }_x\big[ \varrho^{\downarrow, \varepsilon}_n < \varrho^{\uparrow, 2\varepsilon}  \big]\leq \big(\frac{ \mathbf{C}_L }{\log \frac{1}{\varepsilon}}\big)^n$.  Observe that
\begin{align*}
(\mathbf{n}^{\varepsilon})^p \,=\, \Bigg(\sum_{n=1}^{\infty}1_{\varrho^{\downarrow, \varepsilon}_n < \varrho^{\uparrow, 2\varepsilon} } \Bigg)^p   \,= \,  \sum_{n=1}^{\infty}\alpha_{p,n} 1_{\varrho^{\downarrow, \varepsilon}_n < \varrho^{\uparrow, 2\varepsilon} } 
\end{align*}
for coefficients $\alpha_{p,n} \in \mathbb{N} $ with $\alpha_{p,1}=1$ and $\alpha_{p,n} \leq  n^{p} $. Hence we can write
\begin{align}\label{ENNY}
     \mathbf{E}^{T,\lambda }_x\big[\, (\mathbf{n}^{\varepsilon})^p\,\big] \,= \, \sum_{n=1}^{\infty} \alpha_{p,n}  \,\mathbf{P}^{T,\lambda }_x\Big[\, \varrho^{\downarrow, \varepsilon}_n < \varrho^{\uparrow, 2\varepsilon} \,\Big] \,\leq \, \sum_{n=1}^{\infty} n^p  \bigg(\frac{ \mathbf{C}_L }{ \log\frac{1}{\varepsilon}     }\bigg)^n \,\leq \,\frac{ 2\mathbf{C}_L }{\log \frac{1}{\varepsilon}} \,,
\end{align}
where the last inequality holds for small enough $\varepsilon>0$. Thus, our desired inequality follows from~(\ref{Prelimina}).\vspace{.2cm}

Define   $\bar{\mathfrak{p}}_t^{\lambda}(a)=\frac{1}{1+ \bar{H}^{\lambda}_t(a)  }$ for $a\in [0,\infty)$.
  The process $\mathcal{S}^{T,\lambda}_t=\bar{\mathfrak{p}}^{\lambda}_{T-t}(|X_t|)$ is a (bounded)  martingale over the interval $\big[0,\varrho^{\downarrow, \varepsilon}_1 \big]$ in consequence of Proposition~\ref{PropSubMart}.   By the optional stopping theorem,
\begin{align}\label{KBar}
 \bar{\mathfrak{p}}_T^{\lambda}(\varepsilon )  \,=\,  \mathbf{E}^{T,\lambda }_x\Big[\,\mathcal{S}^{T,\lambda}_0\,\Big]\, =\,\mathbf{E}^{T,\lambda }_x\Big[\,\mathcal{S}^{T,\lambda}_{ \varrho^{\uparrow, 2\varepsilon}\wedge \varrho^{\downarrow, \varepsilon}_1  }\,\Big]\,=\,\mathbf{E}^{T,\lambda }_x\Big[\,\bar{\mathfrak{p}}_{T-  \varrho^{\uparrow, 2\varepsilon}}^{\lambda}(2\varepsilon )\, 1_{ \varrho^{\uparrow, 2\varepsilon}<\varrho^{\downarrow, \varepsilon}_1}\, \Big]\,,
\end{align}
in which the third equality uses that $\mathcal{S}^{T,\lambda}_{ \varrho^{\uparrow, 2\varepsilon}\wedge \varrho^{\downarrow, \varepsilon}_1  }=0$ in the event that $\varrho^{\uparrow, 2\varepsilon}>\varrho^{\downarrow, \varepsilon}_1$.  Observe that 
\begin{align*}
    \bigg| & \mathbf{E}^{T,\lambda }_x\Big[\,\bar{\mathfrak{p}}_{T-  \varrho^{\uparrow, 2\varepsilon}}^{\lambda}(2\varepsilon )\,1_{ \varrho^{\uparrow, 2\varepsilon}<\varrho^{\downarrow, \varepsilon}_1}\, \Big] \,-\,  \bar{\mathfrak{p}}_T^{\lambda}(2\varepsilon )\,\mathbf{P}^{T,\lambda }_x\Big[\,\varrho^{\uparrow, 2\varepsilon}<\varrho^{\downarrow, \varepsilon}_1\,\Big] \bigg| \nonumber \,\leq\,\Bigg(\sup_{ \substack{\lambda, t\leq L  \\ a\in [0,2]   }    } \frac{\partial}{\partial t}  \bar{H}^{\lambda}_{t }(a)  \Bigg) \, \mathbf{E}^{T,\lambda }_x\big[\,\varrho^{\uparrow, 2\varepsilon}\,\big]\,\preceq \,\varepsilon^2\,,
\end{align*}
where the last inequality holds by part (i) of Lemma~\ref{LemmaLeave}.  Using~(\ref{KBar}) and that the difference between $\mathbf{E}^{T,\lambda }_x\big[\bar{\mathfrak{p}}_{T-  \varrho^{\uparrow, 2\varepsilon}}^{\lambda}(2\varepsilon )1_{ \varrho^{\uparrow, 2\varepsilon}<\varrho^{\downarrow, \varepsilon}_1} \big]$ and $\bar{\mathfrak{p}}_T^{\lambda}(2\varepsilon )\mathbf{P}^{T,\lambda }_x\big[\varrho^{\uparrow, 2\varepsilon}<\varrho^{\downarrow, \varepsilon}_1\big]$ is bounded by a multiple  $C_L>0$ of $\varepsilon^2$, we get that
\begin{align}\label{K2KDiff}
    \bigg| \bar{\mathfrak{p}}_T^{\lambda}(\varepsilon ) \,-\,\bar{\mathfrak{p}}_T^{\lambda}(2\varepsilon )\,\mathbf{P}^{T,\lambda }_x\Big[\varrho^{\uparrow, 2\varepsilon}<\varrho^{\downarrow, \varepsilon}_1\Big] \bigg| \,\leq \,C_L\,\varepsilon^2\,.
\end{align}
Thus, since the probabilities $\mathbf{P}^{T,\lambda }_x\big[\varrho^{\uparrow, 2\varepsilon}<\varrho^{\downarrow, \varepsilon}_1\big] $ and 
 $\mathbf{P}^{T,\lambda }_x\big[\varrho^{\downarrow, \varepsilon}_1<\varrho^{\uparrow, 2\varepsilon}\big] $ sum to $1$, we can deduce from~(\ref{K2KDiff}) that
\begin{align*}
   \mathbf{P}^{T,\lambda }_x\Big[\,\varrho^{\downarrow, \varepsilon}_1<\varrho^{\uparrow, 2\varepsilon}\,\Big]  \,\leq\,&\, \frac{ \bar{H}^{\lambda}_T(\varepsilon )- \bar{H}^{\lambda}_{T }(2\varepsilon)  }{1+ \bar{H}^{\lambda}_T(\varepsilon )  } \,+\,C_L\Big(  1+ \bar{H}^{\lambda}_{T }(2\varepsilon)   \Big)\,\varepsilon^2 \nonumber \\ 
   \,\leq\,&\, \frac{ \varepsilon }{1+ \bar{H}^{\lambda}_T(\varepsilon )  } \sup_{a\in [\varepsilon, 2\varepsilon ]} -\frac{\partial}{\partial a}\bar{H}^{\lambda}_T(a )  \,+\,C_L\,\Big(  1+ \bar{H}^{\lambda}_{T }(2\varepsilon)   \Big)\,\varepsilon^2\,.
\end{align*}
Applying Lemma~\ref{LemmaUpsilonUpDown} completes the proof.

\vspace{.5cm}

\noindent Part (ii):  Rearranging~(\ref{K2KDiff}), we can write 
\begin{align*}
    \bigg|  \frac{ \bar{H}^{\lambda}_T(\varepsilon )- \bar{H}^{\lambda}_{T }(2\varepsilon)  }{1+ \bar{H}^{\lambda}_T(\varepsilon )  }   \,-\,\mathbf{P}^{T,\lambda }_x\Big[\,\varrho^{\downarrow, \varepsilon}_1 <\varrho^{\uparrow, 2\varepsilon}\,\Big] \bigg| \,\leq\, C_L\,\Big(1+ \bar{H}^{\lambda}_{T }(2\varepsilon)  \Big)\,\varepsilon^2\,.
\end{align*}
In consequence of the above and (i) of Proposition~\ref{PropK},  there is a $\mathbf{C}_L>0$ such that for all $\varepsilon \in (0,\frac{1}{2})$, $T,\lambda\in \big[\frac{1}{L}, L]$ and  $x\in \R^2 $ with $|x|=\varepsilon$
\begin{align*}
\bigg|\frac{ \log 2  }{\log  \frac{1}{\varepsilon}  }  \,-\, \mathbf{P}^{T,\lambda }_x\Big[ \,\varrho^{\downarrow, \varepsilon}_1 < \varrho^{\uparrow, 2\varepsilon}  \,\Big] \bigg| \,\leq \, \frac{ \mathbf{C}_L}{\log^2 \frac{1}{\varepsilon}  }\,. 
\end{align*}
Applying similar reasoning as in~(\ref{ENNY}), we have
\begin{align*}
    \mathbf{E}_{x}^{T,\lambda}[ \,\mathbf{n}^{\varepsilon}\, ]\,-\,\mathbf{P}^{T,\lambda }_x\Big[ \,\varrho^{\downarrow, \varepsilon}_1< \varrho^{\uparrow, 2\varepsilon} \, \Big]\,=\, \sum_{n=2}^{\infty} \mathbf{P}^{T,\lambda }_x\Big[\, \varrho^{\downarrow, \varepsilon}_n< \varrho^{\uparrow, 2\varepsilon} \, \Big]\,\leq\,\sum_{n=2}^{\infty} \frac{ \mathbf{C}_L^n }{\log^n \frac{1}{\varepsilon} } \,\leq \,\frac{ 2\mathbf{C}_L^2 }{\log^2\frac{1}{\varepsilon} }\,,
\end{align*}
where the last inequality holds for small enough $\varepsilon>0$.  The above inequalities imply our desired result.
\end{proof}

\subsection{Proof of Lemma~\ref{LemmaDowncrossingIntegral}}\label{SubsectionLemmaDowncrossingIntegral}

The following lemma presents a Green's function for  the time-occupation density of a dimension-2 Bessel process up to a passage time.
\begin{lemma}\label{LemmaBessel}  Let $\{R_t\}_{t\in [0,\infty)}$ be a dimension-$2$ Bessel process with initial value $x\in [0,\infty)$ on a probability space $\big(\Omega, \mathscr{F}, \mathbb{P}_x\big)$.  For $\ell\in [ x, \infty)$ define the passage time $\varrho^{\uparrow,\ell}:=\inf\{t\in [0,\infty)\,:\,R_t = \ell  \}$. Let $\vartheta_x^{\ell}$ denote the Borel measure  on $[0,\ell]$
assigning $E\subset [0,\ell]$ the value
$$  \vartheta_x^{\ell}(E)\,=\,\mathbb{E}_x\Big[\,\textup{meas}\Big(\big\{t\in \big[0,\varrho^{\uparrow,\ell}\big]\,:\,R_t\in E   \big\}\Big) \,\Big] \, . $$
The Lebesgue density of $\vartheta_x^{\ell} $ has the form $\frac{ \vartheta_x^{\ell}(dy) }{ dy }=G_{\ell}(x,y)$ for $G_{\ell}:[0,\ell]\times [0,\ell]\rightarrow [0,\infty)$ defined by
$$ G_{\ell}(x,y)\,=\,\begin{cases} 2y\log \frac{\ell}{x}  & y\in [0,x)  \, ,  \\ 2y\log \frac{\ell}{y}    &  y\in  [x,\ell] \, . \end{cases} 
    $$
In particular, for any $\phi\in C\big([0,\infty)\big)$ supported on $[0,\varepsilon]$ we have that 
$\mathbb{E}_{\varepsilon}\big[ \int_0^{ \varrho^{\uparrow,2\varepsilon} }\phi(R_t)dt \big]=2\log 2 \int_0^{\varepsilon}a\phi(a)da  $.   

\end{lemma}

\begin{proof} This holds because $G_{\ell}(x,y)$ is the Green's function for the forward Kolmogorov generator $\mathcal{L}=\frac{1}{2}\big(\frac{d^2}{dx^2}-\frac{d}{dx}\frac{1}{x}\big)$ with Dirichlet boundary conditions on $[0,\ell]$ because
$$ \frac{1}{2}\bigg( \frac{\partial^2}{\partial x^2}+\frac{1}{x}\frac{\partial }{\partial x}  \bigg) G_{\ell}(x,y)\,=\,-\delta(x-y)
   \hspace{.5cm} \text{and}\hspace{.5cm}   \frac{1}{2} \bigg( \frac{\partial^2}{\partial y^2}-\frac{\partial}{\partial y}\frac{1}{y}  \bigg) G_{\ell}(x,y)\,=\,-\delta(x-y)  \,, $$
and $G_{\ell}(x,0)=0=G_{\ell}(x,\ell)$ for all $x\in [0,\ell]$.
\end{proof}

\begin{proof} [Proof of Lemma~\ref{LemmaDowncrossingIntegral}] Part (i): Applying (i) of Lemma~\ref{LemmaLeave} and assumption (II) on the family $\{ \phi_{\varepsilon}\}_{\varepsilon\in (0,1)}$ yields the second inequality below.
\begin{align*}
 \mathbf{E}^{T ,\lambda}_{x}\Bigg[\, \bigg(\int_{ 0}^{\varrho^{\uparrow,2\varepsilon} }      \phi_{\varepsilon}\big(|X_r|\big)\,dr\bigg)^p \, \Bigg] \,\leq \, \big\|  \phi_{\varepsilon}\big\|_{\infty}^p\,\mathbf{E}^{T ,\lambda}_{x}\Big[ \,\big(\varrho^{\uparrow,2\varepsilon}\big)^p \, \Big] \,\preceq \,\frac{1}{\log^{2p}\frac{1}{\varepsilon} }
\end{align*}

\vspace{.2cm}

\noindent Part (ii): As before, define the stopping time $\tau :=\inf\big\{r\in [0,\infty)\,\big|\, |X_r|=0   \big\}   $. We can write the expectation of $\int_{ 0}^{\varrho^{\uparrow,2\varepsilon} }      \phi_{\varepsilon}\big(|X_r|\big)dr  $ under $\mathbf{P}^{T ,\lambda}_{x}$  as the sum of the terms (A), (B),  (C) defined by
\begin{align*}
 \textup{(A)}\,:=\,&\, \mathbf{E}^{T,\lambda}_{x}\Bigg[\, 1_{  \tau < \varrho^{\uparrow,2\varepsilon} }\int_{ 0}^{\varrho^{\uparrow,2\varepsilon} }      \phi_{\varepsilon}\big(|X_r|\big)\,dr  \,\Bigg]\,,
\hspace{.5cm}  \textup{(B)}\,:=\,\mathbf{E}^{T,\lambda}_{x}\Bigg[ \,\int_{ 0}^{\varrho^{\uparrow,2\varepsilon} }      \phi_{\varepsilon}\big(|X_r|\big)\,dr \,\bigg|\, \tau > \varrho^{\uparrow,2\varepsilon}  \, \Bigg]\,, \hspace{.2cm} \text{and} \\ \textup{(C)}\,:=\,&\,\Bigg(1-\frac{1 }{  \mathbf{P}^{T,\lambda}_{x}\big[ \, \tau >\varrho^{\uparrow, 2\varepsilon}   \, \big] } \Bigg)\, \mathbf{E}^{T,\lambda}_{x}\Bigg[\, 1_{  \tau > \varrho^{\uparrow,2\varepsilon}  } \int_{ 0}^{\varrho^{\uparrow,2\varepsilon} }      \phi_{\varepsilon}\big(|X_r|\big)\,dr \,\Bigg]\,. \nonumber  \\
\end{align*}
 For the  term (A), we apply Cauchy-Schwarz to get
\begin{align}\label{Pip1}
  \textup{(A)}\,\leq \, &\, \sqrt{\mathbf{P}^{T,\lambda}_{x}\big[\, \tau < \varrho^{\uparrow,2\varepsilon} \,\big] }   \sqrt{\mathbf{E}^{T,\lambda}_{x}\Bigg[ \, \bigg(\int_{ 0}^{\varrho^{\uparrow,2\varepsilon} }      \phi_{\varepsilon}\big(|X_r|\big)\,dr  \bigg)^2 \,\Bigg] } \nonumber \\
   \,\leq \, &\, \|\phi_{\varepsilon}\|_{\infty} \sqrt{ \frac{ C_{L}  }{\log \frac{1}{\varepsilon}  }  }  \sqrt{\mathbf{E}^{T,\lambda}_{x}\Big[ \,\big(\varrho^{\uparrow,2\varepsilon} \big)^2 \,\Big] }   \,\preceq  \, \|\phi_{\varepsilon}\|_{\infty} \frac{  \varepsilon^2 }{\log^{\frac{1}{2}} \frac{1}{\varepsilon}  }  \, \preceq \,\frac{ 1 }{\log^{\frac{5}{2}} \frac{1}{\varepsilon}  }  \,.
\end{align}
The second inequality above holds for some $C_{L} >0$ by~(\ref{Prelimina}) with $\tau\equiv \varrho^{\downarrow,\varepsilon}_1 $, the third inequality applies (i) of Lemma~\ref{LemmaLeave}, and the last inequality uses assumption (II) on the family $\{\phi_{\varepsilon}  \}_{\varepsilon\in (0,1) }$. The absolute value of (C) is bounded by
\begin{align}\label{Pip2}
\big|\textup{(C)}\big| \,\preceq \,\frac{1 }{\log^2 \frac{1}{\varepsilon}  }\mathbf{E}^{T,\lambda}_{x}\Bigg[ \,\int_{ 0}^{\varrho^{\uparrow,2\varepsilon} }      \phi_{\varepsilon}\big(|X_r|\big)\,dr \, \Bigg]   \,\leq  \,\frac{\|\phi_{\varepsilon}\|_{\infty}  }{\log^2 \frac{1}{\varepsilon}  }\,\mathbf{E}^{T,\lambda}_{x}\big[ \,\varrho^{\uparrow,2\varepsilon} \,\big]\, \preceq \,\frac{ 1 }{\log^{4} \frac{1}{\varepsilon}  }\,,
\end{align}
where we  applied~(\ref{Prelimina}) again, this time after writing $\mathbf{P}^{T,\lambda}_{x}[  \tau <\varrho^{\uparrow, 2\varepsilon}    ]=1-\mathbf{P}^{T,\lambda}_{x}[  \tau >\varrho^{\uparrow, 2\varepsilon}    ] $, and we once more used (i) of Lemma~\ref{LemmaLeave} and assumption (II).  Thus the terms (A) and (C) are negligible.

Lastly, we approximate   the term (B).   We can apply (iii) of Theorem~\ref{CorSubMART} with $\mathbf{S}=\varrho^{\uparrow,2\varepsilon}$ to write
\begin{align} \label{Pipple}
\textup{(B)}\,=\,\mathbf{E}_{x}\Bigg[ \, \frac{1+ \bar{H}^{\lambda}_{T-\varrho^{\uparrow,2\varepsilon}}(2\varepsilon) }{ 1+\mathbf{E}_{x}\big[ \bar{H}^{\lambda}_{T-\varrho^{\uparrow,2\varepsilon}}(2\varepsilon)   \big]  } \,\int_{ 0}^{\varrho^{\uparrow,2\varepsilon} }      \phi_{\varepsilon}\big(|X_r|\big) \,dr \Bigg]\,,
\end{align}
in which we have used that $|X_{\varrho^{\uparrow,2\varepsilon}}|=2\varepsilon$.  Since $ |X| $ has the law of a two-dimensional Bessel process with initial value $|x|=\varepsilon$ under $\mathbf{P}_x$, Lemma~\ref{LemmaBessel} yields the second equality below:
\begin{align} \label{Pip3}
\textup{(B)}'\,:=\,\mathbf{E}_{x}\bigg[ \, \int_{ 0}^{\varrho^{\uparrow,2\varepsilon} }      \phi_{\varepsilon}\big(|X_r|\big) \,dr\, \bigg]\,=\,2\log 2
 \int_0^{\varepsilon} x\phi_{\varepsilon}\big(|x|\big)\,dx\, =\,\frac{ \log 2  }{ 2\log^2 \frac{1}{\varepsilon} 
 }\big(1-\mathcal{E}(\varepsilon)  \big)\,,
\end{align}
where $\mathcal{E}(\varepsilon)$ vanishes as $\varepsilon\searrow 0$ by assumption (III) on the family $\{\phi_{\varepsilon}  \}_{\varepsilon\in (0,1) }$. Thus, it remains for us to show that  the difference between (B) and $\textup{(B)}'$ vanishes with   order $\log^{-\frac{5}{2}} \frac{1}{\varepsilon}  $.  Applying Cauchy-Schwarz, we get
\begin{align*}
\big|\textup{(B)}'-\textup{(B)}\big|\,\leq \,\underbrace{\sqrt{\mathbf{E}_{x}\left[ \, \frac{\Big|\bar{H}^{\lambda}_{T-\varrho^{\uparrow,2\varepsilon}}(2\varepsilon) -\mathbf{E}_{x}\big[ \bar{H}^{\lambda}_{T-\varrho^{\uparrow,2\varepsilon}}(2\varepsilon)   \big]   \Big|^2 }{ \big(1+\mathbf{E}_{x}\big[ \bar{H}^{\lambda}_{T-\varrho^{\uparrow,2\varepsilon}}(2\varepsilon)   \big]\big)^2  } \,\right] }}_{ \preceq \,\frac{1}{ \log \frac{1}{\varepsilon}  }    }  \, \underbrace{\sqrt{\mathbf{E}_{x}\Bigg[ \, \bigg(\int_{ 0}^{\varrho^{\uparrow,2\varepsilon} }      \phi_{\varepsilon}\big(|X_r|\big)\,dr  \bigg)^2 \,\Bigg] }}_{ \preceq \, \frac{ 1}{ \log^2 \frac{1}{\varepsilon}  }  }\,\preceq\, \frac{ 1 }{ \log^3 \frac{1}{\varepsilon}  } \,.
\end{align*}
The first term on the right side above is bounded through standard techniques using (i) of Proposition~\ref{PropK} and that $\mathbf{E}_{x}\big[(\varrho^{\uparrow,2\varepsilon})^2\big]\preceq \varepsilon^4$.  The second term is bounded using similar reasoning as in~(\ref{Pip1}).
Combining the above with~(\ref{Pip1})--(\ref{Pip3}) gives us the result.
\end{proof}

\subsection{Proofs of Lemma~\ref{LemmaDetClass}}\label{SubsectionLemmaDetClass}

\begin{proof} For $\beta\in \R$ define $\psi_{\beta}:[0,\infty)\rightarrow [0,\infty)$  by $\psi_{\beta}(0)=1$ and $\psi_{\beta}(a):= \frac{\nu(e^{-\beta} a)  }{ \nu( a) }$ for $a>0$.  Then $ \varphi_T^{\lambda,\lambda'}(a)=   \psi_{\beta}\big(\lambda (T-a)\big) $  for $\beta=\log\frac{\lambda}{\lambda'}$, and   it     suffices to show that $\{\psi_{\beta}\}_{\beta\in \R}$ is a determining class on the measurable space $\big([0,T],\mathscr{B}([0,T])\big)$. 
Let $\vartheta_1$ and $\vartheta_2$ be finite Borel measures on $[0,T]$ such that $\vartheta_1(\psi_{\beta})=\vartheta_2(\psi_{\beta})$ for all $\beta \in \R$.  As $\beta\rightarrow \infty$ we have pointwise convergence of $\psi_{\beta}$ to $ 1_{\{0\}}$, and thus the bounded convergence theorem yields that  
$$  \vartheta_1\big(\{ 0\} \big)\,=\, \lim_{\beta\rightarrow \infty } \vartheta_1\big(\psi_{\beta}\big)   \,=\,\lim_{\beta\rightarrow \infty } \vartheta_2\big(\psi_{\beta}\big)\,=\, \vartheta_2\big(\{ 0\} \big)\,.   $$
Hence it suffices for us to show that $1_{(0,T]}\vartheta_1=1_{(0,T]}\vartheta_2$.  Since $\nu(x):=\int_0^{\infty}\frac{  x^r }{\Gamma(r+1)}  dr$,  we can interchange the order of integrations to write
$$\text{}\hspace{2.5cm} \vartheta_j\big(\psi_{\beta}\big)\,=\,\int_{0}^{\infty} \frac{e^{-\beta r}}{\Gamma(r+1)} \int_0^T \frac{a^r  }{ \nu(a) }\,\vartheta_j(da) \,dr\,,  \hspace{1cm} j=1,2 \,, \hspace{.3cm}\beta\in \R\,.     $$
Thus the map $\beta \mapsto \vartheta_j(\psi_{\beta})$ is the Laplace transform of $ f^T_j(r):=  \frac{1}{\Gamma(r+1)} \int_0^T \frac{a^r  }{ \nu(a) }\vartheta_j(da)   $.  Since $\vartheta_1(\psi_{\beta})=\vartheta_2(\psi_{\beta})$ for all $\beta\in \R$, we have that $ f^T_1=f^T_2 $, and thus $\int_0^T \frac{a^r  }{ \nu(a) }\vartheta_1(da) =\int_0^T \frac{a^r  }{ \nu(a) }\vartheta_2(da) $ for all $r\geq 0$. The map $r\mapsto \int_0^T \frac{a^r  }{ \nu(a) }\vartheta_j(da)$ is the Mellin transform of the measure $  1_{[0,T]}(a)\frac{ a }{ \nu(a) } \vartheta_j(da) $, so we can conclude that the measures $1_{(0,T]}\vartheta_1$  and $1_{(0,T]}\vartheta_2$ are equal because the factor $\frac{a}{\nu(a)}$ is  positive  for $a>0$. 
\end{proof}

\subsection{Proof of Lemma~\ref{lem exp}}\label{Subsectionlem exp}

 \begin{proof} 
        \noindent Part (i): Using that $\Gamma(v+1)=v\Gamma(v)$, we can write
        \begin{align*}
\mathcal{E}_{\alpha,\beta}[\,X\,Y\,]
\,=\,&\,\frac{1  }{ N_{\alpha,\beta}  } \,\int_1^{\infty}\,\int_0^{\infty} \, \frac{e^{-\alpha R}}{R} \,  \frac{  \big(1-\frac{1}{R}\big)^v \beta^v }{  \Gamma(v+1)} \,v\,R\, dv \,dR   \\
\,\leq \,&\,\frac{1  }{ N_{\alpha,\beta}  } \,\bigg(\int_1^{\infty} \,e^{-\alpha R }\,dR\bigg)\bigg( \int_0^{\infty}\,\frac{\beta^v}{  \Gamma(v)} \,dv\bigg)  
\,=  \,\frac{1}{ N_{\alpha,\beta}}\,\frac{ e^{-\alpha  } }{ \alpha}\, \beta \,\nu'( \beta)   \,.
\end{align*}
The claimed bound  holds since $\nu'(x)\sim \frac{1}{x\log^2\frac{1}{x}}   $ for small $x>0$ by (iii) of Lemma~\ref{LemmaEFunAsy}. \vspace{.2cm}

\noindent Part (ii): We can express the covariance between $X$ and $Y$ in terms of a nested conditional expectation as follows:
\begin{align*}
\textup{Cov}_{\alpha,\beta}(X,Y)\,=\, \mathcal{E}_{\alpha,\beta}\Big[ X \,\Big(\mathcal{E}_{\alpha,\beta}[\,Y \,|\, X\,]-\mathcal{E}_{\alpha,\beta}[\,Y\, ]  \Big)  \Big]  \,.
\end{align*}
Note that the marginal density for $X$ has the form
\begin{align*}
\frac{\mathcal{P}_{\alpha,\beta}[ X\in dR] }{dR}\,=\, \frac{1}{N_{\alpha,\beta} } \,\frac{ e^{-\alpha R  } }{ R } \,\nu\Big( \Big(1-\frac{1}{R}\Big)  \beta\Big)\,1_{ [1,\infty) }(R)  \,,\hspace{.5cm} R\in [0,\infty)\,,
\end{align*}
and  the conditional density of the random variable $Y$ given that $X=R\in [1,\infty)$ has the form
\begin{align*}
\frac{\mathcal{P}_{\alpha,\beta}\big[ Y\in dv\,\big| \, X=R  \big] }{dv}\,=\,\frac{1}{ \nu\big((1-\frac{1}{R})  \beta  \big) }\,\frac{ \big(1-\frac{1}{R} 
  \big)^v \beta^v }{ \Gamma(v+1)  }\,,\hspace{.5cm} v\in [0,\infty)\,.
\end{align*}
Thus, the conditional expectation of $Y$ given $X=R$ is 
$$
\mathcal{E}_{\alpha,\beta}\big[\,Y\, \big|\, X=R\,\big]\,=\, \Big(1-\frac{1}{R}\Big)\,\beta\,\frac{  \nu'\big((1-\frac{1}{R})   \beta\big)   }{ \nu\big((1-\frac{1}{R}) \beta  \big)   } \,. $$  It follows that the covariance between $X$ and $Y$ can be written as
\begin{align*}
    \textup{Cov}_{\alpha,\beta}(X,Y)\,= \,&\,\int_1^{\infty}\,R\,\Big(\,\mathcal{E}_{\alpha,\beta}\big[\,Y\, \big|\, X=R\,\big]\,-\, \mathcal{E}_{\alpha,\beta}[\,Y\, ]  \,\Big)\,\mathcal{P}_{\alpha,\beta}[\, X\in dR\,]
    \\ \,= \,&\,\frac{1}{N_{\alpha,\beta} }\,\int^{\infty}_1 \,e^{-\alpha R} \,\nu\,\Big(\Big(1-\frac{1}{R}\Big)\beta  \Big) \, \Bigg(\Big(1-\frac{1}{R}\Big)\, \beta\,\frac{  \nu'\big((1-\frac{1}{R}) \beta  \big)   }{ \nu\big((1-\frac{1}{R})  \beta \big)   } \,-\,\mathcal{E}_{\alpha,\beta}[\,Y\, ]  \,\Bigg)\,dR  \,, \nonumber  
\end{align*}
which is bounded by
\begin{align}\label{Tizlep}
\frac{\nu( \beta) }{N_{\alpha,\beta} }\,\underbrace{\int_1^{\infty}\, e^{-\alpha R}  \, \Bigg|\,\Big(1-\frac{1}{R}\Big)\,\beta\,\frac{  \nu'\big((1-\frac{1}{R}) \beta  \big)   }{ \nu\big((1-\frac{1}{R}) \beta  \big)   } \,-\,\frac{\beta\, \nu'(\beta)   }{ \nu(\beta) } \,\Bigg| \,dR}_{\textup{(I)}} \,+\,\frac{\nu( \beta) }{N_{\alpha,\beta} }\, \underbrace{\bigg| \, \frac{\beta \nu'(\beta)   }{ \nu(\beta) }\,-\,\mathcal{E}_{\alpha,\beta}[\,Y\, ] \,\bigg|}_{\textup{(II)} }\,\frac{e^{-\alpha}}{\alpha} \,.
\end{align}
To get the above bound, we have applied the triangle inequality, that $\nu$ is increasing, and $\frac{e^{-\alpha}}{\alpha} =\int_1^{\infty}e^{-\alpha R}dR $.  We will bound (I) \& (II) in the analysis below.  For this, we will use that there exists a $C_L>0$ such that  for all $ 0\leq  b<\beta \leq L$
\begin{align}\label{3Bounds}
  \Bigg|  \frac{\frac{\beta \nu'(\beta)}{ \nu(\beta)}\,-\,\frac{b \nu'(b)}{ \nu(b)}   }{\beta-b  }  \Bigg|  \,\leq \, &\,\frac{ C_L}{\beta\big(1+\log^+ \frac{1}{\beta}\big) }\,, \quad \bigg|\frac{ \beta \nu'(\beta)\,-\,   b\nu'(b)  }{\beta-b } \bigg| \,\leq \,\frac{ C_L}{\beta\big(1+\log^+ \frac{1}{\beta}\big)^2  }\,,   \\ \bigg| \frac{  \nu(\beta)\,-\,   \nu(b)  }{\beta-b } \bigg| \,\leq \,&\,\frac{ C_L}{\beta\big(1+\log^+ \frac{1}{\beta}\big)  } \,. \label{3BoundsRe}
\end{align}
These inequalities  can be shown using that $\nu(0)=0$   and the asymptotics for $\nu'(x)$ and $\nu''(x)$ for small $x>0$ in (iii) of Lemma~\ref{LemmaEFunAsy}.

For (I) we observe that  for all $ \beta \in (0 , L]  $ and  $R>1$,
\begin{align*}
  \bigg| \frac{ \big(1-\frac{1}{R}\big)\beta \,\nu'\big( \big(1-\frac{1}{R}\big)\beta\big)  }{ \nu\big( \big(1-\frac{1}{R}\big) \beta\big)}  \,-\, \frac{\beta \nu'(\beta)}{ \nu(\beta)} \bigg| 
  \,\leq  \, \frac{\beta}{R}\, \sup_{0\leq b\leq \beta}  \,  \Bigg| \frac{\frac{\beta \nu'(\beta)}{ \nu(\beta)}\,-\,\frac{b \nu'(b)}{ \nu(b)}   }{\beta-b  }  \Bigg|  \,\stackrel{(\ref{3Bounds})}{\leq} \,\frac{C_L}{R\big(1+\log^+ \frac{1}{\beta} \big)  }\,, \nonumber 
\end{align*}
and hence
\begin{align}\label{Izep0}
\textup{(I)}\,\leq \,\frac{C_L}{\big(1+\log^+ \frac{1}{\beta} \big)  }\,\int^{\infty}_1 \,\frac{1}{R} \,e^{-\alpha R} \, dR\,=\,C_L\,\frac{E( \alpha)}{ 1+\log^+ \frac{1}{\beta}   }\,.
\end{align}
For (II) we use the triangle inequality again to separate our analysis into two further subparts:
\begin{align*}
\textup{(II)}\,\leq \,\bigg| \mathcal{E}_{\alpha,\beta}[\, Y\,]\,-\,   \frac{1}{N_{\alpha,\beta} }\,\beta\, \nu'( \beta)\,E(\alpha)\bigg| \,+\,\bigg|  \frac{1}{N_{\alpha,\beta} }\,\beta\, \nu'( \beta)\,E(\alpha) \,-\,\frac{ \beta \nu'( \beta)  }{ \nu( \beta)  }  \bigg|\,=:\,(\textup{II})'\,+\,(\textup{II})''\,.
\end{align*}
We can bound $(\textup{II})'$ for $ \beta\in (0,  L]  $ and $\alpha >0$ through
\begin{align}\label{Izep1}
 (\textup{II})'\,= \,&\,    \Bigg| \frac{1}{N_{\alpha,\beta} }\,\int_1^{\infty}\,\frac{ e^{-\alpha R  } }{ R}\,\Big(1-\frac{1}{R}\Big)\,\beta\, \nu'\Big( \Big(1-\frac{1}{R}\Big) \beta \Big)\,dR\,-\,   \frac{1}{N_{\alpha,\beta} }\,\beta\, \nu'( \beta)\,\int_1^{\infty}\,\frac{ e^{-\alpha R } }{ R }\,dR\Bigg| \nonumber   \\
    \,\leq \,&\, \frac{\beta }{N_{\alpha,\beta} }\,\int_1^{\infty}\,\frac{ e^{-\alpha R } }{ R^2 }\Bigg|\frac{\big(1-\frac{1}{R}\big)\beta \nu'\big( \big(1-\frac{1}{R}\big) \beta \big) \,-\, \beta\, \nu'( \beta) }{\frac{\beta}{R} } \Bigg|\, dR \nonumber 
    \\
    \,\leq \,&\,  \frac{\beta }{N_{\alpha,\beta} }\, \bigg(\sup_{  0 \leq b < \beta\leq L}\, \bigg|\frac{ \beta \nu'(\beta)-   b\nu'(b)  }{\beta-b } \bigg| \bigg)\,\int_1^{\infty}\,\frac{ e^{-\alpha R } }{R^2 }\,dR
       \,\stackrel{ (\ref{3Bounds}) }{ \leq} \,   \frac{C_L }{N_{\alpha,\beta} \big(1+\log^+ \frac{1}{\beta}\big)^2} \,e^{-\alpha}\,,
\end{align}
in which the last inequality uses that $\int_1^{\infty}\frac{ e^{-\alpha R } }{R^2 }dR \leq e^{-\alpha} $, and similarly for  $\textup{(II)}''$ 
\begin{align}\label{Izep2}
(\textup{II})''\,= \,&\, \frac{1}{N_{\alpha,\beta} }\, \frac{ \beta \nu'( \beta)  }{ \nu( \beta)  }  \Big| E(\alpha)\, \nu( \beta)  \,-\,N_{\alpha,\beta} \Big| \nonumber  \\ \,\leq \,&\, \frac{1}{N_{\alpha,\beta} } \, \frac{ \beta \nu'( \beta)  }{ \nu( \beta)  }\,\int_1^{\infty}\,\frac{ e^{-\alpha R  } }{ R }\,\Big|\,\nu( \beta) \,-\, \nu\Big( \beta\Big(1-\frac{1}{R}\Big)  \Big) \, \Big| \, dR
 \nonumber  \\ \,\leq \,&\, \frac{1}{N_{\alpha,\beta} }  \,\frac{ \beta^2\, \nu'( \beta)  }{ \nu( \beta)  } \,\sup_{  0 \leq b < \beta \leq L}\,\bigg| \frac{  \nu(\beta)\,-\,   \nu(b)  }{\beta-b } \bigg|\,\int_1^{\infty}\,\frac{ e^{-\alpha R  } }{ R^2 }\,dR \,\stackrel{ (\ref{3BoundsRe}) }{ \leq} \, \frac{C_L}{N_{\alpha,\beta}\big(1+\log^+ \frac{1}{\beta}\big)^2 }\, e^{-\alpha}\,.
\end{align}
In the above we have used that $ \frac{x \nu'( x)  }{ \nu( x)  } \sim \frac{1}{\log \frac{1}{x} } $ as $x\searrow 0$.  Thus we have shown that (II) is bounded by a multiple of $e^{-\alpha} N_{\alpha,\beta}^{-1} (1+\log^+ \frac{1}{\beta})^{-2}  $.

Going back to~(\ref{Tizlep}), we get the first inequality below by applying our results~(\ref{Izep0})--(\ref{Izep2}).
\begin{align*}
   (\ref{Tizlep})\, \preceq \,&  \,\frac{ \nu(\beta) }{N_{\alpha,\beta}^2 \big(1+\log^+ \frac{1}{\beta}\big)}\bigg( \,E(  \alpha)  \,N_{\alpha,\beta}\,+\,  \frac{e^{-2\alpha}}{\alpha\big(1+\log^+ \frac{1}{\beta}\big)}\,\bigg) \\
 \preceq \,& \,\frac{ 1}{N_{\alpha,\beta}^2 \big(1+\log^+ \frac{1}{\beta}\big)^3}\bigg(\, E^2(  \alpha) \,+\,  \frac{e^{-2\alpha}}{\alpha }\,\bigg)
 \prec  \,\frac{ 1}{N_{\alpha,\beta}^2 \big(1+\log^+\frac{1}{\beta}\big)^3}\,\frac{e^{-2\alpha}}{\alpha }
\end{align*}
The second inequality uses that  $N_{\alpha,\beta}\leq E(\alpha)\nu(\beta)$---see~(\ref{NAlphaBeta})---and that $\nu$ is an increasing function satisfying $\nu(x)\sim\frac{1}{\log\frac{1}{x}}$ when $0 <x\ll 1$.  The third inequality holds since $E(x)$ has the bound in (iv) of Lemma~\ref{LemEM}.
    \end{proof}

\begin{appendix}

\section{Correlations for the two-dimensional SHE}  \label{AppendixSHE}

We will use the convergence~(\ref{ConvKer}) to show that the correlation between  the random variables $u^{\lambda,\varepsilon}_{t}(x,\varphi_1)$ and $u^{\lambda,\varepsilon}_{t}(y,\varphi_2)$ converges to~(\ref{CorrelationLimit}).  For any continuous function $p:[0,\infty)\rightarrow \R^2$, the process $\mathbf{W}^{\varepsilon,p}_t:=\int_0^t\xi_{\varepsilon}\big(r,p(r)\big)dr$ is a one-dimensional Brownian motion with diffusion rate $J_{\varepsilon}(0)=\frac{1}{\varepsilon^2}\|j\|_2^2$, and  the cross variation of $\mathbf{W}^{\varepsilon,p}$ and $  \mathbf{W}^{\varepsilon,q} $ has the form
\begin{align*} \big\langle \mathbf{W}^{\varepsilon,p},  \mathbf{W}^{\varepsilon,q}  \big\rangle_t \,  
 =\,\int_0^t\,J_{\varepsilon}\big(p(r)-q(r)\big)\,dr \,.
\end{align*}
The solution $u^{\lambda,\varepsilon}_{t}(x)\equiv u^{\lambda,\varepsilon}_{t}(x,\varphi)$ of the SPDE~(\ref{SHE2}) can be expressed as an integral of log-normal random variables $\mathscr{E}_t^{\lambda,\varepsilon,p}$ with respect to the  Wiener measure $ \mathbf{P}_x$ on $\boldsymbol{\Omega}:= C\big([0,\infty),\R^2\big) $ as follows:
\begin{align}\label{PathConst}
u^{\lambda,\varepsilon}_{t}(x,\varphi)\,=\,\int_{\boldsymbol{\Omega}} \, \mathscr{E}_t^{\lambda,\varepsilon,p}\,\varphi\big( p(t)\big)\, \mathbf{P}_x(dp) \, ,  \hspace{.5cm} \text{for}\hspace{.3cm} \mathscr{E}_t^{\lambda,\varepsilon,p}\,:=\,\textup{exp}\bigg\{\, \sqrt{\mathbf{\vsm}_{\varepsilon}^{\lambda} }\,\mathbf{W}^{\varepsilon,p}_t\,-\,\mathbf{\vsm}_{\varepsilon}^{\lambda}\,\frac{t\|j\|_2^2}{2\varepsilon^2} \, \bigg\}\,.
\end{align}
For $p,q\in \boldsymbol{\Omega}$ we have
\begin{align*}
   \mathbb{E}\Big[ \,\mathscr{E}_t^{\lambda,\varepsilon,p}\,\mathscr{E}_t^{\lambda,\varepsilon,q}\,\Big]\,=\,\textup{exp}\Big\{ \,  \mathbf{\vsm}_{\varepsilon}^{\lambda}\, \big\langle \mathbf{W}^{\varepsilon,p},  \mathbf{W}^{\varepsilon,q}  \big\rangle_t \, \Big\}
 \,=\,\textup{exp}\bigg\{\, \mathbf{\vsm}_{\varepsilon}^{\lambda} \int_0^t\,J_{\varepsilon}\big(p(r)-q(r)\big)\,dr \, \bigg\}\,.
\end{align*}
Using~(\ref{PathConst}) we can write $\mathbb{E}\big[  u^{\lambda,\varepsilon}_{t}(x,\varphi_1) u^{\lambda,\varepsilon}_{t}(y,\varphi_2)\big]$ in the form
\begin{align*}
   \int_{\boldsymbol{\Omega}} \, \int_{\boldsymbol{\Omega}}\, & \, \mathbb{E}\Big[\, \mathscr{E}_t^{\lambda,\varepsilon,p}\,\mathscr{E}_t^{\lambda,\varepsilon,q}\,\Big]\, \varphi_1\big( p(t)\big)\,\varphi_2\big( q(t)\big)\, \mathbf{P}_x(dp) \, \mathbf{P}_y(dq) \nonumber \\
  \,=\,&\,  \int_{\boldsymbol{\Omega}} \, \int_{\boldsymbol{\Omega}}   \textup{exp}\Bigg\{\, \mathbf{\vsm}_{\varepsilon}^{\lambda} \int_0^t\,J_{\varepsilon}\big(p(r)-q(r)\big)\,dr \, \Bigg\}\, \varphi_1\big( p(t)\big)\,\varphi_2\big( q(t)\big)\, \mathbf{P}_x(dp) \, \mathbf{P}_y(dq) \\
   \,=\,&\,  \int_{\boldsymbol{\Omega}} \, \int_{\boldsymbol{\Omega}}   \textup{exp}\Bigg\{\, \vsm_{\varepsilon}^{\lambda} \int_0^t\,\delta_{\varepsilon}\big(z(r)\big)\,dr \, \Bigg\}\, \varphi_1\bigg( \frac{ a(t)+z(t)  }{\sqrt{2}} \bigg)\,\varphi_2\bigg( \frac{a(t) -z(t)  }{\sqrt{2}} \bigg)\, \mathbf{P}_{\frac{ x-y }{\sqrt{2} } }(dz) \, \mathbf{P}_{\frac{ x+y }{\sqrt{2} } }(da)  \,,
\end{align*}
where $\delta_{\varepsilon}(x):= \frac{1}{\varepsilon^2} D(\frac{x}{\varepsilon}  ) =2 J_{\varepsilon}\big(\sqrt{2} x\big)$, and we used that  $z:=\frac{p-q }{\sqrt{2}}  $, $a:=\frac{p+q }{\sqrt{2}} $ are independent two-dimensional Brownian motions (and that  $\mathbf{\mathsmaller{V}}_{\varepsilon}^{\lambda} :=2 \mathsmaller{V}_{\varepsilon}^{\lambda}$).    Applying~(\ref{ConvKer}), the above converges as $\varepsilon \searrow 0$ to
\begin{align*}
\int_{\R^2}  &\int_{\boldsymbol{\Omega}} \, f^{\lambda}_t\bigg( \frac{x-y}{\sqrt{2} }   , z \bigg)  \, \varphi_1\bigg( \frac{ a(t)+z  }{\sqrt{2}} \bigg)\,\varphi_2\bigg( \frac{a(t) -z }{\sqrt{2}} \bigg)\, \mathbf{P}_{\frac{ x+y }{\sqrt{2} } }(da) \,dz \nonumber  \\
\,=\,&\, \int_{\R^2}  \, \int_{\R^2}  \, f^{\lambda}_t\bigg( \frac{x-y}{\sqrt{2} }   , z \bigg)\, g_t\bigg( z' -\frac{x+y}{\sqrt{2}} \bigg)     \, \varphi_1\bigg( \frac{ z'+z  }{\sqrt{2}} \bigg)\,\varphi_2\bigg( \frac{z' -z }{\sqrt{2}} \bigg)\, dz'\,dz\,.
\end{align*}

\section{Generalized eigenfunctions for the integral kernel $\boldsymbol{f_t^{\lambda}(x,y)}$ }\label{SubsecFKern}

Recall that $K_0$ denotes the order-$0$ modified Bessel function of the second kind. For any  $ k\in \R^2$ with $|k|\neq \sqrt{2\lambda}$, define $\phi^{\lambda}_k:\R^2\rightarrow [0,\infty]$ by
$$  \phi^{\lambda}_k(x)\,=\, e^{k\cdot x}\,+\,\frac{ 2  }{\log \frac{|k|^2}{2\lambda}   }\, K_0\big(|k|\,|x|\big) \,. $$
Note that the function $\psi^{\lambda}_k$ defined in (\ref{PsiLambda}) is the imaginary counterpart to $\phi^{\lambda}_k$.  The proposition below shows that $\phi^{\lambda}_k$ and $\psi^{\lambda}_k$  serve as generalized (unnormalizable)  eigenfunctions for a linear operator having  integral kernel $f_t^{\lambda}(x,y)$.  Our proof relies on the convolution identity  in Lemma~\ref{LemmaEUpsilonBold}.
\begin{proposition}\label{PropRealUnnormEigen} Given $t,\lambda>0$ let  $f_t^{\lambda}:\R^2\times \R^2\rightarrow [0,\infty] $ be defined as in~(\ref{DefFullKer}).
\begin{enumerate}[(i)]
    \item For all  $x\in \R^2\backslash \{0\}$ and  $k\in \R^2$ with $|k|\neq \sqrt{2\lambda}$, we have $
\int_{\R^2} f^{\lambda}_t(x,y)\phi^{\lambda}_k(y)dy =e^{t\frac{|k|^2}{2} }\phi^{\lambda}_k(x) $. 
    
    \item For all  $x\in \R^2\backslash \{0\}$ and  $k\in \R^2$, we have $\int_{\R^2} f^{\lambda}_t(x,y)\psi^{\lambda}_k(y)dy =e^{-t\frac{|k|^2}{2} }\psi^{\lambda}_k(x)$.

\end{enumerate}
Also, for the $L^2(\R^2)$-normalized function $\psi^{\lambda}(x):= (\frac{ 2\lambda }{\pi  })^{1/2}K_0\big(\sqrt{2\lambda}  |x|\big)$, we have  $
\int_{\R^2} f^{\lambda}_t(x,y)\psi^{\lambda}(y)dy =e^{t\lambda }\psi^{\lambda}(x) $. 
\end{proposition}

\begin{proof} The claim about $\psi^{\lambda}$ follows from a simpler version of the computation in the proof of (i) below. \vspace{.1cm}

\noindent Part (i): Since $f^{\lambda}_t(x,y):=g_t(x-y)+h_t^{\lambda}(x,y)$ for  $h_t^{\lambda}(x,y)$ defined in~(\ref{DefLittleH}) and
\begin{align}\label{GeeEigen}
\int_{\R^2}\,g_t(x-y)\,e^{k\cdot y}\,dy\,=\,e^{t\frac{|k|^2}{2} }\,e^{k\cdot x}\,, 
\end{align}
the equation in (i) holds provided that
\begin{align}\label{Sum3}
e^{t\frac{|k|^2}{2} }\,\frac{ 2  }{\log \frac{|k|^2}{2\lambda}   } \,K_0\big(|k|\,|x|\big) \,  =\,&\, \textup{(I)}\,+\,\textup{(II)}\,+\,\textup{(III)} \,,
\end{align}
for (I)--(III) defined by
\begin{align*}
\textup{(I)}\,:=\,&\,\lambda \,\int_{0 < r < s < t}\,\frac{ e^{-\frac{|x|^2  }{2r }  }  }{ r }\,\nu'\big( (s-r)\lambda\big) \,\int_{\R^2}\,\frac{ e^{-\frac{|y|^2  }{2(t-s) }  }  }{2\pi (t-s) }\,e^{k\cdot y}\,dy \,ds\,dr\\ \textup{(II)}\,:=\,&\,\frac{ 2 }{\log \frac{|k|^2}{2\lambda}   }\, \int_{\R^2}\,\frac{ e^{ -\frac{ |x-y|^2  }{ 2t }  }
 }{ 2\pi t}\,K_0\big(|k|\,|y|\big)\,dy \\ \textup{(III)}\,:=\,&\,\frac{2 \lambda }{\log \frac{|k|^2}{2\lambda}   }\, \int_{0 < r < s  < t}\,\frac{ e^{-\frac{|x|^2  }{2r }  }  }{ r }\,\nu'\big((s-r)\lambda\big)\, \int_{\R^2}\,\frac{ e^{-\frac{|y|^2  }{2(t-s) }  }  }{2\pi (t-s) }\,  K_0\big(|k|\,|y|\big)\,dy\,ds\,dr \,.
\end{align*}
Using~(\ref{GeeEigen}) we can write (I) in the form
\begin{align} \label{ForI}
\textup{(I)}\,=\,&\,\int_{0}^t\,\frac{ e^{-\frac{|x|^2  }{2r }  }  }{ r }\, e^{(t-r)\frac{|k|^2  }{2 }}\, \int_r^t\, \lambda\,\nu'\big((s-r)\lambda\big)\,  e^{-(s-r)\frac{|k|^2  }{2 }  } \,ds\,dr   \nonumber 
\\
\,=\,&\, \int_{0}^t\,\frac{ e^{-\frac{|x|^2  }{2r }  }  }{ r } \, e^{(t-r)\frac{|k|^2  }{2 }} \Bigg(\,  \bigg[\,\nu\big( (s-r )\lambda\big)\,  e^{-(s-r)\frac{|k|^2  }{2 }  } \,\bigg]_{s=r}^{s=t}  \,+\, \frac{|k|^2  }{2 }\,\int_r^{t} \,\nu\big((s-r)\lambda \big) \,e^{-(s-r)\frac{|k|^2  }{2 }  }  \,ds \,\Bigg)\, dr   \nonumber 
\\
\,=\,&\, \int_{0}^t\,\frac{ e^{-\frac{|x|^2  }{2r }  }  }{ r } \,\nu\big( (t-r)\lambda\big)\, dr  \,+\, \int_{0}^t\,\frac{ e^{-\frac{|x|^2  }{2r }  }  }{ r } \,\int_0^{(t-r)\frac{|k|^2}{2} }\,e^{(t-r)\frac{|k|^2  }{2 }-a }\, \nu\Big(\frac{2\lambda}{|k|^2} \,a\Big) \, da\,  dr\,. 
\end{align}
The second equality above applies integration by parts and the third makes the change of integration variable to $a=(s-r)\frac{|k|^2}{2}$.

 Using the identity $ K_0(z)=\int_0^{\infty}\frac{1}{2a}e^{-a-\frac{z^2}{4a} }da $ (see, for instance,~\cite[p.\ 183]{Watson}), we can write
 \begin{align}\label{KRescale}
       K_0\big(|k|\,|x| \big)\,=\,\frac{1}{2}\,\int_0^{\infty}\,e^{-a\frac{|k|^2}{2}  }\,\frac{ e^{ -\frac{|x|^2}{2a}  }  }{ a }\,da \,.
 \end{align}
 Applying~(\ref{KRescale}) in the first and fourth equalities below, we can express (II) in the form
\begin{align} \label{ForII} 
\textup{(II)}\,=\, &\,\frac{ 1  }{\log \frac{|k|^2}{2\lambda}   } \,\int_0^{\infty}\,e^{-a\frac{|k|^2}{2}  }\,
\int_{\R^2}\,\frac{e^{ -\frac{ |x-y|^2  }{ 2t }  }
 }{ 2\pi t}\,\frac{e^{-\frac{ |y|^2  }{ 2a } }  }{ a }\,dy \, da  \nonumber  \\
\,=\, &\,\frac{ 1 }{\log \frac{|k|^2}{2\lambda}   } \,\int_0^{\infty}e^{-a\frac{|k|^2}{2}  }\,
\frac{e^{-\frac{ |x|^2  }{ 2(t+a) } }  }{ t+a } \, da  \nonumber \\
\,=\, &\,e^{t\frac{|k|^2}{2} } \,\frac{ 1  }{\log \frac{|k|^2}{2\lambda}   } \,\int_t^{\infty}\,e^{-a\frac{|k|^2}{2}  }\,
\frac{e^{-\frac{ |x|^2  }{ 2a } }  }{ a }  \,da  \nonumber 
\\
\,=\, &\,e^{t\frac{|k|^2}{2} }\,\frac{2  }{\log \frac{|k|^2}{2\lambda}   }\, K_0\big(|k|\,|x|\big)\,-\, \frac{ 1  }{\log \frac{|k|^2}{2\lambda}   }\, \int_0^{t}\,
\frac{e^{-\frac{ |x|^2  }{ 2a } }  }{ a } \, e^{(t-a)\frac{|k|^2}{2}  }\, da\,,  
\end{align}
in which the third equality changes integration variable $t+a\mapsto a$.

For (III)  we again use~(\ref{KRescale}) to write
\begin{align*}
\textup{(III)}\,=\, &\,\frac{  \lambda }{\log \frac{|k|^2}{2\lambda}   }\, \int_{0 < r < s  < t}\,\frac{ e^{-\frac{|x|^2  }{2r }  }  }{ r }\nu'\big((s-r)\lambda\big)\, \underbracket{\int_0^{\infty}\,e^{-a\frac{|k|^2}{2}  } \,\int_{\R^2}\,\frac{ e^{-\frac{|y|^2  }{2(t-s) }  }  }{2\pi (t-s) }\, \frac{e^{-\frac{ |y|^2  }{ 2a } }  }{ a } \, dy \, da}\,ds\,dr \,.
\end{align*}
We can express the underbracketed inner integrals as
\begin{align*}
  \int_0^{\infty}\,e^{-a\frac{|k|^2}{2}  }\,\frac{ 1 }{a+t-s  } \, da \,=\,  e^{(t-s) \frac{|k|^2}{2} }\,\underbrace{\int_{(t-s)\frac{|k|^2}{2}}^{\infty}\,e^{-b  }\,\frac{ 1 }{b }  \,db}_{E\big((t-s)\,\frac{|k|^2}{2}\big) }\,=:\,\boldsymbol{E}\bigg( (t-s)\,\frac{|k|^2}{2}   \bigg)\,,
\end{align*}
where we have made a change of integration variable to $b=(a+t-s)\frac{|k|^2}{2}$. Thus  (III) is equal to
\begin{align} \label{ForIII}
\textup{(III)}
\,=\, &\,\frac{  \lambda }{\log \frac{|k|^2}{2\lambda}   } \,\int_{0 < r < s  < t}\,\frac{ e^{-\frac{|x|^2  }{2r }  }  }{ r }  \,\nu'\big((s-r)\lambda\big) \,\boldsymbol{E}\bigg( (t-s) \, \frac{|k|^2}{2}\bigg)  \,  ds\, dr  \nonumber \\
\,=\, &\,\frac{ 1 }{\log \frac{|k|^2}{2\lambda}   } \,\int_{0}^t\,\frac{ e^{-\frac{|x|^2  }{2r }  }  }{ r } \,\int_0^{(t-r)\frac{|k|^2}{2} } \, \frac{2\lambda}{|k|^2}\,\nu'\Big(\frac{2\lambda}{|k|^2}\, b\Big) \,\boldsymbol{E}\bigg( (t-r)\,\frac{|k|^2}{2} -b  \bigg)   \,   db\, dr  \nonumber 
 \\
\,=\, &\,\frac{ 1 }{\log \frac{|k|^2}{2\lambda}   } \,\int_{0}^t\,\frac{ e^{-\frac{|x|^2  }{2r }  }  }{ r } \,e^{(t-r)\frac{|k|^2}{2} }\, dr\,-\, \int_{0}^t\,\frac{ e^{-\frac{|x|^2  }{2r }  }  }{ r }\,\nu\big( (t-r)  \lambda \big)\, dr  \nonumber  \\ & \,-\, \int_{0}^t\,\frac{ e^{-\frac{|x|^2  }{2r }  }  }{ r }\,\int_0^{(t-r)\, \frac{|k|^2}{2}  }\,e^{ (t-r)\frac{|k|^2}{2} -a }\,\nu\Big( \frac{2\lambda}{|k|^2} \, a\Big)\, da\,  dr \,, 
\end{align}
where we have used a change of integration variable to $b=(s-r) \frac{|k|^2}{2} $ and applied Lemma~\ref{LemmaEUpsilonBold} with $ x\mapsto (t-r)\frac{|k|^2}{2}   $, $ \lambda \mapsto \frac{2\lambda}{|k|^2}   $,  and $z=1$.  Summing up the expressions~(\ref{ForI}), (\ref{ForII}), \& (\ref{ForIII}) for (I)--(III) yields the desired equality~(\ref{Sum3}). \vspace{.2cm}

\noindent Part (ii): For $a>0 $ we can express $H_0^{(1)}(a)$ through a contour integral $\frac{1}{i\pi}\int_{C}e^{(a/2)(z-1/z)} \frac{dz}{z} $, where the contour $C$ starts at zero, initially proceeds in the positive direction (for the purpose of integrability), and then curves around in the upper half plane towards $-\infty+i\epsilon$; see~\cite[Fig.\ 12.9 \& (12.97)]{Weber}.  Equivalently, we can  write $  H_0^{(1)}\big(|k|\,|x| \big)=\frac{1}{i\pi}\int_{C}e^{z\frac{|k|^2}{2}  } e^{ -\frac{|x|^2}{2z}  }  \frac{dz}{z}  $.  We can then obtain the proof by following the steps in the proof of (i), using this representation of $ H_0^{(1)}\big(|k|\,|x| \big)$  in place of~(\ref{KRescale}) and applying Lemma~\ref{LemmaEUpsilonBold} with $z=-1$.
\end{proof}

\section{Generalized eigenfunctions in the radial setting }\label{SubsecFKernII}
To obtain a  technical result stated in Lemma~\ref{LemmaSpanII} below, which is to be employed  in the proof of Lemma~\ref{LemmaCompleteSet}, we will discuss the radial analog of the generalized eigenfunctions considered in Appendix~\ref{SubsecFKern}. Recall that for $t>0$ the transition density function $q_t:[0,\infty)\times [0,\infty)\rightarrow [0,\infty)$ for a two-dimensional Bessel process has the form $q_t(a,b):= \textup{exp}\big\{-\frac{a^2+b^2   }{2t  }\big\} I_0\big(\frac{ab}{t}\big)\frac{b  }{t  }$, wherein $I_0$ is the order-$0$ modified Bessel function of the first kind.  Since  $h_{t}^{\lambda}(x,y)$, which is defined  in~(\ref{DefLittleH}), is radially symmetric in both of the variables $x,y\in \R^2$, there is a function $\bar{h}_{t}^{\lambda}:[0,\infty)\times [0,\infty)\rightarrow [0,\infty)$ satisfying $h_{t}^{\lambda}(x,y)=\bar{h}_{t}^{\lambda}\big(|x|,|y|\big)$.  Define $\bar{\phi}^{\lambda}_{r},\bar{\psi}^{\lambda}_{r}:[0,\infty)\rightarrow [0,\infty] $   for $r>0$ by
\begin{align}\label{BarPsiPhi}
   \bar{\phi}^{\lambda}_{r}(a)\,:=\, &\,I_0(r a )\,+\,\frac{2 }{\log \frac{r^2}{2\lambda} }\, K_0(ra) \,, \hspace{.5cm} r\,\neq \,\sqrt{2\lambda} \,, \nonumber \\
\bar{\psi}^{\lambda}_{r}(a)\,:=\,&\,\frac{ \big(\log \frac{r^2}{2\lambda}-i\pi\big) J_0(r a )+i\pi H_0^{(1)}(ra) }{ \big(\big(\log \frac{r^2}{2\lambda}\big)^2+\pi^2\big)^{\frac{1}{2}}}\,=\,\frac{\log\big(\frac{r^2}{2\lambda}\big) J_0(r a )-\pi  Y_0(ra) }{ \big(\big(\log \frac{r^2}{2\lambda}\big)^2+\pi^2\big)^{\frac{1}{2}} } \,,
\end{align}
in which $Y_0$ is the order-$0$ Bessel function of the second kind, and we have used that $ H_0^{(1)}=J_0+iY_0$. The forms of $\bar{\phi}^{\lambda}_{r}$ and $\bar{\psi}^{\lambda}_{r}$ are derived by averaging $\phi^{\lambda}_k$ and $\psi^{\lambda}_k$ over $k\in \R^2$ with $|k|=r$ as follows:
\begin{align}\label{BarPsiPh2}
\bar{\phi}^{\lambda}_{r}\big(|x|\big)\,= \,\frac{1}{2\pi r}\,\int_{|k|=r}\, \phi^{\lambda}_k(x)\,dk  \hspace{.5cm}\text{and}  \hspace{.5cm} \bar{\psi}^{\lambda}_{r}\big(|x|\big)\,= \,\textup{sgn}\Big( \log \frac{r^2}{2\lambda}-i\pi  \Big)\,\frac{1}{2\pi r}\,\int_{|k|=r}\, \psi^{\lambda}_k(x)\,dk \,,
\end{align}
where $\textup{sgn}(z):=\frac{z}{|z|}$ for  $z\in \C\backslash\{0\}$.  The  sign factor above is introduced to make $\bar{\psi}^{\lambda}_{r}$ real-valued, as seen in~(\ref{BarPsiPhi}).  The functions $\{\bar{\psi}^{\lambda}_{r}\}_{r\in (0,\infty)}$ appear in the theory of scattering for a two-dimensional quantum particle from a point potential;  see~\cite[pp.\ 100-101 \& (5.27)]{AGHH}, wherein different conventions are adopted.  After invoking the definitions of $\bar{\phi}^{\lambda}_{r}$ and $\bar{\psi}^{\lambda}_{r}$, the identities $I_0(a)=\frac{1}{2\pi}\int_0^{2\pi}e^{a\cos(\theta) }d\theta  $ and $J_0(a)=\frac{1}{2\pi}\int_0^{2\pi}e^{ia\cos(\theta) }d\theta  $ connect the expressions in~(\ref{BarPsiPh2}) with those in~(\ref{BarPsiPhi}).
Finally, we set $\bar{\psi}^{\lambda}(a):=2 \lambda^{1/2}K_0\big(\sqrt{2\lambda}  a\big)$, so that $\bar{\psi}^{\lambda}$ is normalized in $ L^2\big([0,\infty),  rdr\big)$ and $\bar{\psi}^{\lambda}(|x|)=\sqrt{2\pi}\psi^{\lambda}(x)$.  The following is a corollary  of Proposition~\ref{PropRealUnnormEigen}.
\begin{corollary}\label{CoroRealUnnormEigen} Given $t,\lambda>0$ let  $\bar{f}_t^{\lambda}:[0,\infty)\times [0,\infty)\rightarrow [0,\infty] $ be defined by $\bar{f}_t^{\lambda}(a,b):= q_t(a,b)+2\pi b\, \bar{h}_{t}^{\lambda}(a,b)$.
\begin{enumerate}[(i)]
    \item For all  $a,r >0$  with $r\neq \sqrt{2\lambda}$, we have $
\int_{0}^{\infty}  \bar{f}^{\lambda}_t(a,b)\bar{\phi}^{\lambda}_r(b)db =e^{t\frac{r^2}{2} }\bar{\phi}^{\lambda}_{r}(a) $. 
    
    \item For all  $a,r>0$, we have $\int_{0}^{\infty}  \bar{f}^{\lambda}_t(a,b)\bar{\psi}^{\lambda}_r(b)db =e^{-t\frac{r^2}{2} }\bar{\psi}^{\lambda}_{r}(a)$.

\end{enumerate}
In addition, for the function $\bar{\psi}^{\lambda}\in L^2\big([0,\infty),  rdr\big)$, we have  $
\int_{0}^{\infty} \bar{f}^{\lambda}_t(a,b)\bar{\psi}^{\lambda}(b)db =e^{t\lambda }\bar{\psi}^{\lambda}(a) $. 
    
\end{corollary}

 Corollary~\ref{CoroRealUnnormEigen} and the symmetry $ a\bar{f}_t^{\lambda}(a,b)=b\bar{f}_t^{\lambda}(b,a) $ imply that $\int_0^{\infty}\bar{\psi}^{\lambda}(a)\bar{\psi}^{\lambda}_r(a)ada=0 $ for all $r>0$, that is the eigenfunction $\bar{\psi}^{\lambda}$ for the linear operator acting on the real Hilbert space $\mathcal{H}:=L^2\big([0,\infty), rdr\big)$ with kernel $\bar{f}^{\lambda}_t(a,b)$  is orthogonal to the family of generalized eigenfunctions $\{\bar{\psi}^{\lambda}_r\}_{r\in (0,\infty)}$, which formally satisfies
 \begin{align*}
 \int_0^{\infty}\, \bar{\psi}^{\lambda}_{r'}(a)\, \bar{\psi}^{\lambda}_{r}(a) \,a\,da   \,=\, \frac{1}{r}\delta(r'-r)\,,\hspace{.5cm}r',r>0\,.
 \end{align*}
 Given an appropriate function $\varphi:[0,\infty)\rightarrow \R$, say $\varphi\in  C_c^1\big([0,\infty)  \big)  
 $ for simplicity, define the following  Hankel-like transform of it, $U^{\lambda}\varphi:[0,\infty)\rightarrow \R$,   by  
$$ \big(U^{\lambda}\varphi\big)(r)\,:=\, \int_0^{\infty}\, \bar{\psi}^{\lambda}_r(b)\,\varphi(b) \,b\,db\, \hspace{.4cm}\text{for  $r>0$ and $\big(U^{\lambda}\varphi\big)(0):=0$\,.}  $$
The following lemma, whose proof we omit, essentially states that $\{\bar{\psi}^{\lambda}, \bar{\psi}^{\lambda}_{r}\}_{r\in (0,\infty)}$ is a  generalized basis for $\mathcal{H}$. Note that $\bar{\psi}^{\lambda}_r(b)$ does not have the homogeneous form $k(rb)$, as do the generalized Fourier kernels considered in~\cite[Ch.\ VIII]{Titch} (although it is a weighted combination, $ A(r) k_1(rb)+ B(r) k_2(rb) $, of kernels of this type).  A different type of inhomogeneity is present in the Fourier-type formula~(\ref{ToVarphi}), this being the   projection onto the vector $\bar{\psi}^{\lambda}\in \mathcal{H}$. 
\begin{lemma}\label{LemHankel}  The linear map $U^{\lambda}$ extends continuously to a partial isometry on $\mathcal{H}$, having $\bar{\psi}^{\lambda}$ as a null vector   and preserving norm on the orthogonal complement of  $\bar{\psi}^{\lambda}$ in $ \mathcal{H} $.   For any $\varphi\in   C_c^1\big([0,\infty)  \big)  $, we have 
\begin{align}\label{ToVarphi}
\varphi\,=\,\big\langle \bar{\psi}^{\lambda} ,\, \varphi\big\rangle_{\mathcal{H}}\,\bar{\psi}^{\lambda}  \,+\, \int_0^{\infty}\,\big(U^{\lambda}\varphi\big)(r)\,\bar{\psi}^{\lambda}_r\,r\, dr\,.  
\end{align} 
\end{lemma}

   For $T,\lambda, r>0$ define the function $\bar{\psi}^{T,\lambda}_r:[0,\infty)\rightarrow \R$ by
\begin{align*}
\bar{\psi}^{T,\lambda}_r(a)\,:=\,  \frac{\bar{\psi}^{\lambda}_r(a)}{ 1+\bar{H}_{T}^{\lambda}(a)}  \hspace{.5cm}\text{for $a>0$ and} \hspace{.5cm}\bar{\psi}^{T,\lambda}_r(0)\,:=\, \lim_{a\searrow 0}\,\bar{\psi}^{T,\lambda}_r(a)\,=\, \frac{  1 }{ \nu(T\lambda)  \big(\log \frac{r^2}{2\lambda}-i\pi\big) }   \,,
\end{align*}
and let  $\bar{\psi}^{T,\lambda}_0$ be defined analogously in terms of $\bar{\psi}^{\lambda}$. 
\begin{lemma}\label{LemmaSpanII} Fix any $T,\lambda> 0$.  The family of functions  $\{\bar{\psi}^{T,\lambda}_r\}_{r\in [0,\infty)}$ is a determining class  on $[0,\infty)$.
\end{lemma}
\begin{proof} Let $\vartheta_1$ and $\vartheta_2$ be finite Borel measures on $[0,\infty)$ such that $\vartheta_1\big(\bar{\psi}^{T,\lambda}_r\big)=\vartheta_2\big(\bar{\psi}^{T,\lambda}_r\big)  $ for all $r\in [0,\infty)$. Since $\bar{\psi}^{T,\lambda}_0(0)>0$, it suffices to show that the restrictions of $\vartheta_1$ and $\vartheta_2$ to $(0,\infty)$ are equal, which holds if the measures $\vartheta_1^{T,\lambda}(da):=\frac{1}{ 1+\bar{H}_{T}^{\lambda}(a)}\vartheta_1(da)  $ and $\vartheta_2^{T,\lambda}(da):=\frac{1}{ 1+\bar{H}_{T}^{\lambda}(a)}\vartheta_2(da)  $ are equal, because $\frac{1}{ 1+\bar{H}_{T}^{\lambda}(a)} >0$ for $a>0$.  Observe that $\vartheta_j^{T,\lambda}(\bar{\psi}^{\lambda})=\vartheta_j(\bar{\psi}^{T,\lambda})$ and  $\vartheta_j^{T,\lambda}(\bar{\psi}^{\lambda}_r)=\vartheta_j(\bar{\psi}^{T,\lambda}_r)$, so we
 can apply Lemma~\ref{LemHankel} twice with any $\varphi\in C^{1}_c\big([0,\infty)\big)$ to get
\begin{align*}
\vartheta_1^{T,\lambda}(\varphi)\stackrel{(\ref{ToVarphi})}{=} &\langle \bar{\psi}^{\lambda} ,\, \varphi\rangle_{\mathcal{H}}\,\vartheta_1\big(\bar{\psi}^{T\lambda}_0\big)  \,+\, \int_0^{\infty}\,(U^{\lambda}\varphi)(r)\,\vartheta_1\big(\bar{\psi}^{T\lambda}_r\big)\,r\, dr \\  \,=\,\,\,&  \langle \bar{\psi}^{\lambda} ,\, \varphi\rangle_{\mathcal{H}}\,\vartheta_2\big(\bar{\psi}^{T\lambda}_0\big)  \,+\, \int_0^{\infty}\,(U^{\lambda}\varphi)(r)\,\vartheta_2\big(\bar{\psi}^{T\lambda}_r\big)\,r\, dr     \stackrel{(\ref{ToVarphi})}{=}  \vartheta_2^{T,\lambda}(\varphi)\,.
\end{align*}
Since $C^{1}_c\big([0,\infty)\big)$ is a determining class, we have that $\vartheta_1^{T,\lambda} =\vartheta_2^{T,\lambda} $.
\end{proof}

\section{The path measure} \label{AppendixYMart}
We will prove Proposition~\ref{PropCP} at the end of this subsection after presenting some technical lemmas involving the transition density function $\mathlarger{d}_{s,t}^{T,\lambda}(x,y) $ and the function $\Rphi_{r}^{\lambda}:\R^2\rightarrow \R^2$ defined by $\Rphi_{r}^{\lambda}(x):=\frac{x}{1+H_{r}^{\lambda}(x)  }$.  The process $\{Y_t^{T,\lambda}\}_{t\in [0,\infty)}$ with $Y_t^{T,\lambda}:=\Rphi_{T-t}^{\lambda}(X_t)$ is a $\mathbf{P}^{T,\lambda}_{\mu}$-martingale, which we will use in Appendix~\ref{AppendixWeakSolConst} to construct the two-dimensional Brownian motion $W^{T,\lambda}$ in Proposition~\ref{PropStochPre}.

Given $  \lambda > 0$ and $T \in \R$, the radial  symmetry of  $H_{T}^{\lambda}(x)$  in $x\in \R^2$ implies that the  Jacobian matrix $\sigma_{T}^{\lambda}(x):=\frac{\partial}{\partial x}\Rphi_{T}^{\lambda}(x)$ has the form
\begin{align*}
\sigma_{T}^{\lambda}(x)\,:=\, \check{\epsilon}_{T}^{\lambda}(x)\,P_x \,+\,\hat{\epsilon}_{T}^{\lambda}(x)\,\big(I-P_x\big)  \,,
\end{align*}
where $P_x:\R^2\rightarrow \R^2$ is the orthogonal projection  onto  $x$, and the eigenvalues are
\begin{align}\label{SigmaForm}
\check{\epsilon}_{T}^{\lambda}(x)\,:=\,\frac{1+H_{T}^{\lambda}(x)\,+\,|x|\,\big|\nabla_x H_{T}^{\lambda}(x)\big|}{\big(1+H_{T}^{\lambda}(x)\big)^2  } \hspace{.7cm}\text{and}\hspace{.7cm} \hat{\epsilon}_{T}^{\lambda}(x)\,:=\,\frac{1}{1+H_{T}^{\lambda}(x)  }\,. \end{align} 
The second eigenvalue $\hat{\epsilon}_{T}^{\lambda}(x)$ is bounded by one since $H_{T}^{\lambda}(x)\geq 0$, and the lemma below provides a range of the variables $T,\lambda,x$ within which $\check{\epsilon}_{T}^{\lambda}(x)$ is bounded.  We omit the proof, which follows easily from Lemma~\ref{LemmaUpsilonUpDown}
\begin{lemma}\label{LemmaSpecialUpperBounds}  For any $L>0$ there exists a $C_L>0$ such that $\check{\epsilon}_{T}^{\lambda}(x) \leq C_L$   for all $x\in \R^2$ and $T,\lambda>0$ with $T\lambda\leq L$.
\end{lemma}

Given a  twice-differentiable function  $f:[0,\infty)\times \R^2\rightarrow \R$ with compact support, the  Kolmogorov equation~(\ref{KolmogorovForJ}) implies that
\begin{align}\label{KolmogorovForJBack}
\frac{\partial}{\partial h} \,\int_{\R^2}\, \mathlarger{d}_{t,t+h}^{T,\lambda}(x,y)\, f\big(t+h,y\big)\,dy\,\bigg|_{h=0}\,=\, 
 \Big(\,\mathcal{L}^{T-t,\lambda}_x  \,+\,\frac{\partial}{\partial t}\,\Big) \,f(t,x) \,,  
\end{align}
 for the differential operator $\mathcal{L}^{t,\lambda}_x :=\frac{1}{2}\Delta_x +\frac{ \nabla_x H_{t}^{\lambda}(x)   }{ 1+H_{t}^{\lambda}(x)   }\cdot \nabla_x$.  Note that we will apply differential operators to a function $F:\R^2\rightarrow \R^2$ component-wise, e.g.\  $\Delta F:=\big(\Delta F_1, \Delta F_2\big)$ when  $F=(F_1,F_2)$. The proof of the next lemma  is  straightforward, using  the definitions of $\mathlarger{d}_{s,t}^{T,\lambda}(x,y)$ and $\Rphi_{T-t}^{\lambda}(y)$ along with symmetry for~(\ref{MartFunEQ}), and  calculus for~(\ref{ChainRULE}). 
\begin{lemma}\label{LemmaMartFUN} The equality below holds for all $T,\lambda>0$, $x\in \R^2$, and $0\leq s<t<\infty$
\begin{align}\label{MartFunEQ}
    \int_{\R^2}  \,\mathlarger{d}_{s,t}^{T,\lambda}(x,y) \,\Rphi_{T-t}^{\lambda}(y)\,dy\,=\,\Rphi_{T-s}^{\lambda}(x)\,.
\end{align}
Moreover, the $\R^2$-valued function  defined by $(t,x)\mapsto \Rphi_{t}^{\lambda}(x)$ satisfies the diffusion equation $\frac{\partial }{\partial t}\Rphi_{t}^{\lambda}(x)=\mathcal{L}^{t,\lambda}_x\Rphi_{t}^{\lambda}(x)$.  For any twice-differentiable function $G:\R^2\rightarrow \R$, we have the formula
\begin{align}\label{ChainRULE}
\Big( \, \mathcal{L}_{x}^{t,\lambda}  \,  -\,\frac{\partial}{\partial t}\,\Big) \, G\big(\Rphi_{t}^{\lambda}(x)  \big)  \,=\,\frac{1}{2}\,\Big(\,\big( \check{\epsilon}_{t}^{\lambda}(x) \big)^2 \, \partial_{x\parallel }^2G(z) \,+\,\big( \hat{\epsilon}_{t}^{\lambda}(x) \big)^2\,\partial_{x\perp}^2G(z) \,\Big)\,\bigg|_{z= \Rphi_{t}^{\lambda}(x)}\,,
\end{align}
where $\partial_{x\parallel}G(z)$ and $\partial_{x\perp}G(z)$  denote the partial derivatives of $G$ at $z\in \R^2$ in the direction of $x$ and  the 90-degree clockwise rotation of $x$, respectively.  
\end{lemma}

The following lemma will be used to check  Kolmogorov's criterion for the process $Y_t^{T,\lambda}=\Rphi_{T-t}^{\lambda}(X_t)$. 
\begin{lemma}\label{LemmaKOLMOGOROV}  Given any $L>0$ there exists a $\mathbf{c}_L>0$ such that for all  $x,y\in \R^2$, $m\in \mathbb{N}_0$, $0\leq s<t<\infty$, and $T,\lambda>0$ with $T\lambda\leq L$
\begin{align*}
\int_{\R^2} \, \mathlarger{d}_{s,t}^{T,\lambda}(x,y)\,\big| \Rphi_{T-t}^{\lambda} (y) -\Rphi_{T-s}^{\lambda}(x) \big|^{2m}\, dy \,\leq \,  \mathbf{c}_{L}^m \, (m!)^2 \, (t-s)^m  \,. 
\end{align*}

\end{lemma}

\begin{proof}   Recall from~(\ref{JExtentd}) that the transition density function $ \mathlarger{d}_{s,t}^{T,\lambda}(x,y)$ shifts to  a Gaussian form when $t>s\geq T$. Thus, it suffices for us to focus on the  case $0<s<t\leq T$.  Furthermore, since $ \mathlarger{d}_{s,t}^{T,\lambda}= \mathlarger{d}_{0,t-s}^{T-s,\lambda}$, we sacrifice no generality in assuming that $s=0$. For $m\in \mathbb{N}_0$, define
$$G_{m}^{T,\lambda}(x,y)\,:=\,\big| y -\Rphi_{T}^{\lambda}(x) \big|^{2m}\hspace{.7cm}\text{and}\hspace{.7cm} 
 \mathbf{H}_{m}^{T,\lambda}(t,x)\,:=\,\int_{\R^2} \, \mathlarger{d}_{0,t}^{T,\lambda}(x,y)\,G_{m}^{T,\lambda}\big(x,\Rphi_{T-t}^{\lambda}(y) \big)\, dy \,. $$  
 Define $\mathbf{c}_L:=2C_L^2$ for the constant $C_L>0$ from Lemma~\ref{LemmaSpecialUpperBounds}.  In the analysis below, we show that for all  $x\in \R^2$, $m\in \mathbb{N}$, and $T,\lambda>0$ with  $T\lambda\leq L$ 
\begin{align}\label{HSUFFICIENT}
\frac{\partial }{\partial t}\mathbf{H}_{m}^{T,\lambda}(t,x)\,\leq \,\mathbf{c}_L \,m^2\,\mathbf{H}_{m-1}^{T,\lambda}(t,x)\,.  
\end{align}
Since $\mathbf{H}_{0}^{T,\lambda}(t,x)=1$ and $\mathbf{H}_{m}^{T,\lambda}(0,x)=0$ for all $x$, we get the desired inequality $\mathbf{H}_{m}^{T,\lambda}(t,x)\leq \mathbf{c}_L^m (m!)^2 t^m $ over the given range of variables through iterated integration in $t$.

By bringing the derivative $\frac{\partial}{\partial t}$ inside the integral and applying the product rule, we have the first equality below.
\begin{align}\label{HINEQUALITY}
  \frac{\partial}{\partial t} \mathbf{H}_{m}^{T,\lambda}(t,x) \nonumber    \,=\,  &\, \int_{\R^2} \,\Big(\frac{\partial}{\partial t} \mathlarger{d}_{0,t}^{T,\lambda}(x,y)\Big)\,G_{m}^{T,\lambda}\big(x,\Rphi_{T-t}^{\lambda}(y) \big)\, dy \,+\,\int_{\R^2}\,  \mathlarger{d}_{0,t}^{T,\lambda}(x,y) \,\frac{\partial }{\partial t}G_{m}^{T,\lambda}\big(x,\Rphi_{T-t}^{\lambda}(y) \big)\, dy \nonumber \\
  \,=\,  &\, \int_{\R^2} \,\Big(\big(\mathcal{L}^{T-t,\lambda}_y\big)^* \,\mathlarger{d}_{0,t}^{T,\lambda}(x,y)\Big)\,G_{m}^{T,\lambda}\big(x,\Rphi_{T-t}^{\lambda}(y) \big)\, dy \,+\,\int_{\R^2}  \,\mathlarger{d}_{0,t}^{T,\lambda}(x,y) \,\frac{\partial }{\partial t}G_{m}^{T,\lambda}\big(x,\Rphi_{T-t}^{\lambda}(y) \big)\, dy \nonumber \\
  \,=\, &\, \int_{\R^2}  \,\mathlarger{d}_{0,t}^{T,\lambda}(x,y)\,\Big(\mathcal{L}_{y}^{T-t,\lambda}\,+\,\frac{\partial}{\partial t} \Big)  \, G_{m}^{T,\lambda}\big(x,\Rphi_{T-t}^{\lambda}(y) \big)\, dy \nonumber 
  \\
  \,=\, &\,\frac{1}{2}\, \int_{\R^2}  \,\mathlarger{d}_{0,t}^{T,\lambda}(x,y) \,\Big(\big(\check{\epsilon}_{T-t}^{\lambda}(y) \big)^2\,\partial^2_{z\parallel}  G_{m}^{T,\lambda}(x, z) + \big(\hat{\epsilon}_{T-t}^{\lambda}(y) \big)^2 \,\partial^2_{z\perp}  G_{m}^{T,\lambda} (x,z)  \Big)\Big|_{z=\Rphi_{T-t}^{\lambda}(y)}\, dy 
\end{align}
The second equality applies the forward Kolmogorov equation~(\ref{KolmogorovForJ}),
the  third equality uses integration by parts, and the fourth equality  holds by~(\ref{ChainRULE}).  The convexity of the function $G_{m}^{T,\lambda}(x,z)$ in the variable $z\in \R^2$ implies that the second-order derivatives  $\partial^2_{z\parallel}  G_{m}^{T,\lambda}(x,z)$ and $\partial^2_{z\perp}  G_{m}^{T,\lambda}(x,z) $ are nonnegative,
and so for $m\geq 1$ we have  
\begin{align}\label{TIZZY}
\frac{1}{2}\,\Big(\big(\check{\epsilon}_{T-t}^{\lambda}(y) \big)^2\,\partial^2_{z\parallel}  G_{m}^{T,\lambda}(x, z) \,&+\, \big(\hat{\epsilon}_{T-t}^{\lambda}(y) \big)^2 \,\partial^2_{z\perp}  G_{m}^{T,\lambda} (x,z)  \Big)\Big|_{z=\Rphi_{T-t}^{\lambda}(y)} \nonumber \\ & \,\leq \,\frac{\mathbf{c}_L}{4}\,\Delta_z G_{m}^{T,\lambda}(x,z)\Big|_{z=\Rphi_{T-t}^{\lambda}(y)} \,=\,\mathbf{c}_{L}\,m^2 \,G_{m-1}^{T,\lambda}\big(x,\Rphi_{T-t}^{\lambda}(y)  \big) \,,
\end{align}
 where the inequality applies  Lemma~\ref{LemmaSpecialUpperBounds}.
 The equality in~(\ref{TIZZY}) holds because $ \Delta_z G_{m}^{T,\lambda}(x,z)=4m^2\big|z-\Rphi_T^{\lambda}(x)  \big|^{2(m-1)}= 4m^2G_{m-1}^{T,\lambda}(x,z)$.  Applying~(\ref{TIZZY}) within~(\ref{HINEQUALITY}) yields~(\ref{HSUFFICIENT}).
\end{proof}

\begin{proof}[Proof of Proposition~\ref{PropCP}] For each $x\in \R^2$, the   transition density function $  \mathlarger{d}_{s,t}^{T,\lambda }  $ determines a consistent family of finite-dimensional distributions $\{ P_{x;\,t_1,\ldots, t_N}^{T,\lambda} \}_{\,0\leq t_1<\cdots <t_N<\infty}^{N\in \mathbb{N}}$, in which $P_{x;\,t_1,\ldots, t_N}^{T,\lambda} $ is the Borel probability measure on $(\R^2)^N$  assigning a  set $E$ the weight
\begin{align}\label{FDD}
P_{x;\,t_1,\ldots, t_N}^{T,\lambda}(E)\,=\, \int_{E} \, \prod_{k=1}^N  \, \mathlarger{d}_{t_{k-1},t_k}^{T,\lambda}(x_{k-1},x_k)\,dx_1\cdots dx_N \,, 
\end{align}
wherein $t_0:=0$ and $x_0:=x$.  By Kolmogorov's extension theorem, this consistent family of finite-dimensional distributions  uniquely determines a probability measure $P^{T,\lambda}_x$ on the measurable space $\big(\mathbf{\widehat{\Omega}}  ,\mathscr{B}\big( \mathbf{\widehat{\Omega}}\big) \big)$ for the function space $\mathbf{\widehat{\Omega}}:=(\R^2)^{[0,\infty)}$, which is  equipped with the topology of pointwise convergence.  Let $\{\mathbf{X}_t\}_{t\in [0,\infty)}$ and $\{\mathbf{F}_t\}_{t\in [0,\infty)}$ denote the coordinate process on $\mathbf{\widehat{\Omega}}$ and  the filtration that it  generates, respectively.  Define the process $\{ \mathbf{Y}^{T,\lambda}_t \}_{t\in [0,\infty)}$  by $\mathbf{Y}^{T,\lambda}_t=\Rphi_{T-t}^{\lambda}(\mathbf{X}_t)$.
 As a consequence of Lemma~\ref{LemmaKOLMOGOROV}, for any given $m>1$ there exists a constant $\mathbf{C}^{T,\lambda}_m>0$ such that the inequality below holds  for all $x\in \R^2$ and $0\leq s<t<\infty$.
\begin{align*}
E^{T,\lambda}_{x}\Big[\, \big|\mathbf{Y}^{T,\lambda}_t-\mathbf{Y}^{T,\lambda}_s \big|^{2m}\, \Big]\,=\,&\,\int_{\R^2}  \,\mathlarger{d}_{0,s}^{T,\lambda}(x,y)\,\int_{\R^2} \, \mathlarger{d}_{s,t}^{T,\lambda}(y,z)\,\big| \Rphi_{T-t}^{\lambda}(z)-\Rphi_{T-s}^{\lambda} (y) \big|^{2m}\, dz\, dy  \,\leq \, \mathbf{C}^{T,\lambda}_m\,|t-s|^m
\end{align*}
 Thus, Kolmogorov's condition is satisfied for the process $\mathbf{Y}^{T,\lambda} $ under $P^{T,\lambda}_x$. Hence for any $x\in \R^2$ there exists a unique   measure $\mathbf{Q}^{T,\lambda}_x $ on $\big(\boldsymbol{\Omega}, \mathscr{B}\big( \boldsymbol{\Omega}  \big) \big)$ such that 
 the finite-dimensional distributions of $\mathbf{Q}^{T,\lambda}_x$ coincide with the finite-dimensional distributions of the process $\mathbf{Y}^{T,\lambda}$ under $P_{x}^{T,\lambda}$.  We define $\mathbf{P}^{T,\lambda}_x$ as the pullback of  $\mathbf{Q}^{T,\lambda}_x$ under  the bicontinuous map  $\Phi_{T}^{\lambda}:\boldsymbol{\Omega}\rightarrow \boldsymbol{\Omega}$ given by
$$  \big(\Phi_{T}^{\lambda} \, \omega\big)(t)\,:=\, \Rphi_{T-t}^{\lambda}\big(\omega(t)  \big) \,, \hspace{.5cm} \omega \in \boldsymbol{\Omega}\,,\,\,t\in [0,\infty)\,,  $$
that is $\mathbf{P}^{T,\lambda}_x:=\mathbf{Q}^{T,\lambda}_x\circ \Phi_{T}^{\lambda}$.   It follows from this construction that the finite-dimensional distributions of the measure $\mathbf{P}^{T,\lambda}_x$ are equal to those of $P_{x}^{T,\lambda}$. 

We  can similarly use Lemma~\ref{LemmaKOLMOGOROV}  to show that the set of path measures $ S_L:= \big\{ \mathbf{P}_{x}^{T,\lambda}\, :\, T,\lambda,|x|\leq L\big\}$ for some $L>0$, that is with a bounded range on the controlling parameters,  is tight under the weak topology.  The claimed continuity and limit property as $\lambda\searrow 0$ can thus be reduced to that for the finite-dimensional distributions, which is trivial given their construction~(\ref{FDD})  from the transition density function $\mathlarger{d}_{s,t}^{T,\lambda}(x,y)$.
\end{proof}

\section{Strong Markov property}\label{AppendixSMP}
We will discuss how to adapt the proof of strong Markovianity for a $d$-dimensional Brownian motion in~\cite[Section 2.6(C)]{Karatzas} to the coordinate process $\{X_t\}_{t\in [0,\infty) }$ under the Markov family $\{\mathbf{P}^{T,\lambda}_{x}\}_{x\in \R^2}$ and with respect to  $\{\mathscr{F}_t^X\}_{t\in [0,\infty) }$. A key part of this argument, which we will formulate in terms of  $\{X_t\}_{t\in [0,\infty) }$  and the family of two-dimensional Wiener measures  $\{\mathbf{P}_{x}\}_{x\in \R^2}$, is the observation that for every $x,k\in \R^2$ the complex-value  process $\{M_t^k\}_{t\in [0,\infty)}$ given by  $M_t^{k} = e^{\frac{t}{2}|k|^2} \psi_k(X_t) $, for $\psi_k:\R^2\rightarrow \C$  defined by $\psi_k(x)=e^{ik\cdot x}$,   is a $\mathbf{P}_{x}$-martingale with respect to $\mathscr{F}^X$.  This, of course, can be seen through a routine application of It\^o's rule, which yields the driftless SDE $dM^{k} =i M^{k}\,k\cdot dX_t $.  The form of $M^{k}_t$ should be understood in terms of $\psi_k $ being a generalized eigenfunction for the backward generator $\frac{1}{2}\Delta_x$  with value $-\frac{1}{2}|k|^2$, that is $\frac{1}{2}\Delta_x \psi_k(x)=-\frac{1}{2}|k|^2\psi_k(x)$. A crucial feature possessed by the family of functions  
$\{\psi_k\}_{k\in \R^2}$ is  that a finite Borel measure $\vartheta$ on $\R^2$ is uniquely determined by the values $ \vartheta(  \psi_k ) $, which is to say through its characteristic function. 
We will define an analogous family of functions $\{\boldsymbol{\psi}_k^{T,\lambda}\}_{k\in \R^2}$ and corresponding martingales  that  can be substituted into the argument~\cite[Section 2.6(C)]{Karatzas} to deduce that $\{\mathbf{P}^{T,\lambda}_{x}\}_{x\in \R^2}$ is  strong Markov.

For $\lambda>0$ and $k\in \R^2\backslash\{0\}$, recall that $\psi_k^{\lambda}:\R^2\rightarrow \C$ is defined as in~(\ref{PsiLambda}).  We further define $\boldsymbol{\psi}^{T,\lambda}_{k }:\R^2\rightarrow \C $ for some $T>0$  by $$ \boldsymbol{\psi}^{T,\lambda}_{k }(x):= \frac{ \psi^{\lambda }_{k}(x)  }{ 1+H_{T}^{\lambda}(x)  }  \hspace{.5cm}  \text{when $x\neq 0$ and} \hspace{.5cm}  \boldsymbol{\psi}^{T,\lambda}_{k}(0):= \lim_{x\rightarrow 0  }\,\boldsymbol{\psi}^{T,\lambda}_{k}(x)\,=\,\frac{1}{\nu(T\lambda)} \frac{1}{\log \frac{|k|^2}{2\lambda} -i\pi  }\,. $$  
Let $\boldsymbol{\psi}^{T,\lambda}_{0 }$ be defined analogously with $\psi^{\lambda }_{k}$ replaced by  $\psi^{\lambda }(x)=(\frac{2\lambda}{\pi})^{1/2} K_0\big(\sqrt{2\lambda} |x|  \big)$.  The function $\boldsymbol{\psi}^{T,\lambda}_{k}(x)$ is  bounded, continuous, and approximately equal to $\psi_k(x)$ when $|x|\gg 1$.  The next lemma  follows trivially from   Proposition~\ref{PropRealUnnormEigen}  after invoking the definitions of $\mathlarger{d}_{s,t}^{T,\lambda}$ and $\boldsymbol{\psi}^{T-s,\lambda}_{k}$.
\begin{lemma}\label{LemmaBasicMart}  Fix some $T,\lambda>0$ and $x,k\in \R^2$.  Define the process   $\{  M^{T,\lambda,k}_t \}_{t\in [0,T]}  $  by $ M^{T,\lambda,k}_t:=e^{\frac{t}{2}|k|^2    } \boldsymbol{\psi}^{T-t,\lambda}_{k}(X_t)$ when $k\in \R^2\backslash\{0\}$ and set $ M^{T,\lambda,0}_t:=e^{-t\lambda    } \boldsymbol{\psi}^{T-t,\lambda}_{0}(X_t)$.   Then $M^{T,\lambda,k}$ 
is a bounded   $\mathbf{P}^{T,\lambda}_{x}$-martingale with respect to $\{\mathscr{F}^X_t\}_{t\in [0,T]}$.
\end{lemma}

The following  lemma  shows that the family of functions $\{ \boldsymbol{\psi}_{k}^{T-t,\lambda}  \}_{k\in \R^2 }  $, which we used in  the definition of the  random variables   $\{M^{T,\lambda,k}_t\}_{k\in \R^2}$,   is rich enough to uniquely  determine Borel measures on $\R^2$.
\begin{lemma}\label{LemmaCompleteSet}  Fix some $T,\lambda>0$. The family $\{ \boldsymbol{\psi}^{T,\lambda}_{k} \}_{k\in \R^2  } $
is a determining class  on $\R^2$.
\end{lemma}

\begin{proof} Let $\vartheta_1$ and $\vartheta_2$ be finite Borel measures on $\R^2$.   Suppose that $\vartheta_1\big( \boldsymbol{\psi}^{T,\lambda}_{k} \big)=\vartheta_2\big( \boldsymbol{\psi}^{T,\lambda }_{k } \big)$ for all $k\in \R^2$.  Since $\boldsymbol{\psi}^{T,\lambda}_{0}(0)\neq 0$, it suffices for us to show that the restrictions  $\vartheta_1 |_{\R^2\backslash \{0\}}  $ and $ \vartheta_2 |_{\R^2\backslash \{0\}}  $ are equal, that is, the measures $\vartheta_1$ and $\vartheta_2$ agree up to an atom at the origin.
For $r\in [0,\infty)$ define the radially symmetric function $\boldsymbol{\bar{\psi}}^{T,\lambda }_{r}:\R^2\rightarrow \R$ by $\boldsymbol{\bar{\psi}}^{T,\lambda }_{r}(x)=\bar{\psi}^{T,\lambda}_r\big(|x|\big)$, where  $\bar{\psi}^{T,\lambda}_r:[0,\infty)\rightarrow \R$ is defined above Lemma~\ref{LemmaSpanII}. It is to be observed that $\vartheta_1\big(  \boldsymbol{\bar{\psi}}^{T,\lambda }_{r}\big)=\vartheta_2\big( \boldsymbol{\bar{\psi}}^{T,\lambda }_{r} \big)$ for all $r\in [0,\infty)$, because  $\boldsymbol{\bar{\psi}}^{T,\lambda}_0=\boldsymbol{\psi}^{T,\lambda }_{0 }$, and we can express $\boldsymbol{\bar{\psi}}^{T,\lambda }_{r}$ in terms of the functions $\psi^{T,\lambda }_{k}$ with $|k|=r$ as follows:
\begin{align*}
\textup{sgn}\Big(  \log \frac{r^2}{2\lambda} -i \pi \Big)\,\frac{1}{2\pi r}\,\int_{|k|=r}\, \psi^{T,\lambda }_{k}(x)\, dk\,&\,=\,\textup{sgn}\Big(  \log \frac{r^2}{2\lambda} -i \pi \Big) \,\frac{1}{2\pi r}\,\int_{|k|=r} \, \frac{ \psi^{\lambda}_k(x)   }{  1+H_{T}^{\lambda}(x) }\, dk\\ \,&  \hspace{-.14cm}\stackrel{ (\ref{BarPsiPh2}) }{=}\, \frac{\bar{\psi}^{\lambda}_r\big(|x|\big)  }{ 1+\bar{H}_{T}^{\lambda}(|x|)  }  \, =:\, \boldsymbol{\bar{\psi}}^{T,\lambda }_{r}(x)\,.
\end{align*}
 Since the family of functions $\{\bar{\psi}^{T,\lambda }_{r}\}_{r\in [0,\infty)}$ is   a determining class  on $[0,\infty)$ by Lemma~\ref{LemmaSpanII} and the measures $\vartheta_1$ and $\vartheta_2$ agree on the family  $\{ \boldsymbol{\bar{\psi}}^{T,\lambda }_{r}\}_{r\in [0,\infty)}$,  it follows that $\vartheta_1( \psi  )=\vartheta_2( \psi )$ for all bounded, continuous  functions $\psi:\R^2\rightarrow \C$ that are radially symmetric. Consequently,  for $k\in \R^2\backslash \{0\}$  we can  rearrange the equality $\psi^{T,\lambda}_{k}(x) :=\frac{ e^{ik\cdot x}  }{  1+H_{T}^{\lambda}(x)   }+\frac{  i\pi }{\log \frac{|k|^2}{2\lambda} -i\pi }\frac{H_0^{(1)}(|k||x|) }{  1+H_{T}^{\lambda}(x)   }$  to write
\begin{align*}
\int_{\R^2}\, \frac{ e^{ik\cdot x}  }{  1+H_{T}^{\lambda}(x)   }\,\vartheta_1(dx)\,=\,&\,\underbracket{\vartheta_1\big( \psi^{T,\lambda}_{k} \big)}\,-\, \frac{  i\pi }{\log \frac{|k|^2}{2\lambda} -i\pi }\,\underbrace{\int_{\R^2}\, \frac{H_0^{(1)}\big(|k|\,|x|\big) }{  1+H_{T}^{\lambda}(x)   }\,\vartheta_1(dx)} \nonumber  \\
\,=\,&\,\underbracket{\vartheta_2\big( \psi^{T,\lambda}_{k} \big)}\,-\, \frac{  i\pi }{\log \frac{|k|^2}{2\lambda}-i\pi }\,\underbrace{\int_{\R^2} \,\frac{H_0^{(1)}\big(|k|\,|x|\big) }{  1+H_{T}^{\lambda}(x)   }\,\vartheta_2(dx)} 
\,=\,\int_{\R^2}\, \frac{ e^{ik\cdot x}  }{  1+H_{T}^{\lambda}(x)   }\,\vartheta_2(dx)\,,
\end{align*}
where the underbracketed terms are equal by assumption, and we have equality between the underbraced terms by radial symmetry. However, since the family $\{ e^{ik\cdot x}\}_{k\in \R^2\backslash\{0\} } $ is a determining class  on $\R^2$, we can conclude that the measures $ \frac{ 1  }{  1+H_{T}^{\lambda}(x)   }\vartheta_1(dx) $ and $ \frac{ 1 }{  1+H_{T}^{\lambda}(x)   }\vartheta_2(dx) $ are equal.  Since $\frac{ 1  }{  1+H_{T}^{\lambda}(x)   }>0 $ when $x\neq 0$, this implies that $\vartheta_1 |_{\R^2\backslash \{0\}}  = \vartheta_2 |_{\R^2\backslash \{0\}}  $,     and therefore   $\vartheta_1 =\vartheta_2$.
\end{proof}

\section{Constructing a weak solution to the SDE}\label{AppendixWeakSolConst}

Given a Borel measure $\mu$ on $\R^2$, recall that the augmented $\sigma$-algebra $\mathscr{B}^{T}_{\mu} $ and the augmented filtration $\{\mathscr{F}_t^{T,\mu} \}_{t\in [0,\infty)}$ are defined as in~(\ref{Augmented}). The next lemma and its corollary below concern the construction of the two-dimensional Brownian motion  $W^{T,\lambda }$ in Proposition~\ref{PropStochPre}. 
\begin{lemma}\label{LemmaMartI}  Fix some $T,\lambda >0$ and a Borel probability measure $\mu$ on $  \R^2$ with $\int_{\R^2}|x|^2\mu(dx)<\infty $.   The $\R^2$-valued process  $\{ Y_{t}^{T,\lambda} \}_{t\in [0,\infty)}$ defined by $Y_{t}^{T,\lambda}:=\Rphi_{T-t}^{\lambda}(X_t)$  is a continuous square-integrable   $\mathbf{P}^{T,\lambda}_{\mu}$-martingale with respect to the filtration $\{\mathscr{F}_t^{T,\mu}  \}_{t\in [0,\infty)}$.  For any $u\in \R^2$, the quadratic variation of the real-valued martingale $\{ u\cdot Y_{t}^{T,\lambda} \}_{t\in [0,\infty)}$ has the form 
$$  \big\langle u\cdot Y^{T,\lambda} \big\rangle_t \,=\, \int_0^t\, \big\|\sigma_{T-s}^{\lambda}(X_s) \,u\big\|_2^2   \,  ds \,. $$
\end{lemma}

\begin{proof} Given $t\geq 0$, the second moment of the random vector $Y^{T,\lambda}_t $ is finite since by Minkowski's inequality and Lemma~\ref{LemmaKOLMOGOROV}
\begin{align*}
    \mathbf{E}_{\mu}^{T,\lambda}\Big[\,\big|Y^{T,\lambda}_t \big|^{2}  \,\Big]^{\frac{1}{2}}\,\leq \,&\,\mathbf{E}_{\mu}^{T,\lambda}\Big[\,\big|Y^{T,\lambda}_0 \big|^{2}\,  \Big]^{\frac{1}{2}}\,+\, \mathbf{E}_{\mu}^{T,\lambda}\Big[\,\big|Y^{T,\lambda}_t-Y^{T,\lambda}_0 \big|^{2}  \,\Big]^{\frac{1}{2}}\\
    \,\leq  \,&\, \bigg( \int_{\R^2}\,|x|^2\,\mu(dx) \bigg)^{\frac{1}{2}}\,+\,\bigg(\int_{\R^2}\,\int_{\R^2} \, \mathlarger{d}_{0,t}^{T,\lambda}(x,y)\,\big|\Rphi_{T-t}^{\lambda}(y)-\Rphi_{T}^{\lambda}(x)   \big|^{2}\, dy \,\mu(dx)\bigg)^{\frac{1}{2}}\,<\,\infty\,,
\end{align*}
where the second inequality uses that $Y^{T,\lambda}_0=\Rphi_{T}^{\lambda}(X_0) $ and $\big|\Rphi_{T}^{\lambda}(x)\big|\leq |x|$.  Since $ \mathbf{P}_{\mu}^{T,\lambda}$ has finite-dimensional distributions determined by the transition density $\mathlarger{d}_{s,t}^{T,\lambda}$,  a standard argument using~(\ref{MartFunEQ}) shows that  $\{ Y_{t}^{T,\lambda}   \}_{t\in [0,\infty)}$ is a $ \mathbf{P}_{\mu}^{T,\lambda}$-martingale with respect to the filtration  $\mathscr{F}^{T,\mu}$.

For  $u\in \R^2$ and $t\in [0,\infty)$, the derivative of the quadratic variation of the real-valued martingale $ u\cdot Y^{T,\lambda} $  can be expressed as
\begin{align}\label{Piffle}
    \frac{d}{dt}\big\langle u\cdot Y^{T,\lambda}\big\rangle_t\,=\,\lim_{h\searrow 0}\, \frac{1}{h} \, \mathbf{E}_{\mu}^{T,\lambda}\bigg[\, \Big(u\cdot Y^{T,\lambda}_{t+h}\, -\,u\cdot Y^{T,\lambda}_{t}\Big)^2\,\bigg|\,\mathscr{F}_t^{T,\mu} \,\bigg] \hspace{.3cm} \text{ a.s.\  } \mathbf{P}_{\mu}^{T,\lambda} \,.
\end{align}
Moreover, by invoking the Markov property and that $X$ has transition density, we can write that
\begin{align*}
\mathbf{E}_{\mu}^{T,\lambda}\bigg[\, \Big(u\cdot Y^{T,\lambda}_{t+h}\, -\,u\cdot Y^{T,\lambda}_{t}\Big)^2\,\bigg|\,\mathscr{F}_t^{T,\lambda} \,\bigg]\,=\,&\,\mathbf{E}_{X_t}^{T-t,\lambda}\bigg[\, \Big(u\cdot Y^{T-t,\lambda}_{h}\, -\, u\cdot  Y^{T-t,\lambda}_{0}\Big)^2 \,\bigg] \nonumber  \\
\,=\,&\,\int_{\R^2}\, \mathlarger{d}_{0,h}^{T-t,\lambda}(X_t,y) \,  \Big(u\cdot \Rphi_{T-t-h}^{\lambda}(y) \, -\, u\cdot \Rphi_{T-t}^{\lambda}(X_t)\Big)^2\, dy\,.
\end{align*}
Thus, returning to~(\ref{Piffle}), we have that
\begin{align*}
    \frac{d}{dt}\big\langle u\cdot Y^{T,\lambda}\big\rangle_t\,=\,&\,
   \frac{\partial}{\partial h}\,\int_{\R^2}\, \mathlarger{d}_{0,h}^{T-t,\lambda}(X_t,y)\,   \Big(u\cdot \Rphi_{T-t-h}^{\lambda}(y) \, -\, u\cdot \Rphi_{T-t}^{\lambda}(X_t)\Big)^2\, dy \,\bigg|_{h=0}  \nonumber \\
    \,=\,&\,\bigg( \mathcal{L}^{T-t,\lambda}_y\,+\,\frac{\partial}{\partial s}  \bigg) \,\Big(  u\cdot \Rphi_{T-s}^{\lambda}(y) \, -\, u\cdot \Rphi_{T-t}^{\lambda}(X_t) \Big)^2\,\bigg|_{(s,y)=(t,X_t)  }\nonumber 
     \\
    \,=\,&\,   \big(\check{\epsilon}^{\lambda}_{T-t}(X_t) \big)^2\, \big\| P_{X_t} u\big\|^2_2    \,+\,  \big(\hat{\epsilon}^{\lambda}_{T-t}(X_t) \big)^2 \,\big\|(I-P_{X_t}) u\big\|_2^2 \,=\,\big\|\sigma_{T-t}^{\lambda}(X_t)u\big\|_2^2\,, 
\end{align*}
where the second and third equalities apply~(\ref{KolmogorovForJBack}) and~(\ref{ChainRULE}), respectively.
\end{proof}

Note that the $2\times 2$ matrix $\sigma_{t}^{\lambda}(x)$  is invertible when $x\neq 0$ with $ \big(\sigma_{t}^{\lambda}\,(x)\big)^{-1}=\frac{1}{\check{\epsilon}_{T}^{\lambda}(x)}P_x +\frac{1}{\hat{\epsilon}_{T}^{\lambda}(x)}\big(I-P_x\big)$ for $\check{\epsilon}_{T}^{\lambda}(x),\hat{\epsilon}_{T}^{\lambda}(x)>0$ defined  in~(\ref{SigmaForm}).
\begin{corollary}\label{CorollaryMartI}  Fix some $T,\lambda >0$ and a Borel probability measure $\mu$ on $  \R^2$.   Let the $\R^2$-valued process  $\{ Y_{t}^{T,\lambda} \}_{t\in [0,\infty)}$ be defined as in Lemma~\ref{LemmaMartI}.  The process $\{  W^{T,\lambda }_t\}_{t\in [0,\infty)}$ defined through the It\^o integral  
$$W^{T,\lambda }_t\,=\,\int_0^t \, 1_{X_s\neq 0}\, \big(\sigma_{T-s}^{\lambda}(X_s)\big)^{-1} \,dY_{s}^{T,\lambda}      $$
is a two-dimensional standard Brownian motion  with respect to $\{\mathscr{F}_t^{T,\mu}  \}_{t\in [0,\infty)}$ under  $\mathbf{P}^{T,\lambda}_{\mu}$, and we have $dY_{t}^{T,\lambda} =\sigma_{T-t}^{\lambda}(X_t)\,dX_t$.
\end{corollary}

\begin{proof}  By L\'evy's theorem,   to prove that the continuous local martingale $W^{T,\lambda}$ is a standard Brownian motion, it suffices to show that for any $v\in \R^2$  the quadratic variation of the process $v\cdot W^{T,\lambda}$ has the form    $\big\langle v\cdot W^{T,\lambda} \big\rangle_t=t\|v\|_2^2$ for all $t\geq 0$. 
Observe that since the diffusion $Y^{T,\lambda}$ has dispersion matrix $\sigma_{T-t}^{\lambda}(X_t) $ at time $t\geq 0$, we have the equality below for $A_{t}^{\lambda}(x):=\big(\sigma_{t}^{\lambda}(x)\big)^{\dagger}\sigma_{t}^{\lambda}(x)=\big(\sigma_{t}^{\lambda}(x)\big)^2$.
\begin{align*}
    \big\langle v\cdot W^{T,\lambda}\big\rangle_t\,=\,\int_0^t\, 1_{X_s\neq 0} \, \Big\langle \big(\sigma_{T-t}^{\lambda}(X_t)  \big)^{-1}\, v \,, \,\,A_{T-t}^{\lambda}(X_t)\, \big(\sigma_{T-t}^{\lambda}(X_t)  \big)^{-1}\, v  \Big\rangle    \, ds   \,=\,\|v\|_2^2\int_0^t \,1_{X_s\neq 0}   \, ds 
\end{align*}
The above is $\mathbf{P}^{T,\lambda}_{\mu}$-almost surely equal to $t\,\|v\|_2^2 $ provided that  $\{s\in [0,\infty)\, :\, X_s = 0   \}   $ has Lebesgue measure zero with probability one under $\mathbf{P}^{T,\lambda}_{\mu}$.  Furthermore, since $\mathlarger{d}^{T,\lambda}_{s,t}(x,y)$ is the transition density function of the coordinate process $X$ under $\mathbf{P}^{T,\lambda}_{x}$, switching the order of integrations yields that
\begin{align*}
    \mathbf{E}^{T,\lambda}_{\mu}\Big[\, \textup{meas}\Big(\big\{s\in [0,\infty)\, :\, | X_s| \,=\, 0   \big\}   \Big) \,  \Big]\,=\,&\,\mathbf{E}^{T,\lambda}_{\mu}\bigg[\, \int_0^{\infty}\, 1_{| X_s|  =  0  }\,ds     \, \bigg]
    \\ \,=\,&\,\int_{\R^2}\,\int_0^{\infty}\,\int_{\R^2}\,\mathlarger{d}^{T,\lambda}_{0,s}(x,y)\,1_{| y| = 0}\,dy\,ds\,\mu(dx)\,=\,0\,.
\end{align*}   
Thus, the Lebesgue measure of $\{s\in [0,\infty)\, :\, | X_s| = 0   \} $ is 
$\mathbf{P}^{T,\lambda}_{\mu}$-almost surely equal to $0$. Finally, since $dW^{T,\lambda }_t= 1_{X_t\neq 0} \big(\sigma_{T-t}^{\lambda}(X_t)\big)^{-1} dY_{t}^{T,\lambda}$, we have that $\sigma_{T-t}^{\lambda}(X_t)dW^{T,\lambda }_t=1_{X_t\neq 0}\, dY_{t}^{T,\lambda}=dY_{t}^{T,\lambda}$.
\end{proof}

\begin{proof}[Proof of Proposition~\ref{PropStochPre}] Let $W^{T,\lambda}$ be the standard two-dimensional Brownian motion from Corollary~\ref{CorollaryMartI}.   
For $t\in \R$ and $\lambda>0$, let $\Rvarphi_{t}^{\lambda}:\R^2\rightarrow \R^2$ denote the function inverse of $\Rphi_{t}^{\lambda}$. Then $X_t=\Rvarphi_{T-t}^{\lambda}\big(Y^{T,\lambda}_t\big)$, and so It\^o's rule gives us the second equality below.
\begin{align}\label{Tbit}
   dX_t\,=\,&\,d\,\Big(\Rvarphi_{T-t}^{\lambda}\big( Y^{T,\lambda}_t\big)\Big)\nonumber  \\ \,=\,&\, \underbrace{\Big(\frac{\partial}{\partial t}\,\Rvarphi_{T-t}^{\lambda}\Big)\big( Y^{T,\lambda}_t\big)}_{\textup{(I)}}\,dt\,+\,\underbrace{\big(D_1 \,\Rvarphi_{T-t}^{\lambda}\big)\big( Y^{T,\lambda}_t\big)\, dY^{T,\lambda}_t}_{\textup{(II)}} \, + \, \frac{1}{2} \underbrace{\big(D_2 \, \Rvarphi_{T-t}^{\lambda}\big)\big( Y^{T,\lambda}_t\big)\, \big(dY^{T,\lambda}_t\big)^{\otimes^2}}_{\textup{(III)}} \nonumber \\
   \,=\,&\,d  W^{T,\lambda}_t \,+\,b^{\lambda}_{T-t}(X_t)\,  dt
    \end{align}
In the above, $D_n h(y)$  denotes the $n^{th}$-order tensor of derivatives of a smooth function $h:\R^2\rightarrow \R$, that is the element in $(\R^2)^{\otimes^n}$ satisfying $D_n h(y)\cdot v^{\otimes^n}= \frac{d^n}{da^n} h(y+a v)\big|_{a=0} $ for any $v\in \R^2$.  Moreover, $D_n$ acts on a vector-valued function component-wise.   In the analysis below, we will verify the third equality in~(\ref{Tbit})  by finding alternative expressions for (I)--(III).    Since $\Rvarphi_{t}^{\lambda}$ is the function inverse of $\Rphi_{t}^{\lambda}$, we have $\Rvarphi_{t}^{\lambda}\big(\Rphi_{t}^{\lambda}(z)\big)=z$ for all $z\in \R^2$, and so the chain rule implies the following equalities:
    \begin{align}\label{CR1} 
    \big(D_1 \,\Rvarphi_{t}^{\lambda}\big)\big(\Rphi_{t}^{\lambda}(z) \big)\,=\,&\big(\sigma_{t}^{\lambda}(z)\big)^{-1}  \\
        \Big(\frac{\partial}{\partial t}\,\Rvarphi_{t}^{\lambda}\Big)\big(\Rphi_{t}^{\lambda}(z)\big) \,=\,&\,-\big(\sigma_{t}^{\lambda}(z)\big)^{-1} \,\Big(\frac{\partial}{\partial t}\,\Rphi_{t}^{\lambda}\Big)(z) \label{CR2} \\
     \big(D_2 \,\Rvarphi_{t}^{\lambda}\big)\big( \Rphi_{t}^{\lambda}(z)\big) \,=\, &\, -\big(\sigma_{t}^{\lambda}(z)\big)^{-1}\,\big(D_2\, \Rphi_{t}^{\lambda}\big)( z)  \,\big(\sigma_{t}^{\lambda}(z)\big)^{-\otimes^2}\,,\label{CR3}
    \end{align}
   where we have used that $\sigma_{t}^{\lambda}(z):=D_1 \Rphi_{t}^{\lambda}( z  )  $ and the notation $ A^{-\otimes^2}:=(A^{-1})^{\otimes^2}$ for an invertible matrix $A$. \vspace{.2cm}
   
   \noindent \textit{Term (I):} Applying~(\ref{CR2}) and then the chain rule yields the first two equalities below. 
    \begin{align}\label{I1} 
    \Big(\frac{\partial}{\partial t}\,\Rvarphi_{T-t}^{\lambda}\Big)\big( Y^{T,\lambda}_t\big)\,=\,&\, -\big(\sigma_{T-t}^{\lambda}( X_t  )\big)^{-1} \,\Big(\frac{\partial}{\partial t}\,\Rphi_{T-t}^{\lambda}\Big)(X_t)\nonumber  \\ \,=\,&\,\big(\sigma_{T-t}^{\lambda}( X_t  )\big)^{-1}  \,\frac{ \big(\frac{ \partial }{\partial t }H_{T-t}^{\lambda}\big)(X_t)   }{ \big(1+H_{T-t}^{\lambda}(X_t) \big)^2 }\,  X_t  \,=\,- \frac{1}{2\, \check{\epsilon}_{T-t}^{\lambda}(X_t) }\, \frac{ \Delta_x  H_{T-t}^{\lambda}(X_t)   }{ \big(1+H_{T-t}^{\lambda}(X_t) \big)^2 } \, X_t  
    \end{align}
   The third equality uses that $\frac{\partial}{\partial t}H_{t}^{\lambda}(x)=\frac{1}{2} \Delta_x H_{t}^{\lambda}(x) $ for nonzero $x \in \R^2$ and that $x$ is an eigenvector for $\big(\sigma_{t}^{\lambda}(x)\big)^{-1}$ with eigenvalue $ \frac{1}{ \check{\epsilon}_{t}^{\lambda}(x) } $. \vspace{.3cm}

   \noindent \textit{Term (II):} Here, we simply apply~(\ref{CR1}) and that $dY^{T,\lambda}_t =\sigma_{T-t}^{\lambda}( X_t  ) dW^{T,\lambda}_t$ to write
    \begin{align}\label{I2}
    \big(D_1 \,\Rvarphi_{T-t}^{\lambda}\big)\big( Y^{T,\lambda}_t\big)\, dY^{T,\lambda}_t \,=\, \big(\sigma_{T-t}^{\lambda}( X_t  )\big)^{-1}\,\sigma_{T-t}^{\lambda}( X_t  )\, dW^{T,\lambda}_t\,=\,dW^{T,\lambda}_t \,. \vspace{.1cm}
    \end{align}

    \noindent \textit{Term (III):} As a preliminary, we observe that by calculus
    \begin{align}\label{DeltaK}
    \Delta_x \,\bigg( \frac{x}{1 +H_{t}^{\lambda}(x)    }  \bigg)\,=\,&\,-\frac{  x }{\big(1 +H_{t}^{\lambda}(x)\big)^2  }\,\Bigg( \Delta_x H_{t}^{\lambda}(x)\,-\, \frac{  2 }{1 +H_{t}^{\lambda}(x) }\,\big| \nabla_x H_{t}^{\lambda}(x)  \big|^2\,-\,\frac{2}{|x|}\,\big|   \nabla_x H_{t}^{\lambda}(x)\big| \Bigg) \nonumber \\ \,=\,&\,-\frac{ \Delta_x H_{t}^{\lambda}(x)  }{\big(1 +H_{t}^{\lambda}(x)\big)^2  }\,x \,+\, 2\, \frac{ 1+ H_{t}^{\lambda}(x)\,+\,|x|\,\big| \nabla_x H_{t}^{\lambda}(x)  \big| }{ \big(1 +H_{t}^{\lambda}(x)\big)^2 }\,\frac{  \big| \nabla_x H_{t}^{\lambda}(x) \big|  }{1 +H_{t}^{\lambda}(x) }\,\frac{ x}{|x| } \nonumber  \\ \,=\,&\,-\frac{ \Delta_x H_{t}^{\lambda}(x)  }{\big(1 +H_{t}^{\lambda}(x)\big)^2  }\,x \,-\, 2\,\check{\epsilon}_{t}^{\lambda}(x)\,b_t^{\lambda}(x)\,,
    \end{align}
where we have used that $b_t^{\lambda}(x):=\frac{ \nabla_x H_{t}^{\lambda}(x)  }{1+H_{t}^{\lambda}(x)   }=-\frac{  | \nabla_x H_{t}^{\lambda}(x) |  }{1 +H_{t}^{\lambda}(x) } \frac{ x}{|x| }$. Using~(\ref{CR3}) and that $dY^{T,\lambda}_t =\sigma_{T-t}^{\lambda}( X_t  )d  W^{T,\lambda}_t$ gives us the following:
    \begin{align}\label{I3}
\big(D_2 \,\Rvarphi_{T-t}^{\lambda}\big)\big( Y^{T,\lambda}_t\big)\, \big(dY^{T,\lambda}_t\big)^{\otimes^2} \,=\,   & \,  -\big(\sigma_{T-t}^{\lambda}(X_t) \big)^{-1}\,\big(D_2\, \Rphi_{T-t}^{\lambda}\big)( X_t ) \, \big(\sigma_{T-t}^{\lambda}(X_t) \big)^{-\otimes^2} \, \Big(\sigma_{T-t}^{\lambda}(X_t)\,d  W^{T,\lambda}_t\Big)^{\otimes^2}\nonumber \\
\,=\,   & \,  -\big(\sigma_{T-t}^{\lambda}(X_t)\big)^{-1}\,\big(D_2 \,\Rphi_{T-t}^{\lambda}\big)( X_t )   \, \big(d  W^{T,\lambda}_t\big)^{\otimes^2}\nonumber \\
\,=\,   & \,  - \big(\sigma_{T-t}^{\lambda}(X_t)\big)^{-1} \,\Delta_x\, \frac{x}{1 +H_{t}^{\lambda}(x)    }  \bigg|_{x=X_t} \, dt\nonumber \\
\,\stackrel{(\ref{DeltaK})}{=}\,&\,\big(\sigma_{T-t}^{\lambda}(X_t)\big)^{-1} \,\Bigg(\frac{ \Delta_x H_{T-t}^{\lambda}(X_t)  }{\big(1 +H_{T-t}^{\lambda}(X_t)\big)^2  }\,X_t \,+\, 2\,\check{\epsilon}_{T-t}^{\lambda}(X_t)\,b_{T-t}^{\lambda}(X_t) \Bigg)  \,dt \nonumber
\\
\,=\,& \,\frac{1}{\check{\epsilon}_{T-t}^{\lambda}(X_t)  } \, \frac{ \Delta_x H_{T-t}^{\lambda}(X_t)  }{\big(1 +H_{T-t}^{\lambda}(X_t)\big)^2  }\,X_t \,dt\,+\, 2\,b_{T-t}^{\lambda}(X_t)    \,dt\,. 
    \end{align}
  The third equality above uses It\^o's rule, meaning that $ \big(e_i\cdot d  W^{T,\lambda}_t\big)\big(e_j\cdot d  W^{T,\lambda}_t\big)=\delta_{i,j}dt $ for orthonormal vectors $e_1,e_2\in \R^2$.    The last equality uses that the vector $b_{t}^{\lambda}(x) $ is a multiple of $x$ and again that $x$  is an eigenvector for $\big(\sigma_{t}^{\lambda}(x)\big)^{-1}$ with eigenvalue  $\frac{1}{\check{\epsilon}_{t}^{\lambda}(x)  }  $.  \vspace{.2cm}
    
\noindent \textit{Returning to~(\ref{Tbit}):}  Substituting the  expressions~(\ref{I1}), (\ref{I2}), \& (\ref{I3}) in for the terms (I)--(III)     yields the third equality in~(\ref{Tbit}), which shows that the coordinate process $X$ satisfies the desired SDE with respect to the Brownian motion $W^{T,\lambda}$.
\end{proof}

  \section{Ramanujan's identity and asymptotics for the Volterra function} \label{AppendixSectNu}
 Recall that  $N(x):=\int_0^{\infty}\frac{ e^{-xs}}{\pi^2+\log^2 s }\frac{1}{s}ds $. 
We will provide another  derivation of the identity $e^x-\nu(x)=N(x)$ using Euler's reflection formula.
\begin{proof}[Proof of (iii) of Proposition~\ref{PropEFunctForm}] The derivative of $\nu$ can be written in the form 
\begin{align}\label{ePrime}
    \nu'(x)\,=\,\int_0^{\infty}\, \frac{s x^{s-1}}{\Gamma(s+1) }\,ds\,=\,\int_0^{1} \,\frac{x^{s-1}}{\Gamma(s) }\,ds\,+\,\int_1^{\infty}\, \frac{x^{s-1}}{\Gamma(s) }\,ds \,=\,\int_0^{1} \,\frac{x^{s-1}}{\Gamma(s) }\,ds\,+\,\nu(x)\,, 
\end{align}
where we have used that $\Gamma(s+1)=s\Gamma(s)$, and the third equality follows from the change of integration variable $ s-1\mapsto s$ and the definition of $\nu$.  Thus, $\nu'(x)-\nu(x)= \int_0^{1} \frac{x^{s-1}}{\Gamma(s) }ds $, and we can apply Euler's reflection formula in the form $\frac{1}{\Gamma(s)}=\frac{1}{\pi} \Gamma(1-s) \sin\big(\pi (1-s)\big) $ for $s\notin \Z$ to get that 
\begin{align} \label{GnunelPre}\int_0^1\, \frac{x^{s-1}}{\Gamma(s)}\,ds\,=\,\frac{1}{\pi}\,\int_0^1 \,\Gamma(1-s)\, \sin\big(\pi (1-s)\big) \,  x^{s-1}\,ds\,=\,\frac{1}{\pi}\,\int_0^1\, \Gamma(r) \,\sin(\pi r)\,x^{-r}\,dr \,.   \end{align}
  Since $\Gamma(r):=\int_0^{\infty}e^{-t}t^{r-1}dt  $, we can swap the order of integration and use calculus to express~(\ref{GnunelPre}) as follows:
\begin{align}\label{Gnunel} 
\frac{1}{\pi}\,\int_0^{\infty}\,e^{-t}\,\int_0^{1}t^{r-1}\,  \sin(\pi r)\,x^{-r}\,dr\, dt 
\,=\,&\,\frac{1}{\pi}\,\int_0^{\infty}\,e^{-t}\,\frac{1}{2 it}\,\int_0^{1}\,e^{r\log\frac{t}{x}}\,\big( e^{i\pi r}-e^{-i\pi r} \big)\, dr\, dt \nonumber \\
\,=\,&\,\frac{1}{\pi}\,\int_0^{\infty}\,e^{-t}\,\frac{1}{2it}\,\Bigg[\,\frac{-\big(\frac{t}{x}+1 \big)  }{i\pi+\log \frac{t}{x}  }\,-\,\frac{-\big(\frac{t}{x}+1 \big)  }{-i\pi+\log \frac{t}{x}   }\, \Bigg]\,dt \nonumber \\ \,=\,&\,\int_0^{\infty}\,e^{-t}\,\frac{\frac{1}{t}+\frac{1}{x}  }{\pi^2+\log^2 \frac{t}{x}  }\,dt \,.
\end{align}
Through the change of integration variable to  $s= \frac{t}{x}$, we can write~(\ref{Gnunel}) as
\begin{align*}
 \int_{0}^{\infty}\,e^{-xs}\,\frac{1+s }{\pi^2+\log^2 s   }\,\frac{1}{s}\,ds
 \,=\, \int_{0}^{\infty}\,e^{-x s}\,\frac{1 }{\pi^2+ \log^2 s  }\,\frac{1}{s}\,ds\,+\, \int_{0}^{\infty}\,e^{-xs}\,\frac{1 }{\pi^2+\log^2 s   }\,ds \,=\,N(x)\,-\,N'(x) \,.
\end{align*}
  We have shown that $ N(x)-N'(x)= \nu'(x)  -\nu(x)$, or equivalently,  $\nu(x)+N(x)= \frac{d}{dx}\big( \nu(x)+N(x)\big)$.  Since $\nu(0)=0$ and $N(0)=1$, we conclude that $\nu(x)+N(x) =e^{x}$, which is the claim. \end{proof}

 Next, we derive the second-order  small and large $x\in (0,\infty)$ asymptotics of $\nu(x)$ using Ramanujan's identity and the first and second moments of the standard Gumbel distribution.

\begin{proof}[Proof of Lemma~\ref{LemmaEFunAsy}]  Part (i): Starting with Ramanujan's identity, we can apply the change of integration variable $s= -\log a-\log x$ in  the integral $N(x)=\int_0^{\infty}\frac{ e^{-xa}}{\pi^2+\log^2 a }\frac{1}{a}da $ followed by
integration by parts to get
\begin{align}\label{DiffEes}
e^x\,-\,\nu(x)\,=\,N(x)\,=\,   1\,-\, \int_{-\infty}^{\infty}\,e^{-s}\,e^{-  e^{-s} }\,\bigg(\frac{1}{2}+ \frac{1}{\pi }\,\tan^{-1}\Big(\frac{s+\log x}{\pi} \Big)    \bigg)\,ds  \,.
\end{align}
 Since the asymptotics $\frac{\pi}{2}+ \tan^{-1}(r)= \frac{1}{-r}-\frac{1}{3(-r)^3}+\mathit{O}\big( \frac{1}{(-r)^5}\big) $ holds as $r\searrow -\infty$,  there is a constant $C>0$ such that for all $x\in (0,\frac{1}{2}]$
$$\Bigg|\frac{1}{2}\,+\,\frac{1}{\pi}\,\tan^{-1}\Big(\frac{s+\log x}{\pi}\Big)\,-\,\Bigg(\frac{1}{\log \frac{1}{x}}+\frac{ s }{ \log^2\frac{1}{x} }\,+\,\frac{ s^2-\frac{\pi^2}{3} }{ \log^3\frac{1}{x} }\Bigg)\Bigg|\, \leq \,C\,\textup{min}\bigg(1,\,\Big(\frac{  1+|s| }{ \log^4\frac{1}{x}  }\Big)^4 \bigg) \,. $$
Thus for small $x$ the integral in~(\ref{DiffEes}) is equal to
\begin{align*}
   \,&\,\frac{1}{\log \frac{1}{x}  }\,\underbrace{\int_{-\infty}^{\infty}\,e^{-s}\,e^{-  e^{-s} }\,ds}_{=\,1} \,+\,\frac{1}{\log^2 \frac{1}{x}  }\,\underbrace{\int_{-\infty}^{\infty}\,e^{-s}\,e^{-  e^{-s} }\,s\,ds}_{=\,\gamma_{\mathsmaller{\textup{EM}}}}\,+\,\frac{1}{\log^3 \frac{1}{x}  }\,\underbrace{\int_{-\infty}^{\infty}\,e^{-s}\,e^{-  e^{-s} }\,\Big(s^2-\frac{\pi^2}{3}\Big)\,ds}_{=\,\gamma_{\mathsmaller{\textup{EM}}}^2-\frac{ \pi^2 }{6}}\,+\,\mathit{O}\bigg(\frac{1}{\log^4 \frac{1}{x}  }\bigg)\,,
\end{align*}
where the  integrals above can be evaluated by recalling that the standard Gumbel distribution has mean $\gamma_{\mathsmaller{\textup{EM}}}$ and second moment $\frac{\pi^2}{6}+\gamma_{\mathsmaller{\textup{EM}}}^2  $. Since $e^{x}-1=\mathit{O}(x)$ for small $x$, the result follows. \vspace{.3cm}

\noindent Part (ii):  We can apply a similar argument as in (i) after writing~(\ref{DiffEes}) in the form
\begin{align*}
   e^x\,-\,\nu(x)\,=\,N(x) \,=\,
    \int_{-\infty}^{\infty}\,e^{-s}\,e^{-  e^{-s} }\,\bigg(\frac{1}{2}- \frac{1}{\pi }\,\tan^{-1}\Big(\frac{s+\log x}{\pi} \Big)    \bigg)\,ds\,.
\end{align*}

\noindent Parts (iii) \& (iv):  The asymptotics for $\nu'$ can be shown using~(\ref{ePrime}) and parts (i)--(ii). Furthermore, this analysis can be extended to $\nu''$ by differentiating~(\ref{ePrime}) to get
$\nu''(x)=\int_0^1 \frac{ s(s-1)x^{s-2} }{\Gamma(s+1)} ds +\nu'(x)  $.
\end{proof}

  \section{The Volterra-Poisson distribution }\label{AppendixSectPoisson}
Given $\lambda>0$ define the probability density $\mathfrak{D}_{\lambda}:\R\rightarrow [0,\infty)$ by
$     \mathfrak{D}_{\lambda}(v)\,=\,\frac{1}{\nu(\lambda)}\frac{  \lambda^v }{\Gamma(v+1)  } 1_{[0,\infty)}(v)   $,
where recall that $\nu(\lambda):=\int_0^{\infty}\frac{\lambda^s}{\Gamma(s+1)}ds$ is  discussed in Section~\ref{SubsecFractExp}.  Note that the density $\mathfrak{D}_{\lambda}(v)$ has the form of a continuous analog of the probability mass function $d_{\lambda}(n)=\frac{1}{e^{\lambda}}\frac{\lambda^n}{n!}$ for a parameter-$\lambda$ Poisson distribution, where the normalizing factor is the reciprocal of the Volterra function rather than of the natural exponential. For this reason, we refer to the Borel measure on $\R$ with Lebesgue density $\mathfrak{D}_{\lambda}$ as the  \textit{parameter-$\lambda$  Volterra-Poisson distribution}. The density function $\mathfrak{D}_{\lambda}$ with parameter $\lambda \mapsto \lambda (T-t)$ appears implicitly in (ii) of Proposition~\ref{PropLocalTimeProp} as the conditional density of the local time $\mathbf{L}_T$ given that the process $X$ first arrives at the origin at time $t\in [0,T]$. The proposition below collects a few casual observations regarding this distribution, where (ii)--(iv) can be proven using (i).
\begin{proposition}\label{PropVolterraPoisson} Let $X_{\lambda}$ be a random variable with parameter-$\lambda$  Volterra-Poisson distribution.
\begin{enumerate}[(i)]
    \item The moment generating function of $X_{\lambda}$ has the form
$  \mathbb{E}[  e^{a X_{\lambda} } ]\,=\,\frac{\nu(\lambda e^a) }{ \nu(\lambda)  } $ for all $a\in \R$.

    \item For any $n\in \mathbb{N}$, the $n^{th}$ cumulant of $X_{\lambda}$  has the form 
    $$\frac{\partial^n}{\partial a^n} \log\Big(\,\mathbb{E}\big[ \, e^{a X_{\lambda} } \,\big] \,\Big)\,\Big|_{a=0}\,=\,\bigg(\lambda\,\frac{\partial}{\partial \lambda  }\bigg)^{n}\log \nu(\lambda ) \,. $$

    \item As $\lambda \searrow 0$ the random variable $\log\big(\frac{1}{\lambda} \big) X_{\lambda} $ converges  in distribution  to a mean one exponential. 

    \item As $\lambda \nearrow \infty $ the random variable $\frac{ X_{\lambda}-\lambda  }{\sqrt{\lambda}  } $ converges  in distribution  to a standard normal. 
    
\end{enumerate}

\end{proposition}

\end{appendix}

\section*{Acknowledgement} We thank Micah Milinovich for helping us find the special function $\nu$ in the third volume of the Bateman Manuscript Project.


\begin{thebibliography}{99}


\bibitem{Abramowitz} M. Abramowitz, I.E. Stegun: \emph{Handbook of Mathematical Functions}, U.S. National Bureau of Standards,  Washington DC, 1965.


\bibitem{Albeverio} S. Albeverio, Z. Brze\'zniak, L. D\c{a}browski: \emph{Fundamental solution of the heat and Schr\"odinger equations with point interaction}, J. Funct. Anal. \textbf{120}, 220-254 (1995).


\bibitem{AGHH} S. Albeverio, G. Gesztesy, R. H{\o}egh-Krohn, H. Holden: \emph{Solvable Models in Quantum Mechanics}, 2nd Edition, AMS Chelsea Publishing, 1988.



\bibitem{Apelblat0}  A. Apelblat: \emph{Volterra Functions}, Nova Science Publ. Inc., New York, 2007.


\bibitem{Apelblat}  A. Apelblat: \emph{Integral Transforms and Volterra Functions}, Nova Science Publ. Inc., New York, 2010.




\bibitem{BC} L. Bertini and N. Cancrini: \emph{The two-dimensional stochastic heat equation: renormalizing a multiplicative
noise}, J. Phys. A: Math. Gen. \textbf{31}, 615-622, (1998).




\bibitem{CSZ2} F. Caravenna, R. Sun, N. Zygouras: \emph{The Dickman subordinator, renewal theorems, and disordered
systems},  Electron J. Probab. \textbf{24}, 40pp (2019).





\bibitem{CSZ4}  F. Caravenna, R. Sun, N. Zygouras: \emph{On the moments of the (2+1)-dimensional directed polymer and stochastic heat equation in the critical window}, Commun. Math. Phys. \textbf{372}, 385-440 (2019).


\bibitem{CSZ5}  F. Caravenna, R. Sun, N. Zygouras: \emph{The critical 2d  stochastic heat flow}, Invent. Math. \textbf{233}, 325–460 (2023).


\bibitem{CSZ6}  F. Caravenna, R. Sun, N. Zygouras: \emph{The critical  2d stochastic heat flow is not a Gaussian multiplicative chaos}, 	\textup{arXiv:2206.08766} (2022).

\bibitem{CCF}  F. Carlone, M. Correggi, R. Figari:\emph{Two-dimensional time-dependent point interactions}, Functional Analysis and Operator Theory for Quantum Physics, J. Dittrich, H. Kovarik, and A. Laptev, eds., EMS Publishing House, Zurick, 189--211 (2017).


\bibitem{Chen} Y.-T. Chen: \emph{The critical 2d delta-Bose gas as mixed-order asymptotics of planar Brownian motion},  \textup{arXiv:2105.05154} (2021). 





\bibitem{Clark3} J.T. Clark: \emph{Continuum models of directed polymers on disordered diamond fractals in the critical case},  Ann. Appl. Probab. \textbf{32}, 4186-4250  (2022).


\bibitem{Clark4} J.T. Clark: \emph{The conditional Gaussian multiplicative chaos   structure underlying a critical  continuum random  polymer model on a diamond fractal}, to appear in Annales de l'Institut Henri Poincar\'e, Probabilit\'es and Statistiques.

    




\bibitem{Dell} G.F. Dell'Antonio, R. Figari, A. Teta:  \emph{Hamiltonians for systems of $N$ particles interacting through point interactions}, Ann. Inst. H. Poincar\'e  Phys. Th\'eor. \textbf{60},  253-290 (1994).





\bibitem{Dimock} J. Dimock, S. Rajeev: \emph{Multi-particle Schr\"odinger operators with point interactions in the plane}, J. Phys. A Math. Gen.  \textbf{37}(39), 9157-9173 (2004).





\bibitem{Bateman} A. Erd\'elyi, W. Magnus, F. Oberhettinger, F.G. Tricomi: \emph{Higher Transcendental Functions}, Vol. III,  McGraw-Hill, New York, (1981).




\bibitem{Garrappa} R. Garrappa, F. Mainardi: \emph{On Volterra functions and Ramanujan integrals}, Analysis, Special Issue devoted to the memory of Prof. A.A. Kilbas, \textbf{36} 89-105 (2016).



\bibitem{GQT} Y. Gu, J. Quastel, L. Tsai: \textit{Moments of the 2d SHE at criticality}, Prob. Math. Phys. \textbf{2}, 179-219  (2021).



\bibitem{Hardy} G.H. Hardy: \textit{Ramanujan: Twelve Lectures on Subjects Suggested by his Life and Work}, Cambridge University Press,  1940.









\bibitem{Karatzas} I. Karatzas, S.E. Shreve: \emph{Brownian Motion and Stochastic Calculus}, Graduate Texts in Mathematics \textbf{113}, 2nd edition, Springer, New York, 1998.


\bibitem{Kashara3} Y. Kashara, S. Kotani: \emph{On limit processes for a class of additive functional of recurrent diffusion processes}, Z. Wahrsch. Verw. Gebiete \textbf{49}, 133-153 (1979).




\bibitem{Llewellyn}   S.G.  Llewellyn Smith: \emph{The asymptotic behavior of Ramanujan's integral and its application to two-dimensional diffusion-like equations}, Euro. J.  Appl. Math. \textbf{11}, 13-28 (2000).






\bibitem{SKM} S.G. Samko, A.A. Kilbas, O.I. Marichev: \emph{Fractional Integrals and Derivatives: Theory and Applications}, Taylor and Francis Books 2002. 


\bibitem{Titch}   E.C. Titchmarsh: \emph{Introduction to the Theory of Fourier Integrals}, 2nd Edition, Clarendon Press, London, 1948.



\bibitem{Volterra0} V. Volterra: \emph{Teoria delle potenze, dei logaritmi e delle funzioni di decomposizione}, R. Acc. Lincei. Memorie, ser 5, \textbf{11}, 167-249, 1916.




\bibitem{Volterra} V. Volterra, J. P\'er\`es: \emph{Le\c{c}ons sur la Composition et les Fonctions Permutables},  Gauthier-Villars, Paris, 1924.



\bibitem{Watson} G.N. Watson: \emph{A treatise on the theory of Bessel functions}, 2nd edition,  Cambridge University Press, Cambridge, 1944.




\bibitem{Weber} H.J. Weber, F. Harris, G.B. Arfken: \emph{Essential Mathematical Methods for Physicists}, 1st edition,  Academic Press, 2004.




\end{thebibliography}
\end{document}